\documentclass[letterpaper, 11pt, thm-restate]{article}

\usepackage[margin=1in]{geometry}
\usepackage{footmisc}
\usepackage{enumitem}
\usepackage[title]{appendix}
\usepackage{xcolor}
\usepackage{array}
\usepackage{amsthm}
\usepackage{amsmath}
\numberwithin{equation}{section}

\usepackage{amssymb}
\usepackage{amsfonts}
\usepackage{graphicx}
\usepackage{thm-restate}
\usepackage[colorlinks = true]{hyperref}
\hypersetup{
    linkcolor=black,
    citecolor=blue,
    urlcolor=blue}

\usepackage{mathtools}
\usepackage{mathrsfs}

\usepackage{tikz}
\usetikzlibrary{positioning,calc,arrows,automata,shapes,shapes.geometric,decorations.pathmorphing}
\usetikzlibrary{shapes}

\usepackage{titlesec}
\titleformat{\subsection}[runin]
       {\normalfont\bfseries}
       {\thesubsection}
       {0.5em}
       {}
       [.]

\newcommand{\Z}{\mathbb{Z}}
\newcommand{\N}{\mathbb{N}}
\newcommand{\Q}{\mathbb{Q}}

\newcommand{\K}{\mathbb{K}}
\newcommand{\F}{\mathbb{F}}
\newcommand{\A}{\mathbb{A}}

\newcommand{\Zpe}{\mathbb{Z}_{/p^e}}
\newcommand{\Zp}{\mathbb{Z}_{/p}}
\newcommand{\Zpt}{\mathbb{Z}_{/p^2}}
\newcommand{\Zpi}{\mathbb{Z}_{/p^i}}
\newcommand{\ZT}{\mathbb{Z}_{/T}}

\newcommand{\GL}{\mathsf{GL}}

\newcommand{\diag}{\operatorname{diag}}
\newcommand{\Ann}{\operatorname{Ann}}
\newcommand{\Aut}{\operatorname{Aut}}
\newcommand{\End}{\operatorname{End}}

\newcommand{\rad}{\operatorname{rad}}
\newcommand{\mG}{\mathcal{G}}

\newcommand{\mF}{\mathcal{F}}

\newcommand{\mM}{\mathcal{M}}
\newcommand{\mN}{\mathcal{N}}

\newcommand{\mV}{\mathcal{V}}
\newcommand{\mT}{\mathcal{T}}

\newcommand{\M}{\mathsf{M}}

\newcommand{\mI}{\mathcal{I}}

\newcommand{\mJ}{\mathcal{J}}

\newcommand{\tS}{\widetilde{S}}
\newcommand{\mS}{\mathcal{S}}
\newcommand{\mH}{\mathcal{H}}
\newcommand{\mA}{\mathcal{A}}
\newcommand{\mZ}{\mathfrak{Z}}

\newcommand{\mB}{\mathcal{B}}

\newcommand{\mR}{\mathcal{R}}

\newcommand{\bv}{\boldsymbol{v}}
\newcommand{\ba}{\boldsymbol{a}}
\newcommand{\bb}{\boldsymbol{b}}
\newcommand{\bc}{\boldsymbol{c}}
\newcommand{\bd}{\boldsymbol{d}}

\newcommand{\be}{\boldsymbol{e}}

\newcommand{\bx}{\boldsymbol{x}}
\newcommand{\br}{\boldsymbol{r}}
\newcommand{\bw}{\boldsymbol{w}}
\newcommand{\bepsilon}{\boldsymbol{\epsilon}}
\newcommand{\tc}{\widetilde{c}}

\newcommand{\tR}{\widetilde{R}}
\newcommand{\tmA}{\widetilde{\mathcal{A}}}
\newcommand{\tmV}{\widetilde{\mathcal{V}}}
\newcommand{\mmU}{\mathscr{U}}
\newcommand{\mmP}{\mathscr{P}}
\newcommand{\tU}{\widetilde{\mathscr{U}}}

\newcommand{\bh}{\boldsymbol{h}}
\newcommand{\bz}{\boldsymbol{z}}

\newcommand{\bzer}{\boldsymbol{0}}

\newcommand{\gen}[1]{\left\langle {#1} \right\rangle}

\newcommand{\bg}{\boldsymbol{g}}

\newcommand{\bsig}{\boldsymbol{\sigma}}
\newcommand{\oX}{\mkern 1.5mu\overline{\mkern-1.5mu X \mkern-1.5mu}\mkern 1.5mu}
\newcommand{\oY}{\mkern 1.5mu\overline{\mkern-1.5mu Y \mkern-1.5mu}\mkern 1.5mu}

\newcommand{\ox}{\overline{x}}
\newcommand{\oA}{\overline{A}}

\newcommand{\oG}{\overline{G}}

\newcommand{\frp}{\mathfrak{p}}
\newcommand{\tfrq}{\widetilde{\mathfrak{q}}}
\newcommand{\frq}{\mathfrak{q}}
\newcommand{\frm}{\mathfrak{m}}
\newcommand{\len}{\mathsf{len}}
\newcommand{\eval}{\mathsf{eval}}

\newtheorem{thrm}{Theorem}[section]
\newtheorem*{thrm*}{Theorem}
\newtheorem{lem}[thrm]{Lemma}
\newtheorem{prop}[thrm]{Proposition}

\newtheorem{cor}[thrm]{Corollary}
\newtheorem{obs}[thrm]{Observation}
\theoremstyle{definition}
\newtheorem{defn}[thrm]{Definition}
\newtheorem*{defn*}{Definition}
\theoremstyle{definition}
\newtheorem{exmpl}[thrm]{Example}
\theoremstyle{definition}

\begin{document}
\title{S-unit equations in modules and linear-exponential Diophantine equations}
\author{Ruiwen Dong\footnote{Magdalen College, University of Oxford, United Kingdom, email: ruiwen.dong@magd.ox.ac.uk} \and Doron Shafrir\footnote{Department of Mathematics, Ben Gurion University of the Negev, Be’er Sheva, Israel}}
\date{}

\maketitle
\thispagestyle{empty}

\begin{abstract}
    Let $T$ be a positive integer, and $\mathcal{M}$ be a finitely presented module over the Laurent polynomial ring $\mathbb{Z}_{/T}[X_1^{\pm}, \ldots, X_N^{\pm}]$.
    We consider S-unit equations over $\mathcal{M}$: these are equations of the form $x_1 m_1 + \cdots + x_K m_K = m_0$, where the variables $x_1, \ldots, x_K$ range over the set of monomials (with coefficient 1) of $\mathbb{Z}_{/T}[X_1^{\pm}, \ldots, X_N^{\pm}]$.
    When $T$ is a power of a prime number $p$, we show that the solution set of an S-unit equation over $\mathcal{M}$ is effectively $p$-normal in the sense of Derksen and Masser (2015), generalizing their result on S-unit equations in fields of prime characteristic.
    When $T$ is an arbitrary positive integer, we show that deciding whether an S-unit equation over $\mathcal{M}$ admits a solution is Turing equivalent to solving a system of linear-exponential Diophantine equations, whose base contains the prime divisors of $T$.
    Combined with a recent result of Karimov, Luca, Nieuwveld, Ouaknine and Worrell (2025), this yields decidability when $T$ has at most two distinct prime divisors.
    This also shows that proving either decidability or undecidability in the case of arbitrary $T$ would entail major breakthroughs in number theory.
    
    We mention some potential applications of our results, such as deciding Submonoid Membership in wreath products of the form $\mathbb{Z}_{/p^a q^b} \wr \mathbb{Z}^d$, as well as progressing towards solving the Skolem problem in rings whose additive group is torsion.
    More connections in these directions will be explored in follow up papers.
\end{abstract}

\vspace{0.5cm}
\noindent
\textbf{Acknowledgements.} The authors would like to thank James Worrell for discussion about S-unit equations. Ruiwen Dong is supported by a Fellowship by Examination at Magdalen College.


\newpage
\setcounter{page}{1}

\section{Introduction}
S-unit equations have a rich history rooted in the study of Diophantine equations and algebraic number theory. They were first introduced in the context of units in number fields, along with the foundational work of Mahler, Siegel and Thue in transcendental number theory. 
Algorithmic solutions to S-unit equations are vital for exploring the algebraic structure of a given field, and have connections to areas such as automata theory, formal verification, cryptography and computational number theory~\cite{10.5555/1481045, adamczewski2012vanishing, 10.1007/978-3-030-80914-0_1, lipton2022skolem, benedikt2023complexity}.
See~\cite{Evertse1988} for an extended survey.

Let $\K$ be a field.
Given a finite subset $S \subseteq \K \setminus \{0\}$, denote by $\gen{S}$ the multiplicative subgroup generated by $S$.
Let $m_0, m_1, \ldots, m_K$ in $\K$, an \emph{S-unit equation} is a linear equation of the form
\begin{equation}\label{eq:Sunitintrofield}
    x_1 m_1 + \cdots + x_K m_K = m_0,
\end{equation}
where we look for solutions $x_1, \ldots, x_K \in \gen{S}$.

When $\K$ is a field of characteristic $0$, Lang~\cite{Lang1960}, generalizing earlier results by Mahler~\cite{Mahler1932}, showed that the S-unit equation has only finitely many solutions when the number of variables $K$ is $2$.
When $K \geq 3$, the \emph{subspace
theorem} can be used to prove that such an equation has only a finite number of nondegenerate solutions; that is, solutions with the property that no proper subsum vanishes~\cite{Evertse1984,PoortenSchlickewei1991}.
However, all general known results concerning more than two variables are ineffective, meaning there is no known algorithm that determines whether a solution exists.

When $\K$ is a field of characteristic $p > 0$ (for example the field $\F_p(X)$ of rational functions), a recent result by Derksen and Masser~\cite{derksen2012linear} showed that the solution set of an S-unit Equation~\eqref{eq:Sunitintrofield} can be effectively written as a \emph{$p$-normal set} (see Definition~\ref{def:pnormal}); thus it is decidable whether the solution set is empty.
A related result was also given by Adamczewski and Bell~\cite{adamczewski2012vanishing}.

Naturally, one can also consider Equation~\eqref{eq:Sunitintrofield} over any commutative ring $\A$.
Given a set $S = \{s_1, \ldots, s_N\}$ of invertible elements of $\A$, one can give $\A$ a $\Z[X_1^{\pm}, \ldots, X_N^{\pm}]$-module structure by letting each $X_i$ act as $s_i$, and consider Equation~\eqref{eq:Sunitintrofield} in the submodule of $\A$ generated by the $m_i$'s.
Therefore, a more general form of S-unit equations can be formulated as follows.
Let $\mM$ be a finitely presented module over the Laurent polynomial ring $\Z[X_1^{\pm}, \ldots, X_N^{\pm}]$, and let $m_0, m_1, \ldots, m_K \in \mM$.
An \emph{S-unit equation} over $\mM$ is the equation
    \begin{equation}\label{eq:Sunitintromod}
        X_1^{z_{11}} X_2^{z_{12}} \cdots X_N^{z_{1N}} \cdot m_1 + \cdots + X_1^{z_{K1}} X_2^{z_{K2}} \cdots X_N^{z_{KN}} \cdot m_K = m_0,
    \end{equation}
where we look for solutions $(z_{11}, \ldots, z_{KN}) \in \Z^{KN}$.
Note that this also allows one to express a \emph{system} of $\ell$ equations of the form~\eqref{eq:Sunitintromod} in a single equation over the module $\mM^{\ell}$.

Equations of the form~\eqref{eq:Sunitintromod} are prevalent in computational algebra, appearing in contexts such as finding sparse polynomials in ideals~\cite{jensen2017finding, draisma2023no}, fragments of arithmetic theories~\cite{Hieronymi2022}, as well as linear recurrence sequences~\cite{chonev2013orbit, ibrahim2024positivity}.
Our original motivation comes from membership problems in metabelian groups, which also reduce to solving such equations~\cite{figelius2020complexity, dong2024submonoid}.
In general, it is undecidable whether a given equation of the form~\eqref{eq:Sunitintromod} admits a solution, even when $N = 1$~\cite{Dong2025LinearEW}.
However, recall that S-unit equations in fields become more tractable when we consider positive characteristics.
Correspondingly, we consider modules $\mM$ with \emph{$T$-torsion}, where $T$ is a positive integer.
This means that $T \mM = 0$, so $\mM$ becomes a $\ZT[X_1^{\pm}, \ldots, X_N^{\pm}]$-module.
It turns out that the difficulty of solving S-unit equations in $T$-torsion modules depends on the number of distinct prime divisors of $T$: this will be the main result of our paper.
We also show connections of S-unit equations to a similar type of Diophantine equations, which we call \emph{linear-exponential Diophantine equations} (see Theorem~\ref{thm:mainequiv}).
These are linear equations over $\Z$, where certain variables are restricted to powers of primes, (e.g.\ $2^x + 2^y - 3^z = 1$).
Despite linear-exponential Diophantine equations being widely studied in number theory, obtaining either decidability or undecidability results for solving the general case remains notoriously difficult~\cite{Hieronymi2022, karimov2025decidability}.

\paragraph*{Main results, comparison with previous work and applications.}

Derksen and Masser's result concerning S-unit equations in fields of positive characteristics is the following.


\begin{thrm}[{Derksen and Masser~\cite{derksen2012linear}}]\label{thm:DM}
    Let $\K$ be a field of characteristic $p$. Let $X_1, \ldots, X_N,$ and $m_0, m_1, \ldots, m_K,$ be elements of $\K$.
    The set of solutions $(z_{11}, \ldots, z_{KN}) \in \Z^{KN}$ to the equation
    \[
        X_1^{z_{11}} X_2^{z_{12}} \cdots X_N^{z_{1N}} \cdot m_1 + \cdots + X_1^{z_{K1}} X_2^{z_{K2}} \cdots X_N^{z_{KN}} \cdot m_K = m_0
    \]
    is an effectively \textbf{$\boldsymbol{p}$-normal set}.
\end{thrm}

Here, a $p$-normal set is defined as follows. Throughout this paper we assume $0 \in \N$.

\begin{defn}[{reformulation of~\cite{derksen2015linear}}]\label{def:pnormal}
    Let $r \in \N$ and $\ell \in \N \setminus \{0\}$.
    Let $H$ be a subgroup of $\Z^{KN}$, and $\ba_0, \ba_1, \ldots, \ba_r$ be vectors in $\Q^{KN}$.
    Define
    \begin{equation}\label{eq:psuccinct}
    S(\ell; \ba_0, \ba_1, \ldots, \ba_r; H) \\
        \coloneqq \; \{\ba_0 + p^{\ell k_1} \ba_1 + \cdots + p^{\ell k_r} \ba_r + \bh \mid k_1, k_2, \ldots, k_r \in \N, \bh \in H\}.
    \end{equation}
    Such a set is called \emph{$p$-succinct} if it is a subset of $\Z^{KN}$.
    A subset $S$ of $\Z^{KN}$ is called \emph{$p$-normal}, if it is a finite union of $p$-succinct sets. 
    We say a set is \emph{effectively} $p$-normal if there is an algorithm that computes all the coefficients ($\ba_0, \ldots, \ba_r$ and generators of $H$) of its $p$-succinct sets.
\end{defn}

Note that the vectors $\ba_i \in \Q^{KN}$ might not have integer coefficients. For example, the set $\left\{\frac{1}{2} + 3^n \cdot \frac{1}{2}\;\middle|\; n \in \N\right\}$ is a subset of $\Z$, but $\frac{1}{2}$ is not an integer. 
The condition $S \subseteq \Z^{KN}$ is equivalent to $(p^{\ell} - 1)\ba_0 \in \Z^{KN}, \ldots, (p^{\ell} - 1)\ba_r \in \Z^{KN}$ and $\ba_0 + \cdots + \ba_r \in \Z^{KN}$,~\cite[p.117]{derksen2015linear}.


The first result of our paper extends Derksen and Masser's theorem from fields to arbitrary modules with $p^e$-torsion:

\begin{restatable}{thrm}{thmprimepower}\label{thm:primepower}
    Let $p$ be a prime number and $e$ be a positive integer.
    Let $\mM$ be a finitely presented module over the Laurent polynomial ring $\Zpe[X_1^{\pm}, \ldots, X_N^{\pm}]$, and let $m_0, m_1, \ldots, m_K \in \mM$.
    Then the set of solutions $(z_{11}, \ldots, z_{KN}) \in \Z^{KN}$ to the S-unit equation
    \begin{equation}\label{eq:Sunitthm}
        X_1^{z_{11}} X_2^{z_{12}} \cdots X_N^{z_{1N}} \cdot m_1 + \cdots + X_1^{z_{K1}} X_2^{z_{K2}} \cdots X_N^{z_{KN}} \cdot m_K = m_0
    \end{equation}
    is effectively $p$-normal.
\end{restatable}

In particular, this means we can algorithmically decide whether Equation~\eqref{eq:Sunitthm} admits a solution.
Theorem~\ref{thm:primepower} will be proven in Section~\ref{sec:stol}.
We point out that the special case of Theorem~\ref{thm:primepower} for $e = 1$ admits a rather direct proof assuming Theorem~\ref{thm:DM}.
Indeed, using the techniques from~\cite[Section~9]{derksen2007skolem} or~\cite[Section~5]{dong2024submonoid}, we can reduce an S-unit equation in an $\F_p[X_1^{\pm}, \ldots, X_N^{\pm}]$-module to a system of S-unit equations in fields of characteristic $p$.
Then by Theorem~\ref{thm:DM}, the solution set of such a system is an intersection of effectively $p$-normal sets, which can be shown to be again effectively $p$-normal (see Proposition~\ref{prop:internormal}).

Unfortunately, this approach fails for $e > 1$.
Indeed, the work of Derksen and Masser makes extensive use of the \emph{Frobenius endomorphism}, and therefore heavily relies on working with $p$-torsion.
Proving Theorem~\ref{thm:primepower} for $e > 1$ will therefore require new insights.

Next, we proceed to consider modules with arbitrary integer torsion.
We show that algorithmically solving S-unit equations in $\ZT[X_1^{\pm}, \ldots, X_N^{\pm}]$-modules is equivalent to algorithmically solving \emph{linear-exponential Diophantine equations}, another infamous open problem in number theory:

\begin{thrm}\label{thm:mainequiv}
    Let $T = p_1^{e_1} p_2^{e_2} \cdots p_k^{e_k}$, where $p_1, p_2, \ldots, p_k$ are distinct primes and $e_1, \ldots, e_k$ are positive integers. The following two decision problems are Turing reducible to one another:
    \begin{enumerate}[noitemsep, label = (\arabic*)]
        \item \textbf{(S-unit equation in a $T$-torsion module)} Given $K, N \in \N$, a finitely presented module $\mM$ over the Laurent polynomial ring $\ZT[X_1^{\pm}, \ldots, X_N^{\pm}]$, as well as elements $m_0, m_1, \ldots, m_K \in \mM$, decide whether the equation
        \begin{equation}\label{eq:Sunit}
        X_1^{z_{11}} X_2^{z_{12}} \cdots X_N^{z_{1N}} \cdot m_1 + \cdots + X_1^{z_{K1}} X_2^{z_{K2}} \cdots X_N^{z_{KN}} \cdot m_K = m_0
        \end{equation}
        admits solutions $(z_{11}, \ldots, z_{KN}) \in \Z^{KN}$.
        \item \textbf{(Linear-exponential Diophantine equations)} Given $d \leq D$ and $L \in \N$, prime numbers $q_1, q_2, \ldots, q_d \in \{p_1, p_2, \ldots, p_k\}$, as well as integer coefficients $(c_{i,j})_{1 \leq i \leq L, 1 \leq j \leq D}, (b_{i})_{1 \leq i \leq L}$, decide whether the system of equations
        \begin{align}\label{eq:linearpower}
            c_{1,1} \cdot q_1^{n_1} + \cdots + c_{1,d} \cdot q_d^{n_d} + c_{1,d+1} \cdot z_{d+1} + \cdots + c_{1, D} \cdot z_{D} & = b_1, \nonumber \\
            & \vdots \nonumber \\
            c_{L,1} \cdot q_1^{n_1} + \cdots + c_{L,d} \cdot q_d^{n_d} + c_{L,d+1} \cdot z_{d+1} + \cdots + c_{L, D} \cdot z_{D} & = b_L,
        \end{align}
        admits solutions $(n_1, \ldots, n_d) \in \N^d, (z_{d+1}, \ldots, z_D) \in \Z^{D-d}$.
    \end{enumerate}
\end{thrm}

The reduction from problem (1) to (2) is a relatively straightforward consequence of Theorem~\ref{thm:primepower} (see Corollary~\ref{cor:intersection}).
The reduction from problem (2) to (1) will be shown in Section~\ref{sec:ltos}.

Note that the only shared parameters between the two problems in Theorem~\ref{thm:mainequiv} are the number $k$ and the primes $p_1, \ldots, p_k$.
For the case of $k = 1$, there are classic algorithms that decide existence of solutions to linear-exponential Diophantine equations over a single prime~\cite{benedikt2023complexity}.
These algorithms can be traced back to a long series of work started by B\"{u}chi and Sem\"{e}nov~\cite{buchi1960weak, semenov1980certain}, who were investigating decidable extensions of \emph{Presburger arithmetic}~\cite{presburger1929uber}.

For the case of $k = 2$, Karimov, Luca, Nieuwveld, Ouaknine and Worrell~\cite{karimov2025decidability} recently proved decidable whether a given system of linear-exponential Diophantine equations over two primes admits a solution.
This breakthrough result uses tools from Diophantine approximation and transcendental number theory, notably Baker's theorem.
This immediately yields the following.

\begin{cor}\label{cor:twoprimes}
    Let $p, q$ be primes and $a, b \in \N$.
    It is decidable whether an S-unit Equation~\eqref{eq:Sunit} over a finitely presented $\Z_{/p^a q^b}[X_1^{\pm}, \ldots, X_N^{\pm}]$-module admits a solution.
\end{cor}

Deciding whether a system of linear-exponential Diophantine equations over $k \geq 3$ primes is an outstanding open problem in number theory.
We mention here~\cite{brenner1982exponential, alma991009821129706011, cao1999note, bajpai2023effective, berthe2024decidability} among a plethora of partial results.
Theorem~\ref{thm:mainequiv} thus shows that any progress in solving S-unit equations over $p^a q^b r^c$-torsion modules for distinct primes $p, q, r$, would require major breakthroughs in number theory.


Note that if we restrict to dimension one, $p$-normal subsets of $\Z$ in fact have a very simple structure~\cite[p.116]{derksen2015linear}.
The procedure in~\cite[Theorem~3.2]{karimov2025decidability} allows us to decide for arbitrary $k$ whether $\mZ_1 \cap \cdots \cap \mZ_k = \emptyset$, where $\mZ_i$ is a $p_i$-normal subset of $\Z$ for different primes $p_1, \ldots, p_k$.
Therefore, Theorem~\ref{thm:primepower} provides a powerful tool for solving the \emph{Skolem problem} (decide whether a linear recurrence sequence contains a zero~\cite{ouaknine2015linear, lipton2022skolem}) in rings whose additive group is torsion.
We also mention that Corollary~\ref{cor:twoprimes} has direct consequences in computational group theory, namely that \emph{Submonoid Membership} is decidable in \emph{wreath products} of the form $\Z_{/p^a q^b} \wr \Z^d$, (see the reduction in~\cite{dong2024submonoid}).
We will explore more connections in these directions in follow up papers.

\section{Preliminaries}\label{sec:prelim}

\paragraph*{Laurent polynomial rings, modules and algebras.}

Let $T$ be a positive integer, denote by $\ZT$ the quotient ring $\Z/T\Z = \{0, 1, \ldots, T-1\}$.
In particular if $p$ is a prime number, then $\Z_{/p}$ is the finite field $\F_p$.
Denote by $\ZT[X_1^{\pm}, \ldots, X_N^{\pm}]$ the Laurent polynomial ring over $\ZT$ with $n$ variables: this is the set of polynomials over the variables $X_1, X_1^{-1}, \ldots, X_N, X_N^{-1}$, with coefficients in $\ZT$, such that $X_i X_i^{-1} = 1$ for all $i$.
In polynomial rings over a finite field $\F_p$, it is useful to point out that $(f + 1)^p = f^p + 1$, for any $f \in \F_p[X_1^{\pm}, \ldots, X_N^{\pm}]$.
This is not true in the polynomial ring $\ZT[X_1^{\pm}, \ldots, X_N^{\pm}]$ for arbitrary $T$.
For a prime $p$, denote by $\F_p(X_1, \ldots, X_N) \coloneqq \left\{\frac{f}{g} \;\middle|\; f, g \in \F_p[X_1^{\pm}, \ldots, X_N^{\pm}], g \neq 0 \right\}$, the fraction field of $\F_p[X_1^{\pm}, \ldots, X_N^{\pm}]$.

Let $R$ be a commutative ring.
An \emph{$R$-module} is defined as an abelian group $(\mM, +)$ along with an operation $\cdot \;\colon R \times \mM \rightarrow \mM$, satisfying $r \cdot (m+m') = r \cdot m + r \cdot m'$, $(r + s) \cdot m = r \cdot m + s \cdot m$, $rs \cdot m = r\cdot (s \cdot m)$ and $1 \cdot m = m$.
For example, for any $d \in \N$, $R^d$ is an $R$-module by $s \cdot (r_1, \ldots, r_d) = (sr_1, \ldots, sr_d)$.
An \emph{ideal} of $R$ is an $R$-submodule of $R$.
An ideal $I \subset R$ is called \emph{maximal} if $I \neq R$ and there is no ideal $J$ with $I \subsetneq J \subsetneq R$.
A (commutative) ring $R$ is called \emph{local} if it has only one maximal ideal: this is equivalent to having an ideal $I \subsetneq R$ such that every element $x \notin I$ is invertible.

Given $m_1, \ldots, m_k$ in an $R$-module $\mM$, let 
$
\gen{m_1, \ldots, m_k} \coloneqq \left\{\sum_{i=1}^k r_i \cdot m_i \;\middle|\; r_1, \ldots, r_k \in R\right\}
$
denote the $R$-submodule generated by $m_1, \ldots, m_k$.
Given two $R$-modules $\mM, \mM'$ such that $\mM \supseteq \mM'$, we define the quotient $\mM/\mM' \coloneqq \{\overline{m} \mid m \in \mM\}$ where $\overline{m_1} = \overline{m_2}$ if and only if $m_1 - m_2 \in \mM'$.
This quotient is also an $R$-module.
We say an $R$-module is \emph{finitely presented} if it can be written as a quotient $R^d/\gen{v_1, \ldots, v_k}$ for some $d \in \N$ and some $v_1, \ldots, v_k \in R^d$. 
Such a quotient is called a \emph{finite presentation}.
Every finitely generated $\ZT[X_1^{\pm}, \ldots, X_N^{\pm}]$-module admits a finite presentation and is effective~\cite{baumslag1981computable}, meaning there is an algorithm that decides equality of any two elements. 

An \emph{$R$-algebra} is defined as a ring $\mA$ such that $(\mA, +)$ is an $R$-module, and such that $r \cdot (ab) = (r \cdot a) b = a(r \cdot b)$ for all $r \in R, \; a, b \in \mA$.
For any $d \in \N$, denote by $\M_{d \times d}(R)$ the set of $d \times d$ matrices with coefficients in $R$. 
Then $\M_{d \times d}(R)$ is an $R$-algebra.
Denote by $\GL_d(R)$ the set of $d \times d$ \emph{invertible} matrices with coefficients in $R$, it is not an $R$-algebra because $0 \notin \GL_d(R)$.

\paragraph*{p-automatic sets}
We recall the standard notion of \emph{$p$-automatic subsets} of $\Z$ and $\Z^d$.

Let $\Sigma$ be a finite alphabet. An \emph{automaton} over $\Sigma$ is a tuple $\mmU = (Q, \Sigma, \delta, q_I, \mF)$,
where $Q$ is a finite set of states, $\delta \coloneqq Q \times \Sigma \rightarrow Q$ is the transition function, $q_I \in Q$ is the initial state, and $\mF \subseteq Q$ is the set of accepting states.
A \emph{word} over the alphabet $\Sigma$ is a finite sequence of elements in $\Sigma$.
For a state $q$ in $Q$ and for a finite word $w = w_1 w_2 \cdots w_n$ over the alphabet $\Sigma$, we define $\delta(q, w)$ recursively by $\delta(q, w) = \delta(\delta(q, w_1 w_2 \cdots w_{n-1}), w_n)$. 
The word $w$ is \emph{accepted} by $\mmU$ if $\delta(q_I, w) \in \mF$.
We call the \emph{language} accepted by $\mmU$ the set of words accepted by $\mmU$ , and we denote it by $L(\mmU)$.

An automaton is usually represented by a graph whose vertices are the states, drawn as circles. For each state $q$ and each $s \in \Sigma$ we draw an arrow from $q$ to $\delta(q, s)$ with label $s$. The accepting states will be drawn as double circles, and the initial state will be marked an arrow with the label ``start''. See Figure~\ref{fig:pow2} and \ref{fig:a2a} for examples.

Let $p \geq 2$ be an integer, define the alphabet $\Sigma_p \coloneqq \{-(p-1), \ldots, -1, 0, 1, \ldots, p-1\}$.
For any word $w = w_0 w_1 \cdots w_{\ell-1}$ over the alphabet $\Sigma_p$, we define its \emph{evaluation} to be $\eval(w) \coloneqq \sum_{i = 0}^{\ell-1} p^i w_i$.
Note that each integer can be represented as $\eval(w)$ for some word $w$ over $\Sigma_p$, though such representation might not be unique.
For example, when $p = 2$, we have $\eval(001) = \eval\big(00(-1)1\big) = 4$.
A subset $S$ of $\Z$ is called \emph{$p$-automatic} if there exists an automaton $\mmU$ over $\Sigma_p$, such that
$
\left\{\eval(w) \;\middle|\; w \in L(\mmU) \right\} = S
$. (Again, such an automaton might not be unique).
For example, the set $\{2^k \mid k \in \N\}$ is 2-automatic, because it can be represented by the language $\{1, 01, 001, 0001, \cdots \}$, which is accepted by the automaton in Figure~\ref{fig:pow2}.

\begin{figure}[ht!]
    \centering
    \vspace{-0.1cm}
    \begin{minipage}[t]{.47\textwidth}
        \centering
        \tikzset{every loop/.style={min distance=10mm,looseness=10}}
        \raisebox{\height+0.5\baselineskip}{
        \begin{tikzpicture}[>=triangle 45]
           \node[initial, state, minimum size=15pt] (0) at (0,0) {};
           \node[accepting, state, minimum size=14pt] (1) at (2,0) {};
           \path[->] (0) edge [loop above] node [align=center]  {$0$} (0)
           (0) edge [above] node [align=center]  {$1$} (1);
        \end{tikzpicture}}
        \vspace{-0.3cm}
        \caption{An automaton for $\{2^k \mid k \in \N\}$.}
        \label{fig:pow2}
    \end{minipage}
    \hfill
    \begin{minipage}[t]{0.47\textwidth}
        \centering
        \tikzset{every loop/.style={min distance=7mm,looseness=5}}
        \begin{tikzpicture}[>=triangle 45]
           \node[initial, state, minimum size=15pt] (2) at (-2.5,0) {};
           \node[accepting, state, minimum size=14pt] (0) at (0,1) {};
           \node[state, minimum size=15pt] (1) at (0,-1) {};
           \path[->] (2) edge [above, sloped] node [align=center]  {\footnotesize$(0,0)$} (0)
           (2) edge [below, sloped] node [align=center]  {\footnotesize$(-1,0)$} (1)
           (0) edge [loop above] node [align=center]  {\footnotesize$(0,0)$} (0)
           (1) edge [loop below] node [align=center]  {\footnotesize$(-1,-1)$} (1)
           (0) edge [above, in=30,out=-30, sloped] node [align=center]  {\footnotesize$(-1,0)$} (1)
           (1) edge [above, in=-120,out=120, sloped] node [align=center]  {\footnotesize$(0,-1)$} (0);
        \end{tikzpicture}
        \vspace{-0.3cm}
        \caption{An automaton for $\{(a, 2a) \mid a \leq 0\}$.}
        \label{fig:a2a}
    \end{minipage}
\end{figure}

Let $d$ be a positive integer. The definition of $p$-automatic subsets of $\Z$ can be naturally generalized to $p$-automatic subsets of $\Z^d$.
For a word $\bw = \bw_0 \bw_1 \cdots \bw_{\ell-1}$ over the alphabet $\Sigma_p^d$, we define $\eval(\bw) \coloneqq \sum_{i = 0}^{\ell-1} p^i \cdot \bw_i$.
For example, when $p = 2, d = 3$, we have $\eval\big((1, 0, 1)(-1, -1, 1)\big) = (-1, -2, 3)$.
A subset $S$ of $\Z^d$ is called \emph{$p$-automatic} if there exists an automaton $\mmU$ over $\Sigma_p^d$, such that
$
\left\{\eval(\bw) \;\middle|\; \bw \in L(\mmU) \right\} = S
$.
For example, the set $\{(a, 2a) \mid a \leq 0\}$ $\subseteq \Z^2$ is 2-automatic, because the language
$
\{(a_1, 0)(a_2, a_1)(a_3, a_2) \cdots (a_k, a_{k-1}) (0, a_k) \mid k \in \N, a_1, \ldots, a_k \in \{-1, 0\}\}
$
is accepted by the automaton in Figure~\ref{fig:a2a}.
We say a subset $S \subseteq \Z^d$ is \emph{effectively} $p$-automatic if an accepting automaton is given explicitly.
Effective $p$-automatic sets enjoy various closure properties:

\begin{lem}[{\cite{wolper2000construction}}]\label{lem:pauto}
    Let $p \geq 2$ be an integer.
    \begin{enumerate}[nosep, label = (\arabic*)]
        \item If $S$ and $T$ are $p$-automatic, then $S \cap T$, $S \cup T$, and $S \setminus T$ are also effectively $p$-automatic.
        \item The set $\Z^d$ is $p$-automatic. If $S \subseteq \Z^d$ is $p$-automatic, and $\varphi \colon \Z^d \rightarrow \Z^n$ is a linear transformation, then $\varphi(S) \subseteq \Z^n$ is also effectively $p$-automatic.
        \item If $S \subseteq \Z^d$ and $T \subseteq \Z^n$ are $p$-automatic, then their direct product $S \times T \coloneqq \{(s, t) \mid s \in S, t \in T\} \subseteq \Z^{d+n}$ is also effectively $p$-automatic.
    \end{enumerate}
\end{lem}

In particular, any subgroup $H$ of $\Z^{KN}$ is $p$-automatic.
It is easy to see that $p$-succinct sets and $p$-normal sets (Definition~\ref{def:pnormal}) are also $p$-automatic.
However, not all $p$-automatic sets are $p$-normal.

\section{S-unit equation to linear-exponential Diophantine equations}\label{sec:stol}

\subsection{Overview and examples}

In this section we prove that the solution set of an S-unit equation over a $p^e$-torsion module is effectively $p$-normal:

\thmprimepower*

From Theorem~\ref{thm:primepower}, we can easily obtain the reduction from solving an S-unit equation to solving linear-exponential Diophantine equations in Theorem~\ref{thm:mainequiv}:

\begin{cor}\label{cor:intersection}
    Let $T = p_1^{e_1} p_2^{e_2} \cdots p_k^{e_k}$.
    Deciding whether an S-unit equation in a $\ZT[X_1^{\pm}, \ldots, X_N^{\pm}]$-module (Equation~\eqref{eq:Sunit}) admits a solution reduces to deciding whether a system of linear-exponential Diophantine equations (Equations~\eqref{eq:linearpower}) admits a solution.
\end{cor}
\begin{proof}
    For each $j = 1, \ldots, k$, consider the quotient $\mM/p_j^{e_j} \mM$.
    It is a finitely presented module over the ring $\Z_{/p_j^{e_j}}[X_1^{\pm}, \ldots, X_N^{\pm}]$. 
    Let $\varphi_j \colon \mM \rightarrow \mM/p_j^{e_j} \mM$ denote the quotient map.
    Since $p_1, \ldots, p_k$ are distinct primes, an element $m \in \mM$ is zero if and only if $\varphi_j(m)$ is zero for all $j$ (by the Chinese remainder theorem for $\mM$ considered as a $\ZT$-module).
    Therefore, a tuple $(z_{11}, \ldots, z_{KN}) \in \Z^{KN}$ is a solution to Equation~\eqref{eq:Sunit} if and only if it is a solution to
    \begin{equation}\label{eq:Sunitmodpe}
        X_1^{z_{11}} X_2^{z_{12}} \cdots X_N^{z_{1N}} \cdot \varphi_j(m_1) + \cdots + X_1^{z_{K1}} X_2^{z_{K2}} \cdots X_N^{z_{KN}} \cdot \varphi_j(m_K) = \varphi_j(m_0)
    \end{equation}
    for all $j = 1, \ldots, k$.
    Let $\mZ_j$ denote the solution set of Equation~\eqref{eq:Sunitmodpe}, it is effectively $p_j$-normal by Theorem~\ref{thm:primepower}.
    Therefore, Equation~\eqref{eq:Sunit} has a solution if and only if the intersection $\mZ_1 \cap \cdots \cap \mZ_k$ is non-empty.
    Each $\mZ_j$ is a finite union of $p_j$-succinct sets, so it suffices to decide whether the intersection of $(p_j)_{j=1, \ldots, k}$-succinct sets is empty.
    
    For each $j$, let $S_j \coloneqq S(\ell_j; \ba_{j0}, \ba_{j1}, \ldots, \ba_{jr_j}; H_j)$ be a $p_j$-succinct set.
    Then deciding whether $S_1 \cap \cdots \cap S_k$ is empty amounts to solving the following system of equation
    \begin{equation}\label{eq:intertoeqflat}
    \ba_{10} + p_1^{n_{11}} \ba_{11} + \cdots + p_1^{n_{1r_1}} \ba_{1r_1} + \bh_1 = \cdots = \ba_{k0} + p_k^{n_{k1}} \ba_{k1} + \cdots + p_k^{n_{kr_k}} \ba_{kr_k} + \bh_k,
    \end{equation}
    over the variables $n_{j1}, \ldots, n_{j1} \in \ell_j \Z, \; \bh_j \in H_j$, for $j = 1, \ldots, k$.
    Let $\bh_{j1}, \ldots, \bh_{js_j}$ be the generators of the group $H_j, j = 1, \ldots, k$.
    Then the system of Equations~\eqref{eq:intertoeqflat} can be written as a system of linear-exponential Diophantine equations of the form~\eqref{eq:linearpower}:
    \begin{align}\label{eq:intertoeq}
    & \; \ba_{10} + p_1^{n_{11}} \cdot \ba_{11} + \cdots + p_1^{n_{1r_1}} \cdot \ba_{1r_1} + z_{11} \cdot \bh_{11} + \cdots + z_{1s_1} \cdot \bh_{1s_1} \nonumber\\
    = & \; \ba_{20} + p_2^{n_{21}} \cdot \ba_{21} + \cdots + p_2^{n_{2r_2}} \cdot \ba_{2r_2} + z_{21} \cdot \bh_{21} + \cdots + z_{2s_2} \cdot \bh_{2s_2} \nonumber\\
    \vdots & \nonumber\\
    = & \; \ba_{k0} + p_k^{n_{k1}} \cdot \ba_{k1} + \cdots + p_k^{n_{kr_k}} \cdot \ba_{kr_k} + z_{k1} \cdot \bh_{k1} + \cdots + z_{ks_k} \cdot \bh_{ks_k},
    \end{align}
    over the variables $n_{11}, \ldots, n_{kr_k} \in \N, z_{11}, \ldots, z_{k s_k} \in \Z$, with the extra constraint that $\ell_j \mid n_{ji}$ for all $j, i$.
    We can multiply each term in Equation~\eqref{eq:intertoeqflat} by their common denominator, and suppose $\ba_{ji} \in \Z^{KN}$ for all $j, i$, so that Equation~\eqref{eq:intertoeq} is indeed of the form~\eqref{eq:linearpower}.
    Furthermore, the extra constraint ``$\ell_j \mid n_{ji}$'' can be expressed as another equation
    ``$
    p^{n_{ji}} - 1 + (p^{\ell_j} - 1) z = 0
    $''
    over the variables $n_{ji} \in \N, z \in \Z$, for any prime $p$. (Indeed, $\ell_j \mid n_{ji} \iff p^{\ell_i} - 1 \mid p^{n_{ji}} - 1$).
    We conclude that solving Equation~\eqref{eq:intertoeq} reduces to solving a system of linear-exponential Diophantine equations of the form~\eqref{eq:linearpower}.
\end{proof}

\paragraph*{Main ideas of proving Theorem~\ref{thm:primepower}.}
We illustrate here with simple examples the key ideas of Derksen and Masser~\cite{derksen2012linear} 
for proving Theorem~\ref{thm:DM}.
Using these examples, we then illustrate how we generalize these ideas to prove Theorem~\ref{thm:primepower}.
Derksen and Masser's method in~\cite{derksen2012linear} builds upon their respective earlier works~\cite{masser2004mixing, derksen2007skolem}.
In~\cite{derksen2007skolem}, Derksen's approach for proving $p$-normality is to first prove \emph{$p$-automaticity}, and then refine it into \emph{$p$-normality} by analysing the accepting automaton.
In~\cite{derksen2012linear}, this approach is simplified and reformulated without the language of automata theory, while retaining many of the same core ideas.

\begin{exmpl}[Derksen's approach~\cite{derksen2007skolem}]\label{exmpl:Derksen}
    Consider the following equation over the variable $n$:
    \begin{equation}\label{eq:exampleDerksen}
    X_2^n - X_1^n = 1,
    \end{equation}
    in the finitely presented $\F_2[X_1^{\pm}, X_2^{\pm}]$-module $\F_2[X_1^{\pm}, X_2^{\pm}]/\gen{X_2 - X_1 - 1}$.
    For simplicity of the illustration consider the variable $n$ over $\N$ instead of $\Z$, and we construct an automaton over the alphabet $\{0, 1\}$ (instead of $\{-1, 0, 1\}$) that accepts the solution set.
    
    Note that Equation~\eqref{eq:exampleDerksen} can be considered as a version of the S-unit equation $X_1^{z_{11}} X_2^{z_{12}} - X_1^{z_{21}} X_2^{z_{22}} = 1$, ``specialized'' at $z_{11} = z_{22} = 0, z_{12} = z_{21} \geq 0$.
    We will illustrate Derksen's approach~\cite{derksen2007skolem} for solving~\eqref{eq:exampleDerksen}, which shares the same ideas for solving the ``full'' S-unit equation.

    Equation~\eqref{eq:exampleDerksen} in $\F_2[X_1^{\pm}, X_2^{\pm}]/\gen{X_2 - X_1 - 1}$ can be rewritten as the equation
    \begin{equation}\label{eq:exampleDerksenvar}
    (X+1)^n - X^n = 1,
    \end{equation}
    in $\F_2[X^{\pm}]$, by the change of variables $X_1 = X,\; X_2 = X_1+1$.
    Consider the parity of $n$:
    \begin{enumerate}[nosep, label = (\roman*)]
        \item If $n$ is even, write $n = 2n'$, and Equation~\eqref{eq:exampleDerksenvar} becomes $(X+1)^{2n'} - X^{2n'} = 1$, which is equivalent to $(X^2+1)^{n'} - {(X^2)}^{n'} = 1$ because $(X+1)^2 = X^2 + 1$.
        Setting $X' \coloneqq X^2$, this can be rewritten as
        \begin{equation}\label{eq:exampleDerkseneven}
        (X'+1)^{n'} - (X')^{n'} = 1.
        \end{equation}
        Note that~\eqref{eq:exampleDerkseneven} has the same form as~\eqref{eq:exampleDerksenvar}.
        \item If $n$ is odd, write $n = 2n' + 1$. Using the equality $(X+1)^2 = X^2 + 1$, Equation~\eqref{eq:exampleDerksenvar} becomes
        $
        (X^2+1)^{n'} \cdot (X+1) - {(X^2)}^{n'} \cdot X = 1.
        $
        Taking all the monomials of \emph{even} degree on both sides yields
        $
        (X^2+1)^{n'} \cdot 1 = 1.
        $
        Taking all the monomials of \emph{odd} degree yields
        $
        (X^2+1)^{n'} \cdot X - {(X^2)}^{n'} \cdot X = 0.
        $
        Thus if we set $X' \coloneqq X^2$, then $
        (X^2+1)^{n'} \cdot (X+1) - {(X^2)}^{n'} \cdot X = 1
        $ becomes the system
        \begin{equation}\label{eq:exampleDerksenodd}
        \begin{cases}
            (X'+1)^{n'} = 1, \\
            (X'+1)^{n'} - {(X')}^{n'} = 0,
        \end{cases}
        \end{equation}
        whose only solution can be easily seen to be $n' = 0$.
    \end{enumerate}
    The above case analysis shows the following. One can construct an automaton $\mmU$ with two states, corresponding respectively to Equation~\eqref{eq:exampleDerksenvar} and (the system of) Equations~\eqref{eq:exampleDerksenodd},
    see Figure~\ref{fig:Derksen}.
    \begin{figure}[h]
        \centering
        \vspace{-0.1cm}
        \begin{tikzpicture}[>=triangle 45]
        \tikzset{elliptic state/.style={draw,ellipse}}
           \node[initial, elliptic state, minimum height=3em] (0) at (0,0) {\small{$(X+1)^n - X^n = 1$}};
           \node[elliptic state,accepting, minimum height=3em] (1) at (6,0) {\footnotesize{$
           \begin{cases}
                (X+1)^{n} = 1, \\
                (X+1)^{n} - {X}^{n} = 0,
            \end{cases}$}};
           \path[->] (0) edge [loop above] node [align=center]  {\footnotesize$0$} (0)
           (0) edge [above] node [align=center]  {\footnotesize$1$} (1);
        \end{tikzpicture}
        \vspace{-0.2cm}
        \caption{The automaton $\mmU$.}
        \label{fig:Derksen}
    \end{figure}
    We start at the state of Equation~\eqref{eq:exampleDerksenvar} and read the base-2 expansion of $n$.
    If the first (least significant) digit of $n$ is $0$, then we stay in the state of Equation~\eqref{eq:exampleDerksenvar}, which now represents the equation for $n'$ with $n = 2n'$.
    If the first digit of $n$ is $1$, then we transition to the state of Equation~\eqref{eq:exampleDerksenodd}, which now represents the equation for $n'$ with $n = 2n' + 1$.
    Since $n' = 0$ is a solution of Equation~\eqref{eq:exampleDerksenodd}, its corresponding state is an accepting state of $\mmU$.
    We have omitted the transitions from the state of~\eqref{eq:exampleDerksenodd}, because we know its solution set is $\{0\}$. 
    Thus, the solution set of Equation~\eqref{eq:exampleDerksenvar} is the language accepted by the automaton $\mmU$, which we can directly see to be $\{2^n \mid n \in \N\}$.

    To solve a general univariate equation of the form~\eqref{eq:exampleDerksenvar}, the key in Derksen's argument for $p$-automaticity is to control the coefficient size and the number of the equations appearing in each state of $\mmU$ (see~\cite[Proposition~5.2]{derksen2007skolem}), so that $\mmU$ has only finitely many states.
    Derksen then proceeds to show $p$-normality of the solution set by analyzing the structure of $\mmU$.
    \hfill $\blacksquare$
\end{exmpl}

Note that instead of considering~\eqref{eq:exampleDerksenvar} as an equation in the polynomial ring $\F_2[X^{\pm}]$, one can equivalently consider it as an equation in the field of fractions $\F_2(X)$.
More generally, Derksen's approach solves Equations of the form~\eqref{eq:exampleDerksenvar} in any field of prime characteristic, using the so-called \emph{Frobenius splitting}.
Namely, if $\K$ is a field of characteristic $p$, then
$
\K^p \coloneqq \{k^p \mid k \in \K\}
$,
is a subfield of $\K$, thus making $\K$ an $\K^p$-vector space.
Hence, as in Example~\ref{exmpl:Derksen}, an equation over $\K$ splits into a system of equations over $\K^p$, which again becomes a system of equations over $\K$ under the variable change $x' \coloneqq x^p$.

In principle, the idea also works for the ``full'' S-unit equation~\eqref{eq:Sunitthm}, \emph{provided} we can embed it in a field of prime characteristic.
Instead of guessing the parity of $n$, we need to guess the residue modulo $p$ of each variable $z_{11}, \ldots, z_{KN}$ as well as their signs.
This approach is taken by the work of Adamczewski and Bell~\cite{adamczewski2012vanishing}, who showed $p$-automaticity of solution sets for S-unit equations in fields of prime characteristics.

The situation becomes much more difficult when we do not work in prime characteristic: we are considering modules over the ring $\Zpe[X_1^{\pm}, \ldots, X_N^{\pm}],\; e > 1$, which is $p^e$-torsion but not $p$-torsion.
This means several key arguments in Derksen's approach no longer work.
The most obvious failure is that we no longer have $(X+1)^p = X^p + 1$. 
However, in the special case of solving Equation~\eqref{eq:exampleDerksenvar} over a \emph{polynomial ring} $\Zpe[X^{\pm}]$, we can conceive an ``improved'' version of Derksen's approach:

\begin{exmpl}[Improved Derksen's approach]\label{exmpl:improved}
    Consider the following equation in $\Z_{/4}[X^{\pm}]$:
    \begin{equation}\label{eq:exampleimproved}
    (X+1)^n - X^n = 1.
    \end{equation}
    
    Suppose $n$ is even and write $n = 2n'$, then Equation~\eqref{eq:exampleimproved} becomes
    \begin{equation}\label{eq:exampleimprovedeven}
    (X^2+2X+1)^{n'} - (X^2)^{n'} = 1,
    \end{equation}
    Note that we have 
    $
    X^2 + 2X + 1 \neq  X^2 + 1
    $
    in $\Z_{/4}[X^{\pm}]$, so we cannot use $X' \coloneqq X^2$ to bring Equation~\eqref{eq:exampleimprovedeven} back to the form~\eqref{eq:exampleimproved}.
    However, the key observation here is that we have
    $
    (X^2+ 2X +1)^2 = X^4+ 2X^2 +1
    $
    in the ring $\Z_{/4}[X^{\pm}]$.
    This means that, if $n'$ is even again and we write $n' = 2n''$, then Equation~\eqref{eq:exampleimprovedeven} becomes $(X^4+2X^2+1)^{n''} - (X^4)^{n''} = 1$.    
    Then we can employ the same approach as in the previous example by taking the variable change $X' \coloneqq X^2$, so that the equation $(X^4+2X^2+1)^{n''} - (X^4)^{n''} = 1$ becomes
    \begin{equation}\label{eq:exampleimprovedfour}
    (X'^2+2X'+1)^{n''} - (X'^2)^{n''} = 1,
    \end{equation}
    which has the same form as~\eqref{eq:exampleimprovedeven}.
    
    This allows us to construct an automaton $\mmU$ using the same approach as in Example~\ref{exmpl:Derksen}.
    See Figure~\ref{fig:improve} for an illustration of the first several states of $\mmU$.
    Here, the dashed transitions $\dashrightarrow$ indicate we do not perform the variable change $X' \coloneqq X^2$ during the transition; whereas regular transitions $\longrightarrow$ indicate the variable change $X' \coloneqq X^2$.
    It is not difficult to show that the degree of the coefficients appearing in the states stays bounded, as each variable change $X' \coloneqq X^2$ decreases the degree by half, up to an additive constant. Therefore, the total number of states stays bounded.
    \hfill $\blacksquare$
\end{exmpl}

    \begin{figure}[h]
        \vspace{-0.5cm}
        \centering
        \begin{tikzpicture}[>=triangle 45]
        \tikzset{elliptic state/.style={draw,ellipse}}
           \node[initial, elliptic state] (0) at (0,0) {\footnotesize$(X+1)^n - X^n = 1$};
           \node[elliptic state] (1) at (3.5,1.2) {\footnotesize{$
               (X^2+2X+1)^n$ $- (X^2)^n = 1
           $}};
            \node[elliptic state] (2) at (3.2,-1.2) {\footnotesize$
           (X^2+2X+1)^n \cdot (X+1) - (X^2)^n \cdot X = 1
            $};
            \node[] (3) at (9, 1.2) {$\cdots$};
            \node[elliptic state, accepting] (4) at (8,0) {\tiny$
            \begin{cases}
                (X^2+2X+1)^n = 1 \\
                (X^2+2X+1)^n - (X^2)^n = 0
            \end{cases}
            $};
            \node[elliptic state] (5) at (8,-2.4) {\tiny$
            \begin{cases}
                (X^2+2X+1)^n \cdot (3X + 1) = 1 \\
                (X^2+2X+1)^n \cdot (X + 3) + (X^2)^n \cdot X = 0
            \end{cases}
            $};
           \path[->] (1) edge [loop above, looseness=6] node [align=center]  {\footnotesize$0$} (1)
           (1) edge [above] node [align=center]  {\footnotesize$1$} (3)
           (2) edge [above, inner sep = 0pt, text depth = 1pt] node [align=center]  {\footnotesize$1$} (5)
           (2) edge [above, inner sep = 0pt, text depth = 1pt] node [align=center]  {\footnotesize$0$} (4);
           \path[dashed, ->]
           (0) edge [above, inner sep = 0pt, text depth = 1pt] node [align=center]  {\footnotesize$0$} (1)
           (0) edge [above, inner sep = 0pt, text depth = 1pt] node [align=center]  {\footnotesize$1$} (2);
        \end{tikzpicture}
        \caption{A fragment of the ``improved'' automaton $\mmU$ for $\Z_{/4}[X^{\pm}]$.}
        \label{fig:improve}
    \end{figure}
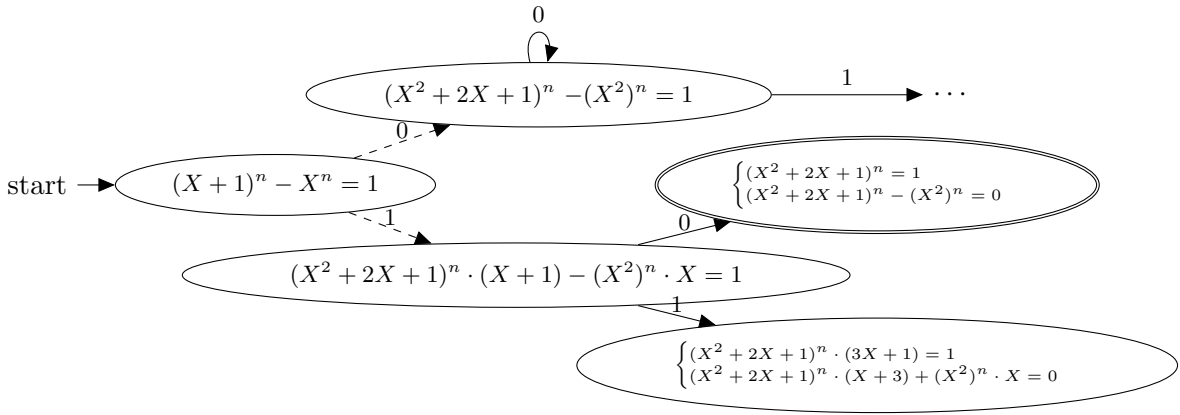
For an equation of the form~\eqref{eq:exampleimproved} in some polynomial ring $\Zpe[X_1^{\pm}, \ldots, X_n^{\pm}]$, we can always apply the idea in Example~\ref{exmpl:improved} to construct an automaton.
Indeed, for any $f \in \Zpe[X_1^{\pm}, \ldots, X_n^{\pm}]$, we can show
$
f^{p^e}(X_1, \ldots, X_n) = f^{p^{e-1}}(X_1^p, \ldots, X_n^p),
$
see Lemma~\ref{lem:powp}.
Therefore, after $e-1$ dashed transitions in the automaton $\mmU$, we can start performing the variable change $(X_1', \ldots, X_n') \coloneqq (X_1^p, \ldots, X_n^p)$ in regular transitions.

Similar to how Example~\ref{exmpl:Derksen} generalizes to arbitrary fields of characteristic $p$, we will generalize Example~\ref{exmpl:improved} from polynomial rings to a family of \emph{Artinian local rings}.
The generalization is highly technical in its exact formulation and proof (see Subsection~\ref{subsec:Frob}), but it is vital for proving Theorem~\ref{thm:primepower} for the following reason.
Our approach in the previous example relies on working over a polynomial ring $\Zpe[X^{\pm}]$ (or $\Zpe[X]$), based on two obvious but important facts:
\begin{enumerate}[nosep]
    \item The polynomial ring $\Zpe[(X^p)^{\pm}]$ is isomorphic to $\Zpe[X^{\pm}]$: this allows us to apply the variable change $X' \coloneqq X^p$ and keep working in equations over the same polynomial ring.
    \item As a $\Zpe[(X^p)^{\pm}]$-module, $\Zpe[X^{\pm}]$ can be written as a direct sum $\bigoplus_{i = 0}^{p-1} \Zpe[(X^p)^{\pm}]$: this allows us to split each equation over $\Zpe[X^{\pm}]$ into a system of at most $p$ equations over $\Zpe[(X^p)^{\pm}]$. This is illustrated in Example~\ref{exmpl:Derksen} case (ii), where an equation is split into a system of two equations~\eqref{eq:exampleDerksenodd}, by grouping monomials of even and odd degrees.
\end{enumerate}
Unfortunately, most S-unit equations in $\Zpe[X_1^{\pm}, \ldots, X_N^{\pm}]$-modules $\mM$ cannot be reduced to equations in polynomial rings.
The two above facts do not make sense if we replace $\Zpe[X^{\pm}]$ with a finitely presented module $\mM$:

\begin{exmpl}[label=exa:cont]\label{exmpl:failure}
    Consider an S-unit equation 
    \begin{equation}\label{eq:failure}
        X^{z_{11}}Y^{z_{12}} - X^{z_{21}}Y^{z_{22}} = 1
    \end{equation}
    in the finitely presented $\Z_{/4}[X^{\pm}, Y^{\pm}]$-module
    $
    \mM \coloneqq \Z_{/4}[X^{\pm}, Y^{\pm}]/\gen{X^3 + 2XY - Y^3}.
    $
    
    It is not clear how $\mM$ can be isomorphic to a polynomial ring, because the quotient $X^3 + 2XY - Y^3$ is not a linear polynomial over any variable. 
    Therefore, we need to work directly over $\mM$.
    
    If all $z_{ij}$ are even, write $z_{ij} = 2 z'_{ij}$ for $i, j \in \{1, 2\}$.
    Then Equation~\eqref{eq:failure} becomes
    \begin{equation}\label{eq:failure2}
    (X^2)^{z'_{11}} (Y^2)^{z'_{12}} - (X^2)^{z'_{21}}(Y^2)^{z'_{22}} = 1.
    \end{equation}
    Now \eqref{eq:failure2} is an equation over $\mM = \Z_{/4}[X^{\pm}, Y^{\pm}]/\gen{X^3 + 2XY - Y^3}$.
    If we consider the indeterminates $X^2, Y^2$, then \eqref{eq:failure2} becomes an equation in the $\Z_{/4}[(X^2)^{\pm}, (Y^2)^{\pm}]$-module
    \[
    \mM' \coloneqq \Z_{/4}[(X^2)^{\pm}, (Y^2)^{\pm}]\Big/\big(\Z_{/4}[(X^2)^{\pm}, (Y^2)^{\pm}] \cap \gen{X^3 + 2XY - Y^3}\big).
    \]
    Unfortunately, the ideal $\Z_{/4}[(X^2)^{\pm}, (Y^2)^{\pm}] \cap \gen{X^3 + 2XY - Y^3}$ is not equal to $\gen{X^6 + 2X^2Y^2 - Y^6}$. (In fact, $X^6 + 2X^2Y^2 - Y^6 \notin \gen{X^3 + 2XY - Y^3}$).
    Therefore $\mM'$ is not isomorphic to $\mM$ under the variable change $X' \coloneqq X^2, Y' \coloneqq Y^2$.
    This means that although Equation~\eqref{eq:failure2} has the same form as~\eqref{eq:failure} after the variable change, it is not the same equation since we are not solving them over the same module.

    What is worse, when considering other possibilities of $z_{ij}$, we can no longer ``split'' the new equation by grouping the monomials in~\eqref{eq:failure2} by their degree parity, as we did in Example~\ref{exmpl:Derksen}~case (ii).
    For instance, in the module $\mM$ we have $2XY = Y^3 - X^3$, but the two sides have different parity in degrees.
    More formally, $\mM$ does not split as a direct product of copies of $\mM'$, which is what would allow us to split an equation into several independent ones.
    \hfill $\blacksquare$
\end{exmpl}

Example~\ref{exmpl:failure} shows that we need additional insights to generalize our idea in Example~\ref{exmpl:improved}.
If we work over a prime characteristic, then in certain cases we can reduce S-unit equations over $\mM$ to S-unit equations over fields of prime characteristic:

\begin{exmpl}[{Embedding in a direct product of fields~\cite{derksen2007skolem, dong2024submonoid}}]\label{exmpl:success}
    Consider the equation
    \begin{equation}\label{eq:success}
        X^{z_{11}}Y^{z_{12}} - X^{z_{21}}Y^{z_{22}} = 1,
    \end{equation}
    but this time in the $\F_2[X^{\pm}, Y^{\pm}]$-module
    $
    \mM \coloneqq \F_2[X^{\pm}, Y^{\pm}]/\gen{X^3 + 2XY - Y^3}.
    $
    
    Note that over $\F_2$, we have the factorization $X^3 + 2XY - Y^3 = (X^2 + XY + Y^2)(X-Y)$, so we have the decomposition of $\mM$ into two modules
    \[
    \mM = \Big(\F_2[X^{\pm}, Y^{\pm}]/\gen{X^2 + XY + Y^2}\Big) \times \Big(\F_2[X^{\pm}, Y^{\pm}]/\gen{X - Y}\Big).
    \]
    One can embed both modules into fields 
    \[
    \varphi_1 \colon \F_2[X^{\pm}, Y^{\pm}]/\gen{X^2 + XY + Y^2} \hookrightarrow \F_2(X)[Y]/\gen{X^2 + XY + Y^2},
    \]
    (the quotient $\F_2(X)[Y]/\gen{X^2 + XY + Y^2}$ is a field because $X^2 + XY + Y^2$ is irreducible), and
    \[
    \varphi_2 \colon \F_2[X^{\pm}, Y^{\pm}]/\gen{X - Y} \hookrightarrow \F_2(X).
    \]
    Thus, the map 
    $
    \varphi_1 \times \varphi_2 \colon \; \mM \hookrightarrow \Big(\F_2(X)[Y]/\gen{X^2 + XY + Y^2} \Big) \times \F_2(X)
    $,
    embeds $\mM$ into a direct product of two fields.
    Let $\pi_1, \pi_2$ denote respectively the projection of $(\varphi_1 \times \varphi_2) (\mM)$ on $\F_2(X)[Y]/\gen{X^2 + XY + Y^2}$ and $\F_2(X)$.
    Then, Equation~\eqref{eq:success} over $\mM$ is equivalent to the system of equations \emph{over fields}
    \begin{equation}\label{eq:successsys}
    \begin{cases}
        \pi_1(X)^{z_{11}} \pi_1(Y)^{z_{12}} - \pi_1(X)^{z_{21}}\pi_1(Y)^{z_{22}} = \pi_1(1), \\
        \pi_2(X)^{z_{11}} \pi_2(Y)^{z_{12}} - \pi_2(X)^{z_{21}}\pi_2(Y)^{z_{22}} = \pi_2(1).
    \end{cases}
    \end{equation}
    We can thus solve both equations effectively by the discussion following Example~\ref{exmpl:Derksen}.
    \hfill $\blacksquare$
\end{exmpl}

In general, if $\mM$ is a module over $\F_p[X_1^{\pm}, \ldots, X_N^{\pm}]$, a similar approach works using the \emph{coprimary decomposition} of $\mM$ instead of factorization.
This is illustrated in~\cite[Section~9]{derksen2007skolem} and more thoroughly in~\cite{dong2024submonoid}.
This approach only works over prime characteristics.
When $\mM$ is a module over $\Zpe[X_1^{\pm}, \ldots, X_N^{\pm}]$, this may fail since $p \mM \neq 0$.
In this case, using an idea similar to Example~\ref{exmpl:success}, we will embed an equation over $\mM$ into equations over \emph{local rings}:

\renewcommand\thmcontinues[1]{continued}
\begin{exmpl}[continues=exa:cont]
    Consider the S-unit equation 
    \begin{equation}\label{eq:failure3}
        X^{z_{11}}Y^{z_{12}} - X^{z_{21}}Y^{z_{22}} = 1
    \end{equation}
    in the $\Z_{/4}[X^{\pm}, Y^{\pm}]$-module
    $
    \mM \coloneqq \Z_{/4}[X^{\pm}, Y^{\pm}]/\gen{X^3 + 2XY - Y^3}.
    $
    We have the factorization 
    \[
    X^3 + 2XY - Y^3 = (X - Y - 2)(X^2 + XY + Y^2 + 2X - 2Y)
    \]
    in the ring $\Z_{/4}[X^{\pm}, Y^{\pm}]$.
    Therefore, the Equation~\eqref{eq:failure3} in $\mM$ can be written as a system of two equations, respectively in
    \[
    \mM_1 \coloneqq \Z_{/4}[X^{\pm}, Y^{\pm}] \big/ \gen{X - Y + 2},
    \]
    and
    \[
    \mM_2 \coloneqq \Z_{/4}[X^{\pm}, Y^{\pm}] \big/ \gen{X^2 + XY + Y^2 + 2X - 2Y}.
    \]
    Let us first consider the ``easier'' module $\mM_1$.
    It is isomorphic to $\Z_{/4}[X^{\pm}, (X+2)^{\pm}]$, which can be embedded into the local ring
    \[
    \Z_{/4}(X) \coloneqq \left\{\frac{f}{g} \;\middle|\; f, g \in \Z_{/4}[X],\; 2 \nmid g \right\}.
    \]
    That is, we allow denominators in $\Z_{/4}(X)$, as long as they are not divisible by $2$.
    The ring $\Z_{/4}(X)$ is local in the sense that it has only one maximal ideal $\gen{2}$.
    It will then be possible to generalize the idea of Example~\ref{exmpl:improved} to solve Equation~\eqref{eq:failure3} over $\Z_{/4}(X)$.

    For the more complicated module $\mM_2$, we can embed it in the ring
    \[
    \mR \coloneqq \Z_{/4}(X)[Y^{\pm}] \big/ \gen{X^2 + XY + Y^2 + 2X - 2Y}.
    \]
    That is, we allow the same denominators in $\mR$ as in $\Z_{/4}(X)$.
    The ring $\mR$ is also \emph{local}, and is an \emph{algebraic extension} of $\Z_{/4}(X)$.
    It turns out that for this reason, $\mR$ shares enough properties with $\Z_{/4}(X)$ and $\Z_{/4}[X]$, so that it will be possible to generalize the ideas of Example~\ref{exmpl:improved} to solve S-unit equations over $\mR$-modules.
    The exact formulation of this generalization is rather technical, and will be gradually introduced in the following subsections.
    \hfill $\blacksquare$
\end{exmpl}

\paragraph{Organization of this section.}
As illustrated in Example~\ref{exmpl:failure}, our approach to proving Theorem~\ref{thm:primepower} will be as follows.
\newcounter{SummaryCounter}
\begin{enumerate}[nosep, label = \arabic*.]
    \item By taking the \emph{coprimary decomposition} of $\mM$ and by \emph{localizing} appropriate variables, we reduce an S-unit equation in $\mM$ to an S-unit equation in modules over a local ring (Proposition~\ref{prop:wrapperlocalization}).
    Though a \emph{stabilization} process (Lemma~\ref{lem:stableffective}) and a \emph{block-diagonalization} process (Proposition~\ref{prop:blockdiagonal}), we further decompose $\mM$ into factors with certain freeness properties (Proposition~\ref{prop:wrapperfreeness}).
    This will be detailed in Subsections~\ref{subsec:localize} and \ref{subsec:free}.
    \item We construct an automaton $\mmU$ in base $p$ that accepts the solution set of an S-unit equation.
    This generalizes Derksen and Masser's~\cite{derksen2007skolem, derksen2012linear} (and independently, Adamczewski and Bell's~\cite{adamczewski2012vanishing}) solution to S-unit equations in fields of characteristic $p$ to S-unit equations in modules over certain $p^e$-torsion local rings.
    This generalization is done in a similar way to how we improved Example~\ref{exmpl:Derksen} to Example~\ref{exmpl:improved}.
    The main idea is to construct a ``pseudo'' Frobenius splitting (Proposition~\ref{prop:highFrobsplit}) via \emph{Hensel lifting} (Lemma~\ref{lem:Hensel}), which will replace the variable change argument in Example~\ref{exmpl:improved}.
    We then bound the number of states appearing in the automaton $\mmU$ by bounding the degree of coefficients appearing in the state equations (Lemma~\ref{lem:finite}).
    This is detailed in Subsections~\ref{subsec:Frob} and \ref{subsec:automaton}.
    \setcounter{SummaryCounter}{\value{enumi}}
\end{enumerate}
The above discussion only serves to prove \emph{$p$-automaticity} instead of \emph{$p$-normality}.
In order to prove Theorem~\ref{thm:primepower}, we will continue to refine our result to $p$-normality.
More precisely:
\begin{enumerate}[nosep, label = \arabic*.]
    \setcounter{enumi}{\value{SummaryCounter}}
    \item We decompose the automaton $\mmU$ into \emph{strongly connected components}.
    We show that each component contributes to a term $p^{\ell k_i} \ba_i$ and a subgroup $H$ in the definition~\eqref{eq:psuccinct} of $p$-succinct sets (Lemma~\ref{lem:stablePi} and Lemma~\ref{lem:singlecomp}).
    If we ``contract'' each strongly connected component into a single point, then $\mmU$ is a graph without cycles, and thus consists of only finitely many paths (see Figure~\ref{fig:SC}).
    Roughly speaking, each path ``corresponds'' to a $p$-succinct set, as it passes through finitely many strongly connected components.
    The union of these paths then ``corresponds'' to a finite union of $p$-succinct sets, thus a $p$-normal set.
    The solution set obtained this way then needs to go through a so-called ``saturation'' process (Lemma~\ref{lem:saturation}) and a so-called ``symmetrization'' process (Proposition~\ref{prop:symmetrize}), in order to truly become a $p$-normal set.
    This is another technical contribution of this section, and will be detailed in Subsections~\ref{subsec:strongconn} and~\ref{subsec:symm}.
\end{enumerate}

\subsection{Decomposition and localization}\label{subsec:localize}
From now on, for a ring $R$ and an $R$-module $M$, we call an \emph{S-unit equation over an $R$-module $M$}, an equation of the form
\[
        x_1^{z_{11}} x_2^{z_{12}} \cdots x_N^{z_{1N}} \cdot m_1 + \cdots + x_1^{z_{K1}} x_2^{z_{K2}} \cdots x_N^{z_{KN}} \cdot m_K = m_0,
\]
where $K, N$ are positive integers, $x_1, \ldots, x_N$ are invertible elements of $R$ and $m_1, \ldots, m_K$ are elements of $M$.

In this subsection we reduce solving S-unit equations over the $\Zpe[X_1^{\pm}, \ldots, X_N^{\pm}]$-module $\mM$, to solving S-unit equations over $\tmA$-modules $\tmV$, where $\tmA$ is a local ring satisfying some additional properties.
More precisely, we will show the following:

\begin{restatable}{prop}{propwrapperlocalization}\label{prop:wrapperlocalization}
    Let $\mZ$ be the solution set of an S-unit equation over a $\Zpe[X_1^{\pm}, \ldots, X_N^{\pm}]$-module $\mM$.
    Then $\mZ$ can be effectively written as a finite positive Boolean combination $\bigcup_i \bigcap_j \mZ_{ij}$, where each $\mZ_{ij}$ is an affine transformation of the solution set of an S-unit equation
    \[
    A_1^{z_{11}} A_2^{z_{12}} \cdots A_N^{z_{1N}} \cdot v_1 + \cdots + A_1^{z_{K1}} A_2^{z_{K2}} \cdots A_N^{z_{KN}} \cdot v_K = v_0
    \]
    over some $\tmA$-module $\tmV$, satisfying:
    \begin{enumerate}[nosep, label = (\roman*)]
        \item $\tmA$ is local, its maximal ideal $\frp$ satisfies $\frp^t = 0$ for some $t \geq 1$.
        \item $\tmA$ is effectively represented as a $\Zpe(\oX)$-algebra for some tuple of variables $\oX = (X_1, \ldots, X_n)$, $n \leq N$. The definition of the ring $\Zpe(\oX)$ will be formalized later. 
        \item As a $\Zpe(\oX)$-module, $\tmA$ is finitely generated.
        \item For any $k \geq 0$, the set of elements $\left\{A_1^{p^k}, A_1^{-p^k} \ldots, A_N^{p^k}, A_N^{-p^k}\right\}$ generates $\tmA$ as a $\Zpe(\oX)$-algebra.
        \item $\tmV$ is finitely presented as an $\tmA$-module.
    \end{enumerate}
\end{restatable}

Note that affine transformations of $p$-normal sets are also $p$-normal.
The following proposition shows that finite intersections of $p$-normal sets are $p$-normal.
Since finite unions of $p$-normal sets are by definition $p$-normal, this will allow us to reduce proving $p$-normality of the solution set of an S-unit equation over the $\Zpe[X_1^{\pm}, \ldots, X_N^{\pm}]$-module $\mM$ to an $\tmA$-module $\tmV$.

\begin{restatable}{prop}{propinternormal}\label{prop:internormal}
    The intersection of two $p$-normal sets is effectively $p$-normal.
\end{restatable}
\begin{proof}
    Proposition~\ref{prop:internormal} can be considered as a generalization of~\cite[Lemma~9.5]{derksen2007skolem}, which deals with $p$-normal sets in $\N$.
    The generalization from $\N$ to $\Z^{KN}$ is technical but does not present conceptual difficulty.
    We provide a full self-contained proof in Appendix~\ref{app:internormal}.
\end{proof}

\paragraph{Step 1. Intersection: coprimary decomposition.}

The idea behind Proposition~\ref{prop:wrapperlocalization} is inspired by the approach taken in~\cite{derksen2007skolem} and~\cite{dong2024submonoid}, which reduced S-unit equations in $\F_p[X_1^{\pm}, \ldots, X_N^{\pm}]$-modules to equations in $\F_p(X_1, \ldots, X_n)$-vector spaces.
See also Example~\ref{exmpl:success} for an illustration of the basic ideas.
First we recall some standard definitions from commutative algebra.
We refer the readers to the textbook~\cite{eisenbud2013commutative} for details and proofs.

\begin{defn}\label{def:commalg}
    Let $R$ be a commutative Noetherian ring (for example, $R = \Zpe[X_1^{\pm}, \ldots, X_N^{\pm}]$).
    \begin{enumerate}[nosep, label = (\arabic*)]
        \item An ideal $I \subseteq R$ is called \emph{prime} if $I \neq R$, and for every $a, b \in R$, $ab \in I$ implies $a \in I$ or $b \in I$. Prime ideals are usually denoted by the Gothic letters $\frp$ or $\frq$.
        \item Let $M$ be a finitely generated $R$-module. The \emph{annihilator} of an element $m \in M$, denoted by $\Ann_R(m)$, is the set $\{r \in R \mid r\cdot m = 0\}$.
        A prime ideal $\frp \subset R$ is called \emph{associated} to $M$ if there exists a non-zero $m \in M$ such that $\frp = \Ann_R(m)$. 
        Let $N$ be a finitely generated $R$-module. A submodule $N'$ of $N$ is called \emph{primary} if $N/N'$ has only one associated prime ideal. If we denote this prime ideal by $\frp$, then $N'$ is called a \emph{$\frp$-primary} submodule of $N$.
        \item Let $N'$ be a submodule of a finitely generated $R$-module $N$. The \emph{primary decomposition} of $N'$ is the writing of $N'$ as a finite intersection $\bigcap_{i = 1}^l N_i$, where $N_i$ is a $\frp_i$-primary submodule of $N$ for some prime ideal $\frp_i \subset R$.
        A primary decomposition always exists~\cite[Theorem~3.10]{eisenbud2013commutative}.
        If $R$ is a quotient of a polynomial ring over an effective base ring (such as $\Zpe$), and $N, N'$ are finitely generated submodules of $R^d$ for some $d \in \N$, then a primary decomposition of $N' \subseteq N$ can be effectively computed~\cite{rutman1992grobner}.       
        \item A finitely generated $R$-module $M$ is called \emph{coprimary} if the submodule $\{0\}$ is primary, that is, if $M$ has only one associated prime ideal. If we denote this prime ideal by $\frp$, then $M$ is called \emph{$\frp$-coprimary}.
        If $M$ is $\frp$-coprimary, and $m$ is a non-zero element in $M$, then $\Ann_R(m) \subseteq \frp$.
    \end{enumerate}
\end{defn}

Let $\mM = \Zpe[X_1^{\pm}, \ldots, X_N^{\pm}]^d/\mN$ be the finite presentation of $\mM$.
Let $\mN = \bigcap_{j = 1}^l \mN_j$ be the primary decomposition of the submodule $\mN$ of $\Zpe[X_1^{\pm}, \ldots, X_N^{\pm}]^d$, where $\mN_j$ is $\frp_j$-primary for a prime ideal $\frp_j \subset \Zpe[X_1^{\pm}, \ldots, X_N^{\pm}]$, $j = 1, \ldots, l$.
Then $\mM_j \coloneqq \Zpe[\oX^{\pm}]^d/\mN_j$ is $\frp_j$-coprimary.
Since $\mN \subseteq \mN_j$, there is a canonical map 
\[
\rho_j \colon \mM = \Zpe[\oX^{\pm}]^d/\mN \rightarrow \mM_j = \Zpe[\oX^{\pm}]^d/\mN_j.
\]
Since $\mN = \bigcap_{j = 1}^l \mN_i$, the intersection of kernels $\bigcap_{j=1}^l \ker(\rho_j)$ is $\{0\}$.

Since each $\mM_j$ is $\frp_j$-coprimary, by~\cite[Proposition~3.9]{eisenbud2013commutative} there exists $t_j \in \N$ such that $\frp_j^{t_j} \mM_j = 0$.
Therefore the $\Zpe[X_1^{\pm}, \ldots, X_N^{\pm}]$-module $\mM_j$ is actually a $\Zpe[X_1^{\pm}, \ldots, X_N^{\pm}]/\frp_j^{t_j}$-module.

\begin{lem}\label{lem:inter}
    Let $m_0, m_1, \ldots, m_K \in \mM$.
    For $j = 1, \ldots, l$, let $\mZ_j$ denote the set of solutions to the following equation over $\mM_j$:
    \begin{equation}\label{eq:Suniti}
        X_1^{z_{11}} X_2^{z_{12}} \cdots X_N^{z_{1N}} \cdot \rho_j(m_1) + \cdots + X_1^{z_{K1}} X_2^{z_{K2}} \cdots X_N^{z_{KN}} \cdot \rho_j(m_K) = \rho_j(m_0).
    \end{equation}
    Then the solution set of $\sum_{i = 1}^K X_1^{z_{i1}} X_2^{z_{i2}} \cdots X_N^{z_{iN}} m_i = m_0$ is exactly the intersection $\bigcap_{i = 1}^l \mZ_i$.
\end{lem}
\begin{proof}
    Since $\bigcap_{j=1}^l \ker(\rho_i) = \{0\}$, we have
    $
    \sum_{i = 1}^K X_1^{z_{i1}} X_2^{z_{i2}} \cdots X_N^{z_{iN}} m_i - m_0 = 0
    $,
    if and only if 
    $
    \sum_{i = 1}^K X_1^{z_{i1}} X_2^{z_{i2}} \cdots X_N^{z_{iN}} \cdot \rho_j( m_i) - \rho_j(m_0) = 0
    $
    for all $j = 1, \ldots, l$.
    Therefore, the solution set of $\sum_{i = 1}^K X_1^{z_{i1}} X_2^{z_{i2}} \cdots X_N^{z_{iN}} m_i = m_0$ is exactly the intersection $\bigcap_{i = 1}^l \mZ_i$.
\end{proof}


By Lemma~\ref{lem:inter}, the solution set of an S-unit equation over a $\Zpe[X_1^{\pm}, \ldots, X_N^{\pm}]$-module is effectively a finite intersection of solution sets of S-unit equations over $\frp$-\emph{coprimary} modules.
Therefore, we can from now on focus on the case where $\mM$ is a $\frp$-coprimary module over $\Zpe[X_1^{\pm}, \ldots, X_N^{\pm}]/\frp^t$.
Here, $\frp$ is a prime ideal of $\Zpe[X_1^{\pm}, \ldots, X_N^{\pm}]$, and $t \in \N$ is such that $\frp^t \mM = 0$.
Since $p^e = 0 \in \frp$ and $\frp$ is prime, we have $p \in \frp$.
Therefore we can also consider $\frp$ as a prime ideal of $\F_p[X_1^{\pm}, \ldots, X_N^{\pm}]$.

\paragraph{Step 2. Localization and definition of $\Zpe(\oX)$.}

Choose a maximal set of variables among $X_1, \ldots, X_N$ that are algebraically independent\footnote{Algebraic independence can be checked effectively by variable elimination in ideals~\cite[Section~15.10.4]{eisenbud2013commutative}.} over $\F_p[X_1^{\pm}, \ldots, X_N^{\pm}]/\frp$. 
Without loss of generality suppose this set of variables is $\{X_1, \ldots, X_n\}$ for some $n \leq N$. 
Denote 
\[
\oX \coloneqq (X_1, \ldots, X_n),
\]
and write $R[\oX^{\pm}] \coloneqq R[X_1^{\pm}, \ldots, X_n^{\pm}]$, $R[\oX] \coloneqq R[X_1, \ldots, X_n]$ for any ring $R$.
Let $\Zpe(\oX)$ denote the \emph{localization} of $\Zpe[\oX^{\pm}]$ at the prime ideal $\gen{p}$:
\[
\Zpe(\oX) \coloneqq \left\{\frac{f}{g} \;\middle|\; f, g \in \Zpe[\oX],\; p \nmid g \right\}.
\]
Then $\Zpe(\oX)$ is a principal ideal ring\footnote{A principal ideal ring is a commutative ring in which every ideal is generated by a single element.} (PIR), whose only ideals are $\gen{1} = \Zpe(\oX), \gen{p}, \gen{p^2}, \ldots$, $\gen{p^{e-1}}$ and $\gen{p^e} = \{0\}$.
Indeed, let $I$ be any ideal of $\Zpe(\oX)$. For each $f \in \Zpe[\oX]$, let $a(f) \in \N$ denote the largest integer $a$ such that $p^{a} \mid f$.
Let $m \coloneqq \min\{a(f) \mid f/g \in I, p \nmid g\}$, then $p^m$ divides every element in $I$, so $I \subseteq \gen{p^m}$. Furthermore, let $\frac{f}{g} \in I, p \nmid g$, be such that $a(f) = m$. Write $f = p^mF$, then $p \nmid F$, so $p^m = \frac{f}{g} \cdot \frac{g}{F} \in I$.
Therefore $I = \gen{p^m}$.

Since the ring $\Zpe(\oX)$ has finitely many ideals, it is Noetherian, so every finitely generated $\Zpe(\oX)$-module admits a finite presentation.
Note that when $e = 1$, the ring $\Zpe(\oX)$ is exactly the fraction field $\F_p(\oX)$.

For any $\Zpe[\oX^{\pm}]$-module $M$, define the localization $M \otimes_{\Zpe[\oX^{\pm}]} \Zpe(\oX)$ to be the $\Zpe(\oX)$-module
\[
\left\{\frac{m}{g} \;\middle|\; m \in M, g \in \Zpe[\oX^{\pm}], \; p \nmid g \right\}.
\]

Consider the localizations
\begin{align*}
    \mN & \coloneqq \mM \otimes_{\Zpe[\oX^{\pm}]} \Zpe(\oX), \\
    \frq & \coloneqq \frp \otimes_{\Zpe[\oX^{\pm}]} \Zpe(\oX), \\
    \mR & \coloneqq \left(\Zpe[X_1^{\pm}, \ldots, X_N^{\pm}]/\frp^t \right) \otimes_{\Zpe[\oX^{\pm}]} \Zpe(\oX) = \Zpe(\oX)[X_{n+1}^{\pm}, \ldots, X_{N}^{\pm}]/\frq^t.
\end{align*}
Since $\mM$ is a $\Zpe[X_1^{\pm}, \ldots, X_N^{\pm}]/\frp^t$-module, its localization $\mN$ is an $\mR$-module.

\begin{lem}\label{lem:AYfgmod}
   The localizations $\mN$ and $\mR$ are finitely generated as $\Zpe(\oX)$-modules.
\end{lem}
\begin{proof}
    Since $\{X_1, \ldots, X_n\}$ is a maximal algebraically independent set over $\F_p[X_1^{\pm}, \ldots, X_N^{\pm}]/\frp$, the quotient $\Zpe[X_1^{\pm}, \ldots, X_N^{\pm}]/\frp = \F_p[X_1^{\pm}, \ldots, X_N^{\pm}]/\frp$ is an algebraic extension of the ring $\F_p[\oX^{\pm}] = \F_p[X_1^{\pm}, \ldots, X_n^{\pm}]$.
    Taking the localization at $\gen{p}$ shows that
    \[
    \mB \coloneqq \left(\Zpe[X_1^{\pm}, \ldots, X_N^{\pm}]/\frp\right) \otimes_{\Zpe[\oX^{\pm}]} \Zpe(\oX) = \Zpe(\oX)[X_{n+1}^{\pm}, \ldots, X_{N}^{\pm}]/\frq
    \]
    is an algebraic extension of the fraction field $\F_p(\oX)$.
    Therefore, $\mB$ is a finite dimensional $\F_p(\oX)$-vector space, and hence a finitely generated $\Zpe(\oX)$-module. 

    Since $\frp^t \mM = 0$ and $\frp^t \left(\Zpe[X_1^{\pm}, \ldots, X_N^{\pm}]/\frp^t \right) = 0$, after localization we have $\frq^t \mN = 0$ and $\frq^t \mR = 0$.
    Let $M$ be either $\mN$ or $\mR$, then we have the chain
    \[
    M \supseteq \frq M \supseteq \frq^2 M \supseteq \cdots \supseteq \frq^t M = 0.
    \]
    Each quotient $\frq^{j} M/\frq^{j+1} M$ is a finitely generated $\Zpe(\oX)[X_{n+1}^{\pm}, \ldots, X_{N}^{\pm}]$-module, thus a finitely generated $\mB = \Zpe(\oX)[X_{n+1}^{\pm}, \ldots, X_{N}^{\pm}]/\frq$-module, because $\frq \cdot (\frq^{j} M/\frq^{j+1} M) = 0$.
    Since $\mB$ is a finitely generated $\Zpe(\oX)$-module, each $\frq^{j} M/\frq^{j+1} M$ is also a finitely generated $\Zpe(\oX)$-module.
    We conclude that $M$ is finitely generated as a $\Zpe(\oX)$-module.
\end{proof}

\paragraph{Step 3. Embedding in S-unit equations over the $\mR$-module $\mN$.}

Since $\mM$ is $\frp$-coprimary, the annihilator of any non-zero element in $\mM$ as a $\Zpe[X_1^{\pm}, \ldots, X_n^{\pm}]$-module is contained in 
\[
\frp \cap \Zpe[X_1^{\pm}, \ldots, X_n^{\pm}] \subseteq \gen{p}.
\]
Therefore the canonical map $\mM \rightarrow \mM \otimes_{\Zpe[\oX^{\pm}]} \Zpe(\oX) = \mN$ is injective since elements of $\gen{p}$ are not localized (do not appear in the denominator).

Let $R_1, \ldots, R_N \in \mR$ be the image of $X_1, \ldots, X_N$ under the composition of maps 
\[
\Zpe[X_1^{\pm}, \ldots, X_N^{\pm}] \rightarrow \Zpe[X_1^{\pm}, \ldots, X_N^{\pm}]/\frp^t \rightarrow \mR.
\]
Let $\nu_0, \nu_1, \ldots, \nu_N \in \mN$ be the images of $m_0, m_1, \ldots, m_K \in \mM$ under the embedding 
\[
\mM \hookrightarrow \mM \otimes_{\Zpe[\oX^{\pm}]} \Zpe(\oX) = \mN.
\]

Since $\mM \hookrightarrow \mN$ is injective, the S-unit equation
    \[
        X_1^{z_{11}} X_2^{z_{12}} \cdots X_N^{z_{1N}} \cdot m_1 + \cdots + X_1^{z_{K1}} X_2^{z_{K2}} \cdots X_N^{z_{KN}} \cdot m_K = m_0
    \]
in the $\Zpe[X_1^{\pm}, \ldots, X_N^{\pm}]/\frp^t$-module $\mM$ is equivalent to the S-unit equation
\[
    R_1^{z_{11}} R_2^{z_{12}} \cdots R_N^{z_{1N}} \cdot \nu_1 + \cdots + R_1^{z_{K1}} R_2^{z_{K2}} \cdots R_N^{z_{KN}} \cdot \nu_K = \nu_0
\]
in the $\mR$-module $\mN$.
Thus, from now on we can work in the $\mR$-module $\mN$.

For $k \in \N$, define $\mR_{k} \subseteq \mR$ to be the $\Zpe(\oX)$-algebra generated by $R_1^{p^{k}}, R_1^{-p^{k}}, \ldots, R_N^{p^{k}}, R_N^{-p^{k}}$.
Then we have a descending chain of $\Zpe(\oX)$-algebras
\begin{equation}\label{eq:descendingmR}
\mR = \mR_0 \supseteq \mR_1 \supseteq \mR_2 \supseteq \cdots.
\end{equation}
The ring $\Zpe(\oX)$ is Artinian (it has finitely many ideals)~\cite[Theorem~2.13]{eisenbud2013commutative}, so $\mR$ is Artinian as a finitely generated module over an Artinian ring. Therefore the chain~\eqref{eq:descendingmR} must stabilize starting from some $\mR_{\ell},\; \ell \in \N$.

\begin{lem}\label{lem:stableffective}
    The number $\ell$ is effectively computable.
\end{lem}
\begin{proof}
    Let $k \in \N$, and let $e \geq 1$ be as above (as in $\Zpe(\oX)$). We claim that if $\mR_{k} = \mR_{k+e}$, 
    then the chain~\eqref{eq:descendingmR} stabilizes after $\mR_{k}$.
    To prove this, we will show $\mR_k = \mR_{k+e} \implies \mR_{k+1} = \mR_{k+e+1}$.

    Indeed, if $\mR_{k+e} = \mR_k$, then each $R_i^{p^k}, i = 1, \ldots, N,$ as well as its inverse, can be written as $f_i(R_1^{p^{k+e}}, R_1^{-p^{k+e}}, \ldots, R_N^{p^{k+e}}, R_N^{-p^{k+e}})$, where $f_i \in \Zpe(\oX)[Z_1, Z_1', \ldots, Z_N, Z_N']$ is a polynomial.
    Then
    \begin{equation}\label{eq:Reqpoly}
    R_i^{p^k} = f_i\left(R_1^{p^{k+e}}, \ldots, R_N^{-p^{k+e}}\right) = f_i\left(f_1\left(R_1^{p^{k+e}}, \ldots, R_N^{-p^{k+e}}\right)^{p^{e}}, \ldots, f_N\left(R_1^{p^{k+e}}, \ldots, R_N^{-p^{k+e}}\right)^{p^{e}}\right).
    \end{equation}
    Each $f_j\left(R_1^{p^{k+e}}, \ldots, R_N^{-p^{k+e}}\right), j = 1, \ldots, N$, can be written as a fraction
    \[
    \frac{g\left(X_1, \ldots, X_n, R_1^{p^{k+e}}, \ldots, R_N^{-p^{k+e}}\right)}{h(X_1, \ldots, X_n)}
    \]
    for some polynomials $g \in \Zpe[X_1, \ldots, X_n, Z_1, Z_1', \ldots, Z_N, Z_N']$ and $h \in \Zpe[X_1, \ldots, X_n]$.
    Applying Lemma~\ref{lem:powp} below for the tuple of variables $(X_1, \ldots, X_n, Z_1, Z_1', \ldots, Z_N, Z_N')$, we have
    \begin{multline*}
    f_j\left(R_1^{p^{k+e}}, \ldots, R_N^{-p^{k+e}}\right)^{p^{e}} = \frac{g\left(X_1, \ldots, X_n, R_1^{p^{k+e}}, \ldots, R_N^{-p^{k+e}}\right)^{p^{e}}}{h(X_1, \ldots, X_n)^{p^{e}}} \\
    = \frac{g\left(X_1^p, \ldots, X_n^p, R_1^{p^{k+e+1}}, \ldots, R_N^{-p^{k+e+1}}\right)^{p^{e-1}}}{h(X_1, \ldots, X_n)^{p^{e}}}.
    \end{multline*}
    This shows that each $f_j\left(R_1^{p^{k+e}}, \ldots, R_N^{-p^{k+e}}\right)^{p^{e}}$ can be written as a polynomial in $R_1^{p^{k+e+1}}$, $R_1^{-p^{k+e+1}}$, $\ldots$, $R_N^{p^{k+e+1}}$, $R_N^{-p^{k+e+1}}$, with coefficients in $\Zpe(\oX)$.
    Therefore Equation~\eqref{eq:Reqpoly} shows that $R_i^{p^k}$ is equal to a polynomial in $R_1^{p^{k+e+1}}, R_1^{-p^{k+e+1}}, \ldots, R_N^{p^{k+e+1}}, R_N^{-p^{k+e+1}}$, with coefficients in $\Zpe(\oX)$.
    This yields $R_i^{p^k} \in \mR_{k+e+1}$ for all $i = 1, \ldots, N$.
    Consequently, $\mR_{k} = \mR_{k+e+1}$.
    Since $\mR_{k+e} = \mR_k \supseteq \mR_{k+1} \supseteq \mR_{k+e}$ in the descending chain~\eqref{eq:descendingmR}, we have $\mR_{k+1} = \mR_k = \mR_{k+e+1}$.
    We have thus shown 
    \[
    \mR_k = \mR_{k+e} \implies \mR_{k+1} = \mR_{k+e+1}.
    \]
    Use this implication iteratively for $k+1, k+2, \ldots,$ we conclude that $\mR_{k} = \mR_{j}$ for all $j \geq k+e$.
    Therefore $\mR_{k} = \mR_{j}$ for all $j \geq k$. 
    
    Thus, in order to compute $\ell$, it suffices to check whether $\mR_{k+e} = \mR_k$ for $k = 1, 2, \ldots$.
    Since the descending chain $\mR_0 \supseteq \mR_1 \supseteq \mR_2 \supseteq \cdots$ eventually stabilizes, such a $k$ can eventually be found.
    Note that checking whether $\mR_{k+e} = \mR_k$ can be done by checking whether $r \in \mR_{k+e}$ for the generators $r$ of the $\Zpe(\oX)$-module $\mR_{k}$ (an algorithm that checks submodule membership can be found in~\cite{baumslag1981computable}).
    Since the descending chain~\eqref{eq:descendingmR} eventually stabilizes, we can find $k$ such that $\mR_{k+e} = \mR_k$ in finite time, and we conclude by letting $\ell \coloneqq k$.
\end{proof}

\begin{lem}\label{lem:powp}
    Let $\oY = (Y_1, \ldots, Y_s)$ be a tuple of variables.
    For any $h \in \Zpe[\oY]$, we have $h^{p^{e}}(Y_1, \ldots, Y_s) = h^{p^{e-1}}(Y_1^p, \ldots, Y_s^p)$.
\end{lem}
\begin{proof}
    Recall that we have $h^p(Y_1, \ldots, Y_s) \equiv h(Y_1^p, \ldots, Y_s^p) \mod p$.
    That is, $h^p(Y_1, \ldots, Y_s) - h(Y_1^p, \ldots, Y_s^p)$ is in the ideal $p \cdot \Zpe[\oY]$.
    Since $p^e = 0$, using Lemma~\ref{lem:lifting} below for $f = h^p(Y_1, \ldots, Y_s)$, $g = h(Y_1^p, \ldots, Y_s^p)$, $t = e$, $r = e-1$, we conclude that 
    $
    h^{p^{e}}(Y_1, \ldots, Y_s) = h^{p^{e-1}}(Y_1^p, \ldots, Y_s^p)
    $.
\end{proof}

\begin{lem}\label{lem:lifting}
    Let $R$ be a ring, and $P$ be an ideal such that $p \in P$, and $P^t = 0$ for some $t \in \N$.
    Let $f, g \in R$ such that $f - g \in P$. Then $f^{p^{r}} = g^{p^{r}}$ for all $r \geq t-1$.
\end{lem}
\begin{proof}
    We claim that for $a \geq 1$, we have $x - y \in P^a \implies x^p - y^p \in P^{a+1}$.
    Indeed, if $x - y \in P^a$ then write $x = y + z$ for some $z \in P^a$.
    So $x^p = (y + z)^p = y^p + \sum_{i = 1}^p \binom{p}{i} y^{p-i} z^i$. 
    For $i = 2, \ldots, p$, we have $z^i \in P^{ai} \subseteq P^{a+1}$; while for $i = 1$, we have $\binom{p}{i} z^i = pz \in p \cdot P^a \subseteq P^{a+1}$.
    In both cases we have $\binom{p}{i} y^{p-i} z^i \in P^{a+1}$.
    Therefore $x^p - y^p = \sum_{i = 1}^p \binom{p}{i} y^{p-i} z^i \in P^{a+1}$.

    Using this claim for $a = 1, 2, \ldots, t-1$, we have
    \[
    f - g \in P \implies f^p - g^p \in P^2 \implies f^{p^2} - g^{p^2} \in P^3 \implies \cdots \implies f^{p^{t-1}} - g^{p^{t-1}} \in P^{t} = 0.
    \]
    Taking $p^{r-t}$-th power to both sides of $f^{p^{t-1}} = g^{p^{t-1}}$ yields $f^{p^{r}} = g^{p^{r}}$.
\end{proof}

\paragraph{Step 4. Union of affine transformations: equations over the $\mR_{\ell}/(\frq \cap \mR_{\ell})^t$-module $\mN$.}

Let
\begin{equation}\label{eq:Rnu}
     R_1^{z_{11}} R_2^{z_{12}} \cdots R_N^{z_{1N}} \cdot \nu_1 + \cdots + R_1^{z_{K1}} R_2^{z_{K2}} \cdots R_N^{z_{KN}} \cdot \nu_K = \nu_0,
\end{equation}
be an S-unit equation over the $\mR$-module $\mN$.

Now consider $\mN$ as an $\mR_{\ell}$-module.
Let 
\[
B_i \coloneqq R_i^{p^\ell} \in \mR_{\ell}, \quad i = 1, \ldots, N.
\]
For each tuple $(r_{11}, \ldots, r_{KN}) \in \{0, 1, \ldots, p^{\ell} - 1\}^{KN}$, the solutions of Equation~\eqref{eq:Rnu} satisfying 
\[
(z_{11}, \ldots, z_{KN}) \equiv (r_{11}, \ldots, r_{KN}) \mod p^{\ell}
\]
are exactly solutions to the equation
\begin{equation}\label{eq:Anuprime}
B_1^{z'_{11}} B_2^{z'_{12}} \cdots B_N^{z'_{1N}} \cdot v_1 + \cdots + B_1^{z'_{K1}} B_2^{z'_{K2}} \cdots B_N^{z'_{KN}} \cdot v_K = v_0,
\end{equation}
where $(z'_{11}, \ldots, z'_{KN}) \in \Z^{KN}$ are such that  
\[
(z_{11}, \ldots, z_{KN}) = p^{\ell} \cdot (z'_{11}, \ldots, z'_{KN}) + (r_{11}, \ldots, r_{KN}),
\]
and $v_j \coloneqq R_1^{r_{11}} R_2^{r_{12}} \cdots R_N^{r_{1N}} \cdot \nu_j$ for $j = 1, \ldots, K$.

Therefore, the solution set to the S-unit equation~\eqref{eq:Rnu} over the $\mR$-module $\mN$ is a finite union of affine transformations of solution sets of S-unit equations~\eqref{eq:Anuprime} over the $\mR_{\ell}$-module $\mN$.
From now on we consider $\mN$ as a $\mR_{\ell}$-module.
Note that for any $k \in \N$, the stability property $\mR_{\ell} = \mR_{\ell+k}$ shows that the elements $B_1^{p^k} = R_1^{p^{k + \ell}}, B_1^{-p^k} = R_1^{-p^{k + \ell}}, \ldots, B_N^{p^{k}} = R_N^{p^{k + \ell}}, B_N^{-p^{k}} = R_N^{-p^{k + \ell}}$, generate $\mR_{\ell}$ as a $\Zpe(\oX)$-algebra.

Recall that $\frq$ is a prime ideal of $\mR$ such that $\frq^t \mN = 0$ for some $t \in \N$. 
Thus, $\tfrq \coloneqq \frq \cap \mR_{\ell}$ is a prime ideal of $\mR_{\ell}$, such that $\tfrq^t \mN = 0$. 
Therefore the $\mR_{\ell}$-module $\mN$ can be considered as an $\mR_{\ell}/\tfrq^t$-module.

Denote
\[
\tmA \coloneqq \mR_{\ell}/\tfrq^t, \quad \tmV \coloneqq \mN,
\]
so Equation~\eqref{eq:Anuprime} can be considered as an S-unit equation over the $\tmA$-module $\tmV$:
\begin{equation}\label{eq:Anuvarphi}
A_1^{z_{11}} A_2^{z_{12}} \cdots A_N^{z_{1N}} \cdot v_1 + \cdots + A_1^{z_{K1}} A_2^{z_{K2}} \cdots A_N^{z_{KN}} \cdot v_K = v_0,
\end{equation}
where $A_1, \ldots, A_N$ are the images of $B_1, \ldots, B_N$ under the projection $\mR_{\ell} \rightarrow \mR_{\ell}/\tfrq^t = \tmA$.

\begin{lem}\label{lem:local}
    The ring $\tmA$ is local. The maximal ideal of $\tmA$ is $\tfrq\tmA$ and satisfies $(\tfrq\tmA)^t = 0$.
\end{lem}
\begin{proof}
    First we show that the quotient $\tmA/\tfrq \tmA$ is a field.

    Since $\tfrq$ is a prime ideal of $\mR_{\ell}$, 
    we have $\tfrq \tmA = \tfrq \cdot \mR_{\ell}/\tfrq^{t}$ is a prime ideal of $\tmA = \mR_{\ell}/\tfrq^{t}$.
    Since $p^e = 0 \in \tfrq \tmA$ and $\tfrq \tmA$ is prime, we have $p \in \tfrq \tmA$.
    Note that $\tmA$ is finitely generated as a $\Zpe(\oX)$-module, and $p \cdot (\tmA/\tfrq \tmA) = 0$.
    Therefore $\tmA/\tfrq \tmA$ is finitely generated as a module over $\Zpe(\oX)/p\Zpe(\oX) = \F_p(\oX)$.
    In other words, the ring $\tmA/\tfrq \tmA$ is a finite extension of the field $\F_p(\oX)$.
    Thus, $\tmA/\tfrq \tmA$ is an integral domain\footnote{A ring $R$ is an \emph{integral domain} if $xy = 0 \implies x = 0 \text{ or } y = 0$, for all $x, y \in R$. If $I$ is a prime ideal of a ring $R$, then $R/I$ is an integral domain.} that is also a finite extension of a field, it is therefore also a field~\cite[Corollary~4.7]{eisenbud2013commutative}.

    The fact that $(\tfrq\tmA)^{t} = 0$ follows from the definition $\tmA = \mR_{\ell}/\tfrq^t$.
    Next we show that every $a \in \tmA \setminus \tfrq \tmA$ is invertible. Since the quotient $\tmA/\tfrq \tmA$ is a field, there exists $b \in \tmA$ such that $ab \equiv -1 \mod \tfrq \tmA$. Then $(ab + 1)^{t} = 0$.
    We can expand $(ab + 1)^{t} = \sum_{j = 0}^{t} \binom{t}{j} a^j b^j$ and write it as $1 + af$ for some $f \in \tmA$.
    Therefore $af = -1$, so $-f$ is the inverse of $a$.
    We conclude that $\tmA$ is local with maximal ideal $\tfrq\tmA$.
\end{proof}

This completes all the ingredients for the proof of Proposition~\ref{prop:wrapperlocalization}:

\propwrapperlocalization*
\begin{proof}
    By the four steps above, the solution set of an S-unit equation in a $\Zpe[X_1^{\pm}, \ldots, X_N^{\pm}]$-module $\mM$ can be written as a finite positive Boolean combination of affine transformation of solutions set of S-unit equations~\eqref{eq:Anuvarphi} over $\tmA$-modules $\tmV$.
    Take the ring $\tmA$, the elements $A_1, \ldots, A_N \in \tmA$, and the $\tmA$-module $\tmV$.
    We show that they satisfy the properties~(i)-(v).
    Property~(i) follows from Lemma~\ref{lem:local}.
    Property~(ii) follows from the definition of $\tmA$.
    Property~(iii) follows from Lemma~\ref{lem:AYfgmod}, since finite generation does not change upon taking quotients or submodules.
    For property~(iv), recall that for any $k \in \N$, the elements $B_1^{p^k}, B_1^{-p^k} \ldots, B_N^{p^{k}}, B_N^{-p^{k}}$ generate $\mR_{\ell}$ as a $\Zpe(\oX)$-algebra.
    Consequently for any $k \in \N$, their projections $A_1^{p^k}, A_1^{-p^k} \ldots, A_N^{p^{k}}, A_N^{-p^{k}}$ generate $\tmA = \mR_{\ell}/\tfrq^t$ as a $\Zpe(\oX)$-algebra.
    Property~(v) follows from the fact that $\tmA$ contains $\Zpe(\oX)$ and that $\tmV$ is finitely generated as a $\Zpe(\oX)$-module, which is a consequence of Lemma~\ref{lem:AYfgmod}.
\end{proof}

By Proposition~\ref{prop:wrapperlocalization} and Proposition~\ref{prop:internormal}, we can now focus on proving $p$-normality of the solution set of an S-unit equation in the $\tmA$-module $\tmV$.

\subsection{Reduction to $\mA$ acting on free $\Zpe(\oX)$-modules}\label{subsec:free}

In this subsection we further reduce solving S-unit equations over the $\tmA$-module $\tmV$, to solving S-unit equations over $\mA$-modules $\mV$, where $\mA$ is some $\Zpi(\oX)$-algebra and $\mV$ is free as a $\Zpi(\oX)$-module (but not necessarily free as an $\mA$-module).
More precisely, we will show the following:

\begin{restatable}{prop}{propwrapperfreeness}\label{prop:wrapperfreeness}
    Let $\mZ$ be the solution set of an S-unit equation over an $\tmA$-module $\tmV$, where $\tmA, \tmV$ satisfy the properties in Proposition~\ref{prop:wrapperlocalization}.
    Then $\mZ$ can be effectively written as a finite intersection $\bigcap_j \mZ_{j}$, where each $\mZ_{j}$ is the solution set of an S-unit equation over some $\mA$-module $\mV$, satisfying
    \begin{enumerate}[nosep, label = (\roman*)]
        \item $\mA$ is local, its maximal ideal $\frm$ satisfies $\frm^t = 0$ for some $t \geq 1$.
        \item $\mA$ is effectively represented as a $\Zpi(\oX)$-algebra for some $i \in \N$. 
        \item As a $\Zpi(\oX)$-module, $\mA$ is finitely generated.
        \item As a $\Zpi(\oX)$-module, $\mV$ is isomorphic to $\Zpi(\oX)^d$ for some $d \in \N$.
        In particular, every element in $\mA$ acts as a $\Zpi(\oX)$-linear transformation on $\mV \cong \Zpi(\oX)^d$.
    \end{enumerate}
\end{restatable}

\paragraph{Step 1: decomposition of $\tmV$ as $\Zpe(\oX)$-module.}

Recall that $\Zpe(\oX)$ is a principal ideal ring (PIR) whose only ideals are $\gen{1}, \gen{p}, \gen{p^2}$, $\ldots$, $\gen{p^{e-1}}, \gen{p^{e}} = \{0\}$.
This gives us a characterization of the structure of $\tmV$ as a $\Zpe(\oX)$-module:

\begin{lem}[{Structure theorem of finitely generated module over a PIR~\cite[Theorem~15.33]{brown1993matrices}}]\label{lem:PIR}
    As a $\Zpe(\oX)$-module, $\tmV$ can be effectively decomposed as a direct sum
        \begin{equation}\label{eq:VdecomposePIR}
        \tmV = \Zp(\oX)^{d_1} \oplus \Zpt(\oX)^{d_2} \oplus \cdots \oplus \Zpe(\oX)^{d_e}
        \end{equation}
    for some $d_1, \ldots, d_e \in \N$.
    Here, the $\Zpe(\oX)$-module $\Zpi(\oX), i = 1, \ldots, e$, is equal to the quotient $\Zpe(\oX)/p^i\Zpe(\oX)$.
\end{lem}

If the action of $\tmA$ stabilizes each component $\Zpi(\oX)^{d_i}$ in the decomposition of the $\Zpe(\oX)$-module $\tmV$, then Equation~\eqref{eq:VdecomposePIR} is also a decomposition of $\tmV$ as an $\tmA$-module.
In this case, Proposition~\ref{prop:wrapperfreeness} easily follows by taking $\mA \coloneqq \tmA/p^i\tmA$ and $\mV \coloneqq \Zpi(\oX)^{d_i}$.
However, in general $\tmA$ does not stabilize each $\Zpi(\oX)^{d_i}$, that is, $\tmA \cdot \Zpi(\oX)^{d_i} \not\subseteq \Zpi(\oX)^{d_i}$.
Therefore Equation~\eqref{eq:VdecomposePIR} is not a decomposition of $\tmV$ as an $\tmA$-module.
The main idea of this subsection is to find a different decomposition of the $\Zpe(\oX)$-module $\tmV$ by changing the ``basis'', so that the components of the new decomposition $\tmV = \Zp(\oX)^{d_1} \oplus \Zpt(\oX)^{d_2} \oplus \cdots \oplus \Zpe(\oX)^{d_e}$ are stabilized by $\tmA$.
To achieve this, we will exploit property~(iv) of $\tmA$ from Proposition~\ref{prop:wrapperlocalization}, so that we can freely take $p^{\N}$-th power of the generators $A_i$.

For any $\Zpe(\oX)$-module $\mN$, let $\End(\mN)$ denote the set of all $\Zpe(\oX)$-linear maps from $\mN$ to $\mN$, then $\End(\mN)$ is a $\Zpe(\oX)$-algebra.
Furthermore, let $\Aut(\mN)$ denote the set of all \emph{invertible} $\Zpe(\oX)$-linear maps from $\mN$ to $\mN$.
In particular, if $\mN = \Zpe(\oX)^d$, then $\End(\mN)$ and $\Aut(\mN)$ are respectively the matrix sets $\M_{d \times d}(\Zpe(\oX))$ and $\GL_{d}(\Zpe(\oX))$.
In general, for $\tmV = \Zp(\oX)^{d_1} \oplus \cdots \oplus \Zpe(\oX)^{d_e}$, the structures of $\End(\tmV)$ and $\Aut(\tmV)$ are more complicated.

Since $A_1, \ldots, A_N \in \tmA$ are invertible, and $\tmV$ is an $\tmA$-module, each $A_i, i = 1, \ldots, N$, can be considered as an element of $\Aut(\tmV)$ by the map
\begin{align*}
    A_i \colon \tmV & \rightarrow \tmV, \\
    v & \mapsto A_i \cdot v.
\end{align*}
Furthermore, $A_1, \ldots, A_N$ commute pairwise. 

As in Lemma~\ref{lem:PIR}, for $i = 1, \ldots, e$, let $\{\bepsilon_{i1}, \ldots, \bepsilon_{i d_i}\}$ be a $\Zpi(\oX)$-basis of the component $\Zpi(\oX)^{d_i}$ of $\tmV$ in the decomposition~\eqref{eq:VdecomposePIR}.
Then, the set $\{\bepsilon_{11}, \ldots, \bepsilon_{1 d_1}, \ldots, \bepsilon_{e1}, \ldots, \bepsilon_{e d_e}\}$ generates $\tmV$ as a $\Zpe(\oX)$-module.
For any $f \in \End(\tmV)$ and each $k = 1, \ldots, e; l = 1, \ldots, d_k$, the element $f \cdot \bepsilon_{k l}$ can be written uniquely as a sum
\begin{equation}\label{eq:matform}
f \cdot \bepsilon_{k l} = \sum_{i = 1}^e \sum_{j = 1}^{d_i} z_{ij, kl} \bepsilon_{ij},
\end{equation}
where 
\begin{equation}\label{eq:memcond}
    z_{ij, kl} \in \Z_{/p^i}(\oX) \quad \text{ for } i = 1, \ldots, e,\; j = 1, \ldots, d_i.
\end{equation}
Furthermore, since 
\[
p^{k} \sum_{i = 1}^e \sum_{j = 1}^{d_i} z_{ij, kl} \bepsilon_{ij} = p^{k} (f \cdot \bepsilon_{kl}) = f \cdot (p^k \bepsilon_{kl}) = 0,
\]
we must have 
\begin{equation}\label{eq:divcond}
    p^{i - k} \mid z_{ij, kl} \quad \text{ for all } i > k,\; j = 1, \ldots, d_i.
\end{equation}

It is easy to see that any tuple $\left(z_{ij, kl}\right)_{i = 1, \ldots, e; j = 1, \ldots, d_i; k = 1, \ldots, e; l = 1, \ldots, d_k}$ satisfying~\eqref{eq:memcond} and \eqref{eq:divcond} defines an element $f \in \End(\tmV)$, by extending $\Zpe(\oX)$-linearly the Definition~\eqref{eq:matform} from the basis $\{\bepsilon_{11}, \ldots, \bepsilon_{1 d_1}, \ldots, \bepsilon_{e1}, \ldots, \bepsilon_{e d_e}\}$ to the whole module $\tmV$.

Note that $\left(z_{ij, kl}\right)_{i = 1, \ldots, e; j = 1, \ldots, d_i; k = 1, \ldots, e; l = 1, \ldots, d_k}$, are the coefficients of $f$ in its matrix form under the basis $\{\bepsilon_{11}, \ldots, \bepsilon_{1 d_1}, \ldots, \bepsilon_{e1}, \ldots, \bepsilon_{e d_e}\}$.
From now on, for any $f \in \End(\tmV)$ and $i,k \in \{1, \ldots, e\}$, we will let 
\[
f_{ik} \coloneqq \left(z_{ij, kl}\right)_{j = 1, \ldots, d_i; l = 1, \ldots, d_k} \in \M_{d_i \times d_k}(\Zpi(\oX))
\]
denote the $(i,k)$-th block of $f$.
Taking into account the divisibility constraints~\eqref{eq:divcond}, the map $f$ can be written as a block matrix
\begin{equation}\label{eq:deffinM}
    \begin{pmatrix}
        M_{11} & M_{12} & M_{13} & \cdots & M_{1e} \\
        p M_{21} & M_{22} & M_{23} & \cdots & M_{2e} \\
        p^2 M_{31} & p M_{32} & M_{33} & \cdots & M_{3e} \\
        \vdots & \vdots & \vdots & \ddots & \vdots \\
        p^{e-1} M_{e1} & p^{e-2} M_{e2} & p^{e-3} M_{e3} & \cdots & M_{ee} \\
    \end{pmatrix},
\end{equation}
where $M_{ik} \in \M_{d_i \times d_k}(\Zpi(\oX))$ for all $i,k \in \{1, \ldots, e\}$.
Then, $f_{ik} = p^{i-k} M_{ik}$ for $i > k$, and $f_{ik} = M_{ik}$ for $i \leq k$.
Note that for any $f \in \End(\tmV)$, the blocks $f_{ik}$ are uniquely defined, but the choices of matrices $M_{ik}, i > k$ are not unique.

For two endomorphisms $f, g \in \End(\tmV)$, their composition $g f$ can be computed by multiplying their corresponding matrices.
That is, 
\[
(g f)_{ik} = \sum_{j = 1}^e \left(g_{ij} f_{jk} \mod p^i \right).
\]
Note that although $f_{jk}$ has coefficients in $\Z_{/p^j}(\oX)$, the expression $\left(g_{ij} f_{jk} \mod p^i \right)$ is well defined even when $i > j$.
Indeed, when $i > j$, we have $p^{i-j} \mid g_{ij}$, so 
\[
f'_{jk} \equiv f''_{jk} \mod p^j \implies g_{ij} f'_{jk} \equiv g_{ij} f''_{jk} \mod p^i.
\]

The next lemma shows that if $f$ is invertible, then so are its diagonal blocks.

\begin{lem}
    If $f \in \Aut(\tmV)$ then $M_{11} \in \GL_{d_1}(\Zp(\oX)),\; M_{22} \in \GL_{d_2}(\Zpt(\oX)),\; \ldots,\; M_{ee} \in \GL_{d_e}(\Zpe(\oX))$.
\end{lem}
\begin{proof}
    Suppose $f \in \Aut(\tmV)$.
    Write $f^{-1}$ as the matrix.
    \[
    \begin{pmatrix}
        M'_{11} & M'_{12} & M'_{13} & \cdots & M'_{1e} \\
        p M'_{21} & M'_{22} & M'_{23} & \cdots & M'_{2e} \\
        p^2 M'_{31} & p M'_{32} & M'_{33} & \cdots & M'_{3e} \\
        \vdots & \vdots & \vdots & \ddots & \vdots \\
        p^{e-1} M'_{e1} & p^{e-2} M'_{e2} & p^{e-3} M'_{e3} & \cdots & M'_{ee} \\
    \end{pmatrix}.
    \]
    Then computing the product $f \cdot f^{-1}$ modulo $p$, we obtain $M_{11} M'_{11} \equiv I \mod p, \ldots, M_{ee} M'_{ee} \equiv I \mod p$.
    This shows that $M_{11}, \ldots, M_{ee}$ are invertible modulo $p$.
    We claim that for any $a \geq 1$, invertibility of a matrix $M \in \GL_d(\Zpi(\oX))$ modulo $p^a$ implies invertibility modulo $p^{a+1}$.
    Indeed, suppose $M M' \equiv I \mod p^a$, then writing $MM' = I + p^a A$, we have
    \[
    M(M' - p^a M' A) = MM' - p^a MM' A = (I + p^a A) - p^a(I + p^a A) A \equiv I \mod p^{2a}.
    \]
    Therefore $M$ is invertible modulo $p^{a+1} \mid p^{2a}$.
    Applying this iteratively for $a = 1, 2, \ldots$, we conclude that $M_{11} \in \GL_{d_1}(\Zp(\oX)), \ldots, M_{ee} \in \GL_{d_e}(\Zpe(\oX))$.
\end{proof}

The main idea of proving Proposition~\ref{prop:wrapperfreeness} is to simultaneously block-diagonalize the matrix forms of $A_1, \ldots, A_N$, up to taking $p$-th powers.
There are two main difficulties.
The first is that the matrix blocks do not commute, nor do they have the same dimension.
The second is that the diagonalization can only use conjugators in $\Aut(\tmV)$, which have certain divisibility constraints for its blocks in the lower-left part of the matrix.
For these reasons, we will perform the diagonalization step by step on its blocks.
First, we simultaneously lower-triangularize the matrices $A_1, \ldots, A_N$.
Some of the arguments used in the following step will also appear in the next subsection.

\paragraph{Step 2: simultaneous block lower-triangularization.}

In this step, we simultaneously lower-triangularize the matrices $A_1, \ldots, A_N$, up to taking their $p^{\N}$-th powers.
More precisely, we will show the following.

\begin{restatable}[simultaneous block lower-triangularization]{prop}{proptriangularize}\label{prop:triangularize}
    Let $A_1, \ldots, A_N$ be pairwise commuting elements of $\Aut(\tmV)$.
    Then there exist effectively computable $\ell \in \N$ and $R \in \Aut(\tmV)$, such that $R^{-1}A_1^{p^{\ell}}R, \ldots, R^{-1}A_N^{p^{\ell}}R$ are block lower-triangular (i.e.\ their $(j,k)$-th blocks are zero for all $1 \leq j < k \leq e$).
\end{restatable}

To prove Proposition~\ref{prop:triangularize}, we will increase the $p$-adic valuation of the upper-right blocks step by step.
Let $a \in \N$ and $1 \leq b < c \leq e$. 
We say that an element $f \in \Aut(\tmV)$ is \emph{$(a,b,c)$-triangular}, if it satisfies
\begin{align}\label{eq:abctriangular}
    p^{a} \mid f_{jk} & \text{ for all } 1 \leq j < k \leq e, \nonumber\\
    p^{a+1} \mid f_{jk} & \text{ for all } 1 \leq j < k \leq e, \; k \geq c +1, \nonumber\\
    p^{a+1} \mid f_{jk} & \text{ for all } b<j< c, \; k = c.
\end{align}
This is in addition to the divisibility constraints~\eqref{eq:divcond} that all elements of $\End(\tmV)$ are subject to.
In this case, $f$ can be written in the matrix form
\[
    \begin{pmatrix}
        M_{11} & p^a M_{12} & \cdots & & & p^a M_{1c} & p^{a+1} M_{1(c+1)} & \cdots \\    
         &  \ddots & & & & \vdots & \vdots \\
         &  & M_{bb} & & & p^a M_{bc} & p^{a+1} M_{b(c+1)} & \cdots \\
         & & & \ddots & & p^{a+1} M_{(b+1)c} & p^{a+1} M_{(b+1)(c+1)} & \cdots \\
        & & & & \ddots & \vdots & \vdots \\
        &*&&&& M_{cc} & p^{a+1} M_{c(c+1)} & \cdots \\
        &&&&& & \ddots & \cdots \\
    \end{pmatrix},
\]
where $*$ denotes entries satisfying the divisibility constraints~\eqref{eq:divcond}.
In particular, every entry represented by $*$ is divisible by $p$.
Note that if $b+1 = c$ then the term $p^{a+1} M_{(b+1)c}$ is overwritten by $M_{cc}$.

\begin{obs}
    The composition of two $(a,b,c)$-triangular automorphisms is $(a,b,c)$-triangular.
\end{obs}

\begin{lem}[single step of triangularization]\label{lem:singletriangularization}
    Let $A_1, \ldots, A_N$ be pairwise commuting elements of $\Aut(\tmV)$, let $a \in \N$ and $1 \leq b < c \leq e$. 
    Suppose that $A_1, \ldots, A_N$ are $(a,b,c)$-triangular.
    Then there exist effectively computable $\ell \in \N$ and $R \in \Aut(\tmV)$, such that $R^{-1}A_1^{p^{\ell}}R, \ldots, R^{-1}A_N^{p^{\ell}}R$ are $(a,b,c)$-triangular and their $(b,c)$-th blocks are divisible by $p^{a+1}$.
\end{lem}
\begin{proof}
    Without loss of generality suppose $a < b$. Otherwise $a \geq b$, so for all $f \in \{A_1, \ldots, A_N\}$, we have $f_{bc} \in p^a \cdot \M_{d_b \times d_c}(\Z_{/p^b}(\oX)) \subseteq p^b \cdot \M_{d_b \times d_c}(\Z_{/p^b}(\oX)) = \{0\}$.
    This means that the $(b,c)$-th blocks of $A_1, \ldots, A_N$ are zero, and hence divisible by $p^{a+1}$.
    So we can take $R$ to be the identity and $\ell = 0$, and there is nothing to prove.

    Let $A_i \in \{A_1, \ldots, A_N\}$.
    Since $A_i$ is $(a,b,c)$-triangular, it can be written in the matrix form
    \[
    \begin{pmatrix}
        M_{i11} & p^a M_{i12} & \cdots & & & p^a M_{i1c} & p^{a+1} M_{i1(c+1)} & \cdots \\    
         &  \ddots & & & & \vdots & \vdots \\
         &  & M_{ibb} & & & p^a M_{ibc} & p^{a+1} M_{ib(c+1)} & \cdots \\
         & & & \ddots & & p^{a+1} M_{i(b+1)c} & p^{a+1} M_{i(b+1)(c+1)} & \cdots \\
        & & & & \ddots & \vdots & \vdots \\
        &*&&&& M_{icc} & p^{a+1} M_{ic(c+1)} & \cdots \\
        &&&&& & \ddots & \cdots \\
    \end{pmatrix}.
    \]
    Let $m \in \N$.
    We claim that $A_i^m$ can be written in the form
    \begin{equation}\label{eq:Aim}
    \begin{pmatrix}
        M'_{i11} & p^a M'_{i12} & \cdots & & & p^a M'_{i1c} & p^{a+1} M'_{i1(c+1)} & \cdots \\    
         &  \ddots & & & & \vdots & \vdots \\
         &  & M'_{ibb} & & & p^a M'_{ibc} & p^{a+1} M'_{ib(c+1)} & \cdots \\
         & & & \ddots & & p^{a+1} M'_{i(b+1)c} & p^{a+1} M'_{i(b+1)(c+1)} & \cdots \\
        & & & & \ddots & \vdots & \vdots \\
        &*&&&& M'_{icc} & p^{a+1} M'_{ic(c+1)} & \cdots \\
        &&&&& & \ddots & \cdots \\
    \end{pmatrix}.
    \end{equation}
    where
    \begin{equation}\label{eq:Mijjcong}
    \left(A_i^{m}\right)_{jj} = M'_{ijj} \equiv M_{ijj}^{m} \mod p, \quad j = 1, \ldots, e,
    \end{equation}
    and $\left(A_i^{m}\right)_{bc} = p^a M'_{ibc}$ with
    \begin{equation}\label{eq:Mibccong}
    M'_{ibc} \equiv \sum_{k=0}^{m-1} M_{ibb}^k M_{ibc} M_{icc}^{m-1-k} \mod p.
    \end{equation}
    Indeed, the congruence~\eqref{eq:Mijjcong} is obvious from the fact that $A_i$ is block upper-triangular modulo $p$.
    (Recall that every entry represented by $*$ is divisible by $p$).
    The congruence~\eqref{eq:Mibccong} follows by induction on $m$ the following way.
    For $m = 1$, the congruence~\eqref{eq:Mibccong} is obvious.
    Suppose~\eqref{eq:Mibccong} holds for $m$, then $A_i^{m+1} \equiv A_i \cdot A_i^{m} \mod p^{a+1}$ yields the recurrence 
    \begin{align*}
    \left(A_i^{m+1}\right)_{bc} & \equiv p \cdot p^a ( \cdots ) + M_{ibb} \left(A_i^{m}\right)_{bc} + p^a \cdot p^{a+1} ( \cdots ) + p^aM_{ibc} \left(A_i^{m}\right)_{cc} + p^{a+1} \cdot p \cdot ( \cdots ) \\
    & \equiv M_{ibb} \left(A_i^{m}\right)_{bc} + p^aM_{ibc} M_{icc}^{m} \mod p^{a+1}.
    \end{align*}
    Therefore for $m + 1$, we have
    \[
    p^a M'_{ibc} \equiv  M_{ibb} \left(p^a\sum_{k=0}^{m-1} M_{ibb}^k M_{ibc} M_{icc}^{m-1-k}\right) + p^a M_{ibc} M_{icc}^{m} \equiv p^a \sum_{k=0}^{m} M_{ibb}^k M_{ibc} M_{icc}^{m-k} \mod p^{a+1},
    \]
    which yields the congruence~\eqref{eq:Mibccong} for $m+1$.
    
    We now take $m \coloneqq p^{\ell}$ and consider the matrices $A_i^{p^{\ell}}, i = 1, \ldots, N$.
    Let $R \in \Aut(\tmV)$ be the automorphism defined by
    \[
    R_{jk} = 
    \begin{cases}
        I \quad & 1\leq j = k \leq e, \\
        p^a Q \quad & j = b, k = c, \\
        0 \quad & \text{otherwise},
    \end{cases}
    \]
    for some $Q \in \M_{d_b \times d_c}(\Z_{/p^b}(\oX))$ to be determined later.
    That is,
    \[
    R = 
    \begin{pmatrix}
        I & 0 & \cdots & & & 0 & 0 & \cdots \\    
         &  \ddots & & && \vdots & \vdots \\
         &  & I & & & p^a Q & 0 & \cdots \\
         & & & \ddots & & 0 & 0 & \cdots \\
        & & &  & \ddots & \vdots \\
        &0&&&& I & 0 & \cdots \\
        &&&&&& \ddots & \cdots \\
    \end{pmatrix},
    \quad
    R^{-1} = 
    \begin{pmatrix}
        I & 0 & \cdots & & & 0 & 0 & \cdots \\    
         &  \ddots & & && \vdots & \vdots \\
         &  & I & & & -p^a Q & 0 & \cdots \\
         & & & \ddots & & 0 & 0 & \cdots \\
        & & &  & \ddots & \vdots \\
        &0&&&& I & 0 & \cdots \\
        &&&&&& \ddots & \cdots \\
    \end{pmatrix}.
    \]
    Both $R, R^{-1}$ are $(a,b,c)$-triangular, and consequently all $R^{-1} A_i^{p^{\ell}} R, i = 1, \ldots, N,$ are $(a,b,c)$-triangular.
    We want to find $Q$ such that $p^{a+1} \mid \left(R^{-1} A_i^{p^{\ell}} R\right)_{bc}$ for all $i = 1, \ldots, N$.

    From the matrix form~\eqref{eq:Aim}, we can directly compute
    \begin{multline*}
    \left(R^{-1} A_i^{p^{\ell}} R\right)_{bc} \equiv p^a\left(M'_{ibc} - Q M'_{icc} + M'_{ibb}Q\right) \equiv p^a \left(\sum_{k=0}^{p^{\ell}-1} M_{ibb}^k M_{ibc} M_{icc}^{p^{\ell}-1-k} - Q M_{icc}^{p^{\ell}} + M_{ibb}^{p^{\ell}} Q \right) \\
    \mod p^{a+1}.
    \end{multline*}
    Therefore in order for $p^{a+1} \mid \left(R^{-1} A_i^{p^{\ell}} R\right)_{bc}$ to be satisfied, it suffices to find 
    \[
    (Q \mod p) \in \M_{d_b \times d_c}(\F_p(\oX))
    \]
    such that 
    \begin{equation}\label{eq:solveQmat}
    \sum_{k=0}^{p^{\ell}-1} M_{ibb}^k M_{ibc} M_{icc}^{p^{\ell}-1-k} \equiv Q M_{icc}^{p^{\ell}} - M_{ibb}^{p^{\ell}} Q \mod p.
    \end{equation}
    We now without loss of generality write $Q$ instead of $(Q \mod p)$.
    
    For each $i = 1, \ldots, N$, let $\varphi_i$ denote the $\F_p(\oX)$-linear transformation
    \begin{align*}
        \varphi_i \colon \M_{d_b \times d_c}(\F_p(\oX)) & \rightarrow \M_{d_b \times d_c}(\F_p(\oX)) \\
        S & \mapsto S M_{icc} - M_{ibb} S.
    \end{align*}
    Here, $\M_{d_b \times d_c}(\F_p(\oX))$ is considered as a $d_bd_c$-dimensional vector space over $\F_p(\oX)$.
    By Lemma~\ref{lem:iterphi} below, Equation~\eqref{eq:solveQmat} can be rewritten as
    \begin{equation}\label{eq:solveQvarphi}
    \varphi_{i}^{p^{\ell}-1} (M_{ibc}) = \varphi_{i}^{p^{\ell}}(Q).
    \end{equation}
    For all $i,j \in \{1, \ldots, N\}$, since $A_i$ and $A_j$ commute, we have 
    \begin{align*}
    M_{ibb}M_{jbb} \equiv (A_iA_j)_{bb} & \equiv (A_jA_i)_{bb} \equiv M_{jbb}M_{ibb} \mod p, \\
    M_{icc}M_{jcc} \equiv (A_iA_j)_{cc} & \equiv (A_jA_i)_{cc} \equiv M_{jcc}M_{icc} \mod p, \\
    p^a(M_{ibb} M_{jbc} + M_{ibc} M_{jcc}) \equiv (A_iA_j)_{bc} & \equiv (A_jA_i)_{bc} \equiv p^a(M_{jbb} M_{ibc} + M_{jbc} M_{icc}) \mod p^{a+1}.
    \end{align*}
    Therefore for all $S \in \M_{d_b \times d_c}(\F_p(\oX))$, we have
    \begin{multline*}
    \varphi_i\varphi_j(S) = S M_{jcc} M_{icc} - M_{jbb} S M_{icc} - M_{ibb} S M_{jcc} + M_{ibb} M_{jbb} S \\
    = S M_{icc} M_{jcc} - M_{jbb} S M_{icc} - M_{ibb} S M_{jcc} + M_{jbb} M_{ibb} S = \varphi_j\varphi_i(S).
    \end{multline*}
    Hence, $\varphi_i$ and $\varphi_j$ commute for all $i,j \in \{1, \ldots, N\}$. 
    Furthermore,
    \[
    \varphi_i(M_{jbc}) = M_{jbc} M_{icc} - M_{ibb} M_{jbc} \equiv M_{ibc} M_{jcc} - M_{jbb} M_{ibc} = \varphi_j(M_{ibc}) \mod p.
    \]
    Hence $\varphi_i(M_{jbc}) = \varphi_j(M_{ibc})$ for all $i,j \in \{1, \ldots, N\}$.
    By Lemma~\ref{lem:invertphi} below, we can compute $Q \in \M_{d_b \times d_c}(\F_p(\oX))$ such that
    \[
    \varphi_1^{p^{\ell}}(Q) = \varphi_1^{p^{\ell}-1}(M_{1bc}), \ldots, \varphi_N^{p^{\ell}}(Q) = \varphi_N^{p^{\ell}-1}(M_{Nbc}).
    \]
    We have thus found $Q$ that satisfies Equation~\eqref{eq:solveQvarphi} (and hence Equation~\eqref{eq:solveQmat}) for all $i = 1, \ldots, N$.
\end{proof}

\begin{lem}\label{lem:iterphi}
    Let $d, d'$ be positive integers.
    Let $C \in \GL_{d}(\F_p(\oX)), B \in \GL_{d'}(\F_p(\oX))$. 
    Define the $\F_p(\oX)$-linear transformation
    \begin{align*}
        \varphi \colon \M_{d \times d'}(\F_p(\oX)) & \rightarrow \M_{d \times d'}(\F_p(\oX)) \\
        M & \mapsto M C - B M.
    \end{align*}
    Then for any $\ell \geq 1$, we have
    \begin{equation}\label{eq:iterpm1}
    \varphi^{p^\ell-1}(M) = \sum_{k=0}^{p^\ell-1} B^k M C^{p^\ell-1-k},
    \end{equation}
    and
    \begin{equation}\label{eq:iterpp}
    \varphi^{p^{\ell}}(M) = M C^{p^{\ell}} - B^{p^{\ell}} M.
    \end{equation}
\end{lem}
\begin{proof}
    We prove by induction on $m$ that
    \begin{equation}\label{eq:iterm}
    \varphi^{m}(M) = \sum_{k=0}^{m} (-1)^{k} \binom{m}{k} B^k M C^{m-k}.
    \end{equation}
    For $m = 1$, Equation~\eqref{eq:iterm} holds by the definition of $\varphi$.
    Suppose Equation~\eqref{eq:iterm} holds for $m$, then
    \begin{align*}
        \varphi^{m+1}(M) & = \varphi\left(\sum_{k=0}^{m} (-1)^{k} \binom{m}{k} B^k M C^{m-k}\right) \\
        & = \left(\sum_{k=0}^{m} (-1)^{k} \binom{m}{k} B^k M C^{m-k}\right) C - B \left(\sum_{k=0}^{m} (-1)^{k} \binom{m}{k} B^k M C^{m-k}\right) \\
        & = \sum_{k=0}^{m} (-1)^{k} \binom{m}{k} B^k M C^{m+1-k} - \sum_{k=0}^{m} (-1)^{k} \binom{m}{k} B^{k+1} M C^{m-k} \\
        & = \sum_{k=0}^{m+1} \left((-1)^{k} \binom{m}{k} - (-1)^{k-1} \binom{m}{k-1}\right) B^k M C^{m+1-k} \\
        & = \sum_{k=0}^{m+1} (-1)^{k} \left(\binom{m}{k} + \binom{m}{k-1}\right) B^k M C^{m+1-k} \\
        & = \sum_{k=0}^{m+1} (-1)^{k} \binom{m+1}{k} B^k M C^{m+1-k}.
    \end{align*}
    This proves Equation~\eqref{eq:iterm} for $m+1$.

    Specializing Equation~\eqref{eq:iterm} for $m = p^{\ell}-1$ and noticing 
    \[
    (-1)^{k} \binom{p^{\ell}-1}{k} \equiv (-1)^{k} \cdot \frac{p^{\ell}-1}{1} \cdot \frac{p^{\ell}-2}{2} \cdots \frac{p^{\ell}-k}{k} \equiv 1 \mod p
    \]
    for prime $p$, we obtain~\eqref{eq:iterpm1}.

    Specializing Equation~\eqref{eq:iterm} for $m = p^{\ell}$.
    By Lucas' theorem~\cite{granville1997arithmetic}, we have
    \[
    \binom{p^{\ell}}{k} \equiv 
    \begin{cases}
        1 \mod p, \quad k = 0 \text{ or } p^{\ell}, \\
        0 \mod p, \quad 1 \leq k \leq p^{\ell}-1.
    \end{cases}
    \]
    Thus we obtain~\eqref{eq:iterpp}.
\end{proof}

\begin{lem}\label{lem:invertphi}
    Let $D$ be a positive integer and $V$ be a $D$-dimensional $\F_p(\oX)$-vector space.
    Let $M_1, \ldots, M_N \in V$, and let $\varphi_1, \ldots, \varphi_N$ be pairwise commuting elements of $\End(V) = \M_{d \times d}(\F_p(\oX))$, such that
    \begin{align*}
        \varphi_i(M_j) = \varphi_j(M_i)
    \end{align*}
    for all $i, j \in \{1, \ldots, N\}$.
    Then there exist effectively computable $Q \in V$ and $\ell \in \N$, such that
    \[
    \varphi_1^{p^{\ell}}(Q) = \varphi_1^{p^{\ell}-1}(M_1), \ldots, \varphi_N^{p^{\ell}}(Q) = \varphi_N^{p^{\ell}-1}(M_N).
    \]
\end{lem}
\begin{proof}
    Let $\mF$ be the $\F_p(\oX)$-subalgebra of $\End(V)$ generated by $\varphi_1, \ldots, \varphi_N$ and the identity element.
    Since the endomorphisms $\varphi_1, \ldots, \varphi_N$ commute pairwise, $\mF$ is commutative.
    Furthermore, $\mF$ has finite dimension over $\F_p(\oX)$, since $\End(V)$ has finite dimension over $\F_p(\oX)$.
    Therefore, $\mF$ is Artinian.
    
    By~\cite[Corollary~2.16]{eisenbud2013commutative}, the Artinian ring $\mF$ can be decomposed into a direct product of local rings
    \[
    \mF = \mR_1 \times \mR_2 \times \cdots \times \mR_q.
    \]
    Let 
    \[
    \mF_i \coloneqq \{0\} \times \cdots \times \{0\}  \times \mR_i \times \{0\} \times \cdots \times \{0\}
    \]
    for $i = 1, \ldots, q$.
    Since $\F_p(\oX) \cdot \mF_i \subseteq \mF \cdot \mF_i \subseteq \mF_i$, each $\mF_i$ is a $\F_p(\oX)$-subalgebra of $\mF$.
    This means that $\mF$ is a direct sum of the local $\F_p(\oX)$-subalgebras $\mF_1, \mF_2, \ldots, \mF_q$:
    \begin{equation}\label{eq:decomposemF}
    \mF = \mF_1 \oplus \mF_2 \oplus \cdots \oplus \mF_q.
    \end{equation}
    See~\cite[Section~3.2,~p.501]{babai1996multiplicative} for an effective algorithm for finding the decomposition~\eqref{eq:decomposemF}.
    For each $i$, the maximal ideal of the local ring $\mF_i$ is
    \[
    \rad(\mF_i) \coloneqq \{f \in \mF_i \mid \exists m \geq 1, f^m = 0\}.
    \]
    Indeed, $\rad(\mF_i)$ is the intersection of all prime ideals of $\mF_i$~\cite[Proposition~1.8]{atiyah1969introduction}, which is exactly the maximal ideal of the Artinian local ring $\mF_i$~\cite[Proposition~8.1]{atiyah1969introduction}.
    Furthermore, there exists an effectively computable $t_i \in \N$, such that $\rad(\mF_i)^{t_i} = 0$~\cite[Proposition~8.6]{atiyah1969introduction}.
    
    For any $i \in \{1, \ldots, N\}$, since $\varphi_i \in \mF$, we can write 
    \[
    \varphi_i = (f_{i1}, f_{i2}, \ldots, f_{iq}) \in \mF_1 \oplus \mF_2 \oplus \cdots \oplus \mF_q
    \]
    according to this decomposition.
    Let $\ell \in \N$ be such that $p^{\ell} \geq 1+\max\{t_1, \ldots, t_q\}$.

    Since $\mF$ acts $\F_p(\oX)$-linearly on $V$, we can decompose 
    \[
    V = V_1 \oplus \cdots \oplus V_q
    \]
    according to the decomposition of $\mF$, that is, $V_1 \coloneqq \mF_1 \cdot V$, $\ldots$, $V_q \coloneqq \mF_q \cdot V$.
    We write each $M_i \in V, i = 1, \ldots, N,$ as 
    \[
    M_i = (v_{i1}, \ldots, v_{iq}) \in V_1 \oplus \cdots \oplus V_q
    \]
    according to this decomposition.
    Then for all $i, j \in \{1, \ldots, N\}$, the equality $\varphi_i(M_j) = \varphi_j(M_i)$ implies 
    $f_{i1}v_{j1} = f_{j1}v_{i1}, \ldots, f_{iq}v_{jq} = f_{jq}v_{iq}$.

    Let $\mI_1 \coloneqq \{i \mid 1 \leq i \leq N, f_{i1} \notin \rad(\mF_1)\}$.
    Since $f_{i1}v_{j1} = f_{j1}v_{i1}$ for all $i, j \in \mI_1$, we have $f_{i1}^{-1}v_{i1} = f_{j1}^{-1}v_{j1}$ for all $i, j \in \mI_1$.
    Therefore, there exists $v_1^* \in V_1$ such that $v_{1}^* = f_{i1}^{-1}v_{i1}$ for all $i \in \mI_1$.
    This yields $f_{i1}^{p^{\ell}}v_1^* = f_{i1}^{p^{\ell}-1}v_{i1}$ for all $i \in \mI_1$.
    For all $i \notin \mI_1$, we have $f_{i1}^{t_1} = 0$, so $f_{i1}^{p^{\ell}}v_1^* = 0 = f_{i1}^{p^{\ell}-1}v_{i1}$ since $p^{\ell} \geq 1+\max\{t_1, \ldots, t_q\}$.
    In both cases, we have
    \[
    f_{i1}^{p^{\ell}}v_1^* = f_{i1}^{p^{\ell}-1}v_{i1}.
    \]

    Similarly, we can find $v_2^* \in V_1, \ldots, v_q^* \in V_q$, such that
    \[
    f_{i2}^{p^{\ell}}v_2^* = f_{i2}^{p^{\ell}-1}v_{i2} \;, \;  \ldots \;, \; f_{iq}^{p^{\ell}}v_q^* = f_{iq}^{p^{\ell}-1}v_{iq},
    \]
    for all $i = 1, \ldots, N$.
    Let 
    \[
    Q \coloneqq (v_1^*, \ldots, v_q^*) \in V_1 \oplus \cdots \oplus V_q,
    \]
    then for all $i = 1, \ldots, N$, we have
    \[
    \varphi_i^{p^{\ell}}(Q) = (f_{i1}^{p^{\ell}} v_1^*, \ldots, f_{iq}^{p^{\ell}} v_q^*) = (f_{i1}^{p^{\ell}-1} v_{i1}, \ldots, f_{iq}^{p^{\ell}-1} v_{iq}) = \varphi_i^{p^{\ell}-1}(M_i).
    \]
\end{proof}

To prove Proposition~\ref{prop:triangularize}, we will apply Lemma~\ref{lem:singletriangularization} repeatedly.
Observe that if $A \in \End(\tmV)$ is $(a, b, c)$-triangular and its $(b,c)$-th block is divisible by $p^{a+1}$, then $A$ is
\[
\begin{cases}
    \text{$(a, b-1, c)$-triangular,} \quad & \text{if $b > 1$,} \\
    \text{$(a, c-2, c-1)$-triangular,} \quad & \text{if $b = 1, c > 2$,} \\
    \text{$(a+1, e-1, e)$-triangular,} \quad & \text{if $b = 1, c = 2$.} \\
\end{cases}
\]

\proptriangularize*
\begin{proof}
    Apply Lemma~\ref{lem:singletriangularization} repeatedly for
    \begin{align*}
    (a, b, c) = \; & (0, e-1, e), (0, e-2, e), \ldots, (0, 1, e), \\
    & (0, e-2, e-1), (0, e-3, e-1), \ldots, (0, 1, e-1), \\
    & \cdots \\
    & (0, 2, 3), (0, 1, 3), \\
    & (0, 1, 2),
    \end{align*}
    the following way.
    We start with the matrices $A_1, \ldots, A_N$, which are $(0, e-1, e)$-triangular (every element in $\Aut(\tmV)$ is $(0, e-1, e)$-triangular).
    After each application of Lemma~\ref{lem:singletriangularization}, we obtain the matrices $R^{-1}A_1^{p^{\ell}}R, \ldots, R^{-1}A_N^{p^{\ell}}R$, and we apply the next repetition of Lemma~\ref{lem:singletriangularization} on $A_1 \coloneqq R^{-1}A_1^{p^{\ell}}R, \ldots, A_N \coloneqq R^{-1}A_N^{p^{\ell}}R$.
    Since 
    \[
    {R'}^{-1}\left(R^{-1}A^{p^{\ell}}R\right)^{p^{\ell'}} R' = {R'}^{-1}R^{-1} A^{p^{\ell} + p^{\ell'}} RR' = (RR')^{-1} A^{p^{\ell + \ell'}} (RR'),
    \]
    the repeated application of Lemma~\ref{lem:singletriangularization} yields 
    \[
    \widehat{\ell} \coloneqq \ell + \ell' + \cdots \in \N, \quad \widehat{R} \coloneqq RR' \cdots \in \Aut(\tmV),
    \]
    such that the $(j,k)$-th blocks of $\widehat{R}^{-1}A_1^{p^{\widehat{\ell}}} \widehat{R}, \ldots, \widehat{R}^{-1}A_N^{p^{\widehat{\ell}}}\widehat{R}$ are divisible by $p$ for all $1 \leq j < k \leq e$.
    
    Then starting from the matrices $A_1 \coloneqq \widehat{R}^{-1}A_1^{p^{\widehat{\ell}}} \widehat{R}, \ldots, A_N \coloneqq \widehat{R}^{-1}A_N^{p^{\widehat{\ell}}}\widehat{R}$, and apply Lemma~\ref{lem:singlecomp} repeated for
    \begin{align*}
    (a, b, c) = \; & (1, e-1, e), (1, e-2, e), \ldots, (1, 1, e), \\
    & (1, e-2, e-1), (1, e-3, e-1), \ldots, (1, 1, e-1), \\
    & \cdots \\
    & (1, 2, 3), (1, 1, 3), \\
    & (1, 1, 2),
    \end{align*}
    we can find $\widehat{\ell}, \widehat{R}$ such that the $(j,k)$-th blocks of $\widehat{R}^{-1}A_1^{p^{\widehat{\ell}}} \widehat{R}, \ldots, \widehat{R}^{-1}A_N^{p^{\widehat{\ell}}}\widehat{R}$ are divisible by $p^2$ for all $1 \leq j < k \leq e$.

    Repeat the above process for $a = 2, 3, \ldots, e-1$. Then we find $\ell, R$, such that the $(j,k)$-th blocks of $R^{-1}A_1^{p^{\ell}}R$, $\ldots$, $R^{-1}A_N^{p^{\ell}}R$ are divisible by $p^e = 0$ for all $1 \leq j < k \leq e$.
    Thus $R^{-1}A_1^{p^{\ell}}R, \ldots, R^{-1}A_N^{p^{\ell}}R$ are block lower-triangular.
\end{proof}

\paragraph{Step 3: simultaneous block diagonalization.} 

In this step, we simultaneously diagonalize the matrices $A_1, \ldots, A_N$, up to taking their $p^{\N}$-th powers.
More precisely, we will show the following.

\begin{restatable}[simultaneous block diagonalization]{prop}{propblockdiagonal}\label{prop:blockdiagonal}
    Let $A_1, \ldots, A_N$ be pairwise commuting elements of $\Aut(\tmV)$.
    Then there exist effectively computable $\ell \in \N$ and $R \in \Aut(\tmV)$, such that $R^{-1}A_1^{p^{\ell}}R, \ldots, R^{-1}A_N^{p^{\ell}}R$ are block-diagonal (i.e.\ their $(j,k)$-th blocks are zero for all $j \neq k$).
\end{restatable}

Following the previous step, we can suppose $A_1, \ldots, A_N$ to be already block lower-triangularized.
That is, we can replace $A_1, \ldots, A_N$ with the elements $R^{-1}A_1^{p^{\ell}}R, \ldots, R^{-1}A_N^{p^{\ell}}R$ obtained from Proposition~\ref{prop:triangularize}.

The proof of Proposition~\ref{prop:blockdiagonal} is similar to that of Proposition~\ref{prop:triangularize}.
Namely, we will increase the $p$-adic valuation of the lower-right blocks step by step.
Let $a \in \N$ and $1 \leq b < c \leq e$. 
We say that an element $f \in \Aut(\tmV)$ is \emph{$(a,b,c)$-diagonal}, if it satisfies
\begin{align}\label{eq:abcdiagonal}
    f_{jk} = 0 & \text{ for all } 1 \leq j < k \leq e, \nonumber\\
    p^{a} \mid f_{jk} & \text{ for all } 1 \leq k < j \leq e, \nonumber\\
    p^{a+1} \mid f_{jk} & \text{ for all } 1 \leq k < j \leq e, \; j \geq c +1, \nonumber\\
    p^{a+1} \mid f_{jk} & \text{ for all } b < k < c, \; j = c.
\end{align}
in addition to the divisibility constraints~\eqref{eq:divcond} which all elements of $\End(\tmV)$ are subject to.
In this case, $f$ can be written in the matrix form
\[
    \begin{pmatrix}
        M_{11} & & & & & \\    
        p^a M_{21} & \ddots & & &&0 \\
        \vdots & & M_{bb} \\
         & & & \ddots & \\
        & & & & \ddots & \\
        p^a M_{c1} & \cdots & p^a M_{cb} & p^{a+1} M_{c(b+1)} & \cdots & M_{cc} \\
        p^{a+1} M_{(c+1)1} & \cdots & p^{a+1} M_{(c+1)b} & p^{a+1} M_{(c+1)(b+1)} & \cdots & p^{a+1} M_{(c+1)c} & \ddots \\
        \vdots & & \vdots & \vdots & & \vdots \\
    \end{pmatrix}.
\]
If $b+1 = c$ then the term $p^{a+1} M_{c(b+1)}$ is overwritten by $M_{cc}$.

\begin{obs}
    The composition of two $(a,b,c)$-diagonal automorphisms is $(a,b,c)$-diagonal.
\end{obs}

\begin{lem}[single step of diagonalization]\label{lem:singlediagonalization}
    Let $A_1, \ldots, A_N$ be pairwise commuting elements of $\Aut(\tmV)$, let $a \in \N$ and $1 \leq b < c \leq e$. 
    Suppose that $A_1, \ldots, A_N$ are $(a,b,c)$-diagonal.
    Then there exist effectively computable $\ell \in \N$ and $R \in \Aut(\tmV)$, such that $R^{-1}A_1^{p^{\ell}}R, \ldots, R^{-1}A_N^{p^{\ell}}R$ are $(a,b,c)$-diagonal and their $(c,b)$-th blocks are divisible by $p^{a+1}$.
\end{lem}
\begin{proof}
    Without loss of generality suppose $c-b \leq a$.
    Otherwise we have $c-b \geq a+1$, so $p^{a+1} \mid p^{c-b}$.
    Since $p^{c-b} \mid f_{cb}$ for all $f \in \End(\tmV)$ by the divisibility constraint~\eqref{eq:divcond}, we have $p^{a+1} \mid A_i, i = 1, \ldots, N$. 
    We can take $\ell = 0$ and $R$ to be the identity map, and we have nothing to prove.

    Let $R$ be the block matrix defined by 
    \[
    R_{jk} = 
    \begin{cases}
        I \quad & 1 \leq j = k \leq e, \\
        -p^a Q \quad & j = c, k = b, \\
        0 \quad & \text{otherwise},
    \end{cases}
    \]
    for some $Q \in \M_{d_c \times d_b}(\Z_{/p^c}(\oX))$ to be determined later.
    That is,
    \[
    R = 
    \begin{pmatrix}
        I & & & & \\    
        0 & \ddots & &&&0 \\
        \vdots & & I \\
         & & & \ddots & \\
        & & & & \ddots & \\
        0 & \cdots & -p^a Q & 0 & \cdots & I \\
        0 & \cdots & 0 & 0 & \cdots & 0 & \ddots \\
        \vdots & & \vdots & \vdots & & \vdots \\
    \end{pmatrix}.
    \]
    We now verify that $R$ satisfies the divisibility constraints~\eqref{eq:divcond}, so that $R$ is indeed in $\Aut(\tmV)$. 
    Indeed, since $p^{a} \mid R_{cb}, c-b \leq a$, we have $p^{c-b} \mid R_{cb}$.
    Therefore $R \in \Aut(\tmV)$.
    
    The rest of the proof is the same as Lemma~\ref{lem:singletriangularization}, but with all the matrices transposed.
\end{proof}

To prove Proposition~\ref{prop:blockdiagonal}, we will apply Lemma~\ref{lem:singlediagonalization} repeatedly.
Observe that if $A \in \End(\tmV)$ is $(a, b, c)$-diagonal and its $(c,b)$-th block is divisible by $p^{a+1}$, then $A$ is
\[
\begin{cases}
    \text{$(a, b-1, c)$-triangular,} \quad & \text{if $b > 1$,} \\
    \text{$(a, c-2, c-1)$-triangular,} \quad & \text{if $b = 1, c > 2$,} \\
    \text{$(a+1, e-1, e)$-triangular,} \quad & \text{if $b = 1, c = 2$.} \\
\end{cases}
\]

\propblockdiagonal*
\begin{proof}
    By Proposition~\ref{prop:triangularize}, we can without loss of generality suppose $A_1, \ldots, A_N$ to be block lower-triangular.
    The rest of the proof is the same as Proposition~\ref{prop:triangularize}.
    Apply Lemma~\ref{lem:singlediagonalization} repeatedly for
    \begin{align*}
    (a, b, c) = \; & (a, e-1, e), (a, e-2, e), \ldots, (a, 1, e), \\
    & (a, e-2, e-1), (a, e-3, e-1), \ldots, (a, 1, e-1), \\
    & \cdots \\
    & (a, 2, 3), (a, 1, 3), \\
    & (a, 1, 2),
    \end{align*}
    for $a = 0, 1, \ldots, e-1$. We can find $\ell, R$ such that the $(j,k)$-th blocks of $R^{-1}A_1^{p^{\ell}}R$, $\ldots$, $R^{-1}A_N^{p^{\ell}}R$ are divisible by $p^e = 0$ for all $1 \leq k < j \leq e$ and all $1 \leq j < k \leq e$.
    This means that $R^{-1}A_1^{p^{\ell}}R, \ldots, R^{-1}A_N^{p^{\ell}}R$ are block diagonal.
\end{proof}

\paragraph{Step 4: reduction to S-unit equations over the $\mA$-module $\Zpe(\oX)^d$.}

Let
\begin{equation}\label{eq:decomposeVZpemod}
\tmV = \Zp(\oX)^{d_1} \oplus \Zpt(\oX)^{d_2} \oplus \cdots \oplus \Zpe(\oX)^{d_e}
\end{equation}
be the decomposition of $\tmV$ as a $\Zpe(\oX)$-module specified in Lemma~\ref{lem:PIR}.
Let $\tmV_1, \tmV_2, \ldots, \tmV_e$ denote respectively the components $\Zp(\oX)^{d_1}, \Zpt(\oX)^{d_2}, \ldots, \Zpe(\oX)^{d_e}$ in~\eqref{eq:decomposeVZpemod}.

By Proposition~\ref{prop:blockdiagonal}, we can compute $\ell \in \N$ and $R \in \Aut(\tmV)$, such that $R^{-1}A_1^{p^{\ell}}R, \ldots, R^{-1}A_N^{p^{\ell}}R$ are block diagonal.
This means that
\[
R^{-1}A_i^{p^{\ell}}R \cdot \tmV_1 \subseteq \tmV_1, \; R^{-1}A_i^{p^{\ell}}R \cdot \tmV_2 \subseteq \tmV_2, \; \ldots, \; R^{-1}A_i^{p^{\ell}}R \cdot \tmV_e \subseteq \tmV_e,
\]
for all $i = 1, \ldots, N$.
Since $R^{-1}A_1^{p^{\ell}}R, \ldots, R^{-1}A_N^{p^{\ell}}R$ are invertible, we have
\[
R^{-1}A_i^{p^{\ell}}R \cdot \tmV_1 = \tmV_1, \; R^{-1}A_i^{p^{\ell}}R \cdot \tmV_2 = \tmV_2, \; \ldots, \; R^{-1}A_i^{p^{\ell}}R \cdot \tmV_e = \tmV_e.
\]
Therefore
\[
A_i^{p^{\ell}} \cdot R \tmV_1 = R \tmV_1, \; A_i^{p^{\ell}} \cdot R \tmV_2 = R \tmV_2, \; \ldots,\; A_i^{p^{\ell}} \cdot R \tmV_e = R \tmV_e
\]
for all $i = 1, \ldots, N$.
Since $A_i^{p^{\ell}}, A_i^{-p^{\ell}}, i = 1, \ldots, N$, generate $\tmA$ as a $\Zpe(\oX)$-algebra (see property~(iv) of Proposition~\ref{prop:wrapperlocalization}), we have
\[
\tmA \cdot R \tmV_1 = R \tmV_1,\; \tmA \cdot R \tmV_2 = R \tmV_2,\; \ldots,\; \tmA \cdot R \tmV_e = R \tmV_e.
\]
This means that we have the decomposition
\begin{equation}\label{eq:decomposeVAmod}
\tmV = R \tmV = R\tmV_1 \oplus R\tmV_2 \oplus \cdots \oplus R\tmV_e
\end{equation}
as an $\tmA$-module.

Let $\pi_i \colon \tmV \rightarrow R \tmV_i, i = 1, \ldots, e$, be the projections according to the decomposition~\eqref{eq:decomposeVAmod}.
Then the solution set of the an equation
\begin{equation}\label{eq:beforeAv}
     A_1^{z_{11}} A_2^{z_{12}} \cdots A_N^{z_{1N}} \cdot v_1 + \cdots + A_1^{z_{K1}} A_2^{z_{K2}} \cdots A_N^{z_{KN}} \cdot v_K = v_0
\end{equation}
over the $\tmA$-module $\tmV$ is equal to the intersection of the solution set of equations
\begin{equation}\label{eq:afterAv}
A_1^{z_{11}} A_2^{z_{12}} \cdots A_N^{z_{1N}} \cdot \pi_i(v_1) + \cdots + A_1^{z_{K1}} A_2^{z_{K2}} \cdots A_N^{z_{KN}} \cdot \pi_i(v_K) = \pi_i(v_0)
\end{equation}
over the $\tmA$-modules $R\tmV_i,\; i = 1, \ldots, e$.

Since $R \in \Aut(\tmV)$ is injective, the map $R\tmV_i \rightarrow \tmV_i, v \mapsto R^{-1}v$ defines an isomorphism between $R\tmV_i$ and $\tmV_i$.
Therefore
\[
R\tmV_i \cong \tmV_i = \Zpi(\oX)^{d_i}
\]
for $i = 1, \ldots, e$.
We consider the S-unit Equations~\eqref{eq:afterAv} over the $\tmA$-module $\mV \coloneqq R\tmV_i \cong \Zpi(\oX)^{d_i}$.
Since $p^i \cdot \mV = 0$, the $\tmA$-module $\mV$ is actually an $\mA \coloneqq \tmA/p^i \tmA$-module.
Hence, Equation~\eqref{eq:afterAv} can be considered as an S-unit equation over the $\mA$-module $\mV \cong \Zpi(\oX)^{d_i}$, by replacing $A_1, \ldots, A_N$ with their image under the projection $\tmA \rightarrow \tmA/p^i \tmA = \mA$.

This completes all the ingredients for the proof of Proposition~\ref{prop:wrapperfreeness}:

\propwrapperfreeness*
\begin{proof}
    It suffices to show that the ring $\mA$ and the $\mA$-module $\mV$ satisfy the properties~(i)-(iv).
    Since $\tmA$ is local with some maximal ideal $\frp$, the quotient $\mA = \tmA/p^i \tmA$ is also local with maximal ideal $\frm \coloneqq \frp/p^i \tmA$.
    Furthermore, since $\frp^t = 0$ for some $t$, we have $\frm^t = 0$ for the same $t$.
    This proves the property~(i) of $\mA$ in Proposition~\ref{prop:wrapperfreeness}.
    Since $\tmA$ is a effectively represented as a $\Zpe(\oX)$-algebra, the quotient $\mA = \tmA/p^i \tmA$ is a effectively represented as a $\Zpe(\oX)/p^i\Zpe(\oX) = \Zpi(\oX)$-algebra
    This shows property~(ii).
    Property~(iii) of $\mA$ is inherited from the property~(iii) of $\tmA$ from Proposition~\ref{prop:wrapperlocalization}.
    Taking $d = d_i$, we can write $\mV$ as $\Zpi(\oX)^{d}$, this yields property~(iv).
\end{proof}

From now on, we focus on S-unit equations
\begin{equation*}
     A_1^{z_{11}} A_2^{z_{12}} \cdots A_N^{z_{1N}} \cdot v_1 + \cdots + A_1^{z_{K1}} A_2^{z_{K2}} \cdots A_N^{z_{KN}} \cdot v_K = v_0
\end{equation*}
in $\mA$-modules $\mV$.
To re-uniformize our notation, we replace the exponent $i$ with $e$, so that $\mA$ is again a local $\Zpe(\oX)$-algebra, and $\mV$ is isomorphic to $\Zpe(\oX)^d$ as a $\Zpe(\oX)$-module for some $d \geq 1$.
In particular, each invertible element $A \in \mA$ acts on $\mV = \Zpe(\oX)^d$ as a matrix in $\GL_d(\Zpe(\oX))$.
From now on, we denote by $\frm$ the maximal ideal of $\mA$.

\subsection{Pseudo Frobenius splitting}\label{subsec:Frob}
As illustrated in Example~\ref{exmpl:Derksen}, the key part in Derksen and Masser's proof of Theorem~\ref{thm:DM} is the so-called \emph{Frobenius splitting}.
Recall that this means for a field $\K$ of characteristic $p$,
\[
\K^p \coloneqq \{k^p \mid k \in \K\}
\]
is a subfield of $\K$, making $\K$ an $\K^p$-vector space.
For the special case of the field $\F_p(\oX)$, we have $f(X_1, \ldots, X_n)^p = f(X_1^p, \ldots, X_n^p)$.
Therefore
\[
\F_p(\oX)^p = \F_p(\oX^p) \coloneqq \{f(X_1^p, \ldots, X_n^p) \mid f \in \F_p(\oX)\}
\]
is a subfield of $\F_p(\oX)$, and $\F_p(\oX)$ splits as an direct sum of $p^n$ different $\F_p(\oX^p)$-vector spaces:
\begin{equation}\label{eq:splitfield}
\F_p(\oX) = \bigoplus_{r_1, \ldots, r_n \in \{0, 1, \ldots, p-1\}} \F_p(\oX^p) \cdot X_1^{r_1} X_2^{r_2} \cdots X_n^{r_n}.
\end{equation}

Ideally, we would like to have a similar result for the algebra $\mA$.
However, the situation here is more delicate, namely $\mA^p \coloneqq \{a^p \mid a \in \mA\}$ is not necessarily a subalgebra of $\mA$ (we no longer have $a^p + b^p = (a+b)^p$).
Therefore, we need to generalize the Frobenius splitting from a field $\K$ to the algebra $\mA$, the same way Example~\ref{exmpl:improved} generalizes Example~\ref{exmpl:Derksen}.
In this subsection, we will construct such a generalization, which we will call \emph{pseudo Frobenius splitting} (Proposition~\ref{prop:highFrobsplit}).


%
%

First, we show that there is a splitting for $\Zpe(\oX)$ similar to the splitting for $\F_p(\oX)$ in Equation~\eqref{eq:splitfield}:

\begin{lem}\label{lem:splitfree}
Define $\Zpe(\oX^p) \coloneqq \{f(X_1^p, \ldots, X_n^p) \mid f \in \Zpe(\oX)\}$.
We have
\[
\Zpe(\oX) = \bigoplus_{r_1, \ldots, r_n \in \{0, 1, \ldots, p-1\}} \Zpe(\oX^p) \cdot X_1^{r_1} X_2^{r_2} \cdots X_n^{r_n}
\]
as a $\Zpe(\oX^p)$-module.
\end{lem}
\begin{proof}
    Consider the $\Zpe(\oX^p)$-linear map
    \begin{align*}
        \varphi \colon \quad \Zpe(\oX^p)^{p^n} & \rightarrow \Zpe(\oX), \\
        (f_{0,0,\ldots,0}, f_{0,0,\ldots,1}, \ldots, f_{p-1, p-1, \ldots, p-1}) & \mapsto \sum_{r_1, \ldots, r_n \in \{0, 1, \ldots, p-1\}} f_{r_1,r_2,\ldots,r_n} \cdot X_1^{r_1} X_2^{r_2} \cdots X_n^{r_n}.
    \end{align*}
    First we show that $\varphi$ is injective.
    Suppose $\sum_{r_1, \ldots, r_n \in \{0, 1, \ldots, p-1\}} f_{r_1,r_2,\ldots,r_n} \cdot X_1^{r_1} X_2^{r_2} \cdots X_n^{r_n} = 0$.
    Write each $f_{r_1,r_2,\ldots,r_n} = \frac{g_{r_1,r_2,\ldots,r_n}}{h_{r_1,r_2,\ldots,r_n}}$ where $g_{r_1,r_2,\ldots,r_n}, h_{r_1,r_2,\ldots,r_n} \in \Zpe[\oX^p]$ and $p \nmid h_{r_1,r_2,\ldots,r_n}$.
    Then
    \begin{multline*}
    0 = \sum_{r_1, \ldots, r_n \in \{0, 1, \ldots, p-1\}} f_{r_1,r_2,\ldots,r_n} \cdot X_1^{r_1} X_2^{r_2} \cdots X_n^{r_n} = 
    \frac{\sum_{r_1, \ldots, r_n \in \{0, 1, \ldots, p-1\}} G_{r_1, \ldots, r_n} X_1^{r_1} X_2^{r_2} \cdots X_n^{r_n}}{\prod_{r_1, \ldots, r_n \in \{0, 1, \ldots, p-1\}} h_{r_1,r_2,\ldots,r_n}},
    \end{multline*}
    where $G_{r_1, \ldots, r_n} \coloneqq g_{r_1, \ldots, r_n} \prod_{(r'_1, \ldots, r'_n) \neq (r_1, \ldots, r_n)} h_{r'_1, \ldots, r'_n} \in \Zpe[\oX^p]$.
    Therefore
    \[
    0 = \sum_{r_1, \ldots, r_n \in \{0, 1, \ldots, p-1\}} G_{r_1, \ldots, r_n} X_1^{r_1} X_2^{r_2} \cdots X_n^{r_n},
    \]
    so we must have $G_{r_1, \ldots, r_n} = 0$ for all $r_1, \ldots, r_n \in \{0, 1, \ldots, p-1\}$.
    Consequently $g_{r_1, \ldots, r_n} = 0$, because $h_{r'_1, \ldots, r'_n} \neq 0$. We conclude that $f_{r_1, \ldots, r_n} = 0$ for all $r_1, \ldots, r_n \in \{0, 1, \ldots, p-1\}$.

    Next we show that $\varphi$ is surjective.
    Let $\frac{g}{h} \in \Zpe(\oX)$ where $g, h \in \Zpe[\oX]$ with $p \nmid h$.
    By Lemma~\ref{lem:powp}, we have $h^{p^{e}} \in \Zpe[\oX^p]$.
    Therefore, $\frac{g}{h} = \frac{g h^{p^{e} - 1}}{h^{p^{e}}}$ where $p \nmid h^{p^{e}}$ and $h^{p^{e}} \in \Zpe[\oX^p]$.
    We can write $g h^{p^{e} - 1}$ as $\sum_{r_1, \ldots, r_n \in \{0, 1, \ldots, p-1\}} H_{r_1, \ldots, r_n} X_1^{r_1} X_2^{r_2} \cdots X_n^{r_n}$ with each $H_{r_1, \ldots, r_n} \in \Zpe[\oX^p]$.
    Then
    \[
    \frac{g}{h} = \frac{g h^{p^{e} - 1}}{h^{p^{e}}} = \sum_{r_1, \ldots, r_n \in \{0, 1, \ldots, p-1\}} \frac{H_{r_1, \ldots, r_n}}{h^{p^{e}}} \cdot X_1^{r_1} X_2^{r_2} \cdots X_n^{r_n} = \varphi\left(\frac{H_{0, \ldots, 0}}{h^{p^{e}}}, \ldots, \frac{H_{p-1, \ldots, p-1}}{h^{p^{e}}}\right).
    \]
    Therefore $\varphi$ is surjective.

    We conclude that $\varphi$ is a bijection, so $\Zpe(\oX) = \bigoplus_{r_1, \ldots, r_n \in \{0, 1, \ldots, p-1\}} \Zpe(\oX^p) \cdot X_1^{r_1} \cdots X_n^{r_n}$.
\end{proof}

Instead of the usual $\Zpe(\oX)$-module structure on $\Zpe(\oX)$, we can define a different $\Zpe(\oX)$-module structure on $\Zpe(\oX)$ by the action
\[
* \; \colon \Zpe(\oX) \times \Zpe(\oX) \rightarrow \Zpe(\oX), \quad f * m \coloneqq f(X_1^p, \ldots, X_k^p) \cdot m.
\]
We denote by $\Phi(\Zpe(\oX))$ this new $\Zpe(\oX)$-module. 
Intuitively, applying $\Phi$ to $\Zpe(\oX)$ can be considered as performing the ``variable change'' $X'_1 \coloneqq X_1^p, \ldots, X'_n \coloneqq X_n^p$, as in Example~\ref{exmpl:Derksen} or~\ref{exmpl:improved}.
By Lemma~\ref{lem:splitfree}, we have $\Phi\left(\Zpe(\oX)\right) = \left(\Zpe(\oX)\right)^{p^n}$.
For any element $f \in \Zpe(\oX)$, it can be considered as an element in $\Phi\left(\Zpe(\oX)\right)$ which we denote by $\Phi(f)$.
In particular, $\Phi(f) = (f_{0, 0, \ldots, 0}, f_{0, 0, \ldots, 1}, \ldots, f_{p-1, p-1, \ldots, p-1})$, where $f_{r_1, r_2 \ldots, r_n} \in \Zpe(\oX)$ are such that
\[
f = \sum_{r_1, \ldots, r_n \in \{0, 1, \ldots, p-1\}} f_{r_1, r_2 \ldots, r_n}(X_1^p, \ldots, X_n^p) \cdot X_1^{r_1} X_2^{r_2} \cdots X_n^{r_n}.
\]

We can extend the domain of definition of $\Phi$ to $\mV = \Zpe(\oX)^{d}$, so that $\Phi(\mV) = \left(\Phi\left(\Zpe(\oX)\right)\right)^d = \Zpe(\oX)^{p^n d}$.
In particular, if $\bv = (f_1, \ldots, f_d) \in \mV$, then $\Phi(\bv) \coloneqq \left(\Phi(f_1), \ldots, \Phi(f_d)\right) \in \Zpe(\oX)^{p^n \cdot d} = \Phi(\mV)$.
Let $A \in \GL_d(\Zpe(\oX))$ be any invertible $\Zpe(\oX)$-linear transformation of $\mV$. 
Then $A$ induces an invertible $\Zpe(\oX)$-linear transformation $\Phi(A) \colon \Phi(\mV) \rightarrow \Phi(\mV)$ defined by $\Phi(A) \cdot \Phi(\bv) \coloneqq \Phi(A\bv)$.
In particular, $\Phi(A) \in \GL_{p^n d}(\Zpe(\oX))$.

Note that the map $\Phi$ commutes with taking modulo $p$. More precisely, we have $\Phi(f + p \cdot \Zpe(\oX)) = \Phi(f) + p \cdot\Phi(\Zpe(\oX))$, therefore $\Phi$ induces a map from $\Zpe(\oX)/p\Zpe(\oX) = \F_p(\oX)$ to $\Zpe(\oX)^{p^n}/p\Zpe(\oX)^{p^n} = \F_p(\oX)^{p^n}$.
Similarly, we have $\Phi(\mV/p\mV) = \Phi(\F_p(\oX)^d) = \F_p(\oX)^{p^n d}$, and $\Phi(\GL_d(\F_p(\oX))) = \GL_{p^n d}(\F_p(\oX))$.

One can also iterate the operation $\Phi$, and we denote $\Phi^k(\cdot) \coloneqq \underbrace{\Phi(\Phi( \cdots \Phi}_{k \text{ times }} ( \cdot ) \cdots ))$.
Thus, we have the chains
\begin{alignat*}{3}
     \Zpe(\oX) \; & \xrightarrow{\;\;\Phi\;\;} \;\left(\Zpe(\oX)\right)^{p^n} && \;\xrightarrow{\;\;\Phi\;\;}\; \left(\Zpe(\oX)\right)^{p^n \cdot p^n} && \;\xrightarrow{\;\;\Phi\;\;}\; \cdots \\
     \mV = \Zpe(\oX)^d & \;\xrightarrow{\;\;\Phi\;\;}\; \left(\Zpe(\oX)\right)^{p^n d} && \;\xrightarrow{\;\;\Phi\;\;}\; \left(\Zpe(\oX)\right)^{p^n \cdot p^nd} && \;\xrightarrow{\;\;\Phi\;\;}\; \cdots \\
     \Aut(\mV) = \GL_d(\Zpe(\oX)) & \;\xrightarrow{\;\;\Phi\;\;}\; \GL_{p^n d}(\Zpe(\oX)) && \;\xrightarrow{\;\;\Phi\;\;}\; \GL_{p^{2n} d}(\Zpe(\oX)) && \;\xrightarrow{\;\;\Phi\;\;}\; \cdots \\
\end{alignat*}
In particular, each $\Phi^k, k \geq 1$, is a bijection, and can be considered as performing the variable change $X'_1 \coloneqq X_1^{p^k}, \ldots, X'_n \coloneqq X_n^{p^k}$.
These also commute with taking modulo $p$.

We are now ready to construct the ``pseudo Frobenius splitting'' for the elements $A_1, \ldots, A_N \in \mA$, considered as invertible matrices in $\GL_d(\Zpe(\oX))$.
The exact formulation is the following proposition.

\begin{restatable}[Pseudo Frobenius splitting]{prop}{prophighFrob}\label{prop:highFrobsplit}
    There exist an effectively computable integer $s \geq 0$, and an effectively computable matrix $R \in \GL_{p^{(s+1)n} d}(\Zpe(\oX))$, such that for all $i = 1, \ldots, N$, we have
    \begin{equation}\label{eq:highFrob}
    R^{-1} \cdot \Phi^{s+1}(A_i)^{p^{s+1}} \cdot R = \diag\Big(\underbrace{\Phi^{s}(A_i)^{p^{s}}, \ldots, \Phi^{s}(A_i)^{p^{s}}}_{p^n \text{ blocks}}\Big).
    \end{equation}
\end{restatable}

Here, $\diag(A, \ldots, A)$ denotes the block-diagonal matrix with the blocks $A, \ldots, A$ on the diagonal.
The rest of this subsection will be dedicated to the proof of Proposition~\ref{prop:highFrobsplit}.
The proof applies similar techniques to the block-diagonalization procedure from Subsection~\ref{subsec:free}.
Notably, Lemma~\ref{lem:iterphi} and Lemma~\ref{lem:invertphi} will be crucial.

Let $\mB \coloneqq \mA/p\mA$. Then $\mB$ is a finite dimensional commutative $\F_p(\oX)$-algebra, which acts on $\mV/p\mV = \F_p(\oX)^d$.
Since $\mA$ is local with maximal ideal $\frm$ satisfying $\frm^t = 0$, the algebra $\mB$ is also local with maximal ideal $\frm/p\mA$, and $(\frm/p\mA)^t = 0$.
Let $B_1, \ldots, B_N$ be the image of $A_1, \ldots, A_N$ in $\mB = \mA/p\mA$.
Since the map $\Phi$ commutes with taking modulo $p$, it can be applied on the quotients $\mV/p\mV$ and $\mB = \mA/p\mA$.

First, we show a special case of Proposition~\ref{prop:highFrobsplit} for $e = 1$: this is a rather classic extension of the Frobenius splitting.

\begin{lem}[Frobenius splitting of $\F_p(\oX)$-algebras]\label{lem:Frobp}
    There exists an effectively computable integer $s \geq 0$, and an effectively computable matrix $R \in \GL_{p^{(s+1)n} d}(\F_p(\oX))$, such that for all $i = 1, \ldots, N$, we have
    \begin{equation}
    R^{-1} \cdot \Phi^{s+1}(B_i)^{p^{s+1}} \cdot R = \diag\Big(\underbrace{\Phi^{s}(B_i)^{p^s}, \ldots, \Phi^{s}(B_i)^{p^s}}_{p^n \text{ blocks}}\Big) 
    \end{equation}
\end{lem}
\begin{proof}
    Let $s$ be such that $p^s \geq t$.
    Since $\mB$ is a finite dimensional $\F_p(\oX)$-algebra, the set 
    \[
    \mB^{p^s} \coloneqq \{b^{p^s} \mid b \in \mB\}
    \]
    is a finite dimensional $\F_p(\oX^{p^s})$-algebra.
    We claim that $\mB^{p^s}$ is an integral domain, and therefore a field~\cite[Corollary~4.7]{eisenbud2013commutative}.
    Indeed, suppose $xy = 0$ in $\mB^{p^s} = (\mA/p\mA)^{p^s}$, then $x = (v + p \mA)^{p^s},\; y = (w + p \mA)^{p^s}$ for some $v, w \in \mA$.
    Then $xy = 0$ yields $v^{p^s} w^{p^s} \in p \mA \subseteq \frm$.
    Since $\frm$ is a prime ideal of $\mA$, we have $v \in \frm$ or $w \in \frm$.
    This yields $v^t = 0$ or $w^t = 0$.
    Since $p^s \geq t$, we have $v^{p^s} = 0$ or $w^{p^s} = 0$.
    We conclude that either $x = 0$ or $y = 0$, and hence $\mB^{p^s}$ is an integral domain and therefore a field.

    The field $\mB^{p^s}$ acts on $\mV/p\mV$, making it a $\mB^{p^s}$-linear space, whose basis we denote by $v_1, \ldots, v_{d_{\mV}} \in \mV/p\mV$.
    Let $E_1, \ldots, E_{d_{\mB}}$ be a basis of $\mB^{p^{s}}$ as a $\F_p(\oX^{p^{s}})$-linear space.
    Then
    \[
    \mV/p\mV = \bigoplus_{j = 1}^{d_{\mV}} \mB^{p^s} v_j = \bigoplus_{j = 1}^{d_{\mV}} \bigoplus_{k = 1}^{d_{\mB}} \F_p(\oX^{p^{s}}) \cdot E_k v_j
    \]
    as a $\F_p(\oX^{p^{s}})$-vector space.
    Thus, 
    \[
    \Phi^s(\mV/p\mV) = \bigoplus_{j = 1}^{d_{\mV}} \bigoplus_{k = 1}^{d_{\mB}} \F_p(\oX) \cdot \Phi^s(E_k v_j).
    \]
    We define the basis matrix $C \coloneqq \left(\Phi^s(E_k v_j)\right)_{j = 1, \ldots, d_{\mV}; k = 1, \ldots, d_{\mB}}$ of $\Phi^s(\mV/p\mV) = \F_p(\oX)^{p^{sn}d}$, treating each $\Phi^s(E_k v_j)$ as a column vector:
    \[
    C \coloneqq \left(\Phi^s(E_1 v_1), \ldots, \Phi^s(E_{d_{\mB}} v_1), \ldots, \Phi^s(E_1 v_{d_{\mV}}), \ldots, \Phi^s(E_{d_{\mB}} v_{d_{\mV}}) \right) \in \GL_{p^{sn}d}(\F_p(\oX)).
    \]

    Since $\left\{E_1, \ldots, E_{d_{\mB}}\right\}$ forms a basis of $\mB^{p^{s}}$ as a $\F_p(\oX^{p^{s}})$-linear space, taking their $p$-th power gives $\left\{E_1^p, \ldots, E_{d_{\mB}}^p\right\}$ as a basis of $\mB^{p^{s+1}}$ as a $\F_p(\oX^{p^{s+1}})$-linear space.
    Since $\mB^{p^s}$ is a field, $\mB^{p^{s+1}} = \left(\mB^{p^s}\right)^p$ is a subfield of $\mB^{p^s}$.
    Let $F_1, \ldots, F_{p^n}$ be a basis of $\mB^{p^s}$ as a $\mB^{p^{s+1}}$-linear space.
    Then,
    \[
    \mV/p\mV = \bigoplus_{j = 1}^{d_{\mV}} \mB^{p^s} v_j = \bigoplus_{j = 1}^{d_{\mV}} \bigoplus_{i = 1}^{p^n} \mB^{p^{s+1}} \cdot F_i v_j = \bigoplus_{i = 1}^{p^n} \bigoplus_{j = 1}^{d_{\mV}} \bigoplus_{k = 1}^{d_{\mB}} \F_p(\oX^{p^{s+1}}) \cdot E_k^p F_i v_j
    \]
    as a $\F_p(\oX^{p^{s+1}})$-vector space.
    So
    \[
    \Phi^{s+1}(\mV/p\mV) = \bigoplus_{i = 1}^{p^n} \bigoplus_{j = 1}^{d_{\mV}} \bigoplus_{k = 1}^{d_{\mB}} \F_p(\oX) \cdot \Phi^{s+1}(E_k^p F_i v_j),
    \]
    and we define the basis matrix of $\Phi^{s+1}(\mV/p\mV)$:
    \begin{multline*}
    Q \coloneqq \left(\Phi^{s+1}(E_1^p F_1 v_1), \ldots, \Phi^{s+1}(E_{d_{\mB}}^p F_1 v_{d_{\mV}}), \ldots, \Phi^{s+1}(E_1^p F_{p^n} v_1), \ldots, \Phi^{s+1}(E_{d_{\mB}}^p F_{p^n} v_{d_{\mV}}) \right) \\
    \in \GL_{p^{(s+1)n}d}(\F_p(\oX)).
    \end{multline*}
    Let $B$ be any element of $\mB^{p^s}$, we will show
    \[
    Q^{-1} \cdot \Phi^{s+1}(B^p) \cdot Q = \diag\Big(\underbrace{C^{-1} \cdot \Phi^{s}(B) \cdot C, \ldots, C^{-1} \cdot \Phi^{s}(B) \cdot C}_{p^n \text{ blocks}}\Big).
    \]

    Recall that $E_1, \ldots, E_{d_{\mB}}$ is a basis of $\mB^{p^{s}}$ as a $\F_p(\oX^{p^{s}})$-linear space.
    For any $k = 1, \ldots, d_{\mB}$, we can write 
    \begin{equation}\label{eq:BE}
        B E_k = b_{1k}^{p^s} E_1 + \cdots + b_{d_{\mB} k}^{p^s} E_{d_{\mB}},
    \end{equation}
    for some $b_{1k}^{p^s}, \ldots, b_{d_{\mB} k}^{p^s} \in \F_p(\oX^{p^s})$.
    Then
    \[
    \Phi^{s}(B) \Phi^{s}(E_k) = b_{1k} \Phi^{s}(E_1) + \cdots + b_{d_{\mB} k} \Phi^{s}(E_{d_{\mB}}),
    \]
    which yields
    \[
    \Phi^{s}(B) \Phi^{s}(E_k v_j) = b_{1k} \Phi^{s}(E_1 v_j) + \cdots + b_{d_{\mB} k} \Phi^{s}(E_{d_{\mB}} v_j)
    \]
    for all $j = 1, \ldots, d_{\mV}$.
    This means that the matrix $C^{-1} \cdot \Phi^{s}(B) \cdot C$ (that is, the matrix form of linear map $\Phi^{s}(B)$ under the new basis $C$) is block diagonal of the form
    \begin{align}\label{eq:CBC}
    C^{-1} \cdot \Phi^{s}(B) \cdot C & =
    \begin{pmatrix}
        b_{11} & \cdots & b_{1d_{\mB}} &   &   &   &   &   &   &   \\
        \vdots & \ddots & \vdots&&&&&&& \\
        b_{d_{\mB}1} & \cdots & b_{d_{\mB}d_{\mB}} & &&&&&& \\
        &&& b_{11} & \cdots & b_{1d_{\mB}} \\
        &&&\vdots & \ddots & \vdots \\
        &&& b_{d_{\mB}1} & \cdots & b_{d_{\mB}d_{\mB}} \\
        &&&&&&\ddots \\
        &&&&&&& b_{11} & \cdots & b_{1d_{\mB}} \\
        &&&&&&& \vdots & \ddots & \vdots \\
        &&&&&&& b_{d_{\mB}1} & \cdots & b_{d_{\mB}d_{\mB}} \\
    \end{pmatrix} \nonumber \\
    & = \diag \left(\underbrace{
    \begin{pmatrix}
        b_{11} & \cdots & b_{1d_{\mB}}   \\
        \vdots & \ddots & \vdots \\
        b_{d_{\mB}1} & \cdots & b_{d_{\mB}d_{\mB}}\\
    \end{pmatrix}, \ldots,
    \begin{pmatrix}
        b_{11} & \cdots & b_{1d_{\mB}}   \\
        \vdots & \ddots & \vdots \\
        b_{d_{\mB}1} & \cdots & b_{d_{\mB}d_{\mB}}\\
    \end{pmatrix}}_{d_{\mV} \text{ blocks }}\right).
    \end{align}
    
    On the other hand, taking power $p$ on both sides of Equation~\eqref{eq:BE} yields 
    \[
    B^p E_k^p = b_{1k}^{p^{s+1}} E_1^p + \cdots + b_{d_{\mB} k}^{p^{s+1}} E_{d_{\mB}}^p.
    \]
    Hence
    \[
    \Phi^{s+1}(B^p) \Phi^{s+1}(E_k^p) = b_{1k} \Phi^{s+1}(E_1^p) + \cdots + b_{d_{\mB} k} \Phi^{s+1}(E_{d_{\mB}}^p),
    \]
    which yields
    \[
    \Phi^{s+1}(B^p) \Phi^{s+1}(E_k^p F_i v_j) = b_{1k} \Phi^{s+1}(E_1^p F_i v_j) + \cdots + b_{d_{\mB} k} \Phi^{s+1}(E_{d_{\mB}}^p F_i v_j)
    \]
    for all $i = 1, \ldots, p^n; j = 1, \ldots, d_{\mV}$.
    This means that the matrix $Q^{-1} \cdot \Phi^{s+1}(B^p) \cdot Q$ (that is, the matrix form of linear map $\Phi^{s+1}(B^p)$ under the new basis $Q$) is block diagonal of the form
    \begin{align}\label{eq:QBQ}
    Q^{-1} \cdot \Phi^{s+1}(B^p) \cdot Q & =
    \begin{pmatrix}
        b_{11} & \cdots & b_{1d_{\mB}} &   &   &   &   &   &   &   \\
        \vdots & \ddots & \vdots&&&&&&& \\
        b_{d_{\mB}1} & \cdots & b_{d_{\mB}d_{\mB}} & &&&&&& \\
        &&& b_{11} & \cdots & b_{1d_{\mB}} \\
        &&&\vdots & \ddots & \vdots \\
        &&& b_{d_{\mB}1} & \cdots & b_{d_{\mB}d_{\mB}} \\
        &&&&&&\ddots \\
        &&&&&&& b_{11} & \cdots & b_{1d_{\mB}} \\
        &&&&&&& \vdots & \ddots & \vdots \\
        &&&&&&& b_{d_{\mB}1} & \cdots & b_{d_{\mB}d_{\mB}} \\
    \end{pmatrix} \nonumber \\
    & = \diag \left(\underbrace{
    \begin{pmatrix}
        b_{11} & \cdots & b_{1d_{\mB}}   \\
        \vdots & \ddots & \vdots \\
        b_{d_{\mB}1} & \cdots & b_{d_{\mB}d_{\mB}}\\
    \end{pmatrix}, \ldots,
    \begin{pmatrix}
        b_{11} & \cdots & b_{1d_{\mB}}   \\
        \vdots & \ddots & \vdots \\
        b_{d_{\mB}1} & \cdots & b_{d_{\mB}d_{\mB}}\\
    \end{pmatrix}}_{p^n \cdot d_{\mV} \text{ blocks }}\right).
    \end{align}
    Comparing~\eqref{eq:CBC} and \eqref{eq:QBQ}, we have
    \[
    Q^{-1} \cdot \Phi^{s+1}(B^p) \cdot Q = \diag\Big(\underbrace{C^{-1} \cdot \Phi^{s}(B) \cdot C, \ldots, C^{-1} \cdot \Phi^{s}(B) \cdot C}_{p^n \text{ blocks}}\Big).
    \]
    Note that this holds for all $B \in \mB^{p^s}$.
    Taking $B$ as $B_1^{p^s}, \ldots, B_N^{p^s} \in \mB^{p^s}$, we can see that for the matrix 
    $
    R = Q \cdot \diag(\underbrace{C^{-1}, \ldots, C^{-1}}_{p^n \text{ blocks}}) 
    $, we have
    \begin{equation*}
    R^{-1} \cdot \Phi^{s+1}(B_i)^{p^{s+1}} \cdot R = \diag\Big(\underbrace{\Phi^{s}(B_i)^{p^s}, \ldots, \Phi^{s}(B_i)^{p^s}}_{p^n \text{ blocks}}\Big)
    \end{equation*}
    for $i = 1, \ldots, N$.
\end{proof}

We then strengthen Lemma~\ref{lem:Frobp} to Proposition~\ref{prop:highFrobsplit} using a variant of \emph{Hensel lifting}, a common technique in number theory.

\begin{lem}[Hensel lifting of the Frobenius splitting]\label{lem:Hensel}
    Let $a \geq 1, D \geq 1$.
    Let $B_1, \ldots, B_N \in \GL_D(\Zpe(\oX))$ be pairwise commuting matrices, and let $C_1, \ldots, C_N \in \GL_D(\Zpe(\oX))$ be another set of pairwise commuting matrices.
    If there exists a matrix $R \in \GL_D(\Zpe(\oX))$ such that 
    \[
    R^{-1} B_i R \equiv C_i \mod p^a
    \]
    for all $i = 1, \ldots, N$,
    then there exists effectively computable $\ell \in \N$ and $\tR \in \GL_D(\Zpe(\oX))$ such that
    \[
    \tR^{-1} B_i^{p^{\ell}} \tR \equiv C_i^{p^{\ell}} \mod p^{a+1}
    \]
    for all $i = 1, \ldots, N$.
\end{lem}
\begin{proof}
    Since $R^{-1} B_i R \equiv C_i \mod p^a$ for all $i$, we can write 
    \[
    C_i = R^{-1} B_i R + p^a M_i
    \]
    for some $M_i \in \M_{D \times D}(\Zpe(\oX))$, $i = 1, \ldots, N$.
    Write $\tR = R + p^{a} Q$, and we want to find $\ell$ and $Q$ such that
    \begin{equation}\label{eq:tRBtR}
    (R + p^{a} Q)^{-1} B_i^{p^{\ell}} (R + p^{a} Q) \equiv C_i^{p^{\ell}} \mod p^{a+1}
    \end{equation}
    for all $i$.
    Since
    \[
    (R + p^{a} Q)^{-1} = R^{-1}  (I + p^{a} Q R^{-1})^{-1} = R^{-1}\Big(I - p^{a} Q R^{-1} + p^{2a} (Q R^{-1})^2 - \cdots \Big),
    \]
    taking modulo $p^{a+1}$ yields
    \[
    (R + p^{a} Q)^{-1} \equiv R^{-1} - p^{a} R^{-1}QR^{-1} \mod p^{a+1}.
    \]
    Equation~\eqref{eq:tRBtR} is thus equivalent to
    \[
    (R^{-1} - p^{a} R^{-1}QR^{-1}) B_i^{p^{\ell}} (R + p^{a} Q) \equiv (R^{-1} B_i R + p^a M_i)^{p^{\ell}} \mod p^{a+1},
    \]
    which can then be rewritten as
    \begin{multline*}
    R^{-1} B_i^{p^{\ell}} R + p^a(R^{-1} B_i^{p^{\ell}} Q - R^{-1} Q R^{-1} B_i^{p^{\ell}} R) \equiv R^{-1} B_i^{p^{\ell}} R + p^a \sum_{k=0}^{p^{\ell}-1} (R^{-1} B_i R)^k M_i (R^{-1} B_i R)^{p^{\ell}-1-k} \\
    \mod p^{a+1}.
    \end{multline*}
    This is equivalent to 
    \[
    R^{-1} B_i^{p^{\ell}} Q - R^{-1} Q R^{-1} B_i^{p^{\ell}} R \equiv \sum_{k=0}^{p^{\ell}-1} (R^{-1} B_i R)^k M_i (R^{-1} B_i R)^{p^{\ell}-1-k} \mod p.
    \]
    Since $R^{-1} B_i R \equiv C_i \mod p$, we have $R^{-1} B_i^{p^{\ell}} R \equiv C_i^{p^{\ell}} \mod p$, and the above equation is equivalent to
    \begin{equation}\label{eq:QRBp}
    C_i^{p^{\ell}} (R^{-1} Q) - (R^{-1} Q) C_i^{p^{\ell}} \equiv \sum_{k=0}^{p^{\ell}-1} C_i^k M_i C_i^{p^{\ell}-1-k} \mod p.
    \end{equation}
    Define the $\F_p(\oX)$-linear transformations
    \begin{align*}
        \varphi_i \colon \M_{D \times D}(\F_p(\oX)) & \rightarrow \M_{D \times D}(\F_p(\oX)) \\
        M & \mapsto M C_i - C_i M,
    \end{align*}
    for $i = 1, \ldots, N$.
    Here $C_1, \ldots, C_N \in \GL_D(\Zpe(\oX))$ are considered as elements in $\M_{D \times D}(\F_p(\oX))$ by taking modulo $p$.

    By Lemma~\ref{lem:iterphi}, we have 
    \[
    \varphi_i^{p^{\ell}-1}(M) = \sum_{k=0}^{p^{\ell}-1} C_i^k M C_i^{p^{\ell}-1-k}, \quad \text{ and } \quad \varphi_i^{p^{\ell}}(M) = M C_i^{p^{\ell}} - C_i^{p^{\ell}} M.
    \]
    Hence, Equation~\eqref{eq:QRBp} is equivalent to 
    \begin{equation}\label{eq:varphipp1}
    - \varphi_i^{p^{\ell}}(R^{-1} Q) = \varphi_i^{p^{\ell}-1}(M_i),
    \end{equation}
    where we have now taken modulo $p$ of the matrices $Q$ and $R$.

    Recall that for all $i, j \in \{1, \ldots, N\}$, the elements $C_i, C_j$ commute.
    So for all $M \in \M_{D \times D}(\F_p(\oX))$ we have
    \begin{multline*}
    \varphi_i \varphi_j (M) = M C_j C_i - C_jMC_i - C_iMC_j + C_iC_j M \\
    = M C_i C_j - C_jMC_i - C_iMC_j + C_jC_i M = \varphi_j \varphi_i (M).
    \end{multline*}
    Therefore $\varphi_i \varphi_j = \varphi_j \varphi_i$.
    Furthermore, since $B_i, B_j$ commute, the elements $C_i - p^a M_i = R^{-1} B_i R$ and $C_j - p^a M_j = R^{-1} B_j R$ also commute.
    Therefore
    \begin{multline*}
    0 \equiv (C_i - p^a M_i)(C_j - p^a M_j) - (C_j - p^a M_j)(C_i - p^a M_i) \equiv p^a(- M_i C_j - C_i M_j + M_j C_i + C_j M_i) \\
    \mod p^{a+1},
    \end{multline*}
    so
    \begin{equation}\label{eq:varMeq}
    - \varphi_j(M_i) + \varphi_i(M_j) = - M_i C_j - C_i M_j + M_j C_i + C_j M_i \equiv 0 \mod p.
    \end{equation}
    That is, we have $\varphi_i(M_j) = \varphi_j(M_i)$ for all $i, j \in \{1, \ldots, N\}$.
    
    By Lemma~\ref{lem:invertphi}, there exist effectively computable $\ell \in \N$ and $\widetilde{Q} \in \M_{D \times D}(\F_p(\oX))$ 
    such that $\varphi_i^{p^{\ell}}(\widetilde{Q}) = \varphi_i^{p^{\ell}-1}(M_i)$ for all $i \in \{1, \ldots, N\}$.
    Let $Q \coloneqq -R\widetilde{Q}$, we have
    \[
    - \varphi_i^{p^{\ell}}(R^{-1} Q) = \varphi_i^{p^{\ell}}(\widetilde{Q}) = \varphi_i^{p^{\ell}-1}(M_i),
    \]
    then Equation~\eqref{eq:varphipp1} (and hence Equation~\eqref{eq:tRBtR}) is satisfied for all $i \in \{1, \ldots, N\}$.
\end{proof}

Combining Lemma~\ref{lem:Frobp} and~\ref{lem:Hensel}, we can finally prove Proposition~\ref{prop:highFrobsplit}:

\prophighFrob*
\begin{proof}
    Let $B_1, \ldots, B_N$ be the image of $A_1, \ldots, A_N$ in $\mB = \mA/p\mA$.
    By Lemma~\ref{lem:Frobp}, there exists $s \geq 0$ and $R \in \GL_{p^{(s+1)n}d}(\F_p(\oX))$ such that
    \begin{equation}\label{eq:splitFp}
    R^{-1} \cdot \Phi^{s+1}(B_i)^{p^{s+1}} \cdot R = \diag\Big(\underbrace{\Phi^{s}(B_i)^{p^s}, \ldots, \Phi^{s}(B_i)^{p^s}}_{p^n \text{ blocks}}\Big).
    \end{equation}
    Take any $R_0 \in \GL_{p^{(s+1)n}d}(\Zpe(\oX))$ such that $(R_0 \mod p) = R$. Then Equation~\eqref{eq:splitFp} yields
    \begin{equation}\label{eq:splitmodp}
    R_0^{-1} \cdot \Phi^{s+1}(A_i)^{p^{s+1}} \cdot R_0 \equiv \diag\Big(\underbrace{\Phi^{s}(A_i)^{p^{s}}, \ldots, \Phi^{s}(A_i)^{p^{s}}}_{p^n \text{ blocks}}\Big) \mod p.
    \end{equation}
    Since $A_1, \ldots, A_N$ are pairwise commuting matrices, the matrices $\Phi^{s+1}(A_i)^{p^{s+1}}, i = 1, \ldots, N$ pairwise commute, and the matrices $\diag\big(\Phi^{s}(A_i)^{p^{s}}, \ldots, \Phi^{s}(A_i)^{p^{s}}\big), i = 1, \ldots, N$, also pairwise commute.
    Therefore we can apply Lemma~\ref{lem:Hensel} with $a = 1$ to Equation~\eqref{eq:splitmodp}.
    This gives us $\ell_0 \in \N$ and a matrix $\tR_0 \in \GL_{p^{(s+1)n}d}(\Zpe(\oX))$ such that 
    \[
    \tR_0^{-1} \cdot \left(\Phi^{s+1}(A_i)^{p^{s+1}}\right)^{p^{\ell_0}} \cdot \tR_0 \equiv \left(\diag\Big(\underbrace{\Phi^{s}(A_i)^{p^{s}}, \ldots, \Phi^{s}(A_i)^{p^{s}}}_{p^n \text{ blocks}}\Big)\right)^{p^{\ell_0}} \mod p^{2}
    \]
    for all $i = 1, \ldots, N$.
    Applying $\Phi^{\ell_0}$ to the above equation yields
    \[
    \Phi^{\ell_0}(\tR_0)^{-1} \cdot \Phi^{s+\ell_0+1}(A_i)^{p^{s+\ell_0+1}} \cdot \Phi^{\ell_0}(\tR_0) \equiv \diag\Big(\underbrace{\Phi^{s+\ell_0}(A_i)^{p^{s+\ell_0}}, \ldots, \Phi^{s+\ell_0}(A_i)^{p^{s+\ell_0}}}_{p^n \text{ blocks}}\Big) \mod p^{2}.
    \]
    Letting $R_1 \coloneqq \Phi^{\ell_0}(\tR_0)$ and iterating the above procedure gives us $\ell_1, \ell_2, \ldots, \ell_{e-1} \in \N$ as well as matrices $R_1, R_2, \ldots, R_{e-1}$, with $R_{a} = \Phi^{\ell_{a-1}}(\tR_{a-1})$ for each $a = 1, 2, \ldots, e-1$, such that
    \begin{multline*}
    R_a^{-1} \cdot \Phi^{s+\ell_0 + \cdots + \ell_{a-1} +1}(A_i)^{p^{s+\ell_0 + \cdots + \ell_{a-1}+1}} \cdot R_a \\
    \equiv \diag\Big(\underbrace{\Phi^{s+\ell_0 + \cdots + \ell_{a-1}}(A_i)^{p^{s+\ell_0 + \cdots + \ell_{a-1}}}, \ldots, \Phi^{s+\ell_0 + \cdots + \ell_{a-1}}(A_i)^{p^{s+\ell_0 + \cdots + \ell_{a-1}}}}_{p^n \text{ blocks}}\Big) \mod p^{a+1}
    \end{multline*}
    for all $i = 1, \ldots, N$.
    Since we work over the base ring $\Zpe(\oX)$, we have $A \equiv B \mod p^e \iff A = B$.
    Therefore, taking $a = e-1$, $R \coloneqq R_{e-1}$ and replacing $s+\ell_0 + \cdots + \ell_{e-1}$ by $s$, we obtain Equation~\eqref{eq:highFrob}.
\end{proof}

\subsection{Constructing the automaton $\mmU$}\label{subsec:automaton}
Recall that
$
\Sigma_p = \{-(p-1), \ldots, -1, 0, 1, \ldots, p-1\}
$.
In this subsection, we construct an automaton $\mmU$ over the alphabet $\Sigma_p^{KN}$, that accepts the solution set to the S-unit equation
\begin{equation}\label{eq:Sunitmat}
    A_1^{z_{11}} A_2^{z_{12}} \cdots A_N^{z_{1N}} v_1 + \cdots + A_1^{z_{K1}} A_2^{z_{K2}} \cdots A_N^{z_{KN}} v_K = v_0
\end{equation}
over the $\mA$-module $\mV$.
The idea is similar to what we did in Example~\ref{exmpl:improved}, but we need to replace the ``stabilization'' argument $(X^2+ 2X +1)^2 = X^4+ 2X^2 +1$, by the ``pseudo Frobenius splitting''~\eqref{eq:highFrob} of Proposition~\ref{prop:highFrobsplit}.
First, we explain a few additional conditions that we can suppose without loss of generality.

\paragraph{Additional condition: $R^{-1} \cdot \Phi(A_i)^{p} \cdot R = \diag(A_i, \ldots, A_i)$ for all $i = 1, \ldots, N$.}

Intuitively, this additional condition can be understood as ignoring all the dashed arrows $\dashrightarrow$ in the automaton constructed in Example~\ref{exmpl:improved} (Figure~\ref{fig:improve}).
This can be done in the same way as in Step~4 of Subsection~\ref{subsec:localize}, by replacing each $A_i$ by $\Phi^s(A_i)^{p^s}$ and decomposing the solution set of Equation~\eqref{eq:Sunitmat} as a union of solution sets according to their residue modulo $p^s$.
Formally, we do the following:
\begin{defn}\label{def:theta}
    Let $j \geq 1$ be an integer and $r_{11}, \ldots, r_{KN} \in \{-(p^j-1), \ldots, 0, \ldots, p^j-1\}$.
    For a set $S \subseteq \Z^{KN}$, define 
    \[
    \Theta_{j; r_{11}, \ldots, r_{KN}} S \coloneqq \{\bz \in \Z^{KN} \mid \left(p^j \cdot \bz + (r_{11}, \ldots, r_{KN})\right) \in S\}.
    \]
    This is analogous to ``truncating'' the length-$j$ prefix $(r_{11}, \ldots, r_{KN})$ of a language over $\Sigma_p^{KN}$.
    When $j = 1$, then $r_{11}, \ldots, r_{KN} \in \Sigma_p^{KN}$, and we write in short $\Theta_{r_{11}, \ldots, r_{KN}}$ instead of $\Theta_{1; r_{11}, \ldots, r_{KN}}$.
\end{defn}

Let $s \geq 0$ be as in Proposition~\ref{prop:highFrobsplit}. Taking $\Phi^s$ on both sides of Equation~\eqref{eq:Sunitmat}, it becomes the equation
\begin{equation}\label{eq:goalPhiAv}
    \sum_{i = 1}^K \Phi^s(A_1)^{z_{i1}} \Phi^s(A_2)^{z_{i2}} \cdots \Phi^s(A_N)^{z_{iN}} \Phi^s(v_i) = \Phi^s(v_0),
\end{equation}
over the $\Phi^s(\mA)$-module $\Phi^s(\mV) = \Zpe(\oX)^{p^{sn}d}$.

Let $\mZ \subseteq \Z^{KN}$ denote the solution set of Equation~\eqref{eq:Sunitmat}.
Then $\mZ$ can be written as a disjoint union 
\[
\mZ = \bigcup_{(r_{11}, \ldots, r_{KN}) \in \{0, 1, \ldots, p^s-1\}^{KN}} p^s \cdot \Theta_{s; r_{11}, \ldots, r_{KN}} \mZ + (r_{11}, \ldots, r_{KN}),
\]
where each $\Theta_{s; r_{11}, \ldots, r_{KN}} \mZ$ is the solution set of the following ``shifted'' S-unit equation
\begin{equation}\label{eq:ThetaZ}
\sum_{i = 1}^K \Phi^s(A_1)^{p^sz'_{i1}} \Phi^s(A_2)^{p^s z'_{i2}} \cdots \Phi^s(A_N)^{p^s z'_{iN}} \cdot \Phi^s(A_1)^{r_{i1}} \Phi^s(A_2)^{r_{i2}} \cdots \Phi^s(A_N)^{r_{iN}} \Phi^s(v_i) = \Phi^s(v_0).
\end{equation}

Note that a finite union of $p$-normal sets is still $p$-normal, and the set $p^s \cdot \Theta_{s; r_{11}, \ldots, r_{KN}} \mZ + (r_{11}, \ldots, r_{KN})$ is $p$-normal if $\Theta_{s; r_{11}, \ldots, r_{KN}} \mZ$ is $p$-normal.
Therefore, it suffices to show that each $\Theta_{s; r_{11}, \ldots, r_{KN}} \mZ$ is $p$-normal.
Note that the defining equation~\eqref{eq:ThetaZ} of $\Theta_{s; r_{11}, \ldots, r_{KN}} \mZ$ can be written as
\begin{equation}\label{eq:Aprime}
    \sum_{i = 1}^K \left(A'_1\right)^{z'_{i1}} \left(A'_2\right)^{z'_{i2}} \cdots \left(A'_N\right)^{z'_{iN}} \cdot v'_i = v'_0,
\end{equation}
where $A'_j \coloneqq \Phi^s(A_j)$ for $j = 1, \ldots, N$; and $v'_i \coloneqq \Phi^s(A_1)^{r_{i1}} \Phi^s(A_2)^{r_{i2}} \cdots \Phi^s(A_N)^{r_{iN}} \Phi^s(v_i)$ for $i = 1, \ldots, K$; and $v'_0 \coloneqq \Phi^s(v_0)$.
Note that \eqref{eq:Aprime} is an equation over the $\Phi^s(\mA)$-module $\Phi^s(\mV) = \Zpe(\oX)^{p^{sn}d}$.
Therefore, we can without loss of generality replace $\mA$ by $\Phi^s(\mA)$ (note that this does not change the ring structure of $\mA$, so it is still local with maximal ideal $\frm$), replace $\mV$ by $\Phi^s(\mV) = \Zpe(\oX)^{p^{sn}d}$ (and consequently replace the dimension $d$ by $p^{sn}d$), as well as replacing each $A'_j = \Phi^s(A_j)$ by $A_j$ and each $v'_i$ by $v_i$.
In this way, by Proposition~\ref{prop:highFrobsplit}, we can suppose from now on that $R \in \GL_{p^nd}(\Zpe(\oX))$ satisfies
\begin{equation}\label{eq:assumeFrob}
    R^{-1} \cdot \Phi(A_i)^{p} \cdot R = \diag\big(A_i, \ldots, A_i\big)
\end{equation}
for $i = 1, \ldots, N$.

\paragraph{Additional condition: homogeneity.}
We now show we can suppose $v_0 = 0$ without loss of generality.
Indeed, let $\widetilde{\mZ}$ denote the set of solutions $(z_{11}, \ldots, z_{KN}, z_{01}, \ldots, z_{0N}) \in \Z^{(K+1)N}$ to the equation
\[
\sum_{i = 1}^K A_1^{z_{i1}} A_2^{z_{i2}} \cdots A_N^{z_{iN}} v_i + A_1^{z_{01}} A_2^{z_{02}} \cdots A_N^{z_{0N}} \cdot (-v_0) = 0.
\]
Then $\mZ = \widetilde{\mZ} \cap \{(z_{11}, \ldots, z_{KN}, z_{01}, \ldots, z_{0N}) \mid z_{01} = \cdots = z_{0N} = 0\}$.
The second set of the intersection is obviously $p$-normal.
Therefore in order to show that $\mZ$ is $p$-normal, it suffices to show that $\widetilde{\mZ}$ is $p$-normal, because the intersection of two $p$-normal sets is $p$-normal (Proposition~\ref{prop:internormal}).
Hence, by replacing $K$ with $K+1$, we reduce to the case of \emph{homogeneous} equations, that is, where the right hand side of \eqref{eq:Sunitmat} is zero.
From now on we suppose without loss of generality that $v_0 = 0$.

\bigskip

For any sequence $\gamma \colon \Z^{KN} \rightarrow \mV$ of the form
\begin{equation}\label{eq:formgamma}
\gamma(z_{11}, \ldots, z_{KN}) = \sum_{i = 1}^K A_1^{z_{i1}} A_2^{z_{i2}} \cdots A_N^{z_{iN}} w_i, \quad \text{ where } w_1, \ldots, w_N \in \mV,
\end{equation}
denote
\[
\mZ(\gamma) \coloneqq \left\{(z_{11}, \ldots, z_{KN}) \in \Z^{KN} \;\middle|\; \gamma(z_{11}, \ldots, z_{KN}) = 0 \right\}.
\]
Our goal now is to show that $\mZ(\alpha)$ is $p$-normal, where $\alpha(z_{11}, \ldots, z_{KN}) \coloneqq \sum_{i = 1}^K A_1^{z_{i1}} A_2^{z_{i2}} \cdots A_N^{z_{iN}} v_i$.

For $j = 0, 1, \ldots, p^n-1$, denote the projection
\[
\pi_j \colon \Phi(\mV) = \Zpe(\oX)^{p^nd} \longrightarrow \mV = \Zpe(\oX)^d, \quad (f_1, f_2, \ldots, f_{p^nd}) \mapsto (f_{jd+1}, \ldots, f_{jd+d}).
\]
To show that $\mZ(\alpha)$ is $p$-automatic, we need to describe $\Theta_{\epsilon_{11}, \ldots, \epsilon_{KN}}(\mZ(\alpha))$ for all $(\epsilon_{11}, \ldots, \epsilon_{KN}) \in \Sigma_p^{KN}$.
The next lemma expresses $\Theta_{\epsilon_{11}, \ldots, \epsilon_{KN}}(\mZ(\alpha))$ in terms of zero sets of other sequences.

\begin{lem}\label{lem:shiftgamma}
    Let 
    $
        \gamma(z_{11}, \ldots, z_{KN}) = \sum_{i = 1}^K A_1^{z_{i1}} A_2^{z_{i2}} \cdots A_N^{z_{iN}} w_i
    $.
    For any $(\epsilon_{11}, \ldots, \epsilon_{KN}) \in \Sigma_p^{KN}$, we have
    \[
    \Theta_{\epsilon_{11}, \ldots, \epsilon_{KN}}(\mZ(\gamma)) = \bigcap_{j = 0}^{p^n - 1} \mZ\left(\Upsilon_{\epsilon_{11}, \ldots, \epsilon_{KN}; j}(\gamma)\right),
    \]
    where
    \begin{equation}\label{eq:defUpsilon}
    \Upsilon_{\epsilon_{11}, \ldots, \epsilon_{KN}; j}(\gamma)(z_{11}, \ldots, z_{KN}) \coloneqq \sum_{i=1}^K A_1^{z_{i1}} A_2^{z_{i2}} \cdots A_N^{z_{iN}} \cdot \pi_j\left(R \Phi(A_1)^{\epsilon_{i1}} \cdots \Phi(A_N)^{\epsilon_{iN}} \Phi(w_i)\right).
    \end{equation}
    Here, $R \in \GL_{p^nd}(\Zpe(\oX))$ is defined in Equation~\eqref{eq:assumeFrob}.
\end{lem}
\begin{proof}
    We have $\gamma(z_{11}, \ldots, z_{KN}) = 0$ if and only if $\Phi(\gamma)(z_{11}, \ldots, z_{KN}) = 0$.
    Write $(z_{11}, \ldots, z_{KN}) = p \cdot (z'_{11}, \ldots, z'_{KN}) + (\epsilon_{11}, \ldots, \epsilon_{KN})$, then
    
    \begin{align*}
    & \; \Phi(\gamma)(z_{11}, \ldots, z_{KN}) \\
    = & \; \sum_{i=1}^K \Phi(A_1^{pz'_{i1} + \epsilon_{i1}} \cdots A_N^{pz'_{iN} + \epsilon_{iN}} w_i) \\
    = & \; \sum_{i=1}^K \left(\Phi(A_1)^p\right)^{z'_{i1}} \cdots \left(\Phi(A_N)^p\right)^{z'_{iN}} \cdot \Phi(A_1^{\epsilon_{i1}} \cdots A_N^{\epsilon_{in}}) \Phi(w_i) \\
    = & \; R^{-1} \sum_{i=1}^K \left(R \Phi(A_1)^p R^{-1}\right)^{z'_{i1}} \cdots \left(R \Phi(A_N)^p R^{-1}\right)^{z'_{iN}} \cdot R \Phi(A_1^{\epsilon_{i1}} \cdots A_N^{\epsilon_{in}}) \Phi(w_i) \\
    = & \; R^{-1} \sum_{i=1}^K \Big(\diag(A_1, \ldots, A_1)\Big)^{z'_{i1}} \cdots \Big(\diag(A_N, \ldots, A_N)\Big)^{z'_{iN}} \cdot R \Phi(A_1^{\epsilon_{i1}} \cdots A_N^{\epsilon_{in}}) \Phi(w_i) \quad \quad (\text{by \eqref{eq:assumeFrob}})\\
    = & \; R^{-1} \sum_{i=1}^K \diag\Big(A_1^{z'_{i1}} A_2^{z'_{i2}} \cdots A_N^{z'_{iN}}, \ldots, A_1^{z'_{i1}} A_2^{z'_{i2}} \cdots A_N^{z'_{iN}}\Big) \cdot R \Phi(A_1^{\epsilon_{i1}} \cdots A_N^{\epsilon_{in}}) \Phi(w_i)
    \end{align*}

    Note that $\Phi(\gamma) = 0$ if and only if $R\Phi(\gamma) = 0$, if and only if $\pi_j(R \Phi(\gamma)) = 0$ for all $j = 0, 1, \ldots, p^n-1$.
    That is,
    \[
    \mZ(\gamma) = \mZ(\Phi(\gamma)) = \bigcap_{j = 0}^{p^n - 1}\mZ(\pi_j(R \Phi(\gamma))),
    \]
    where
    \begin{align*}
    & \; \pi_j(R \Phi(\gamma))(z_{11}, \ldots, z_{KN}) \\
    = & \; \pi_j \left(R R^{-1}\sum_{i=1}^K \diag\Big(A_1^{z'_{i1}} A_2^{z'_{i2}} \cdots A_N^{z'_{iN}}, \ldots, A_1^{z'_{i1}} A_2^{z'_{i2}} \cdots A_N^{z'_{iN}}\Big) \cdot R \Phi(A_1^{\epsilon_{i1}} \cdots A_N^{\epsilon_{in}}) \Phi(w_i) \right) \\
    = & \; \sum_{i=1}^K \pi_j \left(\diag\Big(A_1^{z'_{i1}} A_2^{z'_{i2}} \cdots A_N^{z'_{iN}}, \ldots, A_1^{z'_{i1}} A_2^{z'_{i2}} \cdots A_N^{z'_{iN}}\Big) \cdot R \Phi(A_1^{\epsilon_{i1}} \cdots A_N^{\epsilon_{in}}) \Phi(w_i)\right)\\
    = & \; \sum_{i=1}^K A_1^{z'_{i1}} A_2^{z'_{i2}} \cdots A_N^{z'_{iN}} \cdot \pi_j\left(R\Phi(A_1^{\epsilon_{i1}} \cdots A_N^{\epsilon_{in}}) \Phi(w_i)\right).
    \end{align*}

    Therefore,
    \[
    \Theta_{\epsilon_{11}, \ldots, \epsilon_{KN}}(\mZ(\gamma)) = \bigcap_{j = 0}^{p^n - 1} \mZ\left(\Upsilon_{\epsilon_{11}, \ldots, \epsilon_{KN}; j}(\gamma)\right),
    \]
    where
    \[
    \Upsilon_{\epsilon_{11}, \ldots, \epsilon_{KN}; j}(\gamma)(z_{11}, \ldots, z_{KN}) \coloneqq \sum_{i=1}^K A_1^{z_{i1}} A_2^{z_{i2}} \cdots A_N^{z_{iN}} \cdot \pi_j\left(R\Phi(A_1^{\epsilon_{i1}} \cdots A_N^{\epsilon_{in}}) \Phi(w_i)\right).
    \]
\end{proof}

Based on Lemma~\ref{lem:shiftgamma}, we now construct a first automaton $\tU$ that accepts the zero set $\mZ(\alpha)$.
This automaton will have the critical flaw that it contains an \emph{infinite} number of states.
The final automaton $\mmU$ that we will construct later will be a finite \emph{sub-automaton} of $\tU$.

Let $\mG$ denote the set of all sequences of the form~\eqref{eq:formgamma}:
    \[
    \mG \coloneqq \left\{\gamma \;\middle|\; \gamma(z_{11}, \ldots, z_{KN}) \coloneqq \sum_{i = 1}^K A_1^{z_{i1}} \cdots A_{N}^{z_{iN}} w_i, w_1, \ldots, w_K \in \mV \right\}.
    \]
This is \emph{a priori} an infinite set.
We now construct the automaton $\tU$ as follows. 

\paragraph*{States of $\tU$.}
The state set of $\tU$ is
\[
    2^{\mG} \coloneqq \{W \mid W \subseteq \mG\},
\]
that is, the set of all subsets of $\mG$.
In other words, each state $W$ of $\tU$ is a set of sequences $\{\gamma_1, \gamma_2, \ldots \}$, which can be considered as system of S-unit equations ``$\gamma_1(z_{11}, \ldots, z_{KN}) = \gamma_2(z_{11}, \ldots, z_{KN}) = \cdots = 0$''.
We denote by $\mZ(W)$ its zero set:
\[
\mZ(W) \coloneqq \bigcap_{\gamma \in W} \mZ(\gamma).
\]
When $W$ is a singleton $\{\gamma\}$, we will not distinguish between $\mZ(\{\gamma\})$ and $\mZ(\gamma)$.

\paragraph*{Transitions of $\tU$.}
There is a transition from the state $W$ to $W'$, labeled $(\epsilon_{11}, \ldots, \epsilon_{KN}) \in \Sigma_p^{KN}$, if and only if
\[
    \bigcup_{\gamma \in W} \bigcup_{j = 0}^{p^n - 1} \big\{\Upsilon_{\epsilon_{11}, \ldots, \epsilon_{KN}; j}(\gamma)\big\} = W',
\]
see Figure~\ref{fig:tU}. Note that the above union might not be disjoint.

\begin{figure}[h]
        \centering
        \begin{tikzpicture}[>=triangle 45]
        \tikzset{elliptic state/.style={draw,ellipse}}
           \node[elliptic state] (0) at (0,0) {\small{$\{\gamma \mid \gamma \in W\}$}};
           \node[elliptic state] (1) at (9,0) {\footnotesize{$
           \big\{\Upsilon_{\epsilon_{11}, \ldots, \epsilon_{KN}; j}(\gamma) \;\big|\; \gamma \in W, j \in \{0, 1, \ldots, p^n - 1\}\big\}$}};
           \path[->] (0) edge [above] node [align=center]  {\footnotesize$(\epsilon_{11}, \ldots, \epsilon_{KN})$} (1);
        \end{tikzpicture}
        \caption{A transition of the automaton $\tU$.}
        \label{fig:tU}
    \end{figure}
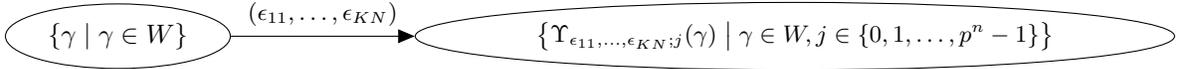

If there is a transition from $W$ to $W'$ labeled $(\epsilon_{11}, \ldots, \epsilon_{KN})$, then
\begin{align}\label{eq:transtU}
    \mZ(W') & = \bigcap_{\gamma' \in W'} \mZ(\gamma') \nonumber\\
        & = \bigcap_{\gamma' \in \bigcup_{\gamma \in W} \bigcup_{j = 0}^{p^k-1} \left\{\Upsilon_{\epsilon_{11}, \ldots, \epsilon_{KN}; j}(\gamma)\right\}} \mZ(\gamma') \nonumber\\
        & = \bigcap_{\gamma \in W} \bigcap_{j = 0}^{p^k-1} \mZ(\Upsilon_{\epsilon_{11}, \ldots, \epsilon_{KN}; j}(\gamma)) \nonumber\\
        & = \bigcap_{\gamma \in W} \Theta_{\epsilon_{11}, \ldots, \epsilon_{KN}}\left(\mZ(\gamma)\right) \nonumber\\
        & = \Theta_{\epsilon_{11}, \ldots, \epsilon_{KN}} \left(\mZ(W)\right)
\end{align}
by Lemma~\ref{lem:shiftgamma}.

\paragraph{Initial and accepting states of $\tU$.}
    
The initial state of $\tU$ is $\{\alpha\} \in 2^{\mG}$.
The accepting states of $\tU$ are those states $W \in 2^{\mG}$ satisfying
$
    (0, \ldots, 0) \in \mZ(W)
$.
Note that whether $(0, \ldots, 0) \in \mZ(W)$ can be checked effectively if $W$ contains finitely many sequences.
This can be done by simply checking whether $\gamma(0, \ldots, 0) = 0$ for all $\gamma \in W$.

From the definition of $\tU$, it is immediate that $\tU$ accepts exactly the set $\mZ(\alpha)$.
Indeed, suppose $\bz = \bepsilon_0 + p \bepsilon_1 + \cdots + p^{\ell} \bepsilon_{\ell} \in \mZ(\alpha)$, where $\bepsilon_{i} \in \Sigma_p^{KN}, i = 0, \ldots, \ell$.
Let $\delta\left(\{\alpha\}, \bepsilon_0\bepsilon_1\cdots\bepsilon_{\ell}\right)$ denote the state reached by reading the word ``$\bepsilon_0\bepsilon_1\cdots\bepsilon_{\ell}$'' starting from $\{\alpha\}$.
Then by Equation~\eqref{eq:transtU},
\[
\mZ\big(\delta\left(\{\alpha\}, \bepsilon_0\bepsilon_1\cdots\bepsilon_{\ell}\right)\big) = \Theta_{\bepsilon_{\ell}} \cdots \Theta_{\bepsilon_1}\Theta_{\bepsilon_0}\left(\mZ(\alpha)\right) \ni \Theta_{\bepsilon_{\ell}} \cdots \Theta_{\bepsilon_1}\Theta_{\bepsilon_0}\bz = \bzer.
\]
Therefore $\delta\left(\{\alpha\}, \bepsilon_0\bepsilon_1\cdots\bepsilon_{\ell}\right)$ is an accepting state.
Similarly, if $\delta\left(\{\alpha\}, \bepsilon_0\bepsilon_1\cdots\bepsilon_{\ell}\right)$ is an accepting state, then
\[
\bzer \in \mZ\big(\delta\left(\{\alpha\}, \bepsilon_0\bepsilon_1\cdots\bepsilon_{\ell}\right)\big) = \Theta_{\bepsilon_{\ell}} \cdots \Theta_{\bepsilon_1}\Theta_{\bepsilon_0}\left(\mZ(\alpha)\right),
\]
so $\bz = \bepsilon_0 + p \bepsilon_1 + \cdots + p^{\ell} \bepsilon_{\ell}$ belongs to $\mZ(\alpha)$.

Although $\tU$ is infinite, not all states of $\tU$ are reachable from $\{\alpha\}$.
We now show that the number of states reachable from $\{\alpha\}$ is in fact finite, by bounding the coefficients defining these states.
The following lemma can be considered as a generalization of~\cite[Proposition~5.2]{derksen2007skolem}.

\begin{lem}\label{lem:finite}
Let $\mS$ be a finite subset of $\mV = \Zpe(\oX)^{d}$.
Let $\mT$ be a finite set of matrices in $\GL_{p^n d}(\Zpe(\oX))$.
Then there exists an effectively computable finite set $\mS' \supseteq \mS$, such that for all $s' \in \mS'$, $T \in \mT$ and $j \in \{0, 1, \ldots, p^n-1\}$, we have $\pi_j(T\Phi(s')) \in \mS'$.
\end{lem}
\begin{proof}
    Let $h \in \Zpe[\oX]$ be a common denominator of all the entries appearing in elements of $\mT$ and $\mS$.
    We construct $\mS'$ as the set
    \begin{equation}\label{eq:defsprime}
    \mS' \coloneqq \left\{\left(\frac{g_1}{h^{2p^{e-1}}}, \ldots, \frac{g_d}{h^{2p^{e-1}}}\right) \in \Zpe(\oX)^d \;\middle|\; g_1, \ldots, g_d \in \Zpe[\oX], \; \deg(g_1) \leq c, \ldots, \deg(g_d) \leq c \right\}
    \end{equation}
    for some $c \in \N$ that we will specify later.
    Note that for any given $c$, the set $\mS'$ is finite since there are only finitely many polynomials in $\Zpe[\oX]$ with bounded degree.

    Let $s' = \left(\frac{g_1}{h^{2p^{e-1}}}, \ldots, \frac{g_d}{h^{2p^{e-1}}}\right) \in \mS'$.
    For any $j = 1, \ldots, d$, the product $g_j h^{p^e - 2p^{e-1}}$ can be written as
    \[
    g_j h^{p^e - 2p^{e-1}} = \sum_{\epsilon_1, \ldots, \epsilon_n \in \{0, 1, \ldots, p-1\}} F_{\epsilon_1, \ldots, \epsilon_n}(X_1^{p}, \ldots, X_n^p) \cdot X_1^{\epsilon_1} \cdots X_n^{\epsilon_n}
    \]
    with $F_{\epsilon_1, \ldots, \epsilon_n} \in \Zpe[\oX]$ for all $\epsilon_1, \ldots, \epsilon_n \in \{0, 1, \ldots, p-1\}$.
    Recall from Lemma~\ref{lem:powp} that we can write $h^{p^e}$ as $h^{p^{e-1}}(X_1^{p}, \ldots, X_n^p)$.
    So
    \[
    \frac{g_j}{h^{2p^{e-1}}} = \frac{g_j h^{p^e - 2p^{e-1}}}{h^{p^e}} = \sum_{\epsilon_1, \ldots, \epsilon_n \in \{0, 1, \ldots, p-1\}} \frac{F_{\epsilon_1, \ldots, \epsilon_n}(X_1^{p}, \ldots, X_n^p)}{h^{p^{e-1}}(X_1^{p}, \ldots, X_n^p)} \cdot X_1^{\epsilon_1} \cdots X_n^{\epsilon_n}.
    \]
    Therefore
    \[
    \Phi\left(\frac{g_j}{h^{2p^{e-1}}}\right) = \left(\frac{F_{0,\ldots,0}}{h^{p^{e-1}}}, \ldots, \frac{F_{p-1, \ldots, p-1}}{h^{p^{e-1}}}\right),
    \]
    where
    \begin{align*}
        \deg(F_{\epsilon_1, \ldots, \epsilon_n}) & \leq \frac{\deg(g_j h^{p^e - 2p^{e-1}}) - (\epsilon_1 + \cdots + \epsilon_n)}{p} \\
        & \leq \frac{\deg(g_j) + (p^e - 2p^{e-1}) \deg(h)}{p} \\
        & \leq \frac{c}{p} + \frac{(p^e - 2p^{e-1})\deg(h)}{p}.
    \end{align*}
    To sum up the above discussion, $\Phi(s') = \left(\Phi\left(\frac{g_1}{h^{2p^{e-1}}}\right), \ldots, \Phi\left(\frac{g_d}{h^{2p^{e-1}}}\right)\right)$ can be written as a tuple $\left(\frac{f_1}{h^{p^{e-1}}}, \ldots, \frac{f_{p^n d}}{h^{p^{e-1}}} \right)$, where each $f_k \in \Zpe[\oX], k = 1, \ldots, p^nd$, satisfies 
    \begin{equation}\label{eq:bounddegf}
    \deg(f_k) \leq \frac{c}{p} + \frac{(p^e - 2p^{e-1})\deg(h)}{p}.
    \end{equation}

    Write $\mT = \{T_1, \ldots, T_m\}$.
    By multiplying both the numerator and the denominator by a suitable polynomial, we can write out the coefficients $(t_{i\ell k})_{\ell, k \in \{1, \ldots, p^n d\}}$ of any $T_i \in \mT$ as
    \[
    t_{i\ell k} = \frac{a_{i\ell k}}{h^{p^{e-1}}}, \quad a_{i\ell k} \in \Zpe[\oX].
    \]
    Then $T_i \Phi(s') = (s_1, \ldots, s_{p^n d})$, where
    \[
    s_\ell = \frac{a_{i\ell1}}{h^{p^{e-1}}} \cdot \frac{f_1}{h^{p^{e-1}}} + \frac{a_{i\ell2}}{h^{p^{e-1}}} \cdot \frac{f_2}{h^{p^{e-1}}} + \cdots + \frac{a_{i\ell(p^nd)}}{h^{p^{e-1}}} \cdot \frac{f_{p^nd}}{h^{p^{e-1}}} = \frac{\sum_{k = 1}^{p^nd} a_{i\ell k} f_k}{h^{2p^{e-1}}}
    \]
    for $\ell = 1, \ldots, p^n d$.
    Furthermore,
    \begin{align*}
        \deg\left(\sum_{k = 1}^{p^nd} a_{i\ell k} f_k\right) & \leq \max_{k \in \{1, \ldots, p^n d\}} \left(\deg(a_{i\ell k}) + \deg(f_k)\right) \\
        & \leq \max_{k \in \{1, \ldots, p^n d\}} \deg(a_{i\ell k}) + \frac{c}{p} + \frac{(p^e - 2p^{e-1})\deg(h)}{p}
    \end{align*}
    by Inequality~\eqref{eq:bounddegf}.  
    Therefore, for any
    \begin{equation}\label{eq:defc}
        c \geq \frac{p \cdot\max_{i \in \{1, \ldots, m\}, \ell, k \in \{1, \ldots, p^n d\}}\deg(a_{ijk}) + (p^e - 2p^{e-1})\deg(h)}{p - 1},
    \end{equation}
    we will have $\deg\left(\sum_{k = 1}^{p^nd} a_{i\ell k} f_k\right) \leq c$ for $i = 1, \ldots, m;\; \ell = 1, \ldots, p^n d$.
    In this case, we have $\pi_j(T_i\Phi(s')) \in \mS'$ for $j = 0, 1, \ldots, p^n -1$.

    Recall that every coefficient appearing in elements of $\mS$ can be written as $\frac{g}{h} = \frac{g h^{2p^{e-1} - 1}}{h^{2p^{e-1}}}$.
    Therefore we can take $c$ large enough so that $\mS'$ contains every element of $\mS$.
    By enlarging $c$ so that it satisfies Condition~\eqref{eq:defc}, we obtain $\mS' \supseteq \mS$ such that $\pi_j(T\Phi(s')) \subseteq \mS'$ for all $s' \in \mS'$, $T \in \mT$ and $j \in \{0, 1, \ldots, p^n-1\}$.
\end{proof}

\paragraph{The finite automaton $\mmU$.}
We now construct the finite sub-automaton $\mmU$ of $\tU$ by bounding the states reachable from $\{\alpha\}$, where
$
\alpha(z_{11}, \ldots, z_{KN}) = \sum_{i = 1}^K A_1^{z_{i1}} A_2^{z_{i2}} \cdots A_N^{z_{iN}} v_i
$.

Recall that taking a transition $(\epsilon_{11}, \ldots, \epsilon_{KN})$ from a state $W$, we reach the new state
\[
    \bigcup_{\gamma \in W} \bigcup_{j = 0}^{p^n - 1} \big\{\Upsilon_{\epsilon_{11}, \ldots, \epsilon_{KN}; j}(\gamma)\big\}
\]
where $\Upsilon_{\epsilon_{11}, \ldots, \epsilon_{KN}; j}(\gamma)(z_{11}, \ldots, z_{KN})$ is defined as
\begin{equation*}
    \sum_{i=1}^K A_1^{z_{i1}} A_2^{z_{i2}} \cdots A_N^{z_{iN}} \cdot \pi_j\left(R\Phi(A_1^{\epsilon_{i1}} \cdots A_N^{\epsilon_{in}}) \Phi(w_i)\right)
\end{equation*}
for $\gamma = \sum_{i = 1}^K A_1^{z_{i1}} A_2^{z_{i2}} \cdots A_N^{z_{iN}} w_i$.

Apply Lemma~\ref{lem:finite} with $\mS \coloneqq \{v_1, \ldots, v_K\}$, and
\[
\mT \coloneqq \left\{R \Phi(A_1^{\epsilon_{i1}} \cdots A_N^{\epsilon_{in}}) \,\middle|\, \epsilon_{11}, \ldots, \epsilon_{KN} \in \Sigma_p \right\},
\]
we obtain a finite set $\mS' \supseteq \mS$ satisfying
\begin{equation}\label{eq:pisins}
\pi_j\left(R\Phi(A_1^{\epsilon_{i1}} \cdots A_N^{\epsilon_{in}}) \Phi(w)\right) \in \mS'
\end{equation}
for all $w \in \mS'$, $\epsilon_{11}, \ldots, \epsilon_{KN} \in \Sigma_p$ and $j \in \{0, 1, \ldots, p^n - 1\}$.
Let $\mH \subset \mG$ be the set of sequences whose coefficients are in $\mS'$:
\[
    \mH \coloneqq \left\{\gamma \;\middle|\; \gamma(z_{11}, \ldots, z_{KN}) \coloneqq \sum_{i = 1}^K A_1^{z_{i1}} \cdots A_{N}^{z_{iN}} w_i,\; w_1, \ldots, w_K \in \mS' \right\}.
\]
Then $\mH$ is finite and contains the sequence $\alpha$, and $\gamma \in \mH \implies \Upsilon_{\epsilon_{11}, \ldots, \epsilon_{KN}; j}(\gamma) \in \mH$ for all $\epsilon_{11}, \ldots, \epsilon_{KN} \in \Sigma_p,\; j = 0, 1, \ldots, p^n -1$. 
Therefore, a state in $2^{\mH} \coloneqq \{W \mid W \subseteq \mH\}$ can only reach other states in $2^{\mH}$.
We now take $\mmU$ to be the sub-automaton of $\tU$ containing all the states in $2^{\mH}$.
Since it contains the initial state $\{\alpha\}$, the finite automaton $\mmU$ accepts the zero set $\mZ(\alpha)$.

\subsection{Decomposition of $\mmU$ into strongly connected components}\label{subsec:strongconn}
In the previous subsection we have shown that the zero set $\mZ(\alpha)$ is $p$-automatic by constructing the automaton $\mmU$.
Our next step is to refine this result from $p$-automaticity to $p$-normality.
This refinement will be done by a combined analysis of the structure of $\mmU$ and the structure of $\alpha$.
In this subsection we analyze the \emph{strongly connected components} of $\mmU$.
We show that, roughly speaking, these strongly connected components will contribute to the subgroup $H$ in the definition~\eqref{eq:psuccinct} of $p$-succinct sets.

\paragraph{Multiplicative independence.} 
We can suppose $A_1, \ldots, A_N \in \mA$ to be multiplicatively independent, that is, 
\[
A_1^{z_1} A_2^{z_2} \cdots A_N^{z_N} = 1 \implies z_1 = z_2 = \cdots = z_N = 0.
\]
We can do so without loss of generality.
In fact, suppose $A_1, \ldots, A_N$ are not multiplicatively independent\footnote{Multiplicative dependence can be effectively computed using Noskov's Lemma \cite{noskov1982conjugacy}~\cite[Proposition~2.4]{baumslag1994algorithmic}}.
Then take a maximal subset of $\{A_1, \ldots, A_N\}$ that is multiplicatively independent, and without loss of generality denote them by $A_1, \ldots, A_s$.
For each $j = s+1, \ldots, N$, there exists $t_j \geq 1$ such that $A_j^{t_j}$ is in the multiplicative subgroup generated by $A_1, \ldots, A_s$.
We can write the zero set $\mZ(\alpha)$ as a finite union
\begin{multline*}
\bigcup_{r_{1j}, \ldots, r_{Kj} \in \{0, 1, \ldots, t_{j}-1\}, j = s+1, \ldots, N} \Bigg\{ (z_{11}, \ldots, z_{1s}, r_{1(s+1)}, \ldots, r_{1N}, \ldots, z_{K1}, \ldots, z_{Ks}, r_{K(s+1)}, \ldots r_{KN}) \\
\;\Bigg|\; \sum_{i = 1}^K A_1^{z_{i1}} \cdots A_s^{z_{s1}} \cdot \left(A_{s+1}^{r_{i(s+1)}} \cdots A_N^{r_{iN}} v_i \right) = 0 \Bigg\}.
\end{multline*}
Thus, we have reduced the problem to showing that the solution set $(z_{11}, \ldots, z_{1s}, \ldots, z_{K1}, \ldots, z_{Ks})$ for each equation 
\[
\sum_{i = 1}^K A_1^{z_{i1}} \cdots A_s^{z_{s1}} \cdot \left(A_{s+1}^{r_{i(s+1)}} \cdots A_n^{r_{iN}} v_i \right) = 0
\]
is $p$-normal, where $A_1, \ldots, A_s$ are multiplicatively independent.
Replacing $s$ by $N$, we therefore reduce to the case where the elements $A_1, A_2, \ldots, A_N$, are multiplicatively independent.

\bigskip

From now on, let $\oA$ denote the tuple $(A_1, \ldots, A_N)$.
For a vector $\bz = (z_1, \ldots, z_N) \in \Z^N$, we write
$
\oA^{\bz} \coloneqq A_1^{z_1} A_2^{z_2} \cdots A_N^{z_N}
$.
Recall that $\mA$ is a local ring with maximal ideal $\frm \ni p$, such that $\frm^t = 0$ for some $t \geq 1$.

\begin{lem}\label{lem:independentmodq}
    The maps $A_1, \ldots, A_N$ are multiplicatively independent in $\mA$ if and only if they are multiplicatively independent in the quotient $\mA/\frm$.
\end{lem}
\begin{proof}
    If $A_1, \ldots, A_N$ are multiplicatively independent in the quotient $\mA/\frm$, then they are obviously multiplicatively independent in $\mA$.
    
    Suppose $A_1, \ldots, A_N$ multiplicatively independent in $\mA$, and let $\bz \in \Z^N$ be such that $\oA^{\bz} \equiv 1 \mod \frm$.
    By Lemma~\ref{lem:lifting}, we have $\oA^{p^{t-1}\bz} = 1$, therefore $p^{t-1}\bz = \bzer$ by the multiplicative independence of $A_1, \ldots, A_N$ in $\mA$.
    Therefore $\bz = \bzer$, and $A_1, \ldots, A_N$ are multiplicatively independent in $\mA/\frm$.
\end{proof}

\paragraph{Prototype of the subgroup $H$.}

Let $\bz_1 = (z_{11}, \ldots, z_{1N}), \ldots, \bz_K = (z_{K1}, \ldots, z_{KN}) \in \Z^N$, then $\alpha(z_{11}, \ldots, z_{KN}) = \sum_{i = 1}^K A_1^{z_{i1}} \cdots A_{N}^{z_{iN}} v_i$ can be written as $\alpha(\bz_1, \ldots, \bz_K) = \sum_{i = 1}^K \oA^{\bz_i} v_i$.

We start by giving some intuition of the subgroup $H$ in the definition~\eqref{eq:psuccinct} of $p$-succinct sets.
For any $b \in \{1, \ldots, N\}$, we can replace $z_{1b}, z_{2b}, \ldots, z_{Kb}$, by $z_{1b}+1, z_{2b}+1, \ldots, z_{Kb}+1$, respectively, this will give us the sequence $A_b \cdot \alpha$.
Therefore, if $(z_{11}, \ldots, z_{KN}) \in \mZ(\alpha)$, then $(z_{11}, \ldots, z_{KN}) + \be_{\{1, \ldots, K\}, b}$ is also in $\mZ(\alpha)$, where
\[
\be_{\{1, \ldots, K\}, b} \coloneqq (\be_1, \ldots, \be_K), \quad \text{ with } \be_1 = \cdots = \be_K = (0, \ldots, 0, \underset{\underset{\text{$b$-th index}}{\uparrow}}{1}, 0, \ldots, 0) \in \Z^N.
\]
Therefore, $\mZ(\alpha)$ is stable under translation by the group generated by $\be_{\{1, \ldots, K\}, 1}, \ldots, \be_{\{1, \ldots, K\}, N}$:
\[
\mZ(\alpha) = \mZ(\alpha) + \sum_{b = 1}^N \Z \be_{\{1, \ldots, K\}, b}.
\]
In this case, the subgroup $\sum_{b = 1}^N \Z \be_{\{1, \ldots, K\}, b}$ a prototype of the subgroup $H$ in the definition~\eqref{eq:psuccinct} of $p$-succinct sets.

Let us consider a more complicated example.
Let $W = \{\beta, \gamma\}$ be a state of $\mmU$, where $\beta = \sum_{i = 1}^{k} \oA^{\bz_i} w_i,\; \gamma = \sum_{i = k+1}^{K} \oA^{\bz_i} w'_i$, for some $1 \leq k \leq K$.
Then by the homogeneity of $\beta$, for any $b \in \{1, \ldots, N\}$, we can replace $z_{1b}, z_{2b}, \ldots, z_{kb}$, by $z_{1b}+1, z_{2b}+1, \ldots, z_{kb}+1$, without changing the solution set to $\beta = 0$.
This also does not change the solution set to $\gamma = 0$, because the variables $z_{1b}, z_{2b}, \ldots, z_{kb}$ do not appear in $\gamma$ at all.
Similarly, we can replace $z_{(k+1)b}, z_{(k+2)b}, \ldots, z_{Kb}$, by $z_{(k+1)b} + 1, z_{(k+2)b} + 1, \ldots, z_{Kb} + 1$, without changing the solution set.
This shows that $\mZ(\{\beta, \gamma\})$ is stable under translation by the group generated by $\be_{\{1, \ldots, k\}, 1}, \ldots, \be_{\{1, \ldots, k\}, N}$, $\be_{\{k+1, \ldots, K\}, 1}, \ldots, \be_{\{k+1, \ldots, K\}, N}$:
\[
\mZ(\{\beta, \gamma\}) = \mZ(\{\beta, \gamma\}) + \sum_{b = 1}^N \Z \be_{\{1, \ldots, k\}, b} + \sum_{b = 1}^N \Z \be_{\{k+1, \ldots, K\}, b},
\]
where for any set $S \subseteq \{1, \ldots, K\}$,
\begin{equation*}
\be_{S, b} \coloneqq (\be_1, \ldots, \be_K), \quad \text{ with } 
\be_i =
\begin{cases}
    (0, \ldots, 0, \underset{\underset{\text{$b$-th index}}{\uparrow}}{1}, 0, \ldots, 0), \quad \text{ if } i \in S \\
    (0, \ldots, 0, \underset{\underset{\text{$b$-th index}}{\uparrow}}{0}, 0, \ldots, 0) , \quad  \text{ if } i \notin S.
\end{cases}
\end{equation*}
If additionally there is a transition labeled $(\epsilon_{11}, \ldots, \epsilon_{KN})$ from $\{\alpha\}$ to $\{\beta, \gamma\}$, then $\Theta_{\epsilon_{11}, \ldots, \epsilon_{KN}}\mZ(\alpha) = \mZ(\{\beta, \gamma\})$ is stable under translation by $\sum_{b = 1}^N \Z \be_{\{1, \ldots, k\}, b} + \sum_{b = 1}^N \Z \be_{\{k+1, \ldots, K\}, b}$.
So $\mZ(\alpha)$ contains a subset that is stable under translation by $p \cdot \left(\sum_{b = 1}^N \Z \cdot \be_{\{1, \ldots, k\}, b} + \sum_{b = 1}^N \Z \cdot \be_{\{k+1, \ldots, K\}, b} \right)$: this is another prototype of the subgroup $H$ in $p$-succinct sets.

More generally, we can replace $\{1, \ldots, k\}, \{k+1, \ldots, K\}$ in the above example by any partition of $\{1, \ldots, K\}$.
This motivates the following definition.

\begin{defn}
    A \emph{partition} of the set $\{1, \ldots, K\}$ is defined as a family $\Pi = \{S_1, S_2, \ldots, S_r\}$, where $S_1, S_2, \ldots, S_r$ are non-empty disjoint subsets of $\{1, \ldots, K\}$, such that $S_1 \cup \cdots \cup S_r = \{1, \ldots, K\}$.
    The sets $S_1, S_2, \ldots, S_r$ are called \emph{blocks} of $\Pi$.
    For any $b \in \{1, \ldots, N\}$ and any subset $S \subseteq \{1, \ldots, K\}$, define
    \begin{equation}\label{eq:defeSb}
        \be_{S, b} \coloneqq (e_{11}, \ldots, e_{KN}), \quad e_{ij} = 1 \text{ for } i \in S, j = b; \text{ and } e_{ij} = 0 \text{ otherwise}.
    \end{equation}
    For any partition $\Pi$ of the set $\{1, \ldots, K\}$, define $(\Z^N)^{\Pi}$ to be the subgroup of $\Z^{KN}$ generated by the elements $\be_{S, b},\; S \in \Pi, b \in \{1, \ldots, N\}$:
    \[
    (\Z^N)^{\Pi} \coloneqq \sum_{S \in \Pi} \sum_{b = 1}^N \Z \be_{S, b}.
    \]
\end{defn}

For a path $\pi$ in the automaton $\mmU$, let $\len(\pi)$ denote the length of $\pi$.
Recall that $\eval(\pi) \in \Z^{KN}$ denotes the evaluation of $\pi$.
For two states $W, V$ of the automaton $\mmU$, denote by $L(W, V)$ the set of paths from $W$ to $V$.

The following lemma shows that, whenever a state $W$ of $\mmU$ appears in two distinct cycles of the same length, the zero set $\mZ(W)$ is stable under translation by a subgroup of the form $(\Z^N)^{\Pi}$.

\begin{lem}\label{lem:stablePi}
    Let $W \in 2^{\mH}$ be a state of the automaton $\mmU$.
    Suppose there are two cycles $C_1, C_2 \in L(W, W)$ such that $\len(C_1) = \len(C_2) = \ell$.
    Write $\eval(C_1) = (\br_1, \ldots, \br_K)$, $\eval(C_2) = (\br_1 + \bsig_{1}, \ldots, \br_K + \bsig_{K})$, with $\br_i, \br_i + \bsig_i \in \{-(p^{\ell} - 1), \ldots, 0, \ldots, p^{\ell} - 1\}^{N}$ for all $i$.
    Let $\Pi$ be the partition of $\{1, \ldots, K\}$ such that $i, j \in \{1, \ldots, K\}$ fall in the same block of $\Pi$ if and only if $\bsig_i = \bsig_j$.\footnote{For example, if $\bsig_1 = (5,6), \bsig_2 = (0, 0), \bsig_3 = (5, 6)$, then $\Pi$ is the family $\big\{\{1, 3\}, \{2\}\big\}$.}
    Then
    \begin{equation}\label{eq:stabletrans}
    \mZ(W) = \mZ(W) + (\Z^N)^{\Pi}.
    \end{equation}
\end{lem}
\begin{proof}
    Let $S \subseteq \{1, \ldots, K\}$ be any block of $\Pi$ and let $b \in \{1, \ldots, N\}$, we will prove
    \begin{equation}
        \mZ(W) = \mZ(W) + \be_{S, b}
    \end{equation}
    for the generator $\be_{S, b}$ of $(\Z^N)^{\Pi}$.

    If $S = \{1, \ldots, K\}$ then Equation~\eqref{eq:stabletrans} is obviously satisfied thanks to the homogeneity of each $\gamma \in W$.
    Therefore suppose $S \neq \{1, \ldots, K\}$, pick another block $S' \neq S$ in the partition $\Pi$.

    For each $\gamma \in W$, $\gamma(\bz_1, \ldots, \bz_K) = \sum_{i = 1}^K \oA^{\bz_i} w_i$, we claim that
    \begin{equation}\label{eq:interdegen}
    \Theta_{\ell; \br_1, \ldots, \br_K}(\mZ(\gamma)) \cap \Theta_{\ell; \br_{1} + \bsig_{1}, \ldots, \br_{K} + \bsig_{K}}(\mZ(\gamma)) = \mZ(\gamma_S) \cap \mZ(\gamma_{S'})
    \end{equation}
    where
    \begin{equation}\label{eq:defbeta}
    \gamma_S(\bz_{1}, \ldots, \bz_{K}) = \sum_{i \in \{1, \ldots, K\} \setminus S} \oA^{p^{\ell} \bz_i} x_i,
    \end{equation}
    with some $x_i \in \mV, i \in \{1, \ldots, K\} \setminus S$, and
    \begin{equation}\label{eq:defbetap}
    \gamma_{S'}(\bz_{1}, \ldots, \bz_{K}) = \sum_{i \in \{1, \ldots, K\} \setminus S'} \oA^{p^{\ell} \bz_i} x'_i,
    \end{equation}
    with some $x'_i \in \mV, i \in \{1, \ldots, K\} \setminus S'$.\footnote{Here, the subscript $S$ in $\gamma_S$ is meant to suggest that the coefficients $x_i$ in $\gamma_S$ vanish for $i \in S$. Same for $\gamma_{S'}$.}
    Here, $\gamma_S$ and $\gamma_{S'}$ can be seen as sequences over the tuple $\oA^{p^{\ell}} \coloneqq \left(A_1^{p^{\ell}}, \ldots, A_N^{p^{\ell}}\right)$.
    
    Indeed, we have $(\bz_1, \ldots, \bz_K) \in \Theta_{\ell; \br_1, \ldots, \br_K}(\mZ(\gamma))$ if and only if
    \begin{equation}\label{eq:thetar}
        \sum_{i = 1}^K \oA^{p^{\ell} \bz_{i}} \cdot \oA^{\br_i} w_i = 0.
    \end{equation}
    Similarly, we have $(\bz_1, \ldots, \bz_K) \in \Theta_{\ell; \br_{1} + \bsig_{1}, \ldots, \br_{K} + \bsig_{K}}(\mZ(\gamma))$ if and only if
    \begin{equation}\label{eq:thetars}
        \sum_{i = 1}^K \oA^{p^{\ell} \bz_{i}} \cdot \oA^{\br_i + \bsig_i} w_i = 0.
    \end{equation}
    Since $S$ is a block in $\Pi$, there exists $\ba \in \Z^N$, such that $\bsig_i = \ba$ for all $i \in S$, and $\bsig_i \neq \ba$ for all $i \notin S$.
    Similarly, there exists $\ba' \in \Z^N$, such that $\bsig_i = \ba'$ for all $i \in S'$, and $\bsig_i \neq \ba'$ for all $i \notin S'$.
    Of course, $\ba \neq \ba'$.
    Let
    \[
        \gamma_S \coloneqq \oA^{\ba} \cdot \left(\sum_{i = 1}^K \oA^{p^{\ell} \bz_{i}} \cdot \oA^{\br_i} w_i\right)
        - \left(\sum_{i = 1}^K \oA^{p^{\ell} \bz_{i}} \cdot \oA^{\br_i + \bsig_i} w_i\right),
    \]
    which can be written in the form~\eqref{eq:defbeta} with $x_i = \oA^{\ba + \br_i} w_i - \oA^{\br_i + \bsig_i} w_i$.
    This is because for $i \in S$, the coefficient $x_i$ vanishes by $\ba = \bsig_i$.
    
    Similarly, let
    \[
        \gamma_{S'} \coloneqq \oA^{\ba'} \cdot \left(\sum_{i = 1}^K \oA^{p^{\ell} \bz_{i}} \cdot \oA^{\br_i} w_i\right)
        - \left(\sum_{i = 1}^K \oA^{p^{\ell} \bz_{i}} \cdot \oA^{\br_i + \bsig_i} w_i\right),
    \]
    which can be written in the form~\eqref{eq:defbetap} with $x'_i = \oA^{\ba' + \br_i} w_i - \oA^{\br_i + \bsig_i} w_i$.
    This is because for $i \in S'$, the coefficient $x'_i$ vanishes.

    The system of equations $\gamma_S = \gamma_{S'} = 0$ is a linear transformation of the system of Equations~\eqref{eq:thetar} and~\eqref{eq:thetars}.
    We claim that the transformation matrix
    $
    \begin{pmatrix}
        \oA^{\ba} & -1 \\
        \oA^{\ba'} & -1 \\
    \end{pmatrix}
    $
    is invertible, so the two systems are equivalent.
    Indeed, the determinant of the transformation matrix is $\oA^{\ba'} - \oA^{\ba} = \oA^{\ba}\left(\oA^{\ba' - \ba} - 1\right)$. We have $\oA^{\ba' - \ba} \not\equiv 1 \mod \frm$, by the multiplicative independence of $A_1, \ldots, A_N$ and Lemma~\ref{lem:independentmodq}. Hence $\oA^{\ba' - \ba} - 1 \notin \frm$, and is therefore invertible.
    Consequently the transformation matrix has an invertible determinant and is therefore invertible.
    We thus conclude that the system $\gamma_S = \gamma_{S'} = 0$ is equivalent to the system of Equations~\eqref{eq:thetar} and \eqref{eq:thetars}.
    In other words,
    \[
    \Theta_{\ell; \br_1, \ldots, \br_K}(\mZ(\gamma)) \cap \Theta_{\ell; \br_{1} + \bsig_{1}, \ldots, \br_{K} + \bsig_{K}}(\mZ(\gamma)) = \mZ(\gamma_S) \cap \mZ(\gamma_{S'}).
    \]
    
    Since there are length-$\ell$ paths from $W$ to $W$ evaluated at $(\br_1, \ldots, \br_K)$ and $(\br_1 + \bsig_1, \ldots, \br_K + \bsig_K)$, we have
    \[
    \mZ(W) = \Theta_{\ell; \br_1, \ldots, \br_K}(\mZ(W)) = \Theta_{\ell; \br_1 + \bsig_1, \ldots, \br_K + \bsig_K}(\mZ(W)).
    \]
    Therefore
    \begin{align*}
        \mZ(W)
        = & \; \Theta_{\ell; \br_1, \ldots, \br_K}(\mZ(W)) \cap \Theta_{\ell; \br_1 + \bsig_1, \ldots, \br_K + \bsig_K}(\mZ(W)) \\
        = & \; \bigcap_{\gamma \in W} \Theta_{\ell; \br_1, \ldots, \br_K}(\mZ(\gamma)) \cap \bigcap_{\gamma \in W} \Theta_{\ell; \br_1 + \bsig_1, \ldots, \br_K + \bsig_K}(\mZ(\gamma)) \\
        = & \; \bigcap_{\gamma \in W} \Big(\Theta_{\ell; \br_1, \ldots, \br_K}(\mZ(\gamma)) \cap \Theta_{\ell; \br_1 + \bsig_1, \ldots, \br_K + \bsig_K}(\mZ(\gamma))\Big) \\
        = & \; \bigcap_{\gamma \in W} \big(\mZ(\gamma_S) \cap \mZ(\gamma_{S'}) \big).
    \end{align*}
    Consider two cases:

    \begin{enumerate}[nosep, wide, label=\textbf{Case~\arabic*:}]
        \item 
    If $S = \{1, \ldots, K\} \setminus S'$.
    Then both $\mZ(\gamma_S)$ and $\mZ(\gamma_{S'})$ are stable under translation by $\boldsymbol{e}_{S, b}$.
    This is because $\gamma_{S'} = \sum_{i \in \{1, \ldots, K\} \setminus S'} \oA^{p^{\ell}\bz_{i}} x'_i = \sum_{i \in S} \oA^{p^{\ell}\bz_{i}} x'_i$ contains only terms for $i \in S$, while $\gamma_S$ does not contain terms for $i \in S$.
    Consequently, $\mZ(W)$ is stable under translation by $\boldsymbol{e}_{S, b}$.

    \item 
    If $S \subsetneq \{1, \ldots, K\} \setminus S'$.
    Then $\mZ(\gamma_S)$ is stable under translation by $\boldsymbol{e}_{S, b}$, but $\mZ(\gamma_{S'})$ is not necessarily stable under translation by $\boldsymbol{e}_{S, b}$.
    Pick another set $S'' \notin \{S, S'\}$ in the partition $\Pi$ and repeat the above process for $(\gamma_S, S'')$ and $(\gamma_{S'}, S'')$ in place of $(\gamma, S')$.
    That is, we can write
    \[
    \Theta_{\ell; \br_1, \ldots, \br_K}(\mZ(\gamma_{S})) \cap \Theta_{\ell; \br_1 + \bsig_1, \ldots, \br_K + \bsig_K}(\mZ(\gamma_{S})) = \mZ(\gamma_{S, S}) \cap \mZ(\gamma_{S, S''}),
    \]
    \[
    \Theta_{\ell; \br_1, \ldots, \br_K}(\mZ(\gamma_{S'})) \cap \Theta_{\ell; \br_1 + \bsig_1, \ldots, \br_K + \bsig_K}(\mZ(\gamma_{S'})) = \mZ(\gamma_{S', S}) \cap \mZ(\gamma_{S', S''}),
    \]
    where each sequence $\gamma_{S_1, S_2}, \; S_1, S_2 \in \{S, S', S''\}$, has the form 
    \[
    \gamma_{S_1, S_2} = \sum_{i \in (\{1, \ldots, K\} \setminus S_1)\setminus S_2} \oA^{p^{2\ell} \bz_{i}} x_i
    \]
    with some $x_i \in \mV$.
    Then
    \begin{align*}
        \mZ(W) = & \; \Theta_{\ell; \br_1, \ldots, \br_K}(\mZ(W)) \cap \Theta_{\ell; \br_1 + \bsig_1, \ldots, \br_K + \bsig_K}(\mZ(W)) \\
        = & \; \Theta_{\ell; \br_1, \ldots, \br_K}\left(\bigcap_{\gamma \in W} \big(\mZ(\gamma_S) \cap \mZ(\gamma_{S'}) \big)\right) \cap \Theta_{\ell; \br_1 + \bsig_1, \ldots, \br_K + \bsig_K}\left(\bigcap_{\gamma \in W} \big(\mZ(\gamma_S) \cap \mZ(\gamma_{S'}) \big)\right) \\
        = & \; \bigcap_{\gamma \in W} \big( \Theta_{\ell; \br_1, \ldots, \br_K}(\mZ(\gamma_{S})) \cap \Theta_{\ell; \br_1 + \bsig_1, \ldots, \br_K + \bsig_K}(\mZ(\gamma_{S})) \big) \\
        & \quad \quad \quad \cap \bigcap_{\gamma \in W} \Theta_{\ell; \br_1, \ldots, \br_K}(\mZ(\gamma_{S'})) \cap \Theta_{\ell; \br_1 + \bsig_1, \ldots, \br_K + \bsig_K}(\mZ(\gamma_{S'})) \\
        = & \; \bigcap_{\gamma \in W} \big(\mZ(\gamma_{S, S}) \cap \mZ(\gamma_{S, S''}) \cap \mZ(\gamma_{S', S}) \cap \mZ(\gamma_{S', S''}) \big)
    \end{align*}
    
    Repeating this process until we have found $S, S', S'', S''', \ldots$, such that the disjoint union $S \cup S' \cup S'' \cup S''' \cup \cdots$ is equal to $\{1, \ldots, K\}$.
    By doing so, we will have written
    \[
    \mZ(W) = \bigcap_{\gamma \in W} \big(\mZ(\gamma_{S, S, S, \ldots}) \cap \mZ(\gamma_{S, S'', S, \ldots}) \cap \cdots \cap \mZ(\gamma_{S', S'', S''', \ldots}) \big),
    \]
    where each set $\mZ(\gamma_{S_1, S_2, S_3, \ldots})$ except the last one contains $S$ in the subscript (and is hence stable under translation by $\be_{S, b}$ because the $\gamma_{S_1, S_2, S_3, \ldots}$ does not contain any term $x_i$ with $i \in S$).
    The last set, $\mZ(\gamma_{S', S'', S''', \ldots})$, is also stable under translation by $\be_{S, b}$ because $S' \cup S'' \cup S''' \cup \cdots = \{1, \ldots, K\} \setminus S$, so $\gamma_{S', S'', S''', \ldots}$ contains only terms $x_i$ with $i \in S$.
    Therefore, the total intersection $\mZ(W)$ is also stable under translation by $\boldsymbol{e}_{S, b}$.
    \end{enumerate}
    
    We have now shown that $\mZ(W) = \mZ(W) + \be_{S, b}$ for each generator $\be_{S, b}$ of $(\Z^N)^{\Pi}$.
    We conclude that $\mZ(W) = \mZ(W) + (\Z^N)^{\Pi}$.
\end{proof}

For two cycles $C_1, C_2 \in L(W, W)$ with $\len(C_1) = \len(C_2)$, we call the partition $\Pi$ defined in Lemma~\ref{lem:stablePi} the partition \emph{induced} by the pair $(C_1, C_2)$, and denote it by $\Pi_{(C_1, C_2)}$.

We say that a partition $\Pi$ is \emph{finer} than a partition $\Pi'$, if each block of $\Pi$ is a subset of some block of $\Pi'$.
For example, the partition $\big\{\{1,2\},\{3\}, \{4, 5\}\big\}$ is finer than the partition $\big\{\{1,2,3\}, \{4, 5\}\big\}$.
If $\Pi$ is finer than $\Pi'$, we write $\Pi \preceq \Pi'$.
For two different partitions $\Pi, \Pi'$, we can define their \emph{meet}, denoted by $\Pi \wedge \Pi'$, as the partition whose blocks are the intersections of a block of $\Pi$ and a block of $\Pi'$.
For example, if $\Pi$ is the family $\big\{\{1,2\},\{3, 4, 5\}\big\}$ and $\Pi'$ is the family $\big\{\{1,3\},\{2, 4, 5\}\big\}$, then $\Pi \wedge \Pi'$ is the family $\big\{\{1\},\{2\},\{3\}, \{4, 5\}\big\}$.
Obviously $\Pi \wedge \Pi' \preceq \Pi$ and $\Pi \wedge \Pi' \preceq \Pi'$.

Let $C_1, C_2, C'_1, C'_2 \in L(W, W)$ be such that $\len(C_1) = \len(C_2),\; \len(C'_1) = \len(C'_2)$, then we have $\len(C_1C_1C_1C'_1) = \len(C_2C_1C_1C'_2)$ for the concatenated cycles $C_1C_1C_1C'_1, C_2C_1C_1C'_2 \in L(W, W)$.

\begin{lem}\label{lem:meet}
    Let $C_1, C_2, C'_1, C'_2 \in L(W, W)$ be such that $\len(C_1) = \len(C_2),\; \len(C'_1) = \len(C'_2)$, then
    \begin{equation*}
        \Pi_{(C_1C_1C_1C'_1, C_2C_1C_1C'_2)} = \Pi_{(C_1, C_2)} \wedge \Pi_{(C'_1, C'_2)}.
    \end{equation*}
\end{lem}
\begin{proof}
    Let $\ell$ denote the length of $C_1$ and $C_2$.
    Write $\eval(C_1) = (\br_1, \ldots, \br_K)$ with $\br_i \in \{-(p^{\ell}-1), \ldots, 0, \ldots, p^{\ell}-1\}^N$, and $\eval(C_2) = (\br_1 + \bsig_1, \ldots, \br_K + \bsig_K)$ with $\br_i + \bsig_i \in \{-(p^{\ell}-1), \ldots, 0, \ldots, p^{\ell}-1\}^N$.
    Then $\|\bsig_i\| \leq 2 p^{\ell} - 2$, where $\|\bz\|$ denote the maximal absolute value among the entries of $\bz$.
    Similarly, let $\ell'$ denote the length of $C'_1$ and $C'_2$, and write $\eval(C'_1) = (\br'_1, \ldots, \br'_K)$, $\eval(C'_2) = (\br'_1 + \bsig'_1, \ldots, \br'_K + \bsig'_K)$.

    By direct computation, $\eval(C_1C_1C_1C'_1)$ amounts to
    \[
    \left(\br_1 + p^{\ell} \br_1 + p^{2\ell} \br_1 + p^{3\ell} \br'_1, \ldots, \br_K + p^{\ell} \br_K + p^{2\ell} \br_K + p^{3\ell} \br'_K\right),
    \]
    whereas $\eval(C_2C_1C_1C'_2)$ amounts to
    \[
    \left((\br_1+\bsig_1) + p^{\ell} \br_1 + p^{2\ell} \br_1 + p^{3\ell} (\br'_1 + \bsig'_1), \ldots, (\br_K+\bsig_K) + p^{\ell} \br_K + p^{2\ell} \br_K + p^{3\ell} (\br'_K + \bsig'_K)\right).
    \]
    Hence, the difference $\eval(C_2C_1C_1C'_2) - \eval(C_1C_1C_1C'_1)$ amounts to
    \begin{equation}
    \left(\bsig_1 + p^{3\ell} \bsig'_1, \ldots, \bsig_K + p^{3\ell} \bsig'_K\right).
    \end{equation}
    
    Recall that $i, j \in \{1, \ldots, K\}$ fall in the same block of $\Pi_{(C_1, C_2)}$ if and only if $\bsig_i = \bsig_j$.
    Similarly, $i, j$ fall in the same block of $\Pi_{(C'_1, C'_2)}$ if and only if $\bsig'_i = \bsig'_j$.
    Therefore $i, j$ fall in the same block of $\Pi_{(C_1, C_2)} \wedge \Pi_{(C'_1, C'_2)}$ if and only if both $\bsig_i = \bsig_j$ and $\bsig'_i = \bsig'_j$.
    Therefore, if $i, j$ fall in the same block of $\Pi_{(C_1, C_2)} \wedge \Pi_{(C'_1, C'_2)}$, then we have $\bsig_i + p^{3\ell} \bsig'_i = \bsig_j + p^{3\ell} \bsig'_j$, so $i, j$ are in the same block of $\Pi_{(C_1C_1C_1C'_1, C_2C_1C_1C'_2)}$.
    
    On the other hand, if $i, j$ fall in the same block of $\Pi_{(C_1C_1C_1C'_1, C_2C_1C_1C'_2)}$, then $\bsig_i + p^{3\ell} \bsig'_i = \bsig_j + p^{3\ell} \bsig'_j$, which can be rewritten as $\bsig_i - \bsig_j = p^{3\ell} (\bsig'_j - \bsig'_i)$.
    But $\|\bsig_i - \bsig_j\| \leq (2 p^{\ell} - 2) + (2 p^{\ell} - 2) < 4 p^{\ell} \leq p^{3 \ell}$, so we must have $\bsig'_i - \bsig'_j = 0$, and consequently $\bsig_i - \bsig_j = 0$.
    This shows that $i, j$ fall in the same block of $\Pi_{(C_1, C_2)} \wedge \Pi_{(C'_1, C'_2)}$.

    We conclude that $\Pi_{(C_1C_1C_1C'_1, C_2C_1C_1C'_2)} = \Pi_{(C_1, C_2)} \wedge \Pi_{(C'_1, C'_2)}$.
\end{proof}

Lemma~\ref{lem:meet} shows that, if $C_1, C_2, C'_1, C'_2 \in L(W, W)$ are such that $\len(C_1) = \len(C_2), \len(C'_1) = \len(C'_2)$, then there exist $C''_1, C''_2 \in L(W, W),\; \len(C''_1) = \len(C''_2)$, such that $\Pi_{(C''_1, C''_2)} = \Pi_{(C_1, C_2)} \wedge \Pi_{(C'_1, C'_2)}$.
This means that the set of partitions
\[
\mmP(W) \coloneqq \left\{\Pi_{(C_1, C_2)} \;\middle|\; C_1, C_2 \in L(W, W),\; \len(C_1) = \len(C_2) \right\}
\]
is closed under the meet operation (that is, $\Pi, \Pi' \in \mmP(W) \implies \Pi \wedge \Pi' \in \mmP(W)$).
Therefore, $\mmP(W)$ contains a finest element, which we denote by $\Pi(W)$.
Namely,
\[
\Pi(W) = \bigwedge_{C_1, C_2 \in L(W, W),\; \len(C_1) = \len(C_2)} \Pi_{(C_1, C_2)}.
\]

\begin{lem}
For any state $W$, the partition $\Pi(W)$ can be effectively computed.
\end{lem}
\begin{proof}
    It suffices to show that for any $i, j \in \{1, \ldots, K\}$, we can check whether $i, j$ belong to the same block of $\Pi(W)$.
    This can be done the following way.
    For any cycle $C$, write $\eval(C) = \left(\eval(C)_1, \ldots, \eval(C)_K\right)$, where each $\eval(C)_i$ is a vector in $\Z^N$.
    Then $i, j$ belong to the same block of $\Pi(W)$ if and only if
    \begin{equation}\label{eq:condPi}
    \forall C_1, C_2 \in L(W, W), \; \len(C_1) = \len(C_2) \implies \eval(C_2)_i - \eval(C_1)_i = \eval(C_2)_j - \eval(C_1)_j.
    \end{equation}
    Modifying the automaton that accepts $L(W, W)$, it is easy to see that the set
    \begin{equation*}
    \Big\{\big(p^{\len(C)}, \eval(C) \big) \;\Big|\; C \in L(W, W) \Big\} \subseteq \Z \times \Z^{KN}
    \end{equation*}
    is effectively $p$-automatic.
    By Lemma~\ref{lem:pauto}(2), the set
    \begin{equation*}
    L_{ji} \coloneqq \Big\{\big(p^{\len(C)}, \eval(C)_j - \eval(C)_i \big) \;\Big|\; C \in L(W, W) \Big\} \subseteq \Z \times \Z^N
    \end{equation*}
    is effectively $p$-automatic.
    Therefore the set
    \begin{multline*}
    L_{ji} - L_{ji} \coloneqq \big\{(a_1-a_2,b_1-b_2) \;\big|\; (a_1,b_1), (a_2, b_2) \in L_{ji} \big\} \\
    = \bigg\{\big(p^{\len(C_1)} - p^{\len(C_2)}, \big(\eval(C_1)_j - \eval(C_1)_i\big) - \big(\eval(C_2)_j - \eval(C_2)_i\big) \;\bigg|\; C_1,C_2 \in L(W, W) \bigg\}
    \end{multline*}
    is effectively $p$-automatic.
    Therefore Condition~\eqref{eq:condPi} is equivalent to ``$(L_{ji} - L_{ji}) \cap (\{0\} \times \Z^N) = \{(0, \bzer)\}$'', which can be effectively verified~\cite{wolper2000construction}.
\end{proof}

The lemmas above gave an intuition of the subgroup $H$ in the definition~\eqref{eq:psuccinct} of $p$-succinct sets.
In fact, $H$ will be a suitable modification of the subgroup $(\Z^{N})^{\Pi(W)}$, where $W$ ranges over the states of $\mmU$.
Next, we start working towards the term $p^{\ell k_i} \ba_i$ in the Equation~\eqref{eq:psuccinct}.

\paragraph{Prototype of the term $p^{\ell k_i} \ba_i$.}

The following lemma characterizes the evaluation of cycles in $L(W, W)$, up to quotient by the subgroup $(\Z^{N})^{\Pi(W)}$ identified in the previous lemmas.

\begin{lem}\label{lem:WW}
    Let $W$ be a state. 
    Then there exists $\bb \in \Q^{KN}$, whose denominators are not divisible by $p$, such that for every cycle $C \in L(W, W)$, we have
    \begin{equation}\label{eq:evalCin}
    \eval(C) \in \left(p^{\len(C)} - 1\right) \bb + (\Z^N)^{\Pi(W)}.
    \end{equation}
\end{lem}
\begin{proof}
    Let $S_1 = \{s_{11}, \ldots, s_{1 n_1}\}, \ldots, S_r = \{s_{r1}, \ldots, s_{r n_r}\}$ be the blocks of $\Pi(W)$.
    For $i \in \{1, \ldots, K\}$, $j \in \{1, \ldots, N\}$, let $\be_{ij}$ denote the vector $(z_{11}, \ldots, z_{KN})$ where $z_{ij} = 1$ and all the other entries are $0$.
    Recall the definition~\eqref{eq:defeSb} of the generators $\be_{S_1, b}, \ldots, \be_{S_r, b}, b = 1, \ldots, N$ of $(\Z^N)^{\Pi(W)}$.
    It is easy to see that they can be extended to a $\Z$-basis
    \[
    \be_{S_1, b}, \be_{s_{11}b}, \be_{s_{12}b}, \ldots,\be_{s_{1(n_1 - 1)}b}, \ldots, \be_{S_r, b}, \be_{s_{r1}b}, \be_{s_{r2}b}, \ldots,\be_{s_{r(n_r - 1)}b}, \; b = 1, \ldots, N,
    \]
    of $\Z^{KN}$.
    Therefore $\Z^{KN}$ splits as a direct sum 
    \begin{equation*}
    \Z^{KN} = (\Z^N)^{\Pi(W)} \oplus \Z^M
    \end{equation*}
    with $M = (K - r)N$.
    We will write each element $\bz$ of $\Z^{KN}$ as a pair $(\bz_{\Pi}, \bz_{\perp}) \in (\Z^N)^{\Pi(W)} \oplus \Z^M$ according to this direct sum.
    Note that this naturally extends to a direct sum $\Q^{KN} = (\Q^N)^{\Pi(W)} \oplus \Q^M$.


    Let $C_1, C_2 \in L(W, W)$ be two cycles of the same length. 
    By definition, $\eval(C_1) - \eval(C_2)$ belongs to $(\Z^N)^{\Pi(W)}$.
    This means that $\eval(C_1)_{\perp} = \eval(C_2)_{\perp}$.

    Let $C, C' \in L(W, W)$ be two cycles, not necessarily of the same length.
    Then the two concatenations $C^{\len(C')} \coloneqq \underbrace{CC \cdots C}_{\text{$\len(C')$ iterations}} \in L(W, W)$ and $(C')^{\len(C)} \coloneqq \underbrace{C'C' \cdots C'}_{\text{$\len(C)$ iterations}} \in L(W, W)$ have the same length.
    So $\eval\left((C')^{\len(C)}\right)_{\perp} = \eval\left(C^{\len(C')}\right)_{\perp}$.
    Since $\eval\left((C')^{\len(C)}\right) = \frac{p^{\len(C')\len(C)} - 1}{p^{\len(C')} - 1} \cdot \eval(C')$, and $\eval\left(C^{\len(C')}\right) = \frac{p^{\len(C)\len(C')} - 1}{p^{\len(C)} - 1} \cdot \eval(C)$, this yields
    \[
    \frac{p^{\len(C')\len(C)} - 1}{p^{\len(C')} - 1} \cdot \eval(C')_\perp = \frac{p^{\len(C)\len(C')} - 1}{p^{\len(C)} - 1} \cdot \eval(C)_\perp.
    \]
    Consequently,
    \[
    \frac{\eval(C')_\perp}{p^{\len(C')} - 1} = \frac{\eval(C)_\perp}{p^{\len(C)} - 1}.
    \]
    This means that, there exists a constant $\ba \in \Q^{M}$, such that $\frac{\eval(C)_\perp}{p^{\len(C)} - 1} = \ba$ for all $C \in L(W, W)$.

    Let $\bb \coloneqq (\bzer_{\Pi}, \ba) \in (\Q^N)^{\Pi(W)} \oplus \Q^M$.
    Since $p \nmid p^{\len(C)} - 1$ and $(p^{\len(C)} - 1) \bb = (\bzer_{\Pi}, \eval(C)_\perp) \in\Z^{KN}$, the denominators of $\bb$ are not divisible by $p$.
    Then,
    \[
    \eval(C) - \left(p^{\len(C)} - 1\right) \bb = \eval(C) - (\bzer_{\Pi}, \eval(C)_\perp) = \left(\eval(C)_{\Pi}, \bzer_{\perp}\right) \in (\Z^N)^{\Pi(W)}.
    \]    
    Therefore $\eval(C) \in \left(p^{\len(C)} - 1\right) \bb + (\Z^N)^{\Pi(W)}$.
\end{proof}

We can easily extend Lemma~\ref{lem:WW} from cycles in $L(W, W)$ to paths in $L(W, V)$, provided that $W, V$ are states in the same strongly connected component of $\mmU$:

\begin{lem}\label{lem:singlecomp}
    Let $W, V$ be two states in the same strongly connected component of $\mmU$.
    Then there exist $\bb, \bc \in \Q^{KN}$, whose denominators are not divisible by $p$, such that for every path $\pi \in L(W, V)$, we have
    \begin{equation}\label{eq:evalpiin}
    \eval(\pi) \in p^{\len(\pi)} \bb + \bc +  (\Z^N)^{\Pi(W)}.
    \end{equation}
\end{lem}
\begin{proof}
    Let $\pi_{VW}$ be any path in $L(V, W)$.
    Then for any path $\pi \in L(W, V)$, the concatenation $\pi \pi_{VW}$ is a cycle in $L(W, W)$.

    By Lemma~\ref{lem:WW}, there exists $\widetilde\bb \in \Q^{KN}$ such that for all $\pi \in L(W, V)$, we have
    \[
    \eval(\pi \pi_{VW}) \in \left(p^{\len(\pi \pi_{VW})} - 1\right) \widetilde\bb + (\Z^N)^{\Pi(W)}.
    \]
    Since $\eval(\pi \pi_{VW}) = \eval(\pi) + p^{\len(\pi)} \cdot \eval(\pi_{VW})$ and $\len(\pi \pi_{VW}) = \len(\pi) + \len(\pi_{VW})$, we have
    \begin{align*}
    \eval(\pi) & \in - p^{\len(\pi)} \cdot \eval(\pi_{VW}) + \left(p^{\len(\pi) + \len(\pi_{VW})} - 1\right) \widetilde\bb + (\Z^N)^{\Pi(W)} \\
    & = p^{\len(\pi)} \left(- \eval(\pi_{VW}) + p^{\len(\pi_{VW})} \widetilde\bb\right)- \widetilde\bb + (\Z^N)^{\Pi(W)}.
    \end{align*}
    Thus we obtain the statement~\eqref{eq:evalpiin} by taking $\bb \coloneqq - \eval(\pi_{VW}) + p^{\len(\pi_{VW})} \widetilde\bb$, and $\bc \coloneqq - \widetilde\bb$.
    The denominators of $\bb, \bc$ are not divisible by $p$ since the denominators of $\widetilde\bb$ are not divisible by $p$.
\end{proof}
For each pair of states $W, V$ in the same strongly connected component, we can find such vectors $\bb, \bc \in \Q^{KN}$ as in Lemma~\ref{lem:singlecomp}.
In what follows we will denote them by $\bb_{W,V}, \bc_{W,V},$ when we want to stress their dependence on $W, V$.

For two states $W, V$ of the automaton $\mmU$, define
\[
\Lambda(W, V) \coloneqq \{\len(\pi) \mid \pi \in L(W, V)\} \subseteq \N,
\]
that is, the set of lengths of paths from $W$ to $V$.
The following folklore result characterizes $\Lambda(W, V)$.

\begin{lem}[{See~\cite{kozen2012automata} or~\cite{haase2018survival}}]\label{lem:semilinear}
Let $W, V$ be two states in the automaton $\mmU$.
Then $\Lambda(W, V)$ can be effectively written as the union of a finite set and finitely many arithmetic progressions.
\end{lem}

We now give an intuition of the the term $p^{\ell k_i} \ba_i$ in the definition~\eqref{eq:psuccinct} of $p$-succinct sets.
One can see from Lemma~\ref{lem:singlecomp} that $p^{\ell k_i} \ba_i$ will be a suitable modification of the term $p^{\len(\pi)} \cdot \bb_{W, V}$, where $\pi$ is the part of an accepting run within a strongly connected component.
Note that Lemma~\ref{lem:semilinear} shows that the value of $\len(\pi)$ must fall in a union of a finite set and finitely many arithmetic progressions.
The common difference $\lambda_j$ of these arithmetic progressions will constitute the value $\ell$ in the term $p^{\ell k_i} \ba_i$.

Next, we start working towards characterizing the zero set $\mZ(\alpha)$ as a finite union of ``prototype'' $p$-succinct sets.

\paragraph{Finite union of ``prototype'' $p$-succinct sets.} 
For a set of paths $L$ in the automaton $\mmU$, denote
\[
\eval(L) \coloneqq \{\eval(\pi) \mid \pi \in L\} \subseteq \Z^{KN}.
\]
Recall that $\mF$ denotes the set of all accepting states of $\mmU$. For any state $W$ in $\mmU$, we have
\[
\mZ(W) = \eval\left(\bigcup_{F \in \mF} L(W, F)\right).
\]

The automaton $\mmU$ can be decomposed into strongly connected components.
Accordingly, each path $\pi \in \bigcup_{F \in \mF} L(\{\alpha\}, F)$ can be decomposed as a concatenation
\begin{equation}\label{eq:decomppi}
\pi = \pi_{W_1, V_1} \cdot \delta_{V_1, W_2} \cdot \pi_{W_2, V_2} \cdot \delta_{V_2, W_3} \cdot \cdots \cdot \delta_{V_{r-1}, W_r} \cdot \pi_{W_r, V_r},
\end{equation}
where
\begin{enumerate}[nosep, label=(\arabic*)]
    \item For $i = 1, \ldots, r$, each $\pi_{W_i, V_i}$ is a path in $L(W_i, V_i)$, where $W_i$ and $V_i$ are in the same strongly connected component of $\mmU$.
    \item For $i = 1, \ldots, r-1$, each $\delta_{V_i, W_{i+1}}$ is a \emph{transition} from $V_i$ to $W_{i+1}$, where $V_i$ and $W_{i+1}$ are in \emph{different} strongly connected components of $\mmU$.
    \item $W_1 = \{\alpha\}$ is the initial state.
    \item $V_r = F$ is an accepting state.
\end{enumerate}
Let $W_1, V_1, \ldots, W_r, V_r$, be states of the automaton $\mmU$.
We will write the diagram
\[
    \{\alpha\} = W_1 \rightsquigarrow V_1 \rightarrow W_2 \rightsquigarrow V_2 \rightarrow \cdots \rightarrow W_r \rightsquigarrow V_r \in \mF
\] 
if
\begin{enumerate}[nosep, label=(\arabic*)]
    \item For $i = 1, \ldots, r$, the states $W_i, V_i$ are in the same strongly connected component of $\mmU$.
    \item For $i = 1, \ldots, r-1$, there exists a transition from $V_i$ to $W_{i+1}$. Furthermore, $V_i$ and $W_{i+1}$ are in different strongly connected components of $\mmU$.
    \item $W_1 = \{\alpha\}$ is the initial state.
    \item $V_r \in \mF$ is an accepting state.
\end{enumerate}
See Figure~\ref{fig:SC} for an illustration. Note that $r$ is bounded by the number of strongly connected components of $\mmU$.

\begin{figure}[h!]
    \centering
    \includegraphics[width=0.9\textwidth,height=1.0\textheight,keepaspectratio, trim={1cm 0cm 1cm 0cm},clip]{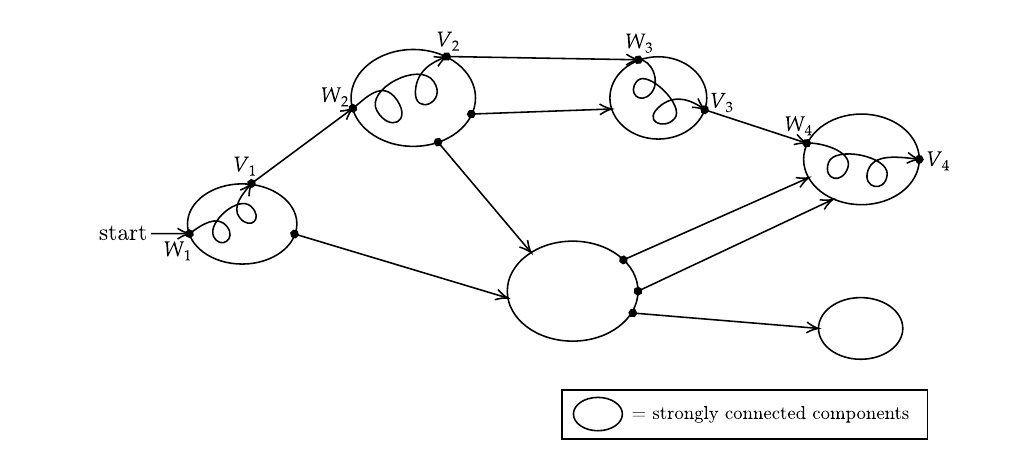}
        \caption{A path validating $\{\alpha\} = W_1 \rightsquigarrow V_1 \rightarrow W_2 \rightsquigarrow V_2 \rightarrow W_3 \rightsquigarrow V_3 \rightarrow W_4 \rightsquigarrow V_4 \in \mF$.}
    \label{fig:SC}
\end{figure}

The above discussion shows that the zero set $\mZ(\alpha) = \eval\left(\bigcup_{F \in \mF} L(\{\alpha\}, F)\right)$ can be written as a finite union
\begin{equation}\label{eq:decomppath}
\bigcup_{\{\alpha\} = W_1 \rightsquigarrow V_1 \rightarrow \cdots \rightarrow W_r \rightsquigarrow V_r \in \mF} \eval\Big(L(W_1, V_1) \cdot \delta_{V_1, W_2} \cdot L(W_2, V_2) \cdot \cdots \cdot \delta_{V_{r-1}, W_r} \cdot L(W_r, V_r)\Big),
\end{equation}
where the concatenation $L \cdot \delta \cdot L' \cdots$ denotes the set of paths $\{\pi \delta \pi' \cdots \mid \pi \in L, \pi' \in L', \ldots\}$.
Note that this union is not necessarily disjoint because different paths may have the same evaluation.

For $i = 1, \ldots, r$, define the integer vector
\begin{align*}
\bd_{V_i, W_{i+1}} \coloneqq 
\begin{cases}
    \eval(\delta_{V_i, W_{i+1}}) & \quad \text{ for } 1 \leq i \leq r-1, \\
    0 & \quad \text{ for } i = r.
\end{cases}
\end{align*}
By Lemma~\ref{lem:singlecomp}, there exist vectors $\bb_{W_i, V_i}, \bc_{W_i, V_i} \in \Q^{KN}, \; i = 1, \ldots, r$, whose denominators are not divisible by $p$, such that
\begin{equation}\label{eq:defcWV}
    \eval(\pi) \in p^{\len(\pi)} \cdot \bb_{W_i, V_i} + \bc_{W_i, V_i} +  (\Z^N)^{\Pi(W_i)}
\end{equation}
for all $\pi \in L(W_i, V_i)$.
We now characterize the zero set $\mZ(\alpha)$ using the vectors $\bd_{V_i, W_{i+1}}, \bb_{W_i, V_i}, \bc_{W_i, V_i}$, $i = 1, \ldots, r$.

\begin{prop}\label{prop:protosuccinct}
    The zero set $\mZ(\alpha)$ is equal to the finite union
    \begin{align}\label{eq:decompeval}
        \bigcup_{\{\alpha\} = W_1 \rightsquigarrow V_1 \rightarrow \cdots \rightarrow W_r \rightsquigarrow V_r \in \mF} & \Bigg\{\sum_{i = 1}^r p^{(\ell_1 + 1) + \cdots + (\ell_{i-1} + 1)} \left(p^{\ell_i} \bb_{W_i, V_i} + \bc_{W_i, V_i} + p^{\ell_i} \bd_{V_i, W_{i+1}} + \bh_i \right) \nonumber \\
        & \qquad \qquad \;\Bigg|\; \forall i, \ell_i \in \Lambda(W_i, V_i), \bh_i \in (\Z^N)^{\Pi(W_i)} \Bigg\}.
    \end{align}
\end{prop}
Before giving the proof of Proposition~\ref{prop:protosuccinct}, we would like to clarify a potentially misleading point.
The similarity between the unions~\eqref{eq:decomppath} and~\eqref{eq:decompeval} might lead one to falsely conjecture a term-wise equality.
However, as we will show later, we only have an inclusion
\begin{multline*}
    \eval\Big(L(W_1, V_1) \cdot \delta_{V_1, W_2} \cdot L(W_2, V_2) \cdot \cdots \cdot \delta_{V_{r-1}, W_r} \cdot L(W_r, V_r)\Big) \subseteq \\
    \Bigg\{\sum_{i = 1}^r p^{(\ell_1 + 1) + \cdots + (\ell_{i-1} + 1)} \left(p^{\ell_i} \bb_{W_i, V_i} + \bc_{W_i, V_i} + p^{\ell_i} \bd_{V_i, W_{i+1}} + \bh_i \right) \\
    \;\Bigg|\; \forall i, \ell_i \in \Lambda(W_i, V_i), \bh_i \in (\Z^N)^{\Pi(W_i)} \Bigg\},
\end{multline*}
where the inclusion ``$\subseteq$'' might be a strict one.
The proof of Proposition~\ref{prop:protosuccinct} is more subtle than simply proving a term-wise equality.
We will need to combine the expression~\eqref{eq:decomppath} with the stability condition $\mZ(W_i) = \mZ(W_i) + (\Z^N)^{\Pi(W_i)}$ from Lemma~\ref{lem:stablePi} for each $i = 1, 2, \ldots, r$.

\begin{proof}[Proof of Proposition~\ref{prop:protosuccinct}]
    First we show the inclusion
    \begin{multline}\label{eq:evalsubset}
        \mZ(\alpha) \subseteq \bigcup_{\{\alpha\} = W_1 \rightsquigarrow V_1 \rightarrow \cdots \rightarrow W_r \rightsquigarrow V_r \in \mF} \Bigg\{\sum_{i = 1}^r p^{(\ell_1 + 1) + \cdots + (\ell_{i-1} + 1)} \left(p^{\ell_i} \bb_{W_i, V_i} + \bc_{W_i, V_i} + p^{\ell_i} \bd_{V_i, W_{i+1}} + \bh_i \right) \\
        \;\Bigg|\; \forall i, \ell_i \in \Lambda(W_i, V_i), \bh_i \in (\Z^N)^{\Pi(W_i)} \Bigg\}.
    \end{multline}
    This is the easier direction.
    Take any accepting path $\pi \in L(\{\alpha\}, F), F \in \mF$, then as in the Decomposition~\eqref{eq:decomppi} we can write $\pi$ as a concatenation
    \[
    \pi = \pi_{W_1, V_1} \cdot \delta_{V_1, W_2} \cdot \pi_{W_2, V_2} \cdot \delta_{V_2, W_3} \cdot \cdots \cdot \delta_{V_{r-1}, W_r} \cdot \pi_{W_r, V_r},
    \]
    where $W_1, V_1, \ldots, W_r, V_r$ satisfy the diagram $\{\alpha\} = W_1 \rightsquigarrow V_1 \rightarrow W_2 \rightsquigarrow V_2 \rightarrow \cdots \rightarrow W_r \rightsquigarrow V_r \in \mF$.
    Denote $\ell_1 \coloneqq \len(\pi_{W_1, V_1}), \ell_2 \coloneqq \len(\pi_{W_2, V_2}), \ldots, \ell_r \coloneqq \len(\pi_{W_r, V_r})$, then
    \begin{align*}
        &\;\eval(\pi) \\
        = &\; \eval\left(\pi_{W_1, V_1} \cdot \delta_{V_1, W_2} \cdot \pi_{W_2, V_2} \cdot \delta_{V_2, W_3} \cdot \cdots \cdot \delta_{V_{r-1}, W_r} \cdot \pi_{W_r, V_r}\right) \\
        = &\; \eval\left(\pi_{W_1, V_1}\right) + p^{\ell_1} \cdot \eval\left(\delta_{V_1, W_2}\right) + p^{\ell_1 + 1} \cdot \eval\left(\pi_{W_2, V_2}\right) + \cdots + p^{\ell_1+1+ \cdots + \ell_{r-1}+1} \cdot \eval\left(\pi_{W_r, V_r}\right) \\ 
        = &\; \eval\left(\pi_{W_1, V_1}\right) + p^{\ell_1} \cdot \bd_{V_1, W_2} + p^{\ell_1 + 1} \cdot \eval\left(\pi_{W_2, V_2}\right) + \cdots + p^{\ell_1+1+ \cdots + \ell_{r-1}+1} \cdot \eval\left(\pi_{W_r, V_r}\right) \\
        = &\; \sum_{i=1}^{r} p^{(\ell_1+1)+ \cdots + (\ell_{i-1}+1)} \cdot \Big(\eval(\pi_{W_i, V_i}) + p^{\ell_i} \cdot \bd_{V_i, W_{i+1}}\Big).
    \end{align*}
    By Equation~\eqref{eq:defcWV}, we have $\eval(\pi_{W_i, V_i}) \in p^{\ell_i} \cdot \bb_{W_i, V_i} + \bc_{W_i, V_i} + (\Z^N)^{\Pi(W_i)}$.
    Therefore
    \[
    \eval(\pi) \in \sum_{i=1}^{r} p^{(\ell_1+1)+ \cdots + (\ell_{i-1}+1)} \cdot \Big(p^{\ell_i} \cdot \bb_{W_i, V_i} + \bc_{W_i, V_i} + (\Z^N)^{\Pi(W_i)} + p^{\ell_i} \cdot \bd_{V_i, W_{i+1}}\Big).
    \]
    Since $\ell_i = \len(\pi_{W_i, V_i}) \in \Lambda(W_i, V_i)$ for all $i = 1, \ldots, r$, this proves the inclusion~\eqref{eq:evalsubset}.

    Next we show the other inclusion
    \begin{multline}\label{eq:evalsupset}
        \bigcup_{\{\alpha\} = W_1 \rightsquigarrow V_1 \rightarrow \cdots \rightarrow W_r \rightsquigarrow V_r \in \mF} \Bigg\{\sum_{i = 1}^r p^{(\ell_1 + 1) + \cdots + (\ell_{i-1} + 1)} \left(p^{\ell_i} \bb_{W_i, V_i} + \bc_{W_i, V_i} + p^{\ell_i} \bd_{V_i, W_{i+1}} + \bh_i \right) \\
        \;\Bigg|\; \forall i, \ell_i \in \Lambda(W_i, V_i), \bh_i \in (\Z^N)^{\Pi(W_i)} \Bigg\} \subseteq \mZ(\alpha).
    \end{multline}
    This is the more difficult direction, see Figure~\ref{fig:supset} for an illustration of the proof.
    Take any states $W_1, V_1, \ldots, W_r, V_r$ satisfying $\{\alpha\} = W_1 \rightsquigarrow V_1 \rightarrow \cdots \rightarrow W_r \rightsquigarrow V_r \in \mF$, and take $\ell_i \in \Lambda(W_i, V_i), \bh_i \in (\Z^N)^{\Pi(W_i)}$ for $i = 1, \ldots, r$.
    
    For $j = r, r-1, \ldots, 2, 1$, define
    \[
    \bz_j \coloneqq \sum_{i = j}^r p^{(\ell_j + 1) + \cdots + (\ell_{i-1} + 1)} \left(p^{\ell_i} \bb_{W_i, V_i} + \bc_{W_i, V_i} + p^{\ell_i} \bd_{V_i, W_{i+1}} + \bh_i \right).
    \]
    In particular, $\bz_r = p^{\ell_r} \bb_{W_r, V_r} + \bc_{W_r, V_r} + \bh_r$, and 
    \begin{equation}\label{eq:reczj}
        \bz_{j} = \left(p^{\ell_i} \bb_{W_i, V_i} + \bc_{W_i, V_i} + p^{\ell_i} \bd_{V_i, W_{i+1}} + \bh_i \right) + p^{\ell_j + 1} \bz_{j+1}
    \end{equation}
    for $j \leq r-1$.    
    We will now show $\bz_j \in \mZ(W_j)$ inductively for $j = r, r-1, \ldots, 2, 1$.
    Note that showing $\bz_1 \in \mZ(W_1)$ will prove the inclusion~\eqref{eq:evalsupset} since $W_1 = \{\alpha\}$.
    
    For $j = r$, take some path $\pi_{W_r, V_r} \in L(W_r, V_r)$ such that $\len(\pi_{W_r, V_r}) = \ell_r$. 
    In particular, we have $\eval(\pi_{W_r, V_r}) \in \mZ(W_r)$.
    By Equation~\eqref{eq:defcWV}, we have
    \[
    \eval(\pi_{W_r, V_r}) \in p^{\ell_r} \cdot \bb_{W_r, V_r} + \bc_{W_r, V_r} + (\Z^N)^{\Pi(W_r)}.
    \]
    Since $\bh_r \in (\Z^N)^{\Pi(W_r)}$, this yields
    \[
    \bz_r - \eval(\pi_{W_r, V_r}) = p^{\ell_r} \bb_{W_r, V_r} + \bc_{W_r, V_r} + \bh_r - \eval(\pi_{W_r, V_r}) \in \bh_r - (\Z^N)^{\Pi(W_r)} = (\Z^N)^{\Pi(W_r)}.
    \]
    Therefore
    \[
    \bz_r \in \eval(\pi_{W_r, V_r}) + (\Z^N)^{\Pi(W_r)} \subseteq \mZ(W_r) + (\Z^N)^{\Pi(W_r)}.
    \]
    But by Lemma~\ref{lem:stablePi}, we have $\mZ(W_r) = \mZ(W_r) + (\Z^N)^{\Pi(W_r)}$, so we obtain $\bz_r \in \mZ(W_r)$.

    Suppose we have proven $\bz_{j+1} \in \mZ(W_{j+1})$, we now prove $\bz_{j} \in \mZ(W_{j})$.
    Take some path $\pi_{W_j, V_j} \in L(W_j, V_j)$ such that $\len(v) = \ell_j$.
    By Equation~\eqref{eq:defcWV}, we have 
    \begin{equation}\label{eq:evalpiin2}
    \eval(\pi_{W_j, V_j}) \in p^{\ell_j} \cdot \bb_{W_j, V_j} + \bc_{W_j, V_j} + (\Z^N)^{\Pi(W_j)}.
    \end{equation}
    By the induction hypothesis $\bz_{j+1} \in \mZ(W_{j+1})$, there exists a path $\pi'_{W_{j+1}} \in L(W_{j+1}, F')$ for some accepting state $F'$ (not necessarily the same as $V_r$), such that $\eval(\pi'_{W_{j+1}}) = \bz_{j+1}$.
    Consider the concatenation 
    \[
    \pi_{W_j, V_j} \cdot \delta_{V_{j}, W_{j+1}} \cdot \pi'_{W_{j+1}} \in L(W_j, F'),
    \]
    we have
    \begin{alignat*}{2}
    & \eval\left(\pi_{W_j, V_j} \cdot \delta_{V_{j}, W_{j+1}} \cdot \pi'_{W_{j+1}}\right) \\
    = &\; \eval(\pi_{W_j, V_j}) + p^{\ell_j} \cdot \eval(\delta_{V_{j}, W_{j+1}}) + p^{\ell_j + 1} \cdot \eval(\pi'_{W_{j+1}}) \\
    = &\; \eval(\pi_{W_j, V_j}) + p^{\ell_j} \cdot \bd_{V_{j}, W_{j+1}} + p^{\ell_j + 1} \cdot \bz_{j+1} \\
    \in &\; p^{\ell_j} \cdot \bb_{W_j, V_j} + \bc_{W_j, V_j} + p^{\ell_j} \cdot \bd_{V_{j}, W_{j+1}} + p^{\ell_j + 1} \cdot \bz_{j+1} + (\Z^N)^{\Pi(W_j)} \qquad \qquad && \text{(by~\eqref{eq:evalpiin2})} \\
    = &\; \bz_j - \bh_j + (\Z^N)^{\Pi(W_j)} && \text{(by~\eqref{eq:reczj})}\\
    = &\; \bz_j + (\Z^N)^{\Pi(W_j)}.
    \end{alignat*}
    Since $\pi_{W_j, V_j} \cdot \delta_{V_{j}, W_{j+1}} \cdot \pi'_{W_{j+1}} \in L(W_j, F')$, we have $\eval(\pi_{W_j, V_j} \cdot \delta_{V_{j}, W_{j+1}} \cdot \pi'_{W_{j+1}}) \in \mZ(W_j)$.
    Consequently,
    \[
    \bz_j \in \eval(v \cdot \delta_{V_{j}, W_{j+1}} \cdot w) + (\Z^N)^{\Pi(W_j)} \subseteq \mZ(W_j) + (\Z^N)^{\Pi(W_j)} = \mZ(W_j)
    \]
    by Lemma~\ref{lem:stablePi}.
    This proves the inclusion~\eqref{eq:evalsupset}.
\end{proof}

\begin{figure}[h!]
    \centering
    \includegraphics[width=1\textwidth,height=1.0\textheight,keepaspectratio, trim={0cm 0cm 0cm 0cm},clip]{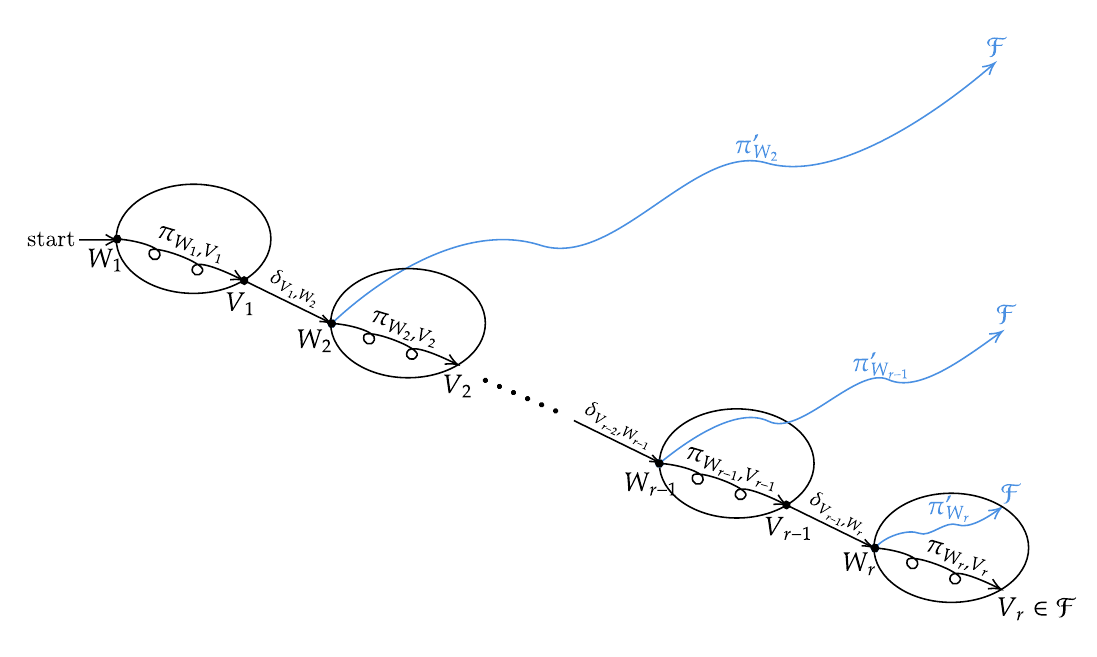}
        \caption{Illustration of proving the inclusion~\eqref{eq:evalsupset}.}
    \label{fig:supset}
\end{figure}

We have thus shown that the zero set $\mZ(\alpha)$ can be written as a finite union of the sets
\begin{align*}
        & \Bigg\{\sum_{i = 1}^r p^{(\ell_1 + 1) + \cdots + (\ell_{i-1} + 1)} \left(p^{\ell_i} \bb_{W_i, V_i} + \bc_{W_i, V_i} + p^{\ell_i} \bd_{V_i, W_{i+1}} + \bh_i \right) \\
        & \hspace{8cm}\;\Bigg|\; \forall i, \ell_i \in \Lambda(W_i, V_i), \bh_i \in (\Z^N)^{\Pi(W_i)} \Bigg\} \\
        = \; & \Bigg\{\bc_{W_1, V_1} + \sum_{i = 1}^r p^{\ell_1 + \cdots + \ell_{i} + (i-1)} \left(\bb_{W_i, V_i} + p \bc_{W_{i+1}, V_{i+1}} + \bd_{V_i, W_{i+1}} \right) + \sum_{i = 1}^r  p^{\ell_1 + \cdots + \ell_{i-1} + (i-1)} \bh_{i}\\
        & \hspace{8cm} \;\Bigg|\; \forall i, \ell_i \in \Lambda(W_i, V_i), \bh_i \in (\Z^N)^{\Pi(W_i)} \Bigg\}, 
\end{align*}
where $\bc_{W_{r+1}, V_{r+1}}$ is defined as zero.
If we denote
\begin{align}\label{eq:defa}
\ba_{i} \coloneqq 
\begin{cases}
    \bc_{W_1, V_1}, & \quad i = 0, \\
    p^{i-1}\left( \bb_{W_i, V_i} + p \bc_{W_{i+1}, V_{i+1}} + \bd_{V_i, W_{i+1}} \right), & \quad i = 1, \ldots, r,\\
\end{cases}
\end{align}
then $\mZ(\alpha)$ can be written as the finite union
\begin{multline}\label{eq:proto}
    \bigcup_{\{\alpha\} = W_1 \rightsquigarrow V_1 \rightarrow \cdots \rightarrow W_r \rightsquigarrow V_r \in \mF} \Bigg\{\sum_{i = 0}^r p^{\ell_1 + \cdots + \ell_{i}} \ba_i + \sum_{i = 1}^r p^{\ell_1 + \cdots + \ell_{i-1} + (i-1)} \bh_i \\
        \;\Bigg|\; \forall i, \ell_i \in \Lambda(W_i, V_i), \bh_i \in (\Z^N)^{\Pi(W_i)} \Bigg\}.
\end{multline}
Furthermore, the denominators of each $\ba_i$ are not divisible by $p$, since this is the case for $\bb_{W_i, V_i}$, $\bc_{W_{i+1}, V_{i+1}}$ and $\bd_{V_i, W_{i+1}}$.

The form of each component in the union~\eqref{eq:proto} is very similar to a $p$-succinct set as defined in Equation~\eqref{eq:psuccinct}.
However, there are two important differences:
\begin{enumerate}[noitemsep, label=(\roman*)]
    \item The set $\left\{\sum_{i = 1}^r p^{\ell_1 + \cdots + \ell_{i-1} + (i-1)} \bh_i \;\middle|\; \forall i, \ell_i \in \Lambda(W_i, V_i), \bh_i \in (\Z^N)^{\Pi(W_i)}\right\}$ doesn't really form a subgroup of $\Z^{KN}$. While for any \emph{fixed} $\ell_1, \ldots, \ell_{r-1}$, the set $\big\{p^{\ell_1 + \cdots + \ell_{i-1} + (i-1)} \bh_i \;\big|\; \bh_i \in (\Z^N)^{\Pi(W_i)} \big\}$ forms a subgroup, it is generally not true when $\ell_1, \ldots, \ell_{r-1}$ are allowed to vary.
    \item The expression $\sum_{i = 0}^r p^{\ell_1 + \cdots + \ell_{i}} \ba_i, \forall i, \ell_i \in \Lambda(W_i, V_i)$ is not really of the form $\ba_0 + p^{\ell k_1} \ba_1 + \cdots + p^{\ell k_r} \ba_r, \forall i, k_i \in \N$.
    The main problem is that we have the extra constraint $\ell_1 \leq \ell_1 + \ell_2 \leq \cdots \leq \ell_1 + \cdots + \ell_{r}$.
\end{enumerate}

The following subsection focuses on eliminating these two differences.
Difference (i) is easy to resolve using a variable elimination process that ``saturates'' the subgroup $p^{\ell_1 + \cdots + \ell_{i-1} + (i-1)} \cdot (\Z^N)^{\Pi(W_i)}$.
Difference (ii) is more difficult to resolve.
To achieve this, we use a so-called ``symmetrization'' process similar to that of~\cite[Lemma~8.1]{derksen2007skolem} and~\cite[Lemma~12.1]{derksen2012linear}, in order to bring down the exponents $\ell_1 + \cdots + \ell_{i}$.

\subsection{Saturation and symmetrization}\label{subsec:symm}
In this subsection we will refine Expression~\eqref{eq:proto} and finally prove $p$-normality of the zero set $\mZ(\alpha)$.

\paragraph{Saturation.}
For a subgroup $H \leq \Z^{KN}$, denote $q \cdot H \coloneqq \{q \bh \mid \bh \in H\}$ for any $q \in \N \setminus \{0\}$, and denote $\ba + H \coloneqq \{\ba + \bh \mid \bh \in H\}$ for any $\ba \in \Z^{KN}$.

\begin{lem}[Subgroup saturation]\label{lem:saturation}
    Let $\Pi$ be a partition of $\{1, \ldots, K\}$, and let $\ba \in \Z^{KN}$.
    Suppose $\ba + q \cdot (\Z^N)^{\Pi} \subseteq \mZ(\alpha)$ for some $q \in \N \setminus \{0\}$. Then $\ba +  (\Z^N)^{\Pi} \subseteq \mZ(\alpha)$.
\end{lem}
\begin{proof}
    Let $\mT$ denote the set $\ba + q \cdot(\Z^N)^{\Pi}$.
    Let $S$ be any block of $\Pi$ and let $b \in \{1, \ldots, N\}$, recall the definition of $\be_{S, b} \in (\Z^N)^{\Pi}$ in~\eqref{eq:defeSb}.
    Take any $\bz = (z_{11}, \ldots, z_{KN}) \in \mT$, we have $\bz, \bz + q\be_{S, b} \in \mZ(\alpha)$.
    The inclusion $\bz \in \mZ(\alpha)$ means
    \begin{equation*}
    \sum_{i = 1}^K A_1^{z_{i1}} \cdots A_N^{z_{iN}} v_i = 0,
    \end{equation*}
    and the inclusion $\bz+ q\be_{S, b} \in \mZ(\alpha)$ means
    \[
    \sum_{i \notin S} A_1^{z_{i1}} \cdots A_b^{z_{ib}} \cdots A_N^{z_{iN}} v_i + \sum_{i \in S} A_1^{z_{i1}} \cdots A_b^{z_{ib} + q} \cdots A_N^{z_{iN}} v_i= 0.
    \]
    Therefore we have the system of equations
    \begin{align*}
    & \sum_{i \notin S} A_1^{z_{i1}} \cdots A_N^{z_{iN}} v_i + \sum_{i \in S} A_1^{z_{i1}} \cdots A_N^{z_{iN}} v_i= 0, \\
    & \sum_{i \notin S} A_1^{z_{i1}} \cdots A_N^{z_{iN}} v_i + A_b^q \cdot \sum_{i \in S} A_1^{z_{i1}} \cdots A_N^{z_{iN}} v_i= 0.
    \end{align*}
    By Lemma~\ref{lem:independentmodq}, we have $A_b^q - 1 \not\in \frm$, so $A_b^q - 1 \in \mA$ is invertible.
    Therefore the transformation matrix
    $
    \begin{pmatrix}
        1 & 1 \\
        1 & A_b^q \\
    \end{pmatrix}
    $
    is invertible.
    This yields
    \[
    \sum_{i \notin S} A_1^{z_{i1}} \cdots A_N^{z_{iN}} v_i = \sum_{i \in S} A_1^{z_{i1}} \cdots A_n^{z_{in}} v_i = 0.
    \]
    Consequently,
    \[
    \sum_{i \notin S} A_1^{z_{i1}} \cdots A_N^{z_{iN}} v_i + A_b \cdot \sum_{i \in S} A_1^{z_{i1}} \cdots A_N^{z_{iN}} v_i= 0.
    \]
    This yields $\bz + \be_{S, b} \in \mZ(\alpha)$.
    Since this holds for all $\bz \in \mT$, we have $\mT + \Z \be_{S, b} \subseteq \mZ(\alpha)$.

    We now set $\mT$ as the new set $\mT + \Z \be_{S, b}$.
    Repeat the above process by taking any other block $S'$ of $\Pi$ and any $b' \in \{1, \ldots, N\}$.
    The process yields $\mT + \Z \be_{S, b} + \Z \be_{S', b'} \subseteq \mZ(\alpha)$.
    Iterate this for all blocks of $\Pi$ and all elements of $\{1, \ldots, N\}$, we obtain 
    \[
    \mT + (\Z^N)^{\Pi} = \mT + \sum_{S \in \Pi, b \in \{1, \ldots, N\}} \Z \be_{S, b} \subseteq \mZ(\alpha).
    \]
    Since $\mT = \ba + q \cdot(\Z^N)^{\Pi}$, this yields $\ba + (\Z^N)^{\Pi} \subseteq \mZ(\alpha)$.
\end{proof}

Recall from Expression~\eqref{eq:proto} that $\mZ(\alpha)$ can be written as the finite union
\begin{multline*}
\bigcup_{\{\alpha\} = W_1 \rightsquigarrow V_1 \rightarrow \cdots \rightarrow W_r \rightsquigarrow V_r \in \mF} \Bigg\{\sum_{i = 0}^r p^{\ell_1 + \cdots + \ell_{i}} \ba_i + \sum_{i = 1}^r p^{\ell_1 + \cdots + \ell_{i-1} + (i-1)} \bh_i \\
        \;\Bigg|\; \forall i, \ell_i \in \Lambda(W_i, V_i), \bh_i \in (\Z^N)^{\Pi(W_i)} \Bigg\}.
\end{multline*}
We can apply Lemma~\ref{lem:saturation} with
$
\ba = \sum_{i = 0}^r p^{\ell_1 + \cdots + \ell_{i}} \ba_i + \sum_{i = 1}^r p^{\ell_1 + \cdots + \ell_{i-1} + (i-1)} \bh_i,
$
and with
$
q = p^{\ell_1 + \cdots + \ell_{i-1} + (i-1)},
$
for each $i = 1, \ldots, r$.
This yields
\[
\mZ(\alpha) = \bigcup_{\{\alpha\} = W_1 \rightsquigarrow V_1 \rightarrow \cdots \rightarrow W_r \rightsquigarrow V_r \in \mF} \Bigg\{\sum_{i = 0}^r p^{\ell_1 + \cdots + \ell_{i}} \ba_i + \sum_{i = 1}^r \bh_i \\
        \;\Bigg|\; \forall i, \ell_i \in \Lambda(W_i, V_i), \bh_i \in (\Z^N)^{\Pi(W_i)} \Bigg\}.
\]
Since
\[
    \sum_{i = 1}^r (\Z^N)^{\Pi(W_i)} \coloneqq \left\{\sum_{i = 1}^r \bh_i \;\middle|\; \forall i, \bh_i \in (\Z^N)^{\Pi(W_i)} \right\}
\]
is a subgroup of $\Z^{KN}$, the above discussion can be summarized as the following corollary:

\begin{cor}\label{cor:saturated}
    The zero set $\mZ(\alpha)$ can be written as the finite union
    \begin{equation}\label{eq:saturated}
        \bigcup_{\{\alpha\} = W_1 \rightsquigarrow V_1 \rightarrow \cdots \rightarrow W_r \rightsquigarrow V_r \in \mF} \Bigg\{\sum_{i = 0}^r p^{\ell_1 + \cdots + \ell_{i}} \ba_i + \bh \\
        \;\Bigg|\; \forall i, \ell_i \in \Lambda(W_i, V_i),\; \bh \in \sum_{i = 1}^r (\Z^N)^{\Pi(W_i)} \Bigg\},
    \end{equation}
    where $\ba_0, \ba_1, \ldots, \ba_r \in \Q^{KN}$ are defined in~\eqref{eq:defa}. Their denominators are not divisible by $p$.
\end{cor}

This resolves the difference (i) in the discussion at the end of the last subsection.
We now start to resolve the difference (ii).

\paragraph{Symmetrization I: arithmetic progressions.}

Recall from Lemma~\ref{lem:semilinear} that for each $i = 1, \ldots, r$, the set $\Lambda(W_i, V_i)$ can be written as a union of a finite set $Q_i$ and finitely many arithmetic progressions $\{s_{ij} + n \cdot \lambda_{ij} \mid n \in \N\}, j = 1, 2, \ldots$.
Let $\ell$ denote the least common multiplier of all $\lambda_{ij}$.
Then each set $\Lambda(W_i, V_i)$ can be written as a union
\[
\Lambda(W_i, V_i) = Q_i \cup \bigcup_{s = 1}^{u_i} \{\sigma_{is} + \ell k \mid k \in \N\}
\]
for some $\sigma_{is} \in \N, i = 1, \ldots, r;\, s = 1, \ldots, u_i$.
In other words, we can suppose without loss of generality that the common difference in all the arithmetic progressions is equal to $\ell$.
Thus, let $\ell_i$ be any element in $\Lambda(W_i, V_i)$, then $\ell_i$ can either be written as $\sigma_{is} + \ell k_i$ for some $k_i \in \N, s \in \{1, \ldots, u_i\}$, or it is equal to some $q_i \in Q_i$.

Hence, the set $\Big\{\sum_{i = 1}^r p^{\ell_1 + \cdots + \ell_{i}} \left(\ba_i + p^i \bh_i\right) \;\Big|\; \forall i, \ell_i \in \Lambda(W_i, V_i), \bh_i \in (\Z^N)^{\Pi(W_i)} \Big\}$
can be written as a finite union
\begin{multline}\label{eq:replacelwithq}
\bigcup_{\underset{q_1 \in Q_1, \ldots, q_r \in Q_r}{\underset{1 \leq s_1 \leq u_{i_1}, \ldots, 1 \leq s_{r'} \leq u_{i_{r'}}}{1 \leq i_1 < i_2 < \cdots < i_{r'} \leq r}}}
\Bigg\{\sum_{i = 0}^{r} p^{q_1 + \cdots + q_{i_1 - 1} + (\sigma_{i_1s_1} + \ell k_{i_1}) + q_{i_1 + 1} + \cdots + q_{i_2 - 1} + (\sigma_{i_2s_2} + \ell k_{i_2}) + q_{i_2 + 1} + \cdots + q_i} \ba_i + \bh \\
\;\Bigg|\; k_{i_1}, k_{i_2}, \ldots, k_{i_{r'}} \in \N, \; \bh \in \sum_{i = 1}^r (\Z^N)^{\Pi(W_i)} \Bigg\}.
\end{multline}
That is, let $\{i_1, i_2, \ldots, i_{r'}\}$ be the set of all indices $i$ such that the value of $\ell_{i}$ falls in an arithmetic progression $\{\sigma_{is} + \ell k \mid k \in \N\}$; for each of these indices we choose $s \in \{1, \ldots, u_i\}$ to determine the arithmetic progression, and write $\ell_i = \sigma_{is} + \ell k_i,\; k_i \in \N$.
For all the other indices $i \notin \{i_1, i_2, \ldots, i_{r'}\}$, the value of $\ell_i$ falls in the finite set $Q_i$, and we choose $q_i \in Q$ so that $\ell_i = q_i$.
These choices give the decomposition~\eqref{eq:replacelwithq}.

By regrouping the terms, each component of the union~\eqref{eq:replacelwithq} can then be rewritten as
\begin{equation}\label{eq:protoform}
\Bigg\{\sum_{j = 1}^{r'} p^{\ell(k_{i_i} + k_{i_2} + \cdots + k_{i_j})} \cdot \ba'_j + \bh \;\Bigg|\; k_{i_1}, k_{i_2}, \ldots, k_{i_{r'}} \in \N,\; \bh \in \sum_{i = 1}^r (\Z^N)^{\Pi(W_i)} \Bigg\},
\end{equation}
for some 
\begin{multline*}
\ba'_j \coloneqq p^{q_1 + \cdots + q_{i_1 - 1} + \sigma_{i_1s_1} + q_{i_1 + 1} + \cdots + \sigma_{i_js_j}} \cdot \ba_{i_j} + p^{q_1 + \cdots + q_{i_1 - 1} + \sigma_{i_1s_1} + q_{i_1 + 1} + \cdots + \sigma_{i_js_j} + q_{i_j + 1}} \cdot \ba_{i_j + 1} + \cdots \\
+ p^{q_1 + \cdots + q_{i_1 - 1} + \sigma_{i_1s_1} + q_{i_1 + 1} + \cdots + \sigma_{i_js_j} + q_{i_j + 1} + \cdots + q_{i_{j+1} - 1}} \cdot \ba_{i_{j+1} - 1} \in \Q^{KN}.
\end{multline*}
Writing $r'$ as $r$, $k_{i_j}$ as $k_j$, $\ba_j'$ as $\ba_j$, and denoting $H = \sum_{i = 1}^r (\Z^N)^{\Pi(W_i)}$, we can summarize the above discussion by the following corollary.

\begin{cor}\label{cor:afterprogression}
    The zero set $\mZ(\alpha)$ can be written as a finite union of sets of the form
    \begin{equation}\label{eq:afterprogression}
        \Bigg\{\sum_{j = 1}^{r} p^{\ell(k_1 + k_2 + \cdots + k_j)} \cdot \ba_j + \bh \;\Bigg|\; k_1, \ldots, k_r \in \N,\; \bh \in H \Bigg\},
    \end{equation}
    where $H$ is a subgroup of $\Z^{KN}$, and the denominators of each $\ba_j$ are not divisible by $p$.
\end{cor}

\paragraph{Symmetrization II: decreasing exponents.}
Recall that $\mA$ is a local $\Zpe(\oX)$-algebra with the maximal ideal $\frm \ni p$, such that $\frm^t = 0$ for some $t \in \N$.
We prove the following generalization of~\cite[Lemma~8.1]{derksen2007skolem} and~\cite[Lemma~12.1]{derksen2012linear}, which will serve to ``decrease'' the exponents $k_1 + k_2 + \cdots + k_j$ in the expression~\eqref{eq:afterprogression}.

\begin{prop}[Symmetrization of exponents]\label{prop:symmetrize}
    Let $\ba, \bb \in \Q^{KN}$ be such that the denominators of $\ba$ are not divisible by $p$.
    Suppose there exists $m \in \N$ such that $p^{\ell n} \cdot \ba + \bb \in \mZ(\alpha)$ holds for all $n \geq m$. Then $p^{\ell n} \cdot \ba + \bb \in \mZ(\alpha)$ holds for all $n \geq t^2 + t$.
\end{prop}

First we characterize the denominators of $\ba, \bb$:

\begin{lem}[{\cite[p.117]{derksen2015linear}}]\label{lem:denom}
    Let $\ba, \bb \in \Q^{KN}$ be such that the denominators of $\ba$ are not divisible by $p$.
    Suppose $p^{\ell n} \cdot \ba + \bb \in \mZ(\alpha)$ holds for all $n \geq m$ for some $m \in \N$, then $(p^{\ell} - 1) \ba \in \Z^{KN}$, $\ba + \bb \in \Z^{KN}$, $(p^{\ell} - 1) \bb \in \Z^{KN}$.
\end{lem}
\begin{proof}
    Since $p^{\ell m} \cdot \ba + \bb$ and $p^{\ell (m+1)} \cdot \ba + \bb$ are in $\Z^{KN}$, their difference is also in $\Z^{KN}$:
    \[
    (p^{\ell (m+1)} \cdot \ba + \bb) - (p^{\ell m} \cdot \ba + \bb) = p^{\ell m} (p^{\ell} - 1) \cdot \ba \in \Z^{KN}.
    \]
    Since $p$ does not divide the denominators of $\ba$, this yields $(p^{\ell} - 1) \cdot \ba \in \Z^{KN}$.
    Since $p^{\ell} - 1 \mid p^{\ell m} - 1$, we have $(p^{\ell m} - 1) \cdot \ba \in \Z^{KN}$.
    Subtract this from $p^{\ell m} \cdot \ba + \bb \in \Z^{KN}$, we obtain $\ba + \bb \in \Z^{KN}$.
    Finally, since $(p^{\ell} - 1) \cdot \ba \in \Z^{KN}$ and $(p^{\ell} - 1) \cdot (\ba + \bb) \in \Z^{KN}$, we have $(p^{\ell} - 1) \cdot \bb \in \Z^{KN}$.
\end{proof}

Since $p \in \frm$, the quotient $\A \coloneqq \mA/\frm$ is a field of characteristic $p$.
Let $\ox = (x_1, \ldots, x_N)$ be a tuple of elements in $\A$.
For any vector $\bz = (z_1, \ldots, z_N) \in \Z^N$, denote by $\ox^{\bz}$ the product $x_1^{z_1} x_2 ^{z_2} \cdots x_N^{z_N}$.
Write $\ba = (\ba_1, \ldots, \ba_K)$ and $\bb = (\bb_1, \ldots, \bb_K)$ with $\ba_i, \bb_i \in \Z^N$.

For any $r \in \N$ and $i = 1, \ldots, K$, denote the sequence
\begin{equation}\label{eq:defbxgeq}
\bx_i^{(\geq r)} \coloneqq (\ox^{p^{\ell r} \ba_i + \bb_i}, \ox^{p^{\ell (r+1)} \ba_i + \bb_i}, \ox^{p^{\ell (r+2)} \ba_i + \bb_i}, \ldots) \in \A^{\N}.
\end{equation}
Note that $\bx_i^{(\geq r+1)}$ is a subsequence of $\bx_i^{(\geq r)}$.
So if the sequences $\bx_{i_1}^{(\geq r)}, \bx_{i_2}^{(\geq r)}, \ldots, \bx_{i_s}^{(\geq r)}$ are $\A$-linearly dependent for some $i_1, \ldots, i_s \in \{1, \ldots, K\}$, then the sequences $\bx_{i_1}^{(\geq r+1)}, \bx_{i_2}^{(\geq r+1)}, \ldots, \bx_{i_s}^{(\geq r+1)}$ are also $\A$-linearly dependent.
Therefore, let $\mI_r \subseteq \{1, \ldots, K\}$ be a \emph{maximal} subset such that $\bx_i^{(\geq r)}, i \in \mI$ are $\A$-linearly independent (that is, each $\bx_j^{(\geq r)}, j \notin \mI$ can be written as an $\A$-linear combination of $\bx_i^{(\geq r)}, i \in \mI$), then there exists $\mI_{r+1} \subseteq \mI_r$ such that $\mI_{r+1}$ is maximal subset such that $\bx_i^{(\geq r+1)}, i \in \mI$ are $\A$-linearly independent.
The chain $\mI_r \supseteq \mI_{r+1} \supseteq \mI_{r+2} \supseteq \cdots$ must stabilize to some $\mI$.
Then $\mI$ is a maximal subset such that $\bx_i^{(\geq r)}, i \in \mI$ are $\A$-linearly independent \emph{for all $r \in \N$}.

\begin{lem}\label{lem:lindepmodp}
    Let $\A$ be a field of characteristic $p$.
    Let $\ox = (x_1, \ldots, x_N)$ be a tuple of non-zero elements in $\A$.
    Pick a maximal subset $\mI \subseteq \{1, \ldots, K\}$ such that $\bx_i^{(\geq r)}, i \in \mI$ are $\A$-linearly independent for all $r$.
    Then for any $j \notin \mI$, we have
    \begin{equation}\label{eq:linconda}
        \ox^{\ba_j + \bb_j} = \sum_{i \in \mI} c_{j,i} \cdot \ox^{\ba_i + \bb_i},
    \end{equation}
    for some $c_{j,i} \in \A, i \in \mI$ satisfying
    \begin{equation}\label{eq:lincondc}
        c_{j,i}^{p^{\ell}} \cdot \ox^{(p^{\ell} - 1) \bb_i} = c_{j,i} \cdot \ox^{(p^{\ell}-1)\bb_j}.
    \end{equation}
\end{lem}
\begin{proof}
    Take any $j \notin \mI$. For brevity we will omit the subscript $j$ from the elements $c_{j,i}$ and write them as $c_i$.
    By the maximality of $\mI$, there exists $c_i \in \A$ for each $i \in \mI$, such that 
    \begin{equation}\label{eq:depenr}
        \ox^{p^{\ell r} \ba_j + \bb_j} = \sum_{i \in \mI} c_i \cdot \ox^{p^{\ell r} \ba_i + \bb_i}
    \end{equation}
    for all large enough $r$.
    Since $\A$ is of characteristic $p$, taking $p^{\ell}$-th power on both sides yields
    \begin{equation}\label{eq:Alin1}
    \ox^{p^{\ell (r+1)} \ba_j + p^{\ell} \bb_j} = \sum_{i \in \mI} c_i^{p^{\ell}} \cdot \ox^{p^{\ell (r+1)} \ba_i + p^{\ell} \bb_i}.
    \end{equation}
    Since Equation~\eqref{eq:depenr} also holds for $r+1$, we have $\ox^{p^{\ell (r+1)} \ba_j + \bb_j} = \sum_{i \in \mI} c_i \cdot \ox^{p^{\ell (r+1)} \ba_i + \bb_i}$.
    Multiplying both sides of $\ox^{p^{\ell (r+1)} \ba_j + \bb_j} = \sum_{i \in \mI} c_i \cdot \ox^{p^{\ell (r+1)} \ba_i + \bb_i}$ by $\ox^{(p^{\ell}-1)\bb_j}$ (note that $(p^{\ell}-1)\bb_j \in \Z^N$ by Lemma~\ref{lem:denom}), we have
    \begin{equation*}
    \ox^{p^{\ell (r+1)} \ba_j + p^{\ell} \bb_j} = \sum_{i \in \mI} c_i \cdot \ox^{p^{\ell (r+1)} \ba_i + \bb_i} \cdot \ox^{(p^{\ell}-1)\bb_j}.
    \end{equation*}
    Subtract this from Equation~\eqref{eq:Alin1}, we obtain
    \[
    0 = \sum_{i \in \mI} \left( c_i^{p^{\ell}} \cdot \ox^{(p^{\ell} - 1) \bb_i} - c_i \cdot \ox^{(p^{\ell}-1)\bb_j}\right) \cdot \ox^{p^{\ell (r+1)} \ba_i + \bb_i}
    \]
    
    Since this holds also for $r+1, r+2, \ldots$, we have
    \[
    0 = \sum_{i \in \mI} \left( c_i^{p^{\ell}} \cdot \ox^{(p^{\ell} - 1) \bb_i} - c_i \cdot \ox^{(p^{\ell}-1)\bb_j}\right) \cdot \bx_i^{(\geq r + 1)}.
    \]
    But $\bx_i^{(\geq r+1)}, i \in \mI$ are $\A$-linearly independent, so we must have 
    \[
    c_i^{p^{\ell}} \cdot \ox^{(p^{\ell} - 1) \bb_i} - c_i \cdot \ox^{(p^{\ell}-1)\bb_j} = 0
    \]
    for all $i \in \mI$. This shows Equation~\eqref{eq:lincondc}.
    
    Next we show Equation~\eqref{eq:linconda}.
    From Equation~\eqref{eq:lincondc} we have
    \[
    c_i = c_i^{p^{\ell}} \cdot \ox^{(p^{\ell} - 1) \bb_i - (p^{\ell} - 1) \bb_j}.
    \]
    Substituting this for $c_i$ in $\ox^{p^{\ell r} \ba_j + \bb_j} = \sum_{i \in \mI} c_i \cdot \ox^{p^{\ell r} \ba_i + \bb_i}$ yields
    \[
    \ox^{p^{\ell r} \ba_j + \bb_j} = \sum_{i \in \mI} c_i^{p^{\ell}} \cdot \ox^{(p^{\ell} - 1) \bb_i - (p^{\ell} - 1) \bb_j} \cdot \ox^{p^{\ell r} \ba_i + \bb_i},
    \]
    so
    \[
    \ox^{p^{\ell r} \ba_j + p^{\ell}\bb_j} = \sum_{i \in \mI} c_i^{p^{\ell}} \cdot \ox^{p^{\ell r} \ba_i + p^{\ell}\bb_i}.
    \]
    Since $\A$ has characteristic $p$, this can be rewritten as
    \[
    \left(\ox^{p^{\ell (r-1)} \ba_j + \bb_j}\right)^{p^{\ell}} = \left(\sum_{i \in \mI} c_i \cdot \ox^{p^{\ell (r-1)} \ba_i + \bb_i}\right)^{p^{\ell}}.
    \]
    Since $y^{p^{\ell}} = z^{p^{\ell}} \implies (y - z)^{p^{\ell}} = y^{p^{\ell}} - z^{p^{\ell}} = 0 \implies y - z = 0$, the equation above yields
    \[
    \ox^{p^{\ell (r-1)} \ba_j + \bb_j} = \sum_{i \in \mI} c_i \cdot \ox^{p^{\ell (r-1)} \ba_i + \bb_i}.
    \]
    This means that Equation~\eqref{eq:depenr} still holds if we replace $r$ by $r-1$.
    Repeat this iteratively for $r - 1, r-2, \ldots, 2, 1$, we obtain
    \[
    \ox^{\ba_j + \bb_j} = \sum_{i \in \mI} c_i \cdot \ox^{\ba_i + \bb_i},
    \]
    which concludes the proof for Equation~\eqref{eq:linconda}.
\end{proof}

The following lemma can be considered as an extension of Lemma~\ref{lem:powp}.
\begin{lem}\label{lem:powpn}
    Let $f \in \Z[\oX^{\pm}]$, then for all $n \geq t$, we have 
    \[
    f^{p^{\ell n}}(Y_1, \ldots, Y_k) \equiv f^{p^{\ell t}}(Y_1^{p^{\ell (n-t)}}, \ldots, Y_k^{p^{\ell (n-t)}}) \mod p^t.
    \]
\end{lem}
\begin{proof}
    We have $f^p(Y_1, \ldots, Y_k) \equiv f(Y_1^p, \ldots, Y_k^p) \mod p$.
    Apply Lemma~\ref{lem:lifting} to the ring $\Z_{/p^t}[\oX^{\pm}]$ and to the ideal generated by $p$, we obtain
    \begin{equation}\label{eq:pluspp}
    f^{p^{r}}(Y_1, \ldots, Y_k) \equiv f^{p^{r - 1}}(Y_1^p, \ldots, Y_k^p) \mod p^t.
    \end{equation}
    for all $r \geq t-1$.
    We can apply~\eqref{eq:pluspp} for $r = \ell n, \ell n-1, \ldots, \ell t + 1$, and obtain
    \begin{multline*}
    f^{p^{\ell n}}(Y_1, \ldots, Y_k) \equiv f^{p^{\ell n-1}}(Y_1^p, \ldots, Y_k^p) \equiv f^{p^{\ell n-2}}(Y_1^{p^2}, \ldots, Y_k^{p^2}) \equiv \cdots \equiv f^{p^{\ell t}}(Y_1^{p^{\ell (n-t)}}, \ldots, Y_k^{p^{\ell (n-t)}}) \\
    \mod p^t.
    \end{multline*}
\end{proof}

Recall that $\alpha(\bz_1, \ldots, \bz_K) = \sum_{i = 1}^K \oA^{\bz_i} v_i$, where $\oA = (A_1, \ldots, A_N)$ is a tuple of elements in the $\Zpe(\oX)$-algebra $\mA$, and $v_1, \ldots, v_K$ are elements in the $\mA$-module $\mV$, with $\frm^t \mV = 0$.
We prove the following slight generalization of Proposition~\ref{prop:symmetrize}.
In particular, if we take $s = 0$, then we immediately obtain Proposition~\ref{prop:symmetrize}.
The reason we introduce the additional variable $s$ is to perform induction.

\begin{lem}\label{lem:syminduction}
    Let $s \leq t$ be an integer, and let $v_1, \ldots, v_K \in \frm^{s}\mV$. 
    Suppose there exists $m \in \N$ such that $\sum_{i = 1}^K \oA^{p^{\ell n} \ba_i + \bb_i} v_i = 0$ holds for all $n \geq m$. Then $\sum_{i = 1}^K \oA^{p^{\ell n} \ba_i + \bb_i} v_i = 0$ holds for all $n \geq (t+1-s)t$.
\end{lem}
\begin{proof}
    We use reverse induction on $s$, starting from $t$ and gradually decreasing to $0$.
    The case where $s = t$ is trivial because $\frm^{t}\mV = 0$. We now focus on the induction step.
    Suppose the statement is true for $s$, we show it for $s-1$.
    Let $v_1, \ldots, v_K \in \frm^{s-1}\mV$.
    
    If $v_i \in \frm^{s} \mV$ for all $i = 1, \ldots, K$, then we conclude directly using the induction hypothesis for $s$.
    Suppose now that $v_i \notin \frm^{s} \mV$ for some $i$. 
    Without loss of generality, we can suppose $v_1 \in \frm^{s-1} \mV \setminus \frm^{s} \mV, \ldots, v_{K'} \notin \frm^{s-1} \mV \setminus \frm^{s} \mV$, and $v_{K'+1} \in \frm^{s} \mV, \ldots, v_K \in \frm^{s} \mV$, where $1 \leq K' \leq K$.
    
    Let $\A \coloneqq \mA/\frm$, it is a field of characteristic $p$.
    Let
    \[
    x_1 \coloneqq A_1 + \frm, x_2 \coloneqq A_2 + \frm, \ldots, x_N \coloneqq A_N + \frm,
    \]
    be elements of $\A$. These are non-zero since $A_1, \ldots, A_N$ are invertible in $\mA$. Denote $\ox \coloneqq (x_1, \ldots, x_N)$.
    
    As in Equation~\eqref{eq:defbxgeq}, for each $i = 1, \ldots, K$, and $r \in \N$, denote the sequence
    \[
    \bx_i^{(\geq r)} \coloneqq (\ox^{p^{\ell r} \ba_i + \bb_i}, \ox^{p^{\ell (r+1)} \ba_i + \bb_i}, \ox^{p^{\ell (r+2)} \ba_i + \bb_i}, \ldots) \in \A^{\N}.
    \]
    Let $\mI \subseteq \{1, \ldots, K'\}$ be a maximal subset such that $\bx_i^{(\geq r)}, i \in \mI$ are $\A$-linearly independent for all $r \in \N$.
    Without loss of generality suppose $\mI = \{1, \ldots, k\}$ for some $k \leq K'$.
    By Lemma~\ref{lem:lindepmodp}, we can write
    \begin{multline*}
    \ox^{\ba_{k+1} + \bb_{k+1}} = c_{k+1,1} \cdot \ox^{\ba_1 + \bb_1} + \cdots + c_{k+1,k} \cdot \ox^{\ba_k + \bb_k}, \\
    \ddots \\
    \ox^{\ba_{K'} + \bb_{K'}} = c_{K',1} \cdot \ox^{\ba_1 + \bb_1} + \cdots + c_{K',k} \cdot \ox^{\ba_k + \bb_k},
    \end{multline*}
    for some $c_{k+1, 1}, \ldots, c_{K', k} \in \A$ that satisfy
    \begin{equation}\label{eq:condc}
    c_{j,i}^{p^{\ell}} \cdot \ox^{(p^{\ell} - 1) \bb_i} = c_{j,i} \cdot \ox^{(p^{\ell}-1)\bb_j}.
    \end{equation}
    For $j = k+1, \ldots, K'; i = 1, \ldots, k$, take any $\tc_{j, i} \in \mA$ such that $\tc_{j, i} = c_{j, i} + \frm$.
    This yields
    \begin{multline*}
    \oA^{\ba_{k+1} + \bb_{k+1}} \equiv \tc_{k+1,1} \oA^{\ba_1 + \bb_1} + \cdots + \tc_{k+1,k} \oA^{\ba_k + \bb_k} \mod \frm, \\\;\ddots\; \\
    \ox^{\ba_{K'} + \bb_{K'}} \equiv \tc_{K',1} \oA^{\ba_1 + \bb_1} + \cdots + \tc_{K',k} \oA^{\ba_k + \bb_k} \mod \frm.
    \end{multline*}
    
    Applying Lemma~\ref{lem:lifting} to the ring $\mA$ and the ideal $\frm$, we have
    \begin{multline*}
    \oA^{p^{\ell n}(\ba_{k+1} + \bb_{k+1})} = \left(\tc_{k+1,1} \oA^{\ba_1 + \bb_1} + \cdots + \tc_{k+1,k} \oA^{\ba_k + \bb_k} \right)^{p^{\ell n}}, \\\;\ddots\; \\
    \oA^{p^{\ell n}(\ba_{K'} + \bb_{K'})} = \left(\tc_{K',1} \oA^{\ba_1 + \bb_1} + \cdots + \tc_{K',k} \oA^{\ba_k + \bb_k} \right)^{p^{\ell n}},
    \end{multline*}
    for all $\ell n \geq t-1$.
    Then for all $n \geq t-1$, the equation $\sum_{i = 1}^K \oA^{p^{\ell n} \ba_i + \bb_i} v_i = 0$ is equivalent to
    \begin{multline}\label{eq:sumyv}
    \sum_{i = 1}^k \oA^{p^{\ell n} \ba_i + \bb_i} v_i + \sum_{j = k+1}^{K'} \left(\tc_{j,1} \oA^{\ba_1 + \bb_1} + \cdots + \tc_{j,k} \oA^{\ba_k + \bb_k} \right)^{p^{\ell n}} \oA^{(1 - p^{\ell n}) \bb_j} v_j + \sum_{i = K'+1}^K \oA^{p^{\ell n} \ba_i + \bb_i} v_i = 0.
    \end{multline}
    For $j = k+1, \ldots, K'$, consider the polynomial $f(Y_1, \ldots, Y_k) \coloneqq Y_1 + \cdots + Y_k \in \Z[Y_1, \ldots, Y_k]$.
    By Lemma~\ref{lem:powpn} we have
    $
    f^{p^{\ell n}}(Y_1, \ldots, Y_k) \equiv f^{p^{\ell t}}(Y_1^{p^{\ell (n-t)}}, \ldots, Y_k^{p^{\ell (n-t)}}) \mod p^t,
    $
    for all $n \geq t$.
    We can write the polynomial $f^{p^{\ell t}}(Y_1, \ldots, Y_k) = (Y_1 + \cdots + Y_k)^{p^{\ell t}}$ in the form
    \[
    Y_1^{p^{\ell t}} + \cdots + Y_k^{p^{\ell t}} + p \sum_{i \in \mJ} u_{i} Y_1^{d_{i1}} Y_2^{d_{i2}} \cdots Y_k^{d_{ik}}
    \]
    for some finite index set $\mJ$, and with $u_i \in \Z, d_{i1} + \cdots + d_{ik} = p^{\ell t},$ for all $i \in \mJ$.
    Then
    \begin{equation}\label{eq:cY}
    (Y_1 + \cdots + Y_k)^{p^{\ell n}} \equiv Y_1^{p^{\ell n}} + \cdots + Y_k^{p^{\ell n}} + p \sum_{i \in \mJ} u_{i} Y_1^{p^{\ell (n - t)} d_{i1}} Y_2^{p^{\ell (n - t)} d_{i2}} \cdots Y_k^{p^{\ell (n - t)} d_{ik}} \mod p^t.
    \end{equation}
    For each $j = k+1, \ldots, K'$, consider the ring homomorphism from $\Z[Y_1, \ldots, Y_k]$ to $\mA$, that sends $Y_1, Y_2, \ldots, Y_k$ respectively to $\tc_{j,1} \oA^{\ba_1 + \bb_1}, \tc_{j,2} \oA^{\ba_2 + \bb_2}, \ldots, \tc_{j,k} \oA^{\ba_k + \bb_k}$.
    Since $p^t = 0$ in $\mA$, applying this homomorphism to Equation~\eqref{eq:cY} yields
    \begin{multline}\label{eq:unfoldpln}
    \left(\tc_{j,1} \oA^{\ba_1 + \bb_1} + \cdots + \tc_{j,k} \oA^{\ba_k + \bb_k}\right)^{p^{\ell n}} = \tc_{j,1}^{p^{\ell n}} \cdot \oA^{p^{\ell n}(\ba_1 + \bb_1)} + \cdots + \tc_{j,k}^{p^{\ell n}} \cdot \oA^{p^{\ell n}(\ba_k + \bb_k)} + \\
    p \sum_{i \in \mJ} u_{i} \tc_{j,1}^{p^{\ell (n - t)} d_{i1}} \cdot \oA^{p^{\ell (n-t)}d_{i1}(\ba_1 + \bb_1)} \cdots \tc_{j,k}^{p^{\ell (n - t)} d_{ik}} \cdot \oA^{p^{\ell (n-t)}d_{ik}(\ba_k + \bb_k)}.
    \end{multline}
    
    Recall from Equation~\eqref{eq:condc} that $c_{j,i}^{p^{\ell}} \cdot \ox^{(p^{\ell} - 1) \bb_i} = c_{j,i} \cdot \ox^{(p^{\ell}-1)\bb_j}$ for all $j, i$.
    So
    \[
    \tc_{j,i}^{p^{\ell}} \cdot \oA^{(p^{\ell} - 1) \bb_i} \equiv \tc_{j,i} \cdot \oA^{(p^{\ell}-1)\bb_j} \mod \frm.
    \]
    For any $n \geq t$, taking $p^{\ell (n-1)}$-th power, $p^{\ell (n-2)}$-th power, $\ldots,$ $p^{\ell (t-1)}$-th power, on both sides and using Lemma~\ref{lem:lifting} yield respectively
    \begin{align*}
    \tc_{j,i}^{p^{\ell n}} \cdot \oA^{(p^{\ell n} - p^{\ell (n-1)}) \bb_i} & = \tc_{j,i}^{p^{\ell (n-1)}} \cdot \oA^{(p^{\ell n} - p^{\ell (n-1)})\bb_j}, \\
    \tc_{j,i}^{p^{\ell (n-1)}} \cdot \oA^{(p^{\ell (n-1)} - p^{\ell (n-2)}) \bb_i} & = \tc_{j,i}^{p^{\ell (n-2)}} \cdot \oA^{(p^{\ell (n-1)} - p^{\ell (n-2)})\bb_j}, \\
    & \vdots \\
    \tc_{j,i}^{p^{\ell t}} \cdot \oA^{(p^{\ell t} - p^{\ell (t-1)}) \bb_i} & = \tc_{j,i}^{p^{\ell (t-1)}} \cdot \oA^{(p^{\ell t} - p^{\ell (t-1)})\bb_j}.
    \end{align*}
    Their product yields for all $n \geq t$,
    \begin{equation}\label{eq:tc1}
    \tc_{j,i}^{p^{\ell n}} = \tc_{j,i}^{p^{\ell (t-1)}} \cdot \oA^{(p^{\ell n} - p^{\ell (t-1)})(\bb_j - \bb_i)}.
    \end{equation}
    For all $n \geq 2t$, replacing $n$ with $n-t$ in Equation~\eqref{eq:tc1} yields
    \begin{equation}\label{eq:tc2}
    \tc_{j,i}^{p^{\ell (n-t)}} = \tc_{j,i}^{p^{\ell (t-1)}} \cdot \oA^{(p^{\ell (n-t)} - p^{\ell (t-1)})(\bb_j - \bb_i)}.
    \end{equation}
    Substituting the terms $\tc_{j,i}^{p^{\ell n}}$ and $\tc_{j,i}^{p^{\ell (n-t)}}$ in the right hand side of Equation~\eqref{eq:unfoldpln} using Equations~\eqref{eq:tc1} and \eqref{eq:tc2}, we obtain
    \begin{multline}\label{eq:sub1}
        \left(\tc_{j,1} \oA^{\ba_1 + \bb_1} + \cdots + \tc_{j,k} \oA^{\ba_k + \bb_k}\right)^{p^{\ell n}} = \\
        \left(\tc_{j,1}^{p^{\ell (t-1)}} \cdot \oA^{p^{\ell n}\ba_1 + p^{\ell (t-1)} \bb_1} + \cdots + \tc_{j,k}^{p^{\ell (t-1)}} \cdot \oA^{p^{\ell n}\ba_k + p^{\ell (t-1)} \bb_k} \right) \cdot \oA^{(p^{\ell n} - p^{\ell (t-1)}) \bb_j} + \\
        p \sum_{i \in \mJ} \left( u_{i} \tc_{j,1}^{p^{\ell (t-1)}d_{i1}} \oA^{\left(p^{\ell (n-t)}\ba_1 + p^{\ell (t-1)} \bb_1  \right)d_{i1}} \cdots \tc_{j,k}^{p^{\ell (t-1)} d_{ik}} \oA^{\left(p^{\ell (n-t)}\ba_k + p^{\ell (t-1)} \bb_k \right) d_{ik} } \right) \oA^{(p^{\ell (n-t)} - p^{\ell (t-1)}) \bb_j \cdot p^{\ell t}},
    \end{multline}
    for all $n \geq 2t$.
    Note that the term $p^{\ell t}$ on the final exponent comes from using $d_{i1} + \cdots + d_{ik} = p^{\ell t}$ in the above substitution.
    Using Equation~\eqref{eq:sub1} to substitute $\left(\tc_{j,1} \oA^{\ba_1 + \bb_1} + \cdots + \tc_{j,k} \oA^{\ba_k + \bb_k}\right)^{p^{\ell n}}$ in Equation~\eqref{eq:sumyv} yields
    \begin{multline}\label{eq:sumyv2}
        \sum_{i = 1}^k \oA^{p^{\ell n} \ba_i + \bb_i} v_i \\
        + \sum_{j = k+1}^{K'} \left(\tc_{j,1}^{p^{\ell (t-1)}} \cdot \oA^{p^{\ell n}\ba_1 + p^{\ell (t-1)} \bb_1} + \cdots + \tc_{j,k}^{p^{\ell (t-1)}} \cdot \oA^{p^{\ell n}\ba_k + p^{\ell (t-1)} \bb_k} \right) \cdot \oA^{(1 - p^{\ell (t-1)}) \bb_j} v_j \\
        + p \sum_{j = k+1}^{K'} \sum_{i \in \mJ} u_{i} \tc_{j,1}^{p^{\ell (t-1)}d_{i1}} \cdots \tc_{j,k}^{p^{\ell (t-1)} d_{ik}} \cdot \oA^{p^{\ell (n-t)}(\ba_1 d_{i1} + \cdots + \ba_k d_{ik}) + p^{\ell (t-1)} (\bb_1 d_{i1} + \cdots + \bb_k d_{ik})} \cdot \oA^{(1 - p^{\ell (2t-1)}) \bb_j} v_j \\
        + \sum_{i = K'+1}^K \oA^{p^{\ell n} \ba_i + \bb_i} v_i = 0.
    \end{multline}
    Combining all the terms containing $\oA^{p^{\ell n} \ba_i}$ yields
    \begin{multline}\label{eq:sumyv3}
        \sum_{i = 1}^k \oA^{p^{\ell n} \ba_i + \bb_i} \left(v_i + \sum_{j = k+1}^{K'} \tc_{j,i}^{p^{\ell (t-1)}} \cdot \oA^{(p^{\ell (t-1)} - 1) (\bb_i - \bb_j)} v_j \right) \\
        + p \sum_{j = k+1}^{K'} \sum_{i \in \mJ} u_{i} \tc_{j,1}^{p^{\ell (t-1)}d_{i1}} \cdots \tc_{j,k}^{p^{\ell (t-1)} d_{ik}} \cdot \oA^{p^{\ell (n-t)}(\ba_1 d_{i1} + \cdots + \ba_k d_{ik}) + p^{\ell (t-1)} (\bb_1 d_{i1} + \cdots + \bb_k d_{ik})} \cdot \oA^{(1 - p^{\ell (2t-1)}) \bb_j} v_j \\
        + \sum_{i = K'+1}^K \oA^{p^{\ell n} \ba_i + \bb_i} v_i = 0.
    \end{multline}
    We have shown that the Equation $\sum_{i = 1}^K \oA^{p^{\ell n} \ba_i + \bb_i} v_i = 0$ is equivalent to Equation~\eqref{eq:sumyv3} for $n \geq 2t$.
    
    Note that $p v_{k+1}, \ldots, p v_{K'}, v_{K'+1}, \ldots, v_K \in \frm^s \mV$.
    Therefore for all $n \geq \max\{m, 2t\}$, Equation~\eqref{eq:sumyv3} yields
    \begin{equation}\label{eq:sumyvmod}
    \sum_{i = 1}^k \oA^{p^{\ell n} \ba_i + \bb_i} \left(v_i + \sum_{j = k+1}^{K'} \tc_{j,i}^{p^{\ell (t-1)}} \cdot \oA^{(p^{\ell (t-1)} - 1) (\bb_i - \bb_j)} v_j \right) \in \frm^s \mV.
    \end{equation}
    Note that $v_1, \ldots, v_k \in \frm^{s-1} \mV$, so Equation~\eqref{eq:sumyvmod} is equivalent the following equation in the $\A$-module $\frm^{s-1} \mV/\frm^s\mV$:
    \[
    \sum_{i = 1}^k \ox^{p^{\ell n} \ba_i + \bb_i} \left(v_i + \sum_{j = k+1}^{K'} c_{j,i}^{p^{\ell (t-1)}} \cdot \ox^{(p^{\ell (t-1)} - 1) (\bb_i - \bb_j)} v_j + \frm^s\mV \right) = 0.
    \]

    Note that $\frm^{s-1} \mV/\frm^s\mV$ is a finitely generated module over the field $\A$, and is therefore a finite dimensional $\A$-vector space.
    Let $\pi$ be any $\A$-linear map from $\frm^{s-1} \mV/\frm^s\mV$ to $\A$, then
    \[
    \sum_{i = 1}^k \ox^{p^{\ell n} \ba_i + \bb_i} \cdot \pi \left(v_i + \sum_{j = k+1}^{K'} c_{j,i}^{p^{\ell (t-1)}} \cdot \ox^{(p^{\ell (t-1)} - 1) (\bb_i - \bb_j)} v_j + \frm^s\mV \right) = 0
    \]
    for all $n \geq \max\{m, 2t\}$.
    Recall that for all $n$, the sequences
    \[
    \bx_i^{(\geq n)} = \left(\ox^{p^{\ell n} \ba_i + \bb_i}, \ox^{p^{\ell (n+1)} \ba_i + \bb_i}, \ldots\right) \in \A^{\N}
    \]
    for $i = 1, 2, \ldots, k$, are $\A$-linearly independent.
    Therefore we must have
    \[
    \pi \left(v_i + \sum_{j = k+1}^{K'} c_{j,i}^{p^{\ell (t-1)}} \cdot \ox^{(p^{\ell (t-1)} - 1) (\bb_i - \bb_j)} v_j + \frm^s\mV \right) = 0.
    \]
    for $i = 1, 2, \ldots, k$. 
    Since this is true for all linear maps $\pi \colon \frm^{s-1} \mV/\frm^s\mV \rightarrow \A$, we have
    \[
    v_i + \sum_{j = k+1}^{K'} c_{j,i}^{p^{\ell (t-1)}} \cdot \ox^{(p^{\ell (t-1)} - 1) (\bb_i - \bb_j)} v_j + \frm^s\mV = 0
    \]
    for $i = 1, 2, \ldots, k$.
    Denote
    \[
    w_i \coloneqq v_i + \sum_{j = k+1}^{K'} c_{j,i}^{p^{\ell (t-1)}} \cdot \ox^{(p^{\ell (t-1)} - 1) (\bb_i - \bb_j)} v_j
    \]
    for $i = 1, \ldots, k$, then $w_i \in \frm^s\mV$ for all $i$.
    Equation~\eqref{eq:sumyv3} can be rewritten as
    \begin{multline}\label{eq:sumyv4}
    \sum_{i = 1}^k \oA^{p^{\ell n} \ba_i + \bb_i} w_i \\
        + p \sum_{j = k+1}^{K'} \sum_{i \in \mJ} u_{i} \tc_{j,1}^{p^{\ell (t-1)}d_{i1}} \cdots \tc_{j,k}^{p^{\ell (t-1)} d_{ik}} \cdot \oA^{p^{\ell (n-t)}(\ba_1 d_{i1} + \cdots + \ba_k d_{ik}) + p^{\ell (t-1)} (\bb_1 d_{i1} + \cdots + \bb_k d_{ik})} \cdot \oA^{(1 - p^{\ell (2t-1)}) \bb_j} \cdot v_j \\
        + \sum_{i = K'+1}^K \oA^{p^{\ell n} \ba_i + \bb_i} v_i = 0.
    \end{multline}
    Note that $w_1, \ldots, w_k, pv_{k+1}, \ldots, pv_{K'}, v_{K'+1}, \ldots, v_{K}$ are now all in $\frm^s \mV$.
    The left hand side of Equation~\eqref{eq:sumyv4} is a sum of terms of the form
    $
        \oA^{p^{\ell (n-t)} \ba + \bb} v
    $,
    for $v \in \frm^s \mV$ and $\ba, \bb \in \Q^{N}$ whose denominators are not divisible by $p$.
    Indeed,
    \begin{enumerate}[nosep, label=(\roman*)]
    \item 
    the terms $\oA^{p^{\ell n} \ba_i + \bb_i} w_i, i = 1, \ldots, k$, can be written in the form $\oA^{p^{\ell (n-t)} \ba + \bb} v$, with
    \[
    \ba \coloneqq (p^{\ell n} - p^{\ell (n-t)}) \ba_i \;, \quad \bb \coloneqq \bb_i \;, \quad v \coloneqq w_i.
    \]
    Similarly, the terms $\oA^{p^{\ell n} \ba_i + \bb_i} v_i, i = K'+1, \ldots, K$, can be written in the form $\oA^{p^{\ell (n-t)} \ba + \bb} v$.
    \item
    The terms
    \[
    p \left( u_{i} \tc_{j,1}^{p^{\ell (t-1)}d_{i1}} \cdots \tc_{j,k}^{p^{\ell (t-1)} d_{ik}} \cdot \oA^{p^{\ell (n-t)}(\ba_1 d_{i1} + \cdots + \ba_k d_{ik}) + p^{\ell (t-1)} (\bb_1 d_{i1} + \cdots + \bb_k d_{ik})} \cdot \oA^{(1 - p^{\ell (2t-1)}) \bb_j} \right) \cdot v_j,
    \]
    for $j = k+1, \ldots, K'; i \in \mJ$,
    can be written in the form $\oA^{p^{\ell (n-t)} \ba + \bb} v$ with
    \[
    \ba \coloneqq \ba_1 d_{i1} + \cdots + \ba_k d_{ik} \;, \quad \bb \coloneqq p^{\ell (t-1)} (\bb_1 d_{i1} + \cdots + \bb_k d_{ik}) + (1 - p^{\ell (2t-1)}) \bb_j,
    \]
    and
    \[
    v \coloneqq p \cdot u_{i} \tc_{j,1}^{p^{\ell (t-1)}d_{i1}} \cdots \tc_{j,k}^{p^{\ell (t-1)} d_{ik}} \cdot v_j.
    \]
    \end{enumerate}
    
    Since Equation~\eqref{eq:sumyv4} can be written as a sum of terms $\oA^{p^{\ell (n-t)} \ba + \bb} v$, $v \in \frm^s \mV$, and it holds for all $n - t \geq \max\{m, 2t\} - t$, we can apply the induction hypothesis on $s$, and conclude that Equation~\eqref{eq:sumyv4} holds for all $n - t \geq (t+1-s)t$.
    Since Equation~\eqref{eq:sumyv4} is equivalent to $\sum_{i = 1}^K \oA^{p^{\ell n} \ba_i + \bb_i} v_i = 0$ for $n \geq 2t$ (which is true whenever $n - t \geq (t+1-s)t$), we conclude that 
    \[
    \sum_{i = 1}^K \oA^{p^{\ell n} \ba_i + \bb_i} v_i = 0
    \]
    holds for all $n \geq t+ (t+1-s)t = \big(t+1 - (s-1)\big) t$.
    This finishes the induction step on $s$.
\end{proof}

\begin{proof}[Proof of Proposition~\ref{prop:symmetrize}]
    Proposition~\ref{prop:symmetrize} follows directly from Lemma~\ref{lem:syminduction} by taking $s = 0$.
\end{proof}

\paragraph{Symmetrization III: conclusion.}
In this subsection we finish the proof of Theorem~\ref{thm:primepower}.
Recall from Corollary~\ref{cor:afterprogression} that $\mZ(\alpha)$ is a finite union of sets of the form
\begin{multline}\label{eq:unionm}
        \Bigg\{\sum_{j = 1}^{r} p^{\ell(k_1 + k_2 + \cdots + k_j)} \cdot \ba_j + \bh \;\Bigg|\; k_1, \ldots, k_r \in \N,\; \bh \in H \Bigg\} \\
        = \Bigg\{\sum_{j = 1}^{r} p^{\ell n_j} \cdot \ba_j + \bh \;\Bigg|\; 0 \leq n_1 \leq n_2 \leq \cdots \leq n_r \in \N,\; \bh \in H \Bigg\}.
\end{multline}
We will use Proposition~\ref{prop:symmetrize} to decrease the exponents $\ell n_i$, and show $\mZ(\alpha)$ is a finite union of $p$-succinct sets.

\begin{proof}[Proof of Theorem~\ref{thm:primepower}]
    See Figures~\ref{fig:triangular}-\ref{fig:succinct} for an illustration of the proof.
    Denote
    \[
    \mT \coloneqq \Bigg\{\sum_{j = 1}^{r} p^{\ell n_j} \cdot \ba_j + \bh \;\Bigg|\; 0 \leq n_1 \leq n_2 \leq \cdots \leq n_r \in \N,\; \bh \in H \Bigg\}.
    \]

    \begin{figure}[h!]
    \centering
    \begin{minipage}[t]{.30\textwidth}
        \centering
        \includegraphics[width=\textwidth,height=1.0\textheight,keepaspectratio, trim={6cm 0cm 5.7cm 0cm},clip]{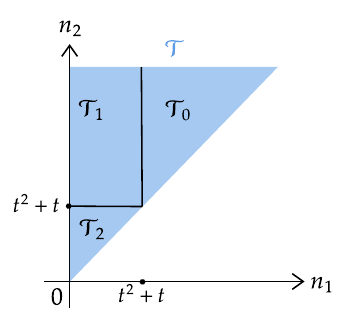}
        \caption{Decomposition of the set $\mT$.}
        \label{fig:triangular}
    \end{minipage}
    \hfill
    \begin{minipage}[t]{.30\textwidth}
        \centering
        \includegraphics[width=1\textwidth,height=1.0\textheight,keepaspectratio, trim={6cm 0cm 5.7cm 0cm},clip]{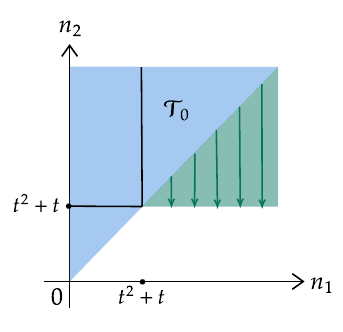}
        \caption{Applying Proposition~\ref{prop:symmetrize}.}
        \label{fig:symmetrize}
    \end{minipage}
    \hfill
    \begin{minipage}[t]{0.30\textwidth}
        \centering
        \includegraphics[width=1\textwidth,height=1.0\textheight,keepaspectratio, trim={6cm 0cm 5.7cm 0cm},clip]{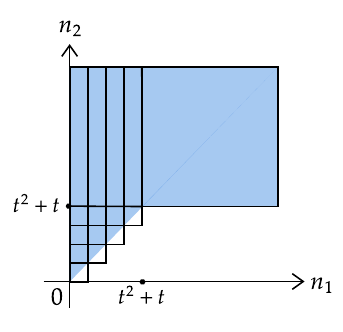}
        \caption{A finite union of $p$-succinct sets.}
        \label{fig:succinct}
    \end{minipage}
\end{figure}

    We can write $\mT$ as a union $\mT_0 \cup \mT_1 \cup \cdots \cup \mT_r$, where
    \begin{multline*}
        \mT_0 \coloneqq \Bigg\{\sum_{j = 1}^{r} p^{\ell n_j} \cdot \ba_j + \bh \;\Bigg|\; t^2+t \leq n_1 \leq n_2 \leq \cdots \leq n_r,\; \bh \in H \Bigg\}, \\
        \mT_1 \coloneqq \Bigg\{\sum_{j = 1}^{r} p^{\ell n_j} \cdot \ba_j + \bh \;\Bigg|\; 0 \leq n_1 \leq t^2+t \leq n_2 \leq \cdots \leq n_r,\; \bh \in H \Bigg\}, \\
        \ddots \\
        \mT_r \coloneqq \Bigg\{\sum_{j = 1}^{r} p^{\ell n_j} \cdot \ba_j + \bh \;\Bigg|\; 0 \leq n_1 \leq n_2 \leq \cdots \leq n_r \leq t^2 + t,\; \bh \in H \Bigg\}.
    \end{multline*}

    For each $i = 0, 1, \ldots, r$, consider the expression $\sum_{j = 1}^{r} p^{\ell n_j} \cdot \ba_j + \bh, 0 \leq n_1 \leq \cdots \leq n_i \leq t^2 + t \leq n_{i+1} \leq \cdots \leq n_r,$ in the set $\mT_i$.
    Apply Proposition~\ref{prop:symmetrize} successively for $n = n_r, n_{r-1}, \ldots, n_{i+2}$, we obtain a new set $\widetilde{\mT}_i \supseteq \mT_i$ such that $\widetilde{\mT}_i \subseteq \mZ(\alpha)$, where
    \[
    \widetilde{\mT}_i \coloneqq \Bigg\{\sum_{j = 1}^{r} p^{\ell n_j} \cdot \ba_j + \bh \;\Bigg|\; 0 \leq n_1 \leq \cdots \leq n_i \leq t^2 + t, \; t^2+t \leq n_{i+1}, \ldots, t^2+t \leq n_r,\; \bh \in H \Bigg\}.
    \]
    Each set $\widetilde{\mT}_i$ can be written as a finite union of $p$-succinct sets
    \begin{align*}
    \widetilde{\mT}_i
    & = \bigcup_{0 \leq n_1 \leq \cdots \leq n_i \leq t^2 + t} \Bigg\{\sum_{j = 1}^{r} p^{\ell n_j} \cdot \ba_j + \bh \;\Bigg|\; t^2+t \leq n_{i+1}, \ldots, t^2+t \leq n_r,\; \bh \in H \Bigg\} \\
    & = \bigcup_{0 \leq n_1 \leq \cdots \leq n_i \leq t^2 + t} \Bigg\{\left(p^{\ell n_1} \ba_1 + \cdots + p^{\ell n_i} \ba_i\right) + \sum_{j = i+1}^{r} p^{\ell n'_j} \cdot \left( p^{\ell(t^2 + t)} \ba_j \right) + \bh \;\Bigg|\; \\
    & \hspace{10cm} n'_{i+1}, \ldots, n'_r \in \N,\; \bh \in H \Bigg\} \\
    & = \bigcup_{0 \leq n_1 \leq \cdots \leq n_i \leq t^2 + t} S\left(\ell; p^{\ell n_1} \ba_1 + \cdots + p^{\ell n_i} \ba_i, p^{\ell(t^2 + t)} \ba_{i+1}, \ldots, p^{\ell(t^2 + t)} \ba_r; H\right).
    \end{align*}
    Since $\mZ(\alpha) \supseteq \left(\widetilde{\mT}_0 \cup \cdots \cup \widetilde{\mT}_r \right) \supseteq \left(\mT_0 \cup \cdots \cup \mT_r \right) = \mT$, and $\mZ(\alpha)$ is a finite union of different $\mT$'s, we obtain that $\mZ(\alpha)$ is a finite union of different $\widetilde{\mT} \coloneqq \widetilde{\mT}_0 \cup \cdots \cup \widetilde{\mT}_r$.
    Since each $\widetilde{\mT}_i$ is a finite union of $p$-succinct sets, we conclude that $\mZ(\alpha)$ is also a finite union of $p$-succinct sets, hence $p$-normal.
\end{proof}

\section{Linear-exponential Diophantine equations to S-unit equation}\label{sec:ltos}
In this section we reduce linear-exponential Diophantine equations to S-unit equations:
\begin{prop}\label{prop:lintoSunit}
    Let $T = p_1^{e_1} p_2^{e_2} \cdots p_k^{e_k}$ be as in Theorem~\ref{thm:mainequiv}.
    Deciding whether a system of linear-exponential Diophantine equations (Equations~\eqref{eq:linearpower}) admits a solution reduces to deciding whether an S-unit equation in a $\ZT[X_1^{\pm}, \ldots, X_N^{\pm}]$-module (Equation~\eqref{eq:Sunit}) admits a solution.
\end{prop}

The main idea is as follows.
Suppose we are given a system of the form~\eqref{eq:linearpower}:
\begin{align*}
        c_{1,1} \cdot q_1^{n_1} + \cdots + c_{1,d} \cdot q_d^{n_d} + c_{1,d+1} \cdot z_{d+1} + \cdots + c_{1, D} \cdot z_{D} & = b_1, \nonumber \\
        & \vdots \nonumber \\
        c_{L,1} \cdot q_1^{n_1} + \cdots + c_{L,d} \cdot q_d^{n_d} + c_{L,d+1} \cdot z_{d+1} + \cdots + c_{L, D} \cdot z_{D} & = b_L,
\end{align*}
with $1 \leq d \leq D$, $q_1, \ldots, q_d \in \{p_1, \ldots, p_k\}$.
We will construct a finitely presented $\ZT[X_1^{\pm}, \ldots, X_D^{\pm}]$-module $\mM$ and its elements $m_0, m_1, \ldots, m_K$, satisfying the following desired property.

\textbf{Desired property:}
Let $(z_{11}, \ldots, z_{1D}) \in \Z^{D}$, the equation
\begin{equation}\label{eq:Sunitidea}
X_1^{z_{11}} X_2^{z_{12}} \cdots X_D^{z_{1D}} \cdot m_1 + \cdots + X_1^{z_{K1}} X_2^{z_{K2}} \cdots X_D^{z_{KD}} \cdot m_K = m_0
\end{equation}
can be satisfied for some $(z_{21}, \ldots, z_{2D}, \ldots, z_{K1}, \ldots,  z_{KD}) \in \Z^{(K-1)D}$, if and only if
\begin{align}\label{eq:mixed}
        & c_{1,1} \cdot z_{11} + \cdots + c_{1, D} \cdot z_{1D} = b_1, \nonumber\\
        & \hspace{3.5cm} \vdots \\
        & c_{L,1} \cdot z_{11} + \cdots + c_{L, D} \cdot z_{1D} = b_L, \nonumber\\
        & z_{11} \in q_1^{\N}, z_{12} \in q_2^{\N}, \ldots, z_{1d} \in q_d^{\N}. \nonumber
\end{align}

Here, $p^{\N}$ denotes the set $\{p^n \mid n \in \N\}$.
This desired property will allow us to reduce solving the system of linear-exponential Diophantine equations~\eqref{eq:mixed} to solving the S-unit equation~\eqref{eq:Sunitidea}.
Our strategy is to, for each equation in~\eqref{eq:mixed}, construct an S-unit equation in some module $\mM$.
We then combine these S-unit equations (in different modules $\mM$) into a single one by taking the direct product of these modules.

First, let us take care of the linear equations in~\eqref{eq:mixed}.

\begin{lem}\label{lem:linear}
    Let $c_1, \ldots, c_D, b \in \Z$. 
    Define the $\ZT[X_1^{\pm}, \ldots, X_D^{\pm}]$-module
    \begin{equation}\label{eq:defMlin}
        \mM \coloneqq \ZT[X_1^{\pm}, \ldots, X_D^{\pm}, Y^{\pm}]/\gen{X_1 - Y^{c_1}, X_2 - Y^{c_2}, \ldots, X_D - Y^{c_D}}.
    \end{equation}
    Then $(z_{11}, \ldots, z_{1D}) \in \Z^D$ satisfies 
    \begin{equation}\label{eq:lintomono}
        X_1^{z_{11}} X_2^{z_{12}} \cdots X_D^{z_{1D}} = Y^b \qquad \text{ in } \mM,
    \end{equation}
    if and only if $c_1 \cdot z_{11} + \cdots + c_D \cdot z_{1D} = b$.
\end{lem}
\begin{proof}
    Note that $\mM$ is isomorphic to $\ZT[Y^{\pm}]$ by the map $X_i \mapsto Y^{c_i}, i = 1, \ldots, D$.
    Under this map, Equation~\eqref{eq:lintomono} becomes $Y^{c_1z_{11}} Y^{c_2z_{12}} \cdots Y^{c_D z_{1D}} = Y^b$, which is equivalent to $c_1 \cdot z_{11} + \cdots + c_D \cdot z_{1D} = b$.
\end{proof}

The $\ZT[X_1^{\pm}, \ldots, X_D^{\pm}]$-module $\mM$ defined in Equation~\eqref{eq:defMlin} might not be finitely generated.
However, both sides of Equation~\eqref{eq:lintomono} fall in the submodule generated by $1$ and $Y^b$, so we can restrict $\mM$ to this submodule.
Then $\mM$ becomes finitely generated and hence finitely presented.\footnote{Computing the finite presentation of a finitely generated submodule of $\mM$ is effective, see~\cite[Theorem~2.14]{baumslag1981computable} or \cite[Theorem~2.6]{baumslag1994algorithmic}.} 

Next, let us take care of the ``exponential'' parts of~\eqref{eq:mixed}, using a similar construction to Example~\ref{exmpl:Derksen} (or~\cite[Example~1.3]{derksen2007skolem}).

\begin{lem}\label{lem:pn}
    Let $z \in \Z$ and $p$ be a prime number. Then $z \in p^{\N}$ if and only if
    \[
    X_2^z - X_1^z = 1
    \]
    in the $\F_p[X_1^{\pm}, X_2^{\pm}]$-module $\F_p[X_1^{\pm}, X_2^{\pm}]/\gen{X_2 - X_1 - 1}$.
\end{lem}
\begin{proof}
    If $z = p^n$ for some $n \in \N$, then $(X_1+1)^{p^n} = X_1^{p^n} + 1$.
    We then have $X_2^z - X_1^z = (X_1+1)^{p^n} - X_1^{p^n} = 1$ in the module $\F_p[X_1^{\pm}, X_2^{\pm}]/\gen{X_2 - X_1 - 1}$.

    For the other implication, suppose $X_2^z - X_1^z = 1$ in $\F_p[X_1^{\pm}, X_2^{\pm}]/\gen{X_2 - X_1 - 1}$.
    If $z \leq 0$, write $z = -y$ with $y \geq 0$.
    Then $X_1^y - X_2^y = X_1^y X_2^y$ in $\F_p[X_1^{\pm}, X_2^{\pm}]/\gen{X_2 - X_1 - 1}$.
    This is equivalent to 
    \begin{equation}\label{eq:X1}
        X_1^y - (X_1+1)^y = X_1^y (X_1+1)^y
    \end{equation}
    in $\F_p[X_1]$.
    The form of Equation~\eqref{eq:X1} suggest we must have $X_1^y \mid (X_1+1)^y$, so $y = 0$. But $y = 0$ is not a solution of~\eqref{eq:X1}.

    If $z > 0$, then the situation is similar to Example~\ref{exmpl:Derksen}.
    For a rigorous proof, the equation $X_2^z - X_1^z = 1$ in $\F_p[X_1^{\pm}, X_2^{\pm}]/\gen{X_2 - X_1 - 1}$ is equivalent to 
    \[
        (X_1+1)^z = X_1^z + 1
    \]
    in $\F_p[X_1]$.
    Write $z = p^n + a$ with $0 \leq a < p^n,\; n \in \N$.
    Then
    \[
    (X_1 + 1)^{z} = (X_1 + 1)^{p^n} (X_1 + 1)^a = (X_1^{p^n} + 1) (X_1 + 1)^a = X_1^{p^n} (X_1 + 1)^a + (X_1 + 1)^a.
    \]
    Every monomial appearing in $X_1^{p^n} (X_1 + 1)^a$ has degree larger than every monomial appearing $(X_1 + 1)^a$. But $(X_1 + 1)^{z} = X_1^{z} + 1$ has only two monomials.
    So both $X_1^{p^n} (X_1 + 1)^a$ and $(X_1 + 1)^a$ must be monomials, meaning $a = 0$.
    Therefore $z = p^n$, and we conclude that $z \in p^{\N}$.   
\end{proof}

We then need to express the equation in Lemma~\ref{lem:pn} as a system of ``full'' S-unit equations:

\begin{lem}\label{lem:pnsys}
    Let $p$ be a prime number and let $z_{1i} \in \Z$ for some $i \in \{1, \ldots, D\}$.
    One can construct a system of S-unit equations~\eqref{eq:Sunitidea} with $K=2$, in finitely presented $\F_p[X_1^{\pm}, \ldots, X_D^{\pm}]$-modules, such that $z_{1i} \in p^{\N}$ if and only if it extends to a solution $(z_{11}, \ldots, z_{1D}, z_{21}, \ldots, z_{2D}) \in \Z^{2D}$ for this system.
\end{lem}
\begin{proof}
    Without loss of generality suppose $i = 1$.
    We claim that $z_{11} \in p^{\N}$ if and only if it extends to a solution for the following system
    \begin{equation}\label{eq:syspn}
    \begin{cases}
        & X_1^{z_{11}} X_2^{z_{12}} \cdots X_D^{z_{1D}} \cdot (-1) + X_1^{z_{21}} X_2^{z_{22}} \cdots X_D^{z_{2D}} \cdot 1 = 1 \\ & \hspace{5cm} \text{ in }\; \F_p[X_1^{\pm}, \ldots, X_D^{\pm}]/\gen{X_2 - X_1 - 1, X_3 - 1, \ldots, X_D - 1} \\
        & X_1^{z_{11}} X_2^{z_{12}} \cdots X_D^{z_{1D}} \cdot 0 + X_1^{z_{21}} X_2^{z_{22}} \cdots X_D^{z_{2D}} \cdot 1 = 1 \\ & \hspace{5cm} \text{ in }\; \F_p[X_1^{\pm}, \ldots, X_D^{\pm}]/\gen{X_2 - 1, X_3 - 1, \ldots, X_D - 1} \\
        & X_1^{z_{11}} X_2^{z_{12}} \cdots X_D^{z_{1D}} \cdot 1 + X_1^{z_{21}} X_2^{z_{22}} \cdots X_D^{z_{2D}} \cdot 0 = 1 \\ & \hspace{5cm} \text{ in }\; \F_p[X_1^{\pm}, \ldots, X_D^{\pm}]/\gen{X_1 - 1, X_3 - 1, \ldots, X_D - 1} \\
        & X_1^{z_{11}} X_2^{z_{12}} \cdots X_D^{z_{1D}} \cdot (-1) + X_1^{z_{21}} X_2^{z_{22}} \cdots X_D^{z_{2D}} \cdot 1 = 0 \\ & \hspace{5cm} \text{ in }\; \F_p[X_1^{\pm}, \ldots, X_D^{\pm}]/\gen{X_2 - X_1, X_3 - 1, \ldots, X_D - 1}. \\
    \end{cases}
    \end{equation}
    
    Note that the elements $X_3 - 1, \ldots, X_D - 1$ in the quotient means that $X_3 = 1, \ldots, X_D = 1$ in the respective modules.
    Thus, the second equation of~\eqref{eq:syspn} is equivalent to $X_1^{z_{21}} = 1$, which means $z_{21} = 0$.
    The third equation of~\eqref{eq:syspn} is equivalent to $X_2^{z_{12}} = 1$, which means $z_{12} = 0$.
    Then, the fourth equation of~\eqref{eq:syspn} becomes $- X_1^{z_{11}} + X_2^{z_{22}} = 0$, in $\F_p[X_1^{\pm}, \ldots, X_D^{\pm}]/\gen{X_2 - X_1, X_3 - 1, \ldots, X_D - 1}$ this yields $z_{11} = z_{22}$.
    Finally, putting $z_{12} = z_{21} = 0,\; z_{11} = z_{22}$, in the first equation of~\eqref{eq:syspn} yields $- X_1^{z_{11}} + X_2^{z_{11}} = 1$ in $\F_p[X_1^{\pm}, X_2^{\pm}]/\gen{X_2 - X_1 - 1}$.
    Using Lemma~\ref{lem:pn}, this is equivalent to $z_{11} = z_{22} \in p^{\N}$.
    Thus, the system~\eqref{eq:syspn} has a solution if and only if $z_{11} \in p^{\N}$.
\end{proof}

We are now ready to complete the proof of Proposition~\ref{prop:lintoSunit}, and consequently, of Theorem~\ref{thm:mainequiv}.
\begin{proof}[Proof of Proposition~\ref{prop:lintoSunit}]
    Suppose we are given a system of the form~\eqref{eq:linearpower}, which we rewrite as
    \begin{align}\label{eq:mixed2}
        & c_{1,1} \cdot z_{11} + \cdots  + c_{1, D} \cdot z_{1D} = b_1, \nonumber\\
        & \hspace{3.5cm} \vdots \\
        & c_{L,1} \cdot z_{11} + \cdots + c_{L, D} \cdot z_{1D} = b_L, \nonumber\\
        & z_{11} \in q_1^{\N}, z_{12} \in q_2^{\N}, \ldots, z_{1d} \in q_d^{\N}. \nonumber
    \end{align}
    Note that Equation~\eqref{eq:lintomono} in Lemma~\ref{lem:linear} can also be written as
    \[
    X_1^{z_{11}} \cdots X_D^{z_{1D}} \cdot 1 + X_1^{z_{21}} \cdots X_D^{z_{2D}} \cdot 0 = Y^b \qquad \text{ in } \mM
    \]
    over the variables $(z_{11}, \ldots, z_{1D}, z_{21}, \ldots, z_{2D}) \in \Z^{2D}$. And we can restrict the module $\mM$ to its finitely presented submodule $\gen{1, Y^b}$.
    
    Therefore by Lemma~\ref{lem:linear} and Lemma~\ref{lem:pnsys}, we can construct a system of S-unit equations
    \begin{alignat}{2}\label{eq:constructSunit}
        & X_1^{z_{11}} \cdots X_D^{z_{1D}} \cdot m_{11} + X_1^{z_{21}} \cdots X_D^{z_{2D}} \cdot m_{12} = m_{10} && \qquad \text{ in } \mM_1, \nonumber\\
        & \hspace{6cm} \vdots && \\
         & X_1^{z_{11}} \cdots X_D^{z_{1D}} \cdot m_{s 1} + X_1^{z_{21}} \cdots X_D^{z_{2D}} \cdot m_{s2} = m_{s0} && \qquad \text{ in } \mM_s, \nonumber
    \end{alignat}
     where each $\mM_i$ is a finitely presented $\F_{q_i}[X_1^{\pm}, \ldots, X_D^{\pm}]$-module for some $q_i \in \{p_1, \ldots, p_k\}$, with the following property.
    For a tuple $(z_{11}, \ldots, z_{1D}) \in \Z^D$, it satisfies the system~\eqref{eq:mixed2} if and only if it extends to a solution $(z_{11}, \ldots, z_{1D}, z_{21}, \ldots, z_{2D}) \in \Z^{2D}$ for the system~\eqref{eq:constructSunit}.

    Since $T = p_1^{e_1} p_2^{e_2} \cdots p_k^{e_k}$ and $q_i \in \{p_1, \ldots, p_k\}$, each $\mM_i$ is also a finitely presented module over the ring $\ZT[X_1^{\pm}, \ldots, X_D^{\pm}]$.
    Let $\mM$ denote the direct product $\mM_1 \times \mM_2 \times \cdots \times \mM_s$, it is also a finitely presented $\ZT[X_1^{\pm}, \ldots, X_D^{\pm}]$-module.
    Thus, the system of equations~\eqref{eq:constructSunit} can be written as a single equation
    \[
    X_1^{z_{11}} \cdots X_D^{z_{1D}} \cdot (m_{11}, \ldots, m_{s1}) + X_1^{z_{21}} \cdots X_D^{z_{2D}} \cdot (m_{12}, \ldots, m_{s2}) = (m_{10}, \ldots, m_{s0})
    \]
    in the $\ZT[X_1^{\pm}, \ldots, X_D^{\pm}]$-module $\mM$.
    The concludes the proof.
\end{proof}

\begin{proof}[Proof of Theorem~\ref{thm:mainequiv}]
    Theorem~\ref{thm:mainequiv} follows directly from Corollary~\ref{cor:intersection} and Proposition~\ref{prop:lintoSunit}.
\end{proof}

\bibliography{torsion}

\appendix
\section{Intersection of $p$-normal sets is $p$-normal}\label{app:internormal}
In this appendix we provide a self-contained proof of Proposition~\ref{prop:internormal}:

\propinternormal*

\begin{defn}\label{def:rectangular}
    A subset $N \subseteq \N^R$ is called a \emph{rectangular coset} if it is of the form $\bepsilon_0 + \N \bepsilon_1 + \cdots + \N \bepsilon_r$ for some $r \in \N$, where each $\bepsilon_i, i = 1, \ldots, r$ is of the form $\bepsilon_i = (\epsilon_{i1}, \ldots, \epsilon_{iR})$,
    \[
    \epsilon_{ij} =
    \begin{cases}
        c_i \quad j \in S_i, \\
        0 \quad j \notin S_i,
    \end{cases}
    \]
    with $c_1, \ldots, c_r \in \N$ and $S_1, \ldots, S_r$ are pairwise disjoint subsets of $\{1, \ldots, R\}$.
\end{defn}

For example, $\{(2a, 2a, 3b, 3b+1, 7) \mid a, b \in \N\}$ is a rectangular coset in $\N^5$, with $\bepsilon_0 = (0, 0, 0, 1, 7)$, $\bepsilon_1 = (2, 2, 0, 0, 0)$ and $\bepsilon_2 = (0, 0, 3, 3, 0)$.

\begin{lem}\label{lem:rectangular}
    Let $q \geq 2$ be an integer, $G$ be a subgroup of $\Z^{KN}$, and $\ba_0, \ba_1, \ldots, \ba_R \in \Z^{KN}$.
    Then the set of $(k_1, \ldots, k_R) \in \N^R$ satisfying
    \begin{equation}\label{eq:recinG}
        \ba_0 + q^{k_1} \ba_1 + \cdots + q^{k_R} \ba_R \in G
    \end{equation}
    is (effectively) a finite union of rectangular cosets. 
\end{lem}
\begin{proof}
    Let $\bg_1, \ldots, \bg_s$ be a $\Z$-basis of $G$.
    Define the following $\Q$-linear subspace of $\Q^{KN}$:
    \[
    \oG \coloneqq \Q \bg_1 + \cdots \Q \bg_s.
    \]
    Then the quotient $\Q^{KN}/\oG \cong \Q^{KN - s}$ is again a $\Q$-linear space.
    Let $\|\cdot\|$ be any norm on $\Q^{KN}/\oG \cong \Q^{KN - s}$.
    
    We use induction on $R$.
    Consider the case $R = 1$.
    If $\ba_1 \in \oG$, then there exists $t \in \N$ such that $t \ba_1 \in G$.
    Let $c > b \geq 1$ be two different integers such that $q^c \equiv q^b \mod t$, then for all $k \geq c$ we have $\ba_0 + q^{k} \ba_1 \in G \iff \ba_0 + q^{k - (c - b)} \ba_1 \in G$.
    Therefore, the set of $k_1 \in \N$ satisfying $\ba_0 + q^{k_1} \ba_1 \in G$ is a finite subset of $\{0, 1, \ldots, b\}$ plus a finite union of arithmetic progressions $\{i + (c- b) \N\}$ with $b \leq i < c$.
    This is a finite union of rectangular cosets.

    If $\ba_1 \notin \oG$, then $\|\ba_1\| > 0$. Therefore if $\ba_0 + q^{k_1} \ba_1 \in G$ we must have $q^{k_1} \leq \frac{\|\ba_0\|}{\|\ba_1\|}$.
    Therefore $k_1$ is bounded, so the solution set is finite.

    For the induction step, consider two cases.
    Let $\Lambda$ denote the solution set of Equation~\eqref{eq:recinG}.
    \begin{enumerate}[wide, label = \textbf{Case~\arabic*:}]
        \item 
    If there is some $i \in \{1, \ldots, R\}$ such that $\ba_i \in \oG$, then there exists $t \in \N$ such that $t \ba_i \in G$. Let $c > b \geq 1$ be two integers such that $q^c \equiv q^b \mod t$, then for all $k \geq c$ we have $\ba_0 + q^{k_1} \ba_1 + \cdots + q^{k_i} \ba_i + \cdots + q^{k_R} \ba_R \in G \iff \ba_0 + q^{k_1} \ba_1 + \cdots + q^{k_i - (c - b)} \ba_i + \cdots + q^{k_R} \ba_R \in G$.
    Therefore, the solution set $\Lambda$ can be decomposed into a finite union
    \[
    \Lambda = \bigcup_{r = 0}^{b-1} \Lambda_r \cup \bigcup_{r = b}^{c-1} \Lambda_r,
    \]
    where for $r = 0, \ldots, b-1$,
    \[
    \Lambda_r = \left\{(k_1, \ldots, k_{i-1}, r, k_{i+1}, \ldots, k_R) \;\middle|\; \ba_0 + q^{k_1} \ba_1 + \cdots + q^{r} \ba_i + \cdots + q^{k_R} \ba_R \in G \right\},
    \]
    and for $r = b, \ldots, c-1$,
    \[
    \Lambda_r = \left\{(k_1, \ldots, k_{i-1}, r + n (c-b), k_{i+1}, \ldots, k_R) \;\middle|\; n \in \N, \ba_0 + q^{k_1} \ba_1 + \cdots + q^{r} \ba_i + \cdots + q^{k_R} \in G \right\}.
    \]
    By the induction hypothesis, for each $r = 0, \ldots, b-1, b, \ldots, c-1$, the set
    \[
    \left\{(k_1, \ldots, k_{i-1}, k_{i+1}, \ldots, k_R) \;\middle|\; \ba_0 + q^{k_1} \ba_1 + \cdots + q^{r} \ba_i + \cdots + q^{k_R} \in G\right\}
    \]
    is a finite union of rectangular cosets of $\N^{R-1}$. 
    Therefore each $\Lambda_r$ is also a finite union of rectangular cosets of $\N^{R}$.

    \item 
    If $\ba_i \notin \oG$ for all $i \in \{1, \ldots, R\}$, that is, $\|\ba_i\| > 0$ for all $i$.
    Choose $D \in \N$ such that
    \begin{equation}\label{eq:Dbig}
    q^D \|\ba_i\| > \sum_{j \neq i} \|\ba_i\|
    \end{equation}
    for all $i = 1, \ldots, R$.
    
    Define $k_0 \coloneqq 0$.
    We claim that for every $(k_1, \ldots, k_R) \in \Lambda$ there exist distinct indices $i, j \in \{0, 1, \ldots, R\}$ such that $|k_j - k_i| \leq D$.
    Take $i$ such that $k_i$ is maximal. Suppose on the contrary that $k_i > D + k_j$ for all $j \neq i$. Then by $\sum_{j = 0}^R q^{k_j} \ba_j = 0$ we have
    \[
    q^{k_i}\|\ba_i\| =  \left\| - \sum_{j \neq i}^R q^{k_j} \ba_j \right\| \leq \sum_{j \neq i}^R q^{k_j} \|\ba_j\| \leq \sum_{j \neq i}^R q^{k_i - D} \|\ba_j\|.
    \]
    This contradicts \eqref{eq:Dbig}.
    Therefore $|k_j - k_i| \leq D$ for some $j \neq i$.

    We can thus write
    \[
    \Lambda = \bigcup_{0 \leq i < j \leq R} \bigcup_{r = -d}^d \Lambda_{i, j, r},
    \]
    where
    \begin{multline*}
    \Lambda_{i, j, r} 
    = \{(k_1, \ldots, k_R) \in \Lambda \mid k_j = k_i + r \} \\
    = \left\{ (k_1, \ldots, k_i, \ldots, k_j = k_i + r, \ldots,  k_R) \;\middle|\; \ba_0 + \cdots + q^{k_i} \ba_i + \cdots + q^{k_i+r} \ba_j + \cdots + q^{k_R} \ba_R \in G \right\}.
    \end{multline*}
    Apply the induction hypothesis with $\ba_i$ as $\ba_i + q^r \ba_j$, we conclude that $\Lambda_{i, j, r}$ is a finite union of rectangular cosets.
    Therefore, $\Lambda$ is a finite union of rectangular cosets.
    \end{enumerate}
\end{proof}

\begin{proof}[Proof of Proposition~\ref{prop:internormal}]
    Since a $p$-normal set is a finite union of $p$-succinct sets, it suffices to show that the intersection of two $p$-succinct sets is effectively $p$-normal.
    Let
    \[
    S = \left\{\ba_0 + p^{\ell k_1} \ba_1 + \cdots + p^{\ell k_r} \ba_r + \bh \;\middle|\; k_1, k_2, \ldots, k_r \in \N, \bh \in H \right\}
    \]
    and
    \[
    S' = \left\{\ba'_0 + p^{\ell' k_1} \ba'_1 + \cdots + p^{\ell' k_{r'}} \ba_{r'} + \bh \;\middle|\; k_1, k_2, \ldots, k_{r'} \in \N, \bh \in H' \right\}
    \]
    be two $p$-succinct sets.
    Let $L$ be a common multiplier of $\ell$ and $\ell'$.
    Then $S$ can be written as a finite union of $p$-succinct sets
    \begin{align*}
    S & = \bigcup_{i_1 = 0}^{\frac{L}{\ell} - 1} \cdots \bigcup_{i_r = 0}^{\frac{L}{\ell} - 1}
    \left\{ \ba_0 + p^{L k_1} (p^{\ell i_1}\ba_1) + \cdots + p^{L k_r} (p^{\ell i_r}\ba_r) + \bh \;\middle|\; k_1, k_2, \ldots, k_r \in \N, \bh \in H \right\}. \\
    & = \bigcup_{i_1 = 0}^{\frac{L}{\ell} - 1} \cdots \bigcup_{i_r = 0}^{\frac{L}{\ell} - 1} S\left(L; \ba_0, p^{\ell i_1}\ba_1, \ldots, p^{\ell i_r}\ba_r ; H \right).
    \end{align*}
    Therefore we can without loss of generality replace $\ell$ with $L$.
    Similarly we can without loss of generality replace $\ell'$ with $L$.
    Therefore from now on we suppose $\ell = \ell'$.
    
    Let $U = \{u_1, \ldots, u_k\}$ be a $\Z$-basis for $H \cap H'$. Extend $U$ to a maximal $\Z$-independent subset $\{u_1, \ldots, u_k, v_1, \ldots, v_m\}$ of $H$. That is, $\sum_{i = 1}^k \Z u_i + \sum_{i = 1}^m \Z v_i$ is a finite index subgroup of $H$.
    Similarly, extend $U$ to a maximal $\Z$-independent subset $\{u_1, \ldots, u_k, v'_1, \ldots, v'_{m'}\}$ of $H'$.
    Note that $u_1, \ldots, u_k, v_1, \ldots, v_m, v'_1, \ldots, v'_{m'}$ are $\Z$-independent. Indeed, suppose $\sum_{i = 1}^k y_i u_i + \sum_{i = 1}^m z_i v_i + \sum_{i = 1}^{m'} z'_i v'_i = 0$ for some $y_i, z_i, z'_i \in \Z$, then $H' \ni \sum_{i = 1}^{m'} z'_i v'_i = - (\sum_{i = 1}^k y_i e_i + \sum_{i = 1}^m z_i v_i) \in H$.
    Therefore $\sum_{i = 1}^m z_i v_i \in H \cap H'$, so $z_i = 0$ for all $i$. Similarly $z'_i = 0$ for all $i$. Consequently $\sum_{i = 1}^k y_i u_i = 0$, so $y_i = 0$ for all $i$.

    Let $h_1, \ldots, h_s \in H$ be the representatives of $H/(\sum_{i = 1}^k \Z u_i + \sum_{i = 1}^m \Z v_i)$, and let $h'_1, \ldots, h'_{s'} \in H'$ be the representatives of $H'/(\sum_{i = 1}^k \Z u_i + \sum_{i = 1}^{m'} \Z v'_i)$. 
    Let
    \[
    \tS \coloneqq \ba_0 + p^{\ell \N} \ba_1 + \cdots + p^{\ell \N} \ba_r + \sum_{i = 1}^k \Z u_i + \sum_{i = 1}^m \Z v_i
    \]
    and
    \[
    \tS' \coloneqq \ba'_0 + p^{\ell \N} \ba'_1 + \cdots + p^{\ell \N} \ba'_{r'} + \sum_{i = 1}^k \Z u_i + \sum_{i = 1}^{m'} \Z v'_i.
    \]
    Then $S \cap S' = \bigcup_{i = 1}^s \bigcup_{j = 1}^{s'} \big( (h_i + \tS) \cap (h'_j + \tS') \big)$.
    Therefore it suffices to show that each $(h_i + \tS) \cap (h'_j + \tS')$ is $p$-normal. By replacing $\ba_0$ with $\ba_0 + h_i$ and $\ba'_0$ with $\ba'_0 + h'_j$ we can without loss of generality suppose $h_i = h'_j = 0$ and show $\tS \cap \tS'$ is $p$-normal.
    
    We then extend the set $\{u_1, \ldots, u_k, v_1, \ldots, v_m, v'_1, \ldots, v'_{m'}\}$ to a maximal $\Z$-independent subset 
    \[
    \{u_1, \ldots, u_k, v_1, \ldots, v_m, v'_1, \ldots, v'_{m'}, w_1, \ldots, w_n\}
    \]
    of $\Z^{KN}$.
    Denote by $U, V, V', W \subseteq \Q^{KN}$ respectively the $\Q$-linear spaces generated by $\{u_1, \ldots, u_k\}$, $\{v_1, \ldots, v_m\}$, $\{v'_1, \ldots, v'_{m'}\}$, $\{w_1, \ldots, w_n\}$.
    Then $U + V + V' + W = \Q^{KN}$.
    For any $x \in \Q^{KN}$, we can uniquely write $x = u + v + v' + w$ with $u \in U, v \in V, v' \in V', w \in W$; in this case we define $\pi_{U+V}(x) \coloneqq u+v$ and $\pi_{V'+W}(x) \coloneqq v' + w$.
    
    Let $\Lambda$ denote the set of solutions $(k_1, \ldots, k_r, k'_1, \ldots, k'_{r'}) \in \N^{r + r'}$ to
    \begin{equation}\label{eq:defLambda}
    \Big(\ba_0 + p^{\ell k_1} \ba_1 + \cdots + p^{\ell k_r} \ba_r\Big) - \Big(\ba'_0 + p^{\ell k'_1} \ba'_1 + \cdots + p^{\ell k'_{r'}} \ba_{r'}\Big) \in \sum_{i = 1}^k \Z u_i + \sum_{i = 1}^m \Z v_i + \sum_{i = 1}^m \Z v'_i.
    \end{equation}
    By Lemma~\ref{lem:rectangular}, we know that $\Lambda$ is a finite union of rectangular cosets.
    
    Consider the set
    \begin{multline*}
    T \coloneqq \\
    \Big\{\pi_{V' + W}(\ba_0) + p^{\ell k_1} \pi_{V' + W}(\ba_1) + \cdots + p^{\ell k_r} \pi_{V' + W}(\ba_r) + \pi_{U + V}(\ba'_0) + p^{\ell k'_1} \pi_{U + V}(\ba'_1) + \cdots + p^{\ell k'_r} \pi_{U + V}(\ba'_r) \\
     \;\Big|\; (k_1, \ldots, k'_{r'}) \in \Lambda \Big\} + \sum_{i = 1}^k \Z u_i.
    \end{multline*}
    We claim that $T$ is $p$-normal and $\tS \cap \tS' = T$, this would show that $\tS \cap \tS'$ is $p$-normal and conclude the proof.

    \begin{enumerate}[wide, label = \arabic*.]
    \item \textbf{First we show $T$ is $p$-normal.}
    Recall that a finite union of $p$-normal sets is still $p$-normal, and $\Lambda$ is a finite union of rectangular cosets. Therefore it suffices to show that $T$ is $p$-normal when $\Lambda$ is a single rectangular coset $\{\bepsilon_0 + n_1 \bepsilon_1 + \cdots + n_s \bepsilon_s \mid n_1, \ldots, n_s \in \N \}$.
    Replacing $(k_1, \ldots, k_r, k'_1, \ldots, k'_{r'})$ with $\bepsilon_0 + n_1 \bepsilon_1 + \cdots + n_s \bepsilon_s$, we can rewrite
    \[
    \pi_{V' + W}(\ba_0) + p^{\ell k_1} \pi_{V' + W}(\ba_1) + \cdots + p^{\ell k_r} \pi_{V' + W}(\ba_r) + \pi_{U + V}(\ba'_0) + p^{\ell k'_1} \pi_{U + V}(\ba'_1) + \cdots + p^{\ell k'_r} \pi_{U + V}(\ba'_r)
    \]
    as
    \[
    \widetilde\ba_0 + p^{\ell c_1 n_1} \widetilde\ba_1 + \cdots + p^{\ell c_s n_s} \widetilde\ba_s
    \]
    for some $\widetilde\ba_0, \ldots, \widetilde\ba_s \in \Q^{KN}$.
    Here $c_1, \ldots, c_s \in \N$ are as in Definition~\ref{def:rectangular}.
    Let $C$ be a common multiplier of $c_1, \ldots, c_s$, then $\{\widetilde\ba_0 + p^{\ell c_1 n_1} \widetilde\ba_1 + \cdots + p^{\ell c_s n_s} \widetilde\ba_s \mid n_1, \ldots, n_s \in \N \}$ can be written as a finite union of sets of the form
    \[
    \{\widetilde\ba'_0 + p^{\ell C n_1} \widetilde\ba'_1 + \cdots + p^{\ell C n_s} \widetilde\ba'_s \mid n_1, \ldots, n_s \in \N \},
    \]
    which are $p$-normal.

    \item
    \textbf{Then we show $\tS \cap \tS' \subseteq T$.}
    Let 
    \[
    s \coloneqq \ba_0 + p^{\ell k_1} \ba_1 + \cdots + p^{\ell k_r} \ba_r + \sum_{i = 1}^k x_i u_i + \sum_{i = 1}^m y_i v_i = \ba'_0 + p^{\ell k'_1} \ba'_1 + \cdots + p^{\ell k'_{r'}} \ba_{r'} + \sum_{i = 1}^k x'_i u_i + \sum_{i = 1}^m y'_i v'_i
    \]
    be any element of $\tS \cap \tS'$.
    Then $(k_1, \ldots, k'_{r'}) \in \Lambda$.
    We have $\pi_{V' + W}(s) = \pi_{V' + W}(\ba_0) + p^{\ell k_1} \pi_{V' + W}(\ba_1) + \cdots + p^{\ell k_r} \pi_{V' + W}(\ba_r)$ and $\pi_{U + V}(s) \in \pi_{U + V}(\ba'_0) + p^{\ell k'_1} \pi_{U + V}(\ba'_1) + \cdots + p^{\ell k'_r} \pi_{U + V}(\ba'_r) + \sum_{i = 1}^k x'_i u_i$.
    Therefore
    \begin{align*}
    s & = \pi_{V' + W}(s) + \pi_{U + V}(s) \\
    & = \pi_{V' + W}(\ba_0) + \cdots + p^{\ell k_r} \pi_{V' + W}(\ba_r) + \pi_{U + V}(\ba'_0) + \cdots + p^{\ell k'_r} \pi_{U + V}(\ba'_r) + \sum_{i = 1}^k x'_i u_i \\
    & \in T.
    \end{align*}
    We conclude that $\tS \cap \tS' \subseteq T$.

    \item \textbf{Finally we show $T \subseteq \tS \cap \tS'$.}
    Let
    \[
    t \coloneqq \pi_{V' + W}(\ba_0) + \cdots + p^{\ell k_r} \pi_{V' + W}(\ba_r) + \pi_{U + V}(\ba'_0) + \cdots + p^{\ell k'_r} \pi_{U + V}(\ba'_r) + \sum_{i = 1}^k x'_i u_i
    \]
    be any element in $T$, with $x'_1, \ldots, x'_k \in \Z$ and $(k_1, \ldots, k'_{r'}) \in \Lambda$.
    Since $(k_1, \ldots, k'_{r'}) \in \Lambda$, by the definition of $\Lambda$ we have
    \[
    \Big(\ba_0 + p^{\ell k_1} \ba_1 + \cdots + p^{\ell k_r} \ba_r\Big) - \Big(\ba'_0 + p^{\ell k'_1} \ba'_1 + \cdots + p^{\ell k'_{r'}} \ba_{r'}\Big) \in \sum_{i = 1}^k \Z u_i + \sum_{i = 1}^m \Z v_i + \sum_{i = 1}^m \Z v'_i.
    \]
    So
    \begin{align*}
        & t - \Big(\ba_0 + p^{\ell k_1} \ba_1 + \cdots + p^{\ell k_r} \ba_r\Big) \\
        = \; & \big(\pi_{V' + W}(\ba_0) - \ba_0 \big) + \cdots + p^{\ell k_r} \big(\pi_{V' + W}(\ba_r) - \ba_r \big) + \pi_{U + V}(\ba'_0) + \cdots + p^{\ell k'_r} \pi_{U + V}(\ba'_r) + \sum_{i = 1}^k x'_i u_i \\
        = \;& - \pi_{U + V}(\ba_0) - \cdots - p^{\ell k_r} \pi_{U + V}(\ba_r) + \pi_{U + V}(\ba'_0) + \cdots + p^{\ell k'_r} \pi_{U + V}(\ba'_r) + \sum_{i = 1}^k x'_i u_i \\
        = \; & \pi_{U+V} \Big( \big(\ba'_0 + p^{\ell k'_1} \ba'_1 + \cdots + p^{\ell k'_{r'}} \ba_{r'}\big) - \big(\ba_0 + p^{\ell k_1} \ba_1 + \cdots + p^{\ell k_r} \ba_r\big) \Big) + \sum_{i = 1}^k x'_i u_i \\
        \in \; & \pi_{U+V} \left( \sum_{i = 1}^k \Z u_i + \sum_{i = 1}^m \Z v_i + \sum_{i = 1}^m \Z v'_i \right) + \sum_{i = 1}^k x'_i u_i \\
        \subseteq \; & \sum_{i = 1}^k \Z u_i + \sum_{i = 1}^m \Z v_i.
    \end{align*}
    Therefore $t \in \ba_0 + p^{\ell \N} \ba_1 + \cdots + p^{\ell \N} \ba_r + \sum_{i = 1}^k \Z u_i + \sum_{i = 1}^m \Z v_i = \tS$.
    
    Similarly,
    \begin{align*}
        & t - \Big(\ba'_0 + p^{\ell k'_1} \ba'_1 + \cdots + p^{\ell k'_r} \ba'_r\Big) \\
        = \; & \pi_{V' + W}(\ba_0) + \cdots + p^{\ell k_r} \pi_{V' + W}(\ba_r) + \big(\pi_{U+V}(\ba'_0) - \ba'_0 \big) + \cdots + p^{\ell k_r} \big(\pi_{U+V}(\ba'_r) - \ba'_r \big) + \sum_{i = 1}^k x'_i u_i \\
        = \;& \pi_{V' + W}(\ba_0) + \cdots + p^{\ell k_r} \pi_{V' + W}(\ba_r) - \pi_{V' + W}(\ba'_0) - \cdots - p^{\ell k'_r} \pi_{V' + W}(\ba'_r) + \sum_{i = 1}^k x'_i u_i \\
        = \; & \pi_{V' + W} \Big( \big(\ba_0 + p^{\ell k_1} \ba_1 + \cdots + p^{\ell k_r} \ba_r\big) - \big(\ba'_0 + p^{\ell k'_1} \ba'_1 + \cdots + p^{\ell k'_{r'}} \ba_{r'}\big) \Big) + \sum_{i = 1}^k x'_i u_i \\
        \in \; & \pi_{V' + W} \left( \sum_{i = 1}^k \Z u_i + \sum_{i = 1}^m \Z v_i + \sum_{i = 1}^m \Z v'_i \right) + \sum_{i = 1}^k x'_i u_i \\
        \subseteq \; & \sum_{i = 1}^k \Z u_i + \sum_{i = 1}^m \Z v'_i.
    \end{align*}
    Therefore $t \in \ba'_0 + p^{\ell \N} \ba'_1 + \cdots + p^{\ell \N} \ba'_r + \sum_{i = 1}^k \Z u_i + \sum_{i = 1}^m \Z v'_i = \tS'$.
    This yields $T \subseteq \tS \cap \tS'$.
    \end{enumerate}
    We conclude that $\tS \cap \tS' = T$, which is $p$-normal.
\end{proof}
\end{document}